\documentclass{amsart}     
\usepackage{amssymb} 
\usepackage{amsmath,amscd,eucal,latexsym} 
\usepackage{graphicx}

\newtheorem{theorem}{Theorem}[section]
\newtheorem{lemma}[theorem]{Lemma}
\newtheorem{cor}[theorem]{Corollary}
\newtheorem{prop}[theorem]{Proposition}

\theoremstyle{definition} 
\newtheorem*{rem}{Remark}

\newtheorem*{nota}{Notation}
\newtheorem{example}[theorem]{Example}
\newtheorem{defn}[theorem]{Definition} 
\newtheorem{axiom}{Axiom} 
 
\newtheorem{hyp}{Hypothesis}
\newtheorem*{*hyp}{Temporary Hypothesis}

\DeclareMathOperator{\hh}{h}
\DeclareMathOperator{\vv}{v}
\DeclareMathOperator{\intr}{int}
\DeclareMathOperator{\fr}{Fr}
\DeclareMathOperator{\s}{s}
\DeclareMathOperator{\T}{T}
\DeclareMathOperator{\E}{E}

\DeclareMathOperator{\thh}{th}
\DeclareMathOperator{\id}{id}

\DeclareMathOperator{\card}{card}
\DeclareMathOperator{\g}{g}

\newcommand{\ra}{\rightarrow} 
\newcommand{\Ra}{\Rightarrow} 
\newcommand{\bd}{\partial} 
\newcommand{\AAA}{\mathcal{A}} 
\newcommand{\BB}{\mathcal{B}} 
\newcommand{\trb}{\bd_{\pitchfork}} 
\newcommand{\tr}{\pitchfork} 
\newcommand{\tb}{\bd_{\tau}}
\newcommand{\sm}{\smallsetminus} 
\newcommand{\FF}{\mathcal{F}} 
\newcommand{\EE}{\mathcal{E}} 
\newcommand{\CC}{\mathcal{C}} 
 
\newcommand{\QQ}{\mathcal{Q}}
\newcommand{\GG}{\mathcal{G}}
\newcommand{\II}{\mathcal{I}}
\newcommand{\LL}{\mathcal{L}} 
\newcommand{\MM}{\mathcal{M}} 
\newcommand{\NN}{\mathcal{N}} 
\newcommand{\PP}{\mathcal{P}} 
\newcommand{\RR}{\mathcal{R}} 
\newcommand{\TT}{\mathcal{T}} 
\newcommand{\UU}{\mathcal{U}}
\newcommand{\VV}{\mathcal{V}}
\newcommand{\WW}{\mathcal{W}}
\newcommand{\XX}{\mathcal{X}}
\newcommand{\YY}{\mathcal{Y}}
\newcommand{\ZZ}{\mathcal{Z}}
\newcommand{\KK}{\mathcal{K}}
\newcommand{\HH}{\mathcal{H}}
\newcommand{\JJ}{\mathcal{J}}
\newcommand{\R}{\mathbb{R}}
\newcommand{\Z}{\mathbb{Z}}

\newcommand{\D}{\mathbb{D}}

\newcommand{\A}{\mathfrak{A}}
\newcommand{\B}{\mathfrak{B}} 
\newcommand{\C}{\mathfrak{C}}
\newcommand{\G}{\mathfrak{G}}
\newcommand{\F}{\mathbb{F}}

\renewcommand{\L}{\mathfrak{L}}
\renewcommand{\S}{\mathfrak{S}}
\newcommand{\X}{\mathfrak{X}}
\renewcommand{\ss}{\subset}
\renewcommand{\SS}{\mathcal{S}}
\newcommand{\sseq}{\subseteq}
\newcommand{\0}{\emptyset}

\newcommand{\wt}{\widetilde} 
\renewcommand{\tilde}{\wt}
\newcommand{\ol}{\overline} 
\newcommand{\wh}{\widehat} 
\newcommand{\x}{\times} 
\newcommand{\CII}{C^{2}} 
\newcommand{\CI}{C^{1}} 
\newcommand{\Ci}{C^{\infty}} 
\newcommand{\CO}{C^{0}}
\newcommand{\COO}{C^{0+}} 
\newcommand{\Si}{S^{1}_{\infty}} 
\newcommand{\SI}{S^{1}} 
\newcommand{\Se}{S_{\infty}} 
 
\renewcommand{\tilde}{\wt}
\newcommand{\vs}{\vspace{.15in}}
\renewcommand{\ni}{\noindent}
\renewcommand{\phi}{\varphi}

\renewcommand{\o}{\circ}
\renewcommand{\epsilon}{\varepsilon}
\newcommand{\hra}{\hookrightarrow}

\newcommand{\lra}{\leftrightarrow}

\begin{document}

\title[Endperiodic Automorphisms]{Endperiodic Automorphisms of Surfaces and Foliations}

\author[J. Cantwell]{John Cantwell}
\address{Saint Louis University, St. Louis, MO 63103}
\email{cantwelljc@slu.edu}

\author[L. Conlon]{Lawrence Conlon}
\address{Washington University,  St. Louis, MO  63130 }
\email{conlonlawrence@icloud.com}

\author[S. Fenley]{Sergio R. Fenley}
\address{Florida State University,Tallahassee, FL 32306-4510 }
\email{fenley@math.fsu.edu}


\begin{abstract}
We   extend   the  unpublished work of M.~Handel and R.~Miller on the classification, up to isotopy, of endperiodic  automorphisms of surfaces.  We give the Handel-Miller construction of the geodesic laminations, give an axiomatic theory for pseudo-geodesic laminations, show the geodesic laminations satisfy the axioms, and prove that pseudo-geodesic laminations satisfying our axioms are ambiently isotopic to the geodesic laminations. The axiomatic approach allows us to show that the given endperiodic automorphism is isotopic to a smooth endperiodic automorphism preserving  smooth laminations ambiently isotopic to the original ones. Using the axioms, we also  prove the  ``transfer theorem'' for foliations of 3-manifolds, namely that, if two depth one foliations $\FF$ and $\FF'$ are transverse to a common one-dimensional foliation $\LL$ whose monodromy on the noncompact leaves of $\FF$ exhibits the nice dynamics of Handel-Miller theory, then $\LL$ also induces   monodromy on the noncompact leaves of $\FF'$ exhibiting the same nice dynamics.   Our theory also applies to surfaces with infinitely many ends.
\end{abstract}

\maketitle


\section{Introduction}

The Nielsen-Thurston theory of automorphisms of compact surfaces~\cite{bca, th:surfaces,ha:th,mill}  classifies the isotopy class of an  automorphism $f$ of a compact, hyperbolic surface. 
For endperiodic automorphisms of noncompact surfaces, M.~Handel and R.~Miller outlined an analogous theory (unpublished).

Both theories produce a pair of transverse geodesic laminations and a map $h$ (endperiodic in the Handel-Miller case), isotopic to $f$  and preserving the laminations.

In the compact case, there are $h$-invariant reducing circles which decompose the surface into periodic pieces and pseudo-Anosov pieces.  Similarly, in the endperiodic case, there  are reducing circles and reducing lines.  These reduce the surface into finitely many (finite or infinite) $h$-orbits of compact subsurfaces, finitely many noncompact pieces on which a power of $h$ is a translation, and finitely many noncompact ``pseudo-anosov'' pieces. (The lower case ``a'' indicates that the analogy with pseudo-Anosov automorphisms of compact surfaces is weak.) In our exposition, reduction arises late in the game when an analysis of the laminations yields the reducing curves in a very natural way.

The first nine sections of this paper treat the basics, defining endperiodicity, presenting the Handel-Miller construction of the geodesic laminations arising from an endperiodic automorphism,  giving a detailed analysis of their structure, constructing the endperiodic automorphism preserving the laminations, and analyzing its dynamics. These sections fill in roughly a thirty year gap in the literature, the first goal of this paper.  

Our second goal is to  prove two important new theorems, Theorem~\ref{HMsmooth}  and Theorem~\ref{transfer} (see below), which are critical for    applications to foliation theory.  For this, it becomes necessary to relax the condition that the laminations be geodesic. We do this in Section~\ref{uniq} where we state four axioms for the ``pseudo-geodesic'' laminations and show that these laminations are ambiently isotopic to the geodesic laminations of the Handel-Miller theory.  We note that the geodesic laminations satisfy our axioms, so we are not axiomatizing the empty set, and the isotopy theorem then shows that our axioms are complete.  The entire theory developed in the geodesic case becomes immediately available in the pseudo-geodesic case.

As in the Nielsen-Thurston theory, smoothness is a problem.  Thurston's technique of ``blowing down'' the laminations to produce a pair of transverse foliations with finitely many $p$-pronged singularities made it possible to smooth $h$ and the foliations except at the singularities.  This used a pair of projectively invariant measures to produce the smooth coordinate atlas on the complement of the singular set.  In our case, the blow-down does not yield foliations  and the projectively invariant measures may not have full support.  In our ``Smoothing Theorem'' (Theorem~\ref{HMsmooth}), having relaxed the geodesic condition on the laminations, we directly construct a pair of transverse smooth laminations preserved by an endperiodic diffeomorphism $h$ and  verify the axioms.

Handel-Miller theory has applications to  foliations analogous to the applications of Nielsen-Thurston theory to fibrations (cf.~\cite{fried}).  In smooth foliations of depth one,  the monodromy $f$ of the noncompact leaves is endperiodic and there is a smooth   representative $h$ of the isotopy class of $f$ preserving the laminations.  We will prove the fundamental ``Transfer Theorem'' (Theorem~\ref{transfer}).  By this theorem, if $\FF$ and $\FF'$ are depth one foliations on $M$, both transverse to a $1$-dimensional foliation $\LL$, and if the first return map (monodromy) induced by $\LL$ on a depth one leaf $L$ of $\FF$ preserves a pair of pseudo-geodesic  laminations satisfying the axioms, then the monodromy it induces on a depth one leaf $L'$ of $\FF'$  also preserves such a pair of laminations. For this, the  laminations on  the leaf of $\FF'$ cannot be assumed to be geodesic, even if those on the leaf of $\FF$ are, necessitating our axiomatic characterization.   The proof of the theorem  proceeds by showing that the the truth of our axioms for the monodromy of $\FF$ implies their truth for  the monodromy of $\FF'$.


\section{Endperiodic Automorphisms}\label{endp}  

We fix a temporary hypothesis,
\begin{*hyp}

\textbf{Until Section~\ref{jnctrs}, we assume that $L$ is a noncompact, connected  $n$-manifold with finitely or infinitely many ends and possibly with boundary}.

\end{*hyp}
\ni Let $f:L\to L$ be a homeomorphism. We  do not assume that $L$ is orientable nor, if it is, that $f$ is orientation preserving.

Suppose $\displaystyle L \supset V_1 \supset \ldots \supset V_n 
\supset \overline V_{n+1} \supset V_{n+1} \supset \ldots$ with the 
$V_n$ open, connected, $\bigcap_{n=1}^{\infty} V_n = \emptyset$, and 
$\overline V_n \sm V_n$ compact. Then the nested sequence of sets $\{V_n\}$ defines an 
{\it end} of $L$. 

If $\{V_n\}$ and $\{U_n\}$ define ends of $L$, then $\{V_n\}$ is said to be 
equivalent to $\{U_n\}$ if for every $n$ there exists an $m$ such 
that $V_n \supset U_m$ and for every $m$ there exists an $n$ 
such that $U_m \supset V_n$.
The equivalence classes, $e = [\{V_n\}]$, are called the \emph{ends}\label{enddefn} of $L$ and the set $\EE(L)$ of equivalence classes is called the {\it endset} of $L$.\label{endendset}

Often one gives an ``exhaustion''   $K_{0}\ss K_{1}\ss\cdots\ss K_{n}\ss\cdots\ss L$\label{exhaust} where the $K_{n}$ are compact and $\bigcup_{n=0}^{\infty} K_{n} = L$. Then for any $e\in\EE(L)$, $e=[\{U_{n}\}]$ where $U_{n}$ is an unbounded component of the complement of $K_{n}$.

Let $\TT$ be the topology on $L$, that is $\TT$ is the set of open sets in $L$. For  $V\in\TT$   let,
  $$\wh V = V\cup\{e = [\{V_n\}]\in \EE(L)\ | {\rm\ there\ exists\ an\ } n {\rm\ with\ } V\supset V_n\}.$$
Then   it is well known that  $\wh\BB = \TT\cup\{\wh V\ |\ V\in\TT\}$  is  a base for a   compact,  separable metrizable  topology on $L\cup \EE(L)$ which restricts to a totally disconnected topology on the closed set $\EE(L)$.\label{endsetL}

If $e\in\EE(L)$, we will  say that $U\ss L$ is a neighborhood of the end $e$ if $\wh U$ is a neighborhood of $e$ in the space $(L\cup\EE(L),\wh\TT)$.\label{nhbend}

Notice that $f$ induces an automorphism on the space of ends of $L$ which, by abuse, we will also denote by $f$.

\begin{defn}[$p_{e}$]\label{perend}
 An end $e$ of $L$ is \emph{periodic} of \emph{period} $p_{e} > 0$ if $f^{p_{e}}(e) = e$ and $p_{e}$ is the  least   positive integer with this property.
\end{defn}

\begin{rem}
If there are finitely many ends, every end is periodic.

\end{rem}

\begin{defn}[positive and negative  ends, $\EE_{\pm}(L)$]\label{perends}
 An end $e$ of period $p_{e}$ is a  \emph{positive end} if there is a closed, connected neighborhood $U_{e}$ of $e$ such that $L\sm U_{e}$ is connected and
 \begin{enumerate}
 \item $f^{p_{e}}(U_{e})\ss U_{e}$;\label{firstitem}
 \item $\bigcap_{n=0}^{\infty}f^{np_{e}}(U_{e})=\0$;\label{seconditem}
 \item $\fr U_{e}$ is compact.\label{thirditem}
 \end{enumerate}
 The end $e$ is a \emph{negative end} if the parallel assertions hold with $p_{e}$ replaced by $-p_{e}$. We denote the set of positive (respectively negative) ends by $\EE_{+}(L)$ (respectively $\EE_{-}(L)$).
 \end{defn}
 
 \begin{rem}\label{bdfr}
We will write $\partial M$ for the boundary  of a manifold $M$ and the symbol $\fr A$ for the topological boundary (frontier) of the subset $A$ of a topological space. 
\end{rem}

\begin{rem}\label{attrrep}
 The positive ends are the attracting ends and the negative ends are the repelling ends.
 
 \end{rem}
 
Ends may fail to be positive or negative.  For instance, on a four times punctured sphere, one easily produces a homeomorphism  $f$ that cyclically permutes three punctures,  fixes the fourth and satisfies $f^3=\id$.  These ends are   neither positive nor negative.  Similar examples can be produced  in which the ends are
nonplanar.

\begin{defn}[$f$-neighborhood of an end]\label{fnbhd}

A set $U_{e}$  as in Definition~\ref{perends} will be called an \emph{$f$-neighborhood} of $e$.

\end{defn}

\begin{defn}[endperiodic automorphism]\label{epdefn}
 The homeomorphism $f:L\to L$ is called an \emph{endperiodic automorphism} of $L$ if all periodic ends are positive or negative.
\end{defn}

\begin{figure}[ht]
\begin{center}
\begin{picture}(300,150)(-35,-140)

\rotatebox{270}{\includegraphics[width=140pt]{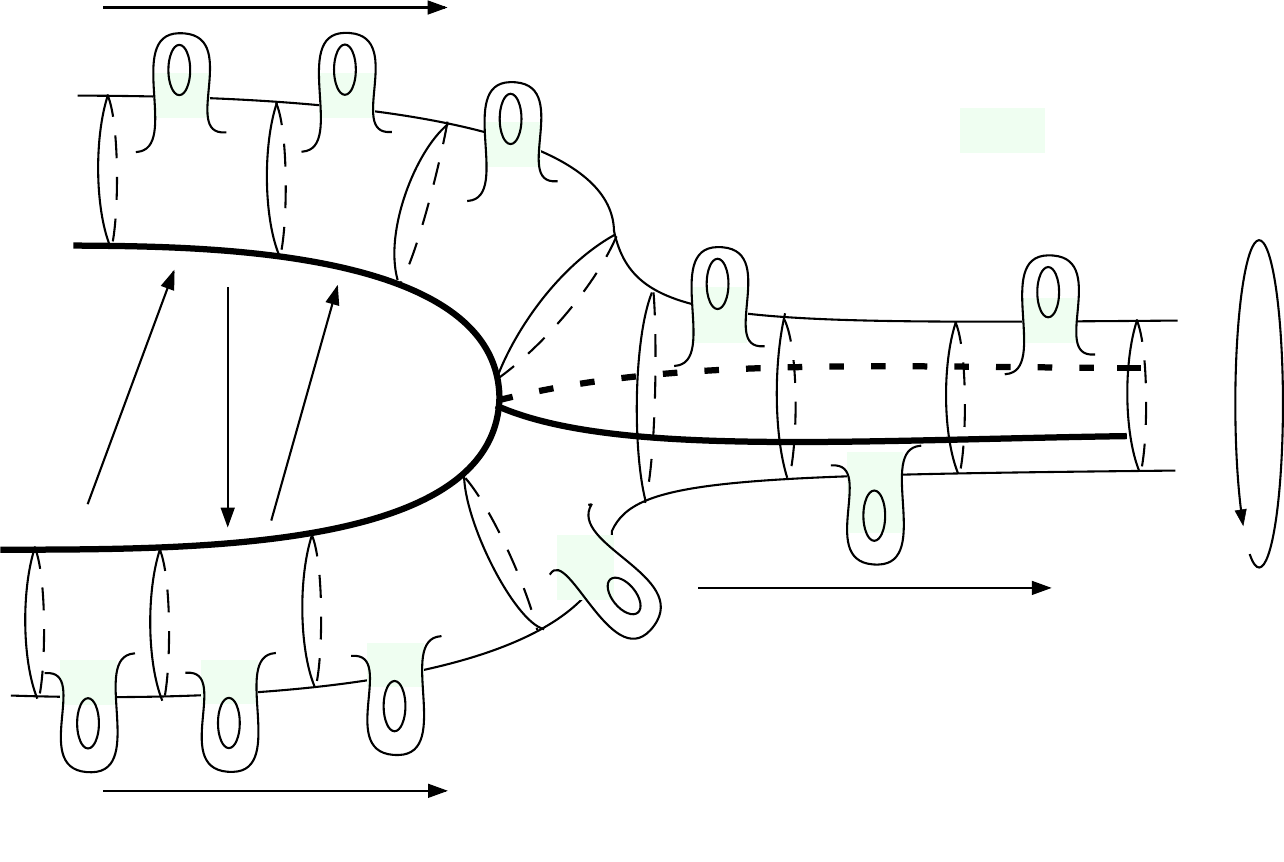}}

\put(-210,-25){$e_{1}$}
\put(-220,-100){$e_{2}$}
\put(-20,-70){$e$}

\end{picture}
\caption{An example with two negative ends}\label{TwoEnds}
\end{center}
\end{figure}

\begin{example}\label{2e}
  A simple example of an endperiodic automorphism is depicted in Figure~\ref{TwoEnds}.  Here, the ends $e_{1}$ and $e_{2}$ are periodic of period $2$ and negative, with $p_{e_{2}}=p_{e_{1}}=2$.  The end $e$ is periodic of period~$1$,  positive, with $p_{e}=1$. The circle $f$-junctures (Definition~\ref{fjunct}) separate the neighborhoods of the ends into  $f$-domains (Definition~\ref{fdom}) having negative Euler characteristic. The arrows indicate the action of $f$. Notice that all of the ``interesting'' dynamics of $f$ occurs in the compact region complementary to neighborhoods of the three ends.  This is called the ``core'' and it is only there that fixed points, periodic points, other invariant sets can occur.  Actually, this example can be constructed so that the only interesting dynamics is a single fixed point.  The two boldface curves will be explained later.
\end{example}

\begin{rem}
In this paper, we study surfaces with finitely many ends but the definition of endperiodic automorphism also makes sense for surfaces with infinitely many ends and for $n$-manifolds with $n>2$. In Section~\ref{endpinf} we show how  the theory of endperiodic automorphisms of surfaces $L$ with infinitely many ends reduces to the case of a surface $L$ with finitely many ends. 

\end{rem}

\begin{rem}
In~\cite[Section~4]{cc:tisch}, we give an infinite family of examples of endperiodic automorphisms of $1$-ended $3$-manifolds in which the end is negative. Included are Whitehead's example of a contractible open three manifold which is not $\R^{3}$~\cite{rolf,whjhc} and $3$-manifolds with nontrivial fundamental group.

\end{rem}

\begin{lemma}\label{emptint}

If $e$ and $e'$ are distinct positive or distinct negative ends of $L$ and $U_{e}$ and $U_{e'}$ are $f$-neighborhoods of $e$ and $e'$ respectively, then $U_{e}\cap U_{e'} = \0$.

\end{lemma}

\begin{proof}
Suppose $e$ and $e'$ are both positive ends. Let $V$ and $V'$ be disjoint neighborhoods of $e$ and $e'$ respectively. If there exists $x\in U_{e}\cap U_{e'}$, then for $k$ sufficiently large, $f^{kp_{e}p_{e'}}(x)$ lies in both $V$ and $V'$ which is a contradiction. The case where $e$ and $e'$ are negative ends is parallel.
\end{proof}

\begin{lemma}\label{power}
Let $f:L\to L$ be a homeomorphism. Then $f$ is endperiodic if and only if $f^{p}$ is endperiodic, for some integer $p>0$.
\end{lemma}

\begin{proof}
The ``only if'' direction is trivial. 
Assume that $f^{p}$ is endperiodic and let $e$ be an   end with period $p_{e}$. Then, $f^{pp_{e}}(e)=e$.  Since $f^{p}$ is endperiodic, it follows that $e$ is either a positive or negative end under the homeomorphism $f^{p}$. Without loss, assume that $e$ is positive. Thus, there is a closed neighborhood $U_{e}$ of $e$ such that $L\sm U_{e}$ is nonempty and connected, $f^{pp_{e}}(U_{e})\ss U_{e}$, and
$$
  \bigcap_{n=1}^{\infty}f^{npp_{e}}(U_{e})=\0.
$$
Set
$$
V_{e}=U_{e}\cap f^{p_{e}}(U_{e})\cap f^{2p_{e}}(U_{e})\cap\cdots\cap f^{(p-1)p_{e}}(U_{e}).
$$
Then,
$$
L\sm V_{e}=\bigl(L\sm U_{e}\bigr)\cup\bigl(L\sm f^{p_{e}}(U_{e})\bigr)\cup\cdots\cup \bigl(L\sm f^{(p-1)p_{e}}(U_{e})\bigr)
$$
is connected as the union of connected sets with  intersection 
$$\bigcap_{i=0}^{p-1}\bigl(L\sm f^{ip_{e}}(U_{e})\bigr)$$
which is nonempty as the finite intersection of open neighborhoods of each negative end and
$$
f^{p_{e}}(V_{e})=f^{p_{e}}(U_{e})\cap f^{2p_{e}}(U_{e})\cap\cdots\cap f^{pp_{e}}(U_{e})\ss V_{e}.
$$
Since $V_{e}$ is a neighborhood of $e$, it has exactly one noncompact component $V'_{e}$ which is a neighborhood of the end $e$. Hence $f^{p_{e}}(V'_{e})\ss V'_{e}$ and we can replace $V_{e}$ with $V'_{e}$.
Finally, $L\sm V'_{e}$ is connected and,
$$
\bigcap_{n=1}^{\infty}f^{np_{e}}(V'_{e})\ss\bigcap_{n=1}^{\infty}f^{np_{e}}(U_{e})\ss\bigcap_{n=1}^{\infty}f^{npp_{e}}(U_{e})=\0.
$$
Thus $e$ is a positive   end.  A similar argument works for negative ends, proving that $f$ is endperiodic. 
\end{proof}

\begin{lemma}\label{ne}
If $e$ is a positive \upn{(}respectively negative\upn{)} end and $U_{e}$ is an $f$-neighbor-hood of $e$, then there is an integer $p>0$, divisible by $p_{e}$, such that $f^{p}( U_{e})\ss\intr U_{e}$ \upn{(}respectively $f^{-p}( U_{e})\ss\intr U_{e}$\upn{)}.
\end{lemma}

\begin{proof}
Recall from Definition~\ref{perends} that $U_{e}$ has compact frontier. If $f^{np_{e}}(U_{e})$ meets $\fr U_{e}$ for all $n\ge 0$ then, by the compactness of $\fr U_{e}$, there exists 
$$x\in\Bigl(\bigcap_{n=0}^{\infty} f^{np_{e}}(U_{e})\Bigr)\cap\fr U_{e}$$
 contradicting $\bigcap_{n=0}^{\infty}f^{np_{e}}(U_{e})=\0$. Thus there exists an $n>0$ such that $f^{np_{e}}(U_{e})\ss   \intr U_{e}$.
\end{proof}

\begin{defn}[$f$-junctures]\label{fjunct}
For  each  $f$-neighborhood $U_{e}$ (Definition~\ref{fnbhd})  of a positive or negative end $e$, the set  $J   = \fr U_{e}$ is called an  \emph{$f$-juncture} for $e$. The $f$-juncture is \emph{positive \upn{(}respectively negative\upn{)}} if $U_{e}$ is an $f$-neighborhood of a positive (respectively negative) end. 
\end{defn}

\begin{rem}
A given $f$-juncture $J$ for a positive or negative end $e$ gives rise to a whole bi-infinite sequence $\{J_{n} = f^{n}(J)\}_{n\in\Z}$ of $f$-junctures for the ends in the $f$-cycle of $e$.
\end{rem}

\begin{rem}
Given an end $e$, there are uncountably many choices of $f$-neighborhood $U_{e}$ and $f$-juncture $J = \fr U_{e}$. In Section~\ref{definjunct}, we will pick and fix a countable set of $f$-junctures.

\end{rem}

\begin{figure}[ht]
\begin{center}
\begin{picture}(300,200)(-60,0)

\includegraphics[width=150pt]{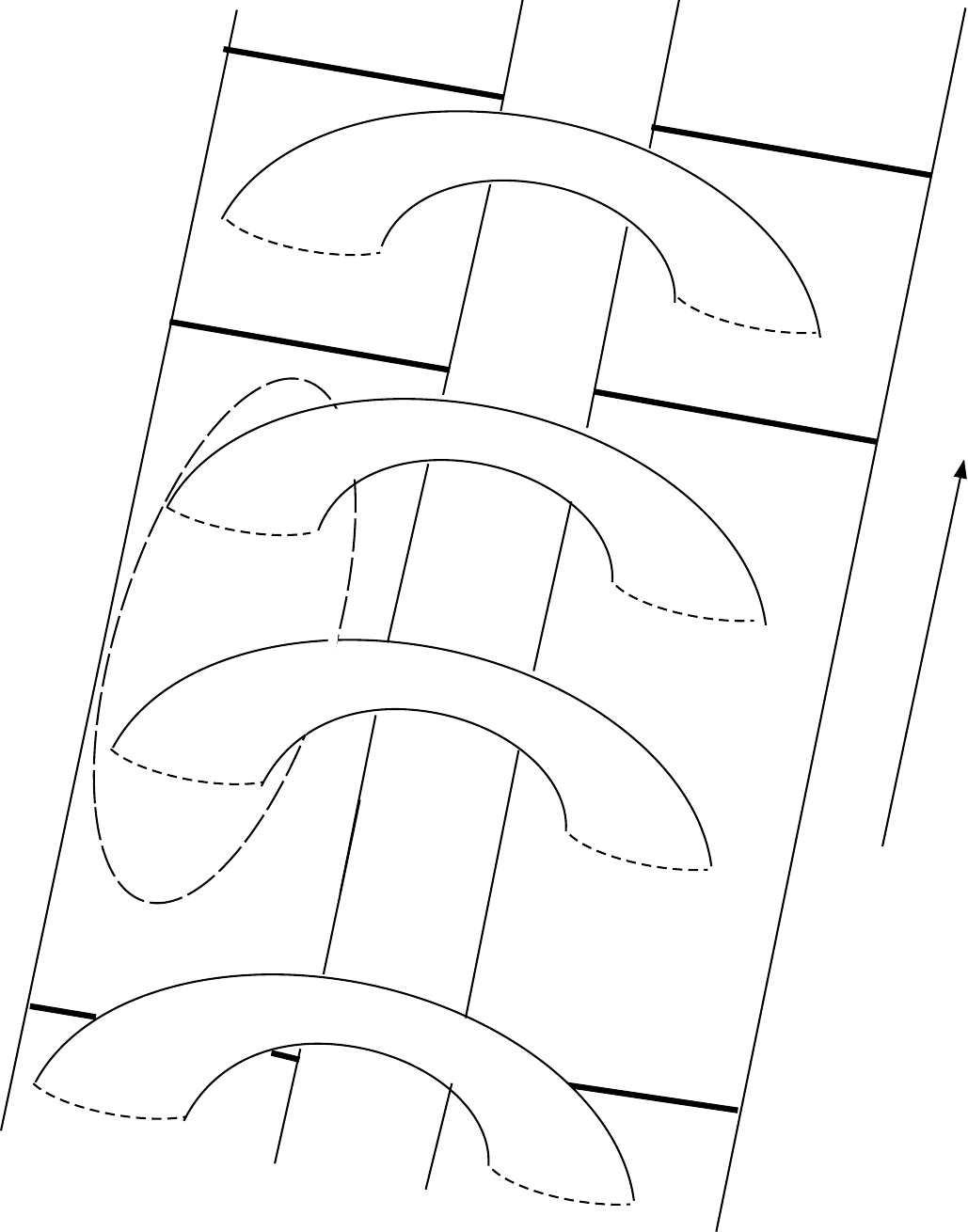}
\put(-130,45){$_{C}$}

\end{picture}
\caption{$f$ is not a translation but some components of $f$-junctures escape}\label{oneescapes}
\end{center}
\end{figure}

\begin{example}\label{itj}
We give an example to show how the images of $f$-junctures behave under iteration of $f$. Let $L$ be the surface depicted in Figure~\ref{oneescapes}, two strips connected by an infinite sequence of tubes.  The endperiodic automorphism is $f=\tau\o g$, where $g$ is the translation from the negative end to the positive end moving each handle to the next handle as indicated by the arrow and $\tau$ is a Dehn twist in the oval $C$.   The $f$-junctures are each the pair of properly embedded arcs as pictured in boldface.  Under forward iteration of $f$, the left component of an $f$-juncture in the negative end gets caught by the Dehn twist and starts stretching unboundedly into the positive end.  Note that the right component of the $f$-juncture escapes without distortion to the positive end.  A similar description holds for $f$-junctures in the positive end and their images under backward iteration of $f$.  This distortion of $f$-junctures (and of the geodesic and pseudo-geodesic junctures defined in Sections~\ref{constr} and~\ref{uniq}) is typical behavior and is what will create the stable and unstable laminations of Handel-Miller theory.
\end{example}

We replace our temporary hypothesis by the hypothesis,

\begin{hyp}\label{hyp1}
\textbf{Hereafter, unless we explicitly state otherwise, $L$ is a non-compact  surface with finite endset}.

\end{hyp}

 \subsection{Asymptotic construction of $f$-junctures}\label{jnctrs}

 We are going to give a method of constructing $f$-junctures which, being  closely related to the way that they arise in foliations, will be called an \emph{asymptotic} construction.  The $f$-junctures constructed in this way will intersect only in common components.  We will see this in Example~\ref{weird} which, in fact, is a thinly disguised  example of the asymptotic construction.

The  junctures in the following definition are either  $f$-junctures or sets of geodesic or pseudo-geodesic junctures as defined in Sections~\ref{constr} or~\ref{uniq}.

 \begin{defn}[juncture intersection property]\label{intprop}
 
 A  juncture $J$ has the \emph{juncture intersection property} if any two  junctures in the set  $\{J_{n} = f^{n}(J)\ |\ n\in\Z\}$ of junctures   intersect, if at all, in common components.  A set of junctures has the \emph{juncture intersection property} if every juncture in the set has the juncture intersection property.
 \end{defn}

Assume that $f:L\ra L$ is endperiodic and $e$ is a positive or negative end of period  $p_{e}$.  

\begin{prop}\label{comcomp}The choice of $f$-neighborhood $U_{e}$ can be made so that the $f$-juncture $J = \fr U_{e}$   is a compact, properly embedded, transversely oriented $1$-manifold satisfying the juncture intersection property.  
\end{prop}

 Proposition~\ref{comcomp} will be proven in a series of lemmas.  Let $e$ be a positive end of $L$ and set $$c=\{e_{0}=e,e_{1}=f(e),e_{2}=f^{2}(e),\dots,e_{p_{e}-1}=f^{p_{e}-1}(e)\},$$ the complete $f$-cycle  of ends containing $e$.  The reader can adapt the following discussion for the case that $e$ is a negative end.  

For $e'$ an end in the cycle $c$, set $\UU_{e'}=\bigcup_{n=0}^{\infty} f^{-np_{e'}}(U_{e'})$, where $p_{e'}=p_{e}$ and $U_{e'}$ is any  $f$-neighborhood of $e'$.  Set 
$$\UU_{c}=\bigcup_{n=-\infty}^{\infty}f^{n}(U_{e})=\UU_{e_{0}}\cup\UU_{e_{1}}\cup\cdots\cup\UU_{e_{p_{e}-1}}.$$  
Remark that $\UU_{c}$ is an open, $f$-invariant set with no periodic points.  The connected components $\UU_{e_{i}}$ of $\UU_{c}$ are permuted cyclically by $f$.

The action of $f$ partitions $\UU_{c}$ into orbits $x=\{x_{n}\}_{n\in\Z}$, where $x_{n}=f^{n}(x_{0})$.  Let $F$ be the space of orbits with the quotient topology, and remark that the quotient map $q:\UU_{c}\to F$ is a regular covering map.  The group of deck transformations is infinite cyclic generated by $f$.  In particular $F$ is a surface.

\begin{rem}
Let $e$ be a positive (respectively negative) end of $L$ and $U_{e}$ be any  $f$-neighborhood of $e$. Since $L$  has a finite endset, the set $X = U_{e}\sm \intr f^{p_{e}}(U_{e})$ (respectively $X = U_{e}\sm \intr f^{-p_{e}}(U_{e})$) is compact. Since the  compact set   $X$  and the connected set $U_{e}$ both surject onto $F$, it follows that $F$ is compact and connected. 

\end{rem}

\begin{nota}
Fix a basepoint $*\in F$, $*=\{x_{n}\}_{n\in\Z}$, where $x_{n+1}=f(x_{n})$, $-\infty< n<\infty$, and $x_{0}\in\UU_{e}$.  
 
 \end{nota}
 
 If $\sigma$ is a directed (oriented) loop in $F$ based at $*$, $\sigma$ lifts to a directed path $\tau$ in $\UU_{e}$ starting at $x_{n}$ and ending at a point $x_{n+\kappa(\sigma)}$.  Here, one sees that $\kappa(\sigma)\in\Z$ is independent of the choice of  $n$.  Simply apply powers of $f$ to $\tau$.  Furthermore, a basepoint-preserving homotopy of $\sigma$ lifts to an endpoint preserving homotopy of all  lifts $\tau$.  This construction defines a group homomorphism
 $$
 \kappa:\pi_{1}(F,*)\to\Z,
 $$
which can be viewed as a cohomology class $\kappa\in H^{1}(F;\Z)$. The period of such a class $\kappa$ is the least positive value taken by $\kappa$, evidently the greatest common divisor of the  set of values of $\kappa$.  The class is divisible if its period is greater than $1$.

\begin{lemma}\label{perper}
The class $\kappa$ is divisible if and only if $p_{e}>1$, in which case the period of $\kappa$  is the period $p_{e}$ of $e$.
\end{lemma}
 
 Indeed, for $n\in\Z$, there is a path in $\UU_{c}$ from $x_{n}$ to $x_{n+p_{e}}$ which projects by $q$ to a loop $\sigma$ in  $F$. Thus $\kappa(\sigma)=p_{e}$.  Evidently, $p_{e}$ is the smallest positive value of $\kappa$.
 
 \begin{lemma}\label{Poincdu}
 There is a compact, transversely oriented, properly embedded $1$-manifold $J_{\kappa}$ in $F$ such that $\kappa(\sigma)=\sigma\cdot J_{\kappa}$, the algebraic intersection number.
 \end{lemma}
 
 \begin{proof}
  As is well known, there is a map $f_{\kappa}:F\to\SI$, unique up to homotopy, such that $f_{\kappa}^{*}[\SI]= \kappa$,  $[\SI]\in H^{1}(\SI)$ being the fundamental class.  We may take $f_{\kappa}$ to be smooth.   By Sard's theorem, there is a regular value $p$ both for $f_{\kappa}$ and for $f_{\kappa}|\bd F$.  Then, $J_{\kappa}=f_{\kappa}^{-1}(p)$ is a properly embedded, compact 1-manifold, transversely oriented by the orientation of $\SI$.    Let $\sigma:\SI\to F$ be in general position relative to $J_{\kappa}$.  Then $f_{\kappa}\o\sigma:\SI\to\SI$ has $p$ as a regular value and has degree equal to the number of times (counted with sign) that it crosses $p$.  Clearly, this degree is equal to $\sigma\cdot J_{\kappa}$ and is the value of $ \kappa$ on $\sigma$. 
 \end{proof}

 \begin{rem}
 If $F$ is oriented, the transverse orientation of $J_{\kappa}$ induces an orientation, allowing us to view $J_{\kappa}$ as a 1-\emph{cycle}. In this case, Lemma~\ref{Poincdu} is really just Poincar\'e duality.  But we do not require orientability and 
 the  transversely oriented 1-manifold $J_{\kappa}$ can only be thought of as a $1$-\emph{cocycle}, evaluating on $1$-cycles via the algebraic intersection product.  Throughout this discussion, cocycles typically will be transversely oriented, properly embedded 1-manifolds.
 \end{rem}

 \begin{defn}[$\kappa$-juncture]\label{kapJunct}
 We call $J_{\kappa}$ a \emph{$\kappa$-juncture}.
 \end{defn}
 
 \begin{rem}
 Note that a $\kappa$-juncture is a   compact submanifold of the compact surface $F$ and is not a juncture in the sense that we usually use the term juncture in this paper.
 
 \end{rem}
 
 \begin{rem}
  The $\kappa$ in the symbol $J_{\kappa}$ for a $\kappa$-juncture refers to the fact that the geometric object $J_{\kappa}$ represents the specific cohomology class $\kappa\in H^{1}(F;\Z)$.
 
 \end{rem}

 One sometimes calls the cohomology class $\kappa\in H^{1}(F;\Z)$ a ``juncture'', but we prefer to reserve this term for a  geometric object representing $\kappa$.   
 
Notice that two components of $J_{\kappa}$ might be ``parallel'' circles if they cobound an annulus  $A\ss\intr F$. They will be parallel, properly embedded arcs if, together with two arcs in $\bd F$, they bound a rectangle $A$ with   $\intr A\ss\intr F$.  In either case, their transverse orientations are said to be coherent if one is oriented out of $A$ and the other into $A$.

 We want to modify a $\kappa$-juncture $J_{\kappa}$ to be \emph{weakly groomed} in the following sense.
 
 \begin{defn}[weakly groomed]\label{groomed}
 A compact, properly embedded, transversely oriented $1$-manifold $J\ss L$ is \emph{weakly groomed} if every pair of parallel circle (respectively arc) components have coherent transverse orientations.
 \end{defn}
 
 \begin{rem}
 The term ``groomed'' is already in use by $3$-manifold topologists, having been introduced by D.~Gabai in~\cite{ga3}.  It places stronger conditions on $J$ than we require, but includes our condition of ``weakly groomed''. It also assumes that $F$ is orientable, which we do not. The metaphor, of course, has to do with a nicely combed head of hair.
 \end{rem}
 
 \begin{lemma}\label{groomed'}
 Every $\kappa$-juncture $J_{\kappa}$ is cohomologous to a weakly groomed $\kappa$-juncture.
 \end{lemma}
 
 \begin{proof}
  If $J_{\kappa}$ is not weakly groomed, let $\tau_{1}$ and $\tau_{2}$ be parallel components with noncoherent transverse orientation.  It is clear that the algebraic intersection number of any closed, oriented curve $\sigma$ with $\tau_{1}\cup\tau_{2}$  is zero, hence these components can be removed.  Repeating this procedure finitely often produces the desired weakly groomed $\kappa$-juncture.
 \end{proof}
  
Thus, one can assume that the properly embedded arc components of $J_{\kappa}$ fall into ``packets'' of parallel, coherently transversely oriented arcs and that no two parallel components of $J_{\kappa}$ have opposing   transverse orientation.  Similarly, the components that are essential, embedded, transversely oriented circles fall into packets  of parallel, coherently transversely oriented circles with no oppositely oriented parallel packets.  Each of these packets  can be represented by one of its elements together with a positive integer weight (the number of components of the packet).  That is, the cohomology class $\kappa$ can be represented as a union of disjoint, transversely oriented arcs and circles $s$, no two of which are parallel, each with an attached integer weight $w_{s}>0$.  The union of these arcs and circles, neglecting the weights and transverse orientations, is called the \emph{support} of the $\kappa$-juncture and will be denoted by $|J_{\kappa}|$. 

\begin{rem}
Note that the term support and notation $|\cdot|$ for support has a different meaning for laminations in the rest of the paper.

\end{rem}
 
\begin{lemma}\label{nosep}
If $|J_{\kappa}|$ separates $F$, there is a weakly groomed $\kappa$-juncture  $J_{\kappa}^{*}$ such that  $|J^{*}_{\kappa}|$ does not separate $F$ and the set of its components  is a subset of the set of components of $|J_{\kappa}|$.
\end{lemma}

\begin{proof}
Let $s_{1},s_{2},\dots,s_{r}$ be oriented components of $|J_{\kappa}|$ that separate off a connected subsurface $S$ of $F$. Let $w_{i}$ be the weight associated to $s_{i}$ and let $w_{j}$ be the minimum, $1\le j\le r$. Let $s_{i}'$ be the same arc or circle as $s_{i}$, but  with  transverse orientation inward to $S$.  Then the coboundary $\sum_{i=1}^{r}w_{j}s'_{i}$ can either be added to or subtracted from $J_{\kappa}$, reducing the number of components of $|J_{\kappa}|$.  Finite repetition of this process produces the desired weakly groomed $\kappa$-juncture $J_{\kappa}^{*}$.
\end{proof}

We fix a choice of $J_{\kappa}$ satisfying the conclusions of Lemma~\ref{nosep}. In particular, through each component of $|J_{\kappa}|$ there is a transverse loop that does not intersect any other component of $|J_{\kappa}|$.
The following is a consequence of Lemma~\ref{nosep} and Lemma~\ref{perper}.

\begin{cor}\label{divbype}
Each weight $w_{s}$ is divisible by $p_{e}$.
\end{cor}

We can assume that our basepoint $*$ is disjoint from $|J_{\kappa}|$.  Let $F'$ denote the compact, connected surface with boundary (and possibly corners) obtained by cutting $F$ apart along the components of $|J_{\kappa}|$.  Then each component $s$ of $|J_{\kappa}|$ determines two copies $s_{\pm}$ of itself in $\bd F'$, $s_{+}$ being the copy along which the transverse orientation points out of $F'$ and $s_{-}$ the one along which the transverse orientation is inward.

Loops in $F$ based at $*$ which do not properly intersect $|J_{\kappa}|$ remain loops in $F'$ based at $*$.  The following is an easy consequence.

\begin{lemma}\label{F'embedded}
For each $n\in\Z$, there is a unique copy $F_{n}$ of $F'$ embedded in $\UU_{c}$ and containing $x_{n}$. 
\end{lemma}

The projection $q:F_{n}\to F$ induces a homeomorphism $q':F_{n}\to F'$.  The curve (an arc or circle) $s^{n}_{\pm}\ss\bd F_{n}$ is the one carried by $q'$ onto $s_{\pm}\ss\bd F'$.

\begin{lemma}\label{F'attached}
For each $n\in\Z$ and each component $s$ of $|J_{\kappa}|$, $F_{n}$ is attached to $F_{n+w_{s}}$ by an identification $s^{n}_{+}\equiv s^{n+w_{s}}_{-}$.
\end{lemma}

\begin{proof}
By Lemma~\ref{nosep}, there is a loop $\sigma$ in $F$, based at $*$, which intersects $s$ once and has algebraic intersection number $+1$ at that point.  Thus, $\sigma$ lifts to a path joining $x_{n}$ to $x_{n+w_{s}}$, exiting $F_{n}$ through $s^{n}_{+}$ and entering $F_{n+w_{s}}$ through $s^{n+w_{s}}_{-}$.
\end{proof}

List the conponents of $|J_{\kappa}|$ as $s_{1},s_{2},\dots,s_{p}$. Write $w_{s_{i}}=w_{i}$.  The juncture components $s^{n}_{i-}$, $1\le i\le p$, are inwardly oriented components of $\fr F_{n}$ and the juncture components $s^{n}_{i+}$, $1\le i\le p$, are outwardly oriented ones.

Let
$$
V_{0}=F_{0}\cup F_{p_{e}}\cup F_{2p_{e}}\cup\cdots\cup F_{kp_{e}}\cup\cdots.
$$

\begin{lemma}
For each integer $1\le i\le p$, the number of inwardly oriented components of $\fr V_{0}$ of the form $s^{j}_{i-}$ is $w_{i}/p_{e}$.
\end{lemma}

\begin{proof}
Fix $s_{i}$.  Since $F_{kp_{e}}$ is attached to $F_{kp_{e}+w_{i}}$ by identifying $s^{kp_{e}}_{i+}$ with $s^{kp_{e}+w_{i}}_{i-}$, it is clear that exactly the $w_{i}/p_{e}$  inwardly oriented components $s^{j}_{i-}$ of $F_{j}$, 
$$j = 0,p_{e},2p_{e},\ldots, w_{i} - p_{e},$$
 are in $\fr V_{0}$.
\end{proof}

Thus, the union of the  components $s^{j}_{i-}$ in $\fr V_{0}$ is a compact, properly embedded, transversely oriented $1$-manifold. We take $U_{e} = V_{0}$ as $f$-neighborhood of $e$ in Proposition~\ref{comcomp}  and $J = \fr V_{0}$.

\begin{cor}\label{Je0}
The image of $J$ under the covering projection $q:\UU_{c}\to F$ is exactly $|J_{\kappa}|$, each component $s$ of $|J_{\kappa}|$ being the image of exactly $w_{s}/p_{e}$ components of $J$.  
\end{cor}

Since $f$ is a deck transformation, this remains true for all $f^{n}(J)$, all $n\in\Z$, and it is clear that any two of these positive $f$-junctures intersect, if at all, only in common components. 

The proof of Proposition~\ref{comcomp} is now complete.

 \begin{hyp}\label{hypjunct}
\textbf{Hereafter, we will require that any $f$-juncture $J$ is a compact $1$-manifold and has the juncture intersection property}.
\end{hyp}

That is,  hereafter we are modifying the definition of $f$-juncture (Definition~\ref{fjunct}) to require that any $f$-juncture $J$ be a compact $1$-manifold and have the juncture intersection property.

\begin{rem}
Note that $q:V_{0}\to F$ is a semi-covering with covering semi-group $\Z^{+}$ generated by $f$.  It mirrors perfectly the way that an end $e$ of a depth one leaf of a foliation semi-covers the compact leaf $F$ to which it is asymptotic. (One commonly says that $e$ ``spirals'' on $F$.)
\end{rem}

\subsection{Behavior of an endperiodic automorphism near the ends of $L$}

An endperiodic automorphism is well behaved near the ends of $L$. In this subsection we give some lemmas that show the behavior can be more  complex than expected. Example~\ref{weird} shows how complicated the structure of a positive or negative end can be.

\begin{rem}
The notation and terminology  defined in the next paragraph is used only  in Section~\ref{endp}. 

\end{rem}

 It is sometimes necessary to work with a positive multiple $kp_{e}$ of $p_{e}$ as power of $f$.  If $e$ is a positive (respectively negative) end, choose an $f$-neighborhood $U_{e}$. Let $p$ be a positive multiple of $p_{e}$ and define $U_{e}^{i}= f^{ip}(U_{e})$ (respectively $U_{e}^{i}= f^{-ip}(U_{e})$), $J_{e}^{i}=\fr U_{e}^{i}$, and $B_{e}^{i} = \ol{U_{e}^{i}\sm   U_{e}^{i+1}}$, $i\in\Z$.

\begin{rem}
This notation depends on the choice of  positive multiple $p$ of $p_{e}$.

\end{rem}
 
\begin{rem}
The notation is such that all the sets $U_{e}^{n}, B_{e}^{n}, J_{e}^{n}\ss U_{e}$ and  tend to $e$ in the topology of $L\cup\EE(L)$ as $n\to+\infty$ whether $e$ is a positive or negative end. We will also use the notation introduced above, $J_{n}= f^{n}(J)$, for $J$ an $f$-juncture. For $J$ a positive $f$-juncture, $J_{n}$ will approach a cycle of positive ends as $n\to+\infty$ and for $J$ a negative $f$-juncture, $J_{n}$ will approach a cycle of negative ends as $n\to-\infty$

\end{rem}

\begin{defn}[$f$-domain]\label{fdom}

The  set $B_{e}^{i}$, $i\in\Z$, is the \emph{$f$-domain for $e$ corresponding to $p$}.

\end{defn}

\begin{lemma}\label{disjointequal}

If $f$ be an endperiodic automorphism, then the following are equivalent,

\begin{enumerate}

\item If $e$ is a positive end, $f^{p}( U_{e}) = U_{e}^{1}\ss\intr U_{e}$. \\If $e$ is a negative end, $f^{-p}( U_{e}) = U_{e}^{1}\ss\intr U_{e}$.\label{ctone}

\item $U_{e}^{i+1}\ss\intr U_{e}^{i}$, $i\in\Z$.\label{cttwo}

\item $J_{e}^{i}\cap J_{e}^{i+1} = \0$, $i\in\Z$.\label{ctthree}

\item $\fr B_{e}^{i} = J_{e}^{i}\cup J_{e}^{i+1}$, $i\in\Z$.\label{ctfour}

\item $B_{e}^{i}$ separates $L$, $i\in\Z$.\label{ctfive}

\end{enumerate}

\end{lemma}

\begin{proof}
Clearly $(\ref{cttwo})\Leftrightarrow(\ref{ctone})$. In fact $(\ref{ctone})$ is a special case of $(\ref{cttwo})$ with $i=0$ and $(\ref{cttwo})$ follows from $(\ref{ctone})$ since $f$ is a homeomorphism. We will show that $(\ref{cttwo})\Ra (\ref{ctfive})\Ra (\ref{ctfour})\Ra (\ref{ctthree})\Ra (\ref{cttwo})$.  
\medskip

\ni$(\ref{cttwo})\Ra(\ref{ctfive})$. Since $U_{e}^{i+1}\ss\intr U_{e}^{i}$, it follows that $J_{e}^{i} = U_{e}^{i} \sm \intr U_{e}^{i} \ss U_{e}^{i}\sm   U_{e}^{i+1}\ss B_{e}^{i}$. Since $J_{e}^{i}$,  separates $L$, $B_{e}^{i}$ separates $L$.
\medskip

\ni$(\ref{ctfive})\Ra(\ref{ctfour})$.  First remark that since $U_{e}^{i+1}$ is a connected subset of a surface with $\fr U_{e}^{i+1}$ a compact $1$-manifold, $\intr U_{e}^{i+1}$ is connected. 
Further it follows from the definition of $f$-neighborhood that $L\sm U_{e}^{i}$ is connected. 
By (\ref{ctfive}), $L\sm B_{e}^{i}$ is not connected so $L\sm B_{e}^{i}$ is the  union of nonempty disjoint open sets $U,V$. Since  $\intr U_{e}^{i+1}$ and $L\sm U_{e}^{i}$ are connected, only one of $U,V$ can meet each of $\intr U_{e}^{i+1}$ and $L\sm U_{e}^{i}$. It follows that $U = L\sm U_{e}^{i}$ and $V = \intr U_{e}^{i+1}$.  Therefore 
$$L\sm B_{e}^{i} = U\cup V = (L\sm U_{e}^{i})\cup\intr U_{e}^{i+1} = L\sm(U_{e}^{i}\sm\intr U_{e}^{i+1})$$
so  $B_{e}^{i} = U_{e}^{i}\sm\intr U_{e}^{i+1}\supset J_{e}^{i}\cup J_{e}^{i+1}$. Further, if $x\in J_{e}^{i}$ (respectively $x\in J_{e}^{i+1}$) then $x$ has a neighborhood meeting $L\sm U_{e}^{i}$ (respectively $\intr U_{e}^{i+1}$). Thus, every $x\in J_{e}^{i}\cup J_{e}^{i+1}$ has a neighborhood meeting both $B_{e}^{i}$ and its complement so $J_{e}^{i}\cup J_{e}^{i+1}\ss\fr B_{e}^{i}$. The reverse containment is clear so (\ref{ctfour}) follows.
\medskip

\ni$(\ref{ctfour})\Ra(\ref{ctthree})$. If $J_{e}^{i}\cap J_{e}^{i+1} \ne \0$, then by assumption $J_{e}^{i}$ and $J_{e}^{i+1}$ have a component $\sigma$ in common. If $x\in\intr\sigma$, then the point  $x$ has a neighborhood $V$ disjoint from $U_{e}^{i}\sm   U_{e}^{i+1}$. Thus, $x\notin B_{e}^{i}$ so $x\notin\fr B_{e}^{i}$  contradicting $(\ref{ctfour})$.
\medskip

\ni$(\ref{ctthree})\Ra(\ref{cttwo})$. If $J_{e}^{i}\cap J_{e}^{i+1} = \0$, then $J_{e}^{i+1}\ss\intr U_{e}^{i}$ so $U_{e}^{i+1}\ss\intr U_{e}^{i}$.
\end{proof}

\begin{lemma}\label{fdconn}
The integer $p >0$ of \emph{Lemma~\ref{ne}} can be chosen so that the   $B_{e}^{i}$ are connected.
\end{lemma}

\begin{proof}
Without loss assume that $e$ is a positive end. Suppose that  
$$B_{e}^{0} = \ol{U_{e}\sm   f^{p}(U_{e})}$$     
has more than one connected component.  By  Definition~\ref{fnbhd}, $ \fr B_{e}^{0}$ is a compact $1$-manifold, hence has finitely many components. Since $U_{e}$ is connected, we find finitely many paths $s_{1},\dots,s_{r}$ in $U_{e}$ that connect the components of $B_{e}^{0}$.   That is $B_{e}^{0}\cup s_{1}\cup\cdots \cup s_{r}$ is connected.  For a minimal integer $m\ge1$, these paths all lie in $B=B_{e}^{0}\cup B_{e}^{1}\cup\cdots\cup B_{e}^{m}$ and we claim that $B$ is connected.  Indeed, each component of $B_{e}^{1}$ attaches to $B_{e}^{0}$ along at least one component of $J_{e}^{1}$, and so $B_{e}^{0}\cup B_{e}^{1}\cup s_{1}\cup\cdots\cup s_{r}\ss B$ is connected. Repeating this reasoning finitely often, we obtain that $B$ is connected.  If we replace the integer $p$ of Lemma~\ref{ne} with $q=p(m+1)$, then $B = \ol{U_{e}\sm  f^{q}(U_{e})}$ can be taken as a new   connected  $B_{e}^{0}$.  
\end{proof}

We introduce the following nonstandard term which will come up frequently throughout this paper.

\begin{defn}[simple end]\label{simpend}
 An end of $L$ is \emph{simple} if it has a neighborhood homeomorphic to $S^{1}\x[0,\infty)$ or $[0,1]\x[0,\infty)$.
\end{defn}

The next lemma follows from the fact that $L$ has no simple ends. 

\begin{lemma}\label{negeuler}

If $L$ has no simple ends and $p >0$ is large enough, then the   $B_{e}^{i}$ have negative Euler characteristic.

\end{lemma}

\begin{example}\label{weird}
This example is meant to illustrate how complex the structure of a positive or negative end might be.

\begin{figure}[ht]
\begin{center}
\begin{picture}(300,90)(-65,0)

\scalebox{.8}{\includegraphics[width=220pt]{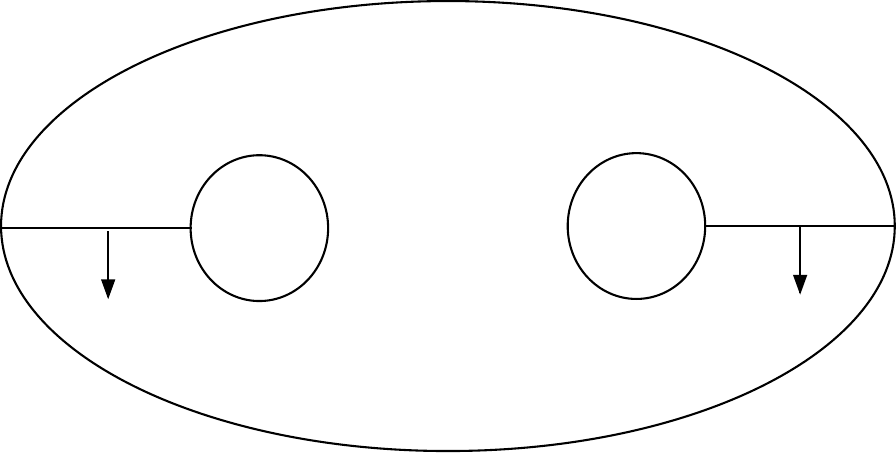}}
\put (-160,50){$_{A_{i}}$}
\put (-25,50){$_{D_{i}}$}
\put(-92.5,45){$_{P_{i}}$}

\end{picture}
\caption{A pair of pants $P_{i}$}\label{AD}
\end{center}
\end{figure}

We describe a family of two-ended surfaces  and an endperiodic automorphism on each which is a  translation (this term will be carefully defined in Section~\ref{translations}).  The surface $L = \bigcup_{i\in\Z}P'_{i}$ will be formed by cutting  pairs of pants $P_{i}$ along certain essential, properly embedded subarcs and then pasting the resulting disks $P'_{i}$ to one another along these subarcs.  The endperiodic automorphism $f:L\to L$ will take $P_{i}'$ to $P'_{i+1}$.

In Figure~\ref{AD}, we depict the typical pair of pants and the essential arcs with transverse orientation.  After cutting, $P_{i}$ becomes a disk $P_{i}'$, with $A_{i}$ and $D_{i}$ split and indexed   as indicated in Figure~\ref{indices}. Here $i$ varies over the integers. Let $m$ and $n$ be fixed, relatively prime integers, positive and/or negative.  The index $i$ on $A^{+}_{i},D^{+}_{i}$ indicates that these arcs are identified with the original $A_{i}$ and $D_{i}$, while the index on $A^{-}_{i+m}$  indicates that it is to be attached to $A^{+}_{i+m}$, forming a single arc to be labeled $A_{i+m}$, and the index on $D^{-}_{i+n}$ indicates that it is to be attached to $D^{+}_{i+n}$, forming a single arc to be labeled $D_{i+n}$.  It will be convenient to represent $P'_{i}$ and its boundary arcs symbolically as $X_{i}$ and its vertices as in Figure~\ref{XX}.

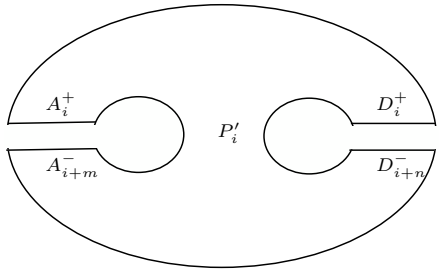
\begin{figure}[ht]
\begin{center}
\begin{picture}(300,100)(-60,0)

\rotatebox{90}{\includegraphics[width=100pt]{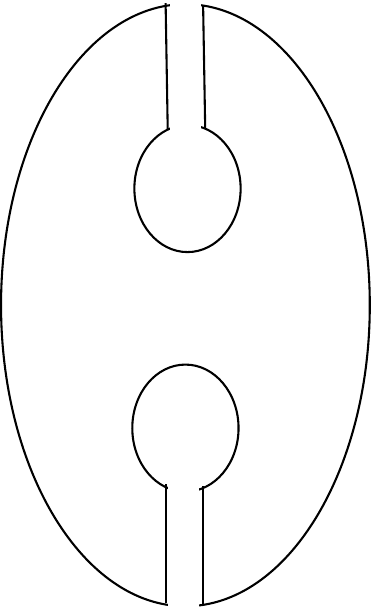}}
\put (-150,63){$_{A^{+}_{i}}$}
\put (-150,40){$_{A_{i+m}^{-}}$}
\put (-25,63){$_{D^{+}_{i}}$}
\put (-25,40){$_{D^{-}_{i+n}}$}
\put(-85,51.5){$_{P'_{i}}$}

\end{picture}
\caption{Arcs $A^{\pm},D^{\pm}$ with indices}\label{indices}
\end{center}
\end{figure}

\begin{figure}[ht]
\begin{center}
\begin{picture}(300,120)(-100,0)

\scalebox{.8}{\includegraphics[width=120pt]{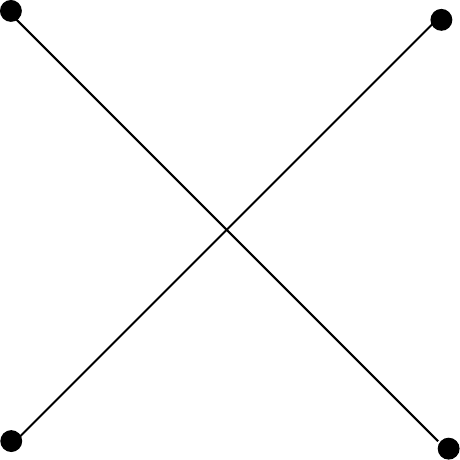}}
\put (0,0){$_{D_{i+n}^{-}}$}
\put (-107,0){$_{A_{i+m}^{-}}$}
\put (-108,90){$_{A_{i}^{+}}$}
\put (0,90){$_{D_{i}^{+}}$}
\put(-53,40){$_{X_{i}}$}

\end{picture}
\caption{A symbolic representation of Figure~\ref{indices}}\label{XX}
\end{center}
\end{figure}
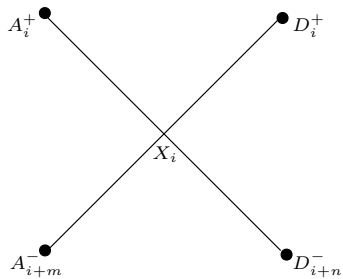

\begin{figure}[ht]
\begin{center}
\begin{picture}(250,100)(50,-150)
\rotatebox{270}{\includegraphics[width=270pt]{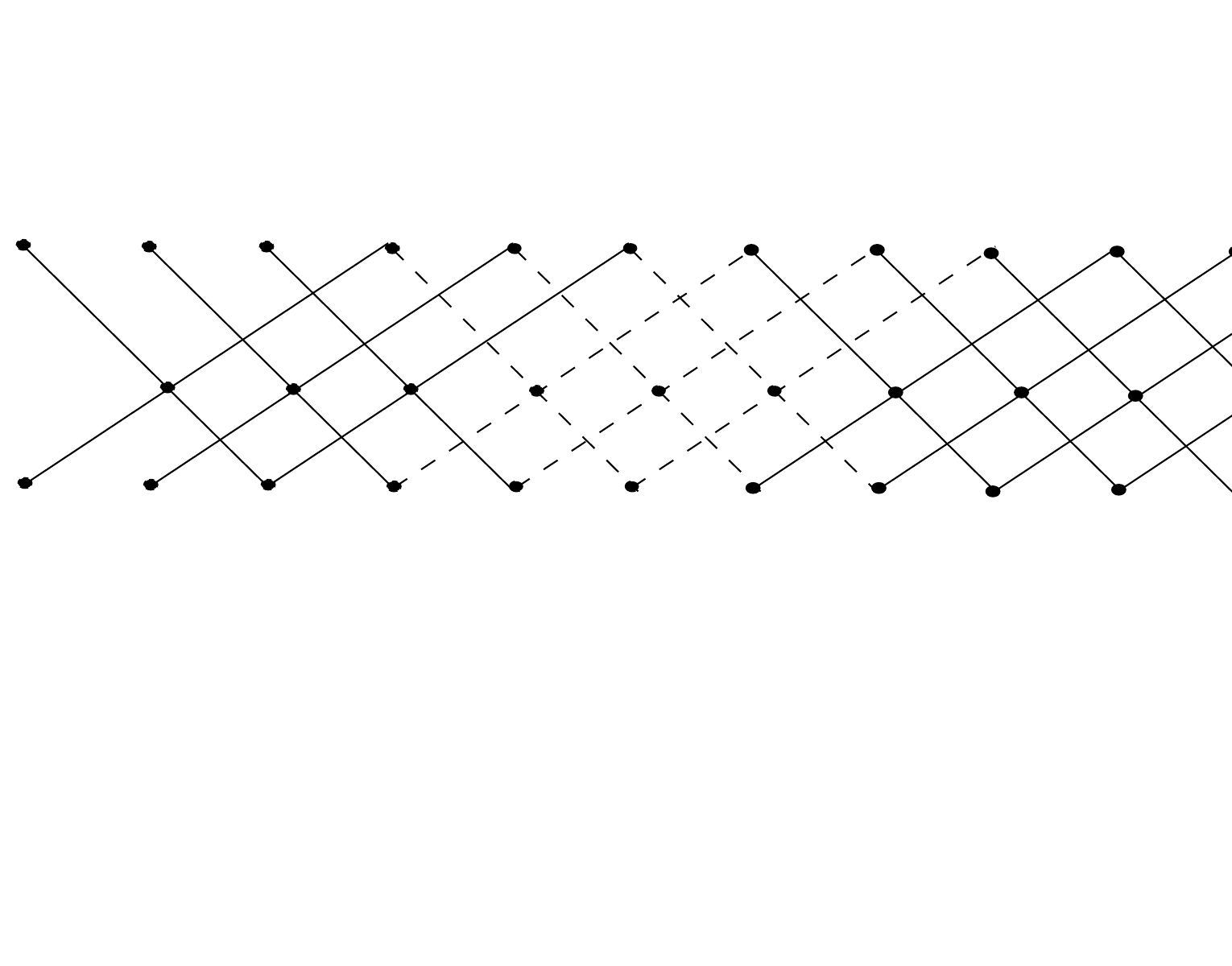}}

\put (-340,-64){\small$D_0$}
\put (-306,-64){\small$D_1$}
\put (-273,-64){\small$D_2$}
\put (-240,-64){\small$D_3$}
\put (-205,-64){\small$D_4$}
\put (-173,-64){\small$D_5$}
\put (-140,-64){\small$D_6$}
\put (-106,-64){\small$D_7$}
\put (-70,-64){\small$D_8$}
\put (-35,-64){\small$D_9$}

\put (-340,-145){\small$A_0$}
\put (-306,-145){\small$A_1$}
\put (-273,-145){\small$A_2$}
\put (-240,-145){\small$A_3$}
\put (-205,-145){\small$A_4$}
\put (-173,-145){\small$A_5$}
\put (-140,-145){\small$A_6$}
\put (-106,-145){\small$A_7$}
\put (-70,-145){\small$A_8$}
\put (-35,-145){\small$A_9$}

\put (-315,-110){\small$X_0$}
\put (-278,-110){\small$X_1$}
\put (-245,-110){\small$X_2$}
\put (-210,-110){\small$X_3$}
\put (-176,-110){\small$X_4$}
\put (-143,-110){\small$X_5$}
\put (-110,-110){\small$X_6$}
\put (-74,-110){\small$X_7$}
\put (-43,-110){\small$X_8$}

\end{picture}
\caption{Symbolic picture of the  surface with  $m=2, n=3$}\label{WeirdEx}
\end{center}
\end{figure}
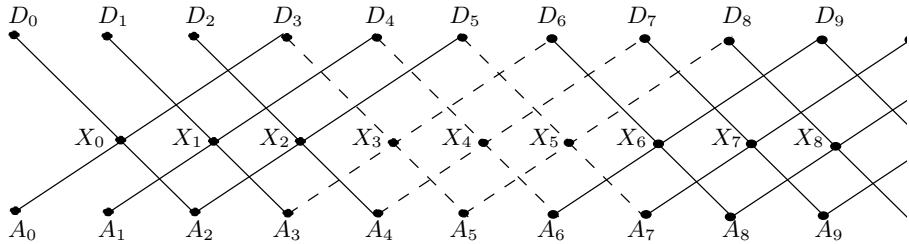

The case $m=2$  and $n=3$ is indicated symbolically by the graph in Figure~\ref{WeirdEx}. The endperiodic automorphism carries $X_{i}$ to $X_{i+1}$. An  $f$-neighborhood of the positive end $e$ is given by $U_{e} = P'_{0}\cup P'_{1}\cup\cdots$ and is represented symbolically as $X_{0}\cup X_{1}\cup\cdots$ in Figure~\ref{WeirdEx}. Thus, $p_{e}=1$.  The $f$-juncture $J_{e}^{0} = \fr U_{e}$ consists of the five properly embedded arcs represented by the points $A_{0},A_{1},D_{0},D_{1},D_{2}$ in Figure~\ref{WeirdEx}. As usual, let $U_{e}^{i} = f^{ip}(U_{e})$, $J_{e}^{i}=f^{ip}(J_{e}^{0}) = \fr U_{e}^{i}$, and $B_{e}^{i} = \ol{U_{e}^{i}\sm   U_{e}^{i+1}}$, $i\ge 0$ where $p$ is a positive integer multiple of $p_{e}=1$.

If $p=1$ or $2$, the  $U_{e}^{i}$, $J_{e}^{i}$, and $B_{e}^{i}$ do not satisfy the conditions of Lemma~\ref{disjointequal}. In particular, $J_{e}^{0}\cap J_{e}^{1}\ne\0$ and $B_{e}^{0}$ does not separate $L$.

If $p=3$, then   $B_{e}^{0}$ is symbolically represented by $X_{0}\cup X_{1}\cup X_{2}$ in Figure~\ref{WeirdEx} where it is drawn with solid lines, $B_{e}^{1}$ is represented by $X_{3}\cup X_{4}\cup X_{5}$ in Figure~\ref{WeirdEx} where it is drawn with  dotted lines and $B_{e}^{2}$ is represented by $X_{6}\cup X_{7}\cup X_{8}$ in Figure~\ref{WeirdEx} where it is drawn with  solid lines. In this case,  $B_{e}^{0}$ does  separate $L$, so  the  $U_{e}^{i}$, $J_{e}^{i}$, and $B_{e}^{i}$ do  satisfy the conditions of Lemma~\ref{disjointequal}. However, $B_{e}^{0}$ has two components.

If $p\ge 4$,  the  $U_{e}^{i}$, $J_{e}^{i}$, and $B_{e}^{i}$ do  satisfy the conditions of Lemma~\ref{disjointequal} and $B_{e}^{0}$ is connected. If $p=4$, $B_{e}^{0}$ is symbolically represented by $X_{0}\cup X_{1}\cup X_{2}\cup X_{3}$ in Figure~\ref{WeirdEx} and is a disk. If $p=5$, $B_{e}^{0}$ is symbolically represented by $X_{0}\cup X_{1}\cup X_{2}\cup X_{3}\cup X_{4}$ in Figure~\ref{WeirdEx} and is an annulus. If $p\ge 6$, $B_{e}^{0}$ has negative Euler charateristic, separates $L$, and is connected and $\fr B_{e}^{0} = J_{e}^{0}\cup J_{e}^{1}$ with $J_{e}^{0}\cap J_{e}^{1} = \0$ .

\end{example}

\subsection{Properties of endperiodic automorphisms}

We  prove some elementary facts about endperiodic automorphisms.  

\begin{prop}\label{att-rep}
If  $L$ has at least one nonsimple end, then the endperiodic automorphism $f:L\to L$ has both positive  and negative ends.  
\end{prop}

\begin{proof}
Assume that $L$ has no negative ends. Let $e$ be a positive end and let $U_{e}$ be an $f$-neighborhood of $e$.  Let $B_{e}^{i}$, $i\in\Z$,  be the   $f$-domains (Definition~\ref{fdom}) for $e$ corresponding to a choice of $p$, a large enough multiple of $p_{e}$ that it can serve as a choice of $p$ in Lemmas~\ref{ne},~\ref{fdconn} and~\ref{negeuler}. Thus, $B_{e}^{i+1}=f^{p}(B_{e}^{i})$.  By Lemma~\ref{negeuler},   the   $f$-domains $B_{e}^{i}$ for the end $e$ have negative Euler characteristics. Then the sequence $\{B_{e}^{i}\}_{i=0}^{-\infty}$ is a sequence of compact surfaces with disjoint interiors, each of negative Euler characteristic, which does not accumulate at any end.  This gives the contradiction that some compact subsurface $S\ss L$ has infinite Euler characteristic.  
\end{proof}

\begin{rem}
In particular, $L$ must have at least two ends. The argument  of this proof obviously fails if all ends  are simple.  For example, it is easy to produce an endperiodic automorphism on $\R^{2}$ with one negative  end and one on $I\x\R$ with two negative ends.  
\end{rem}

\begin{rem}
Each end of every noncompact boundary component of a surface $L$ limits on some end of $L$.

\end{rem}

\begin{rem}
One might think that at most finitely many noncompact boundary components can have an end limiting on a given end of $L$.  But see~\cite[Figure~12.5.11]{condel1} for a  surface with one end $e$	 and infinitely many noncompact boundary components with both ends limiting on $e$.  Of course, this surface does not admit an endperiodic automorphism.
\end{rem}

\begin{lemma}\label{noncompactbd}
If $f:L\to L$ is endperiodic, then, only finitely many noncompact boundary components can limit on a given end  $e$ of $L$.
\end{lemma}

\begin{proof}
 Since $L$ has finitely many ends, $e$ is a periodic end. Let $U_{e}$ be an $f$-neighborhood of $e$. Consider the $f$-junctures $\{J_{e}^{i}\}_{i=1}^{\infty}$, a sequence of mutually homeomorphic $1$-manifolds.  If  $J_{e}^{i}$ is a collection of simple closed curves, no noncompact boundary component limits on  $e$.  In general, $J_{e}^{i}$ contains finitely many arc components and the finite number of endpoints of these components is independent of $i$. Each noncompact boundary component $\sigma$ of $L$ with an end  limiting on $e$ contains an endpoint of an arc component of $J_{e}^{i}$ for some $i$ and if $\sigma$ contains an endpoint of an arc component of $J_{e}^{i}$, it contains an endpoint of an arc component of $J_{e}^{i+r}$, $r\ge0$. Thus, if infinitely many boundary components of $L$ have an end limiting on $e$, the number of endpoints of arc components of $J_{e}^{i}$ is unbounded which is a contradiction.
\end{proof}

\begin{rem}
 It should be noted that it is possible that a component of $\bd L$ might be a line $\ell$ joining a positive end to a positive end (possibly the same end) with the parallel possibility for negative ends.  In that event, $f$ will have at least one periodic point on $\ell$.  It will turn out that $\ell$ will then be a leaf of one of the two transverse laminations to be constructed in Section~\ref{constr}, while a ray issuing from the periodic point will be a leaf of the  other.  This causes some difficulties for us, but the problem can be eliminated without loss of generality by doubling along all such components $\ell$ (and only along such components). Denote the partially doubled surface by $L'$ and the corresponding partial double of $f$ by $f'$.   The original surface remains as an $f'$-invariant subsurface in $L'$.  The remaining noncompact boundary components which issue from periodic ends will join negative ends to positive ones and it is easy to modify the endperiodic automorphism near these boundary lines so that there are no periodic points on them.  Accordingly, we make the following assumption.
 \end{rem}

\begin{hyp}\label{noperpt}
 \textbf{Each noncompact component of $\bd L$,  joins a negative end to a positive end and contains no periodic point}.
\end{hyp}

 \subsection{Tunneling}\label{Tunnel}
 
 A certain special class of geometric modifications of junctures that leave them unchanged homologically is a process which we call ``tunneling''. In Figure~\ref{tunneling} we represent tunneling between components $\tau_{1}$ and $\tau_{2}$ of $J_{\kappa}$ along an arc $\alpha$ issuing from $\tau_{1}$ and ending at $\tau_{2}$ and not otherwise meeting $J_{\kappa}$ so that, whatever orientation is given to $\alpha$, it agrees with the transverse orientation of $\tau_{1}\cup\tau_{2}$ at one end and opposes it at the other.  In this paper we only need this  in the proof of the transfer theorem (Theorem~\ref{transfer}).
 
 \begin{figure}[ht]
\begin{center}
\begin{picture}(100,50)(70,0)

\includegraphics[width=250pt]{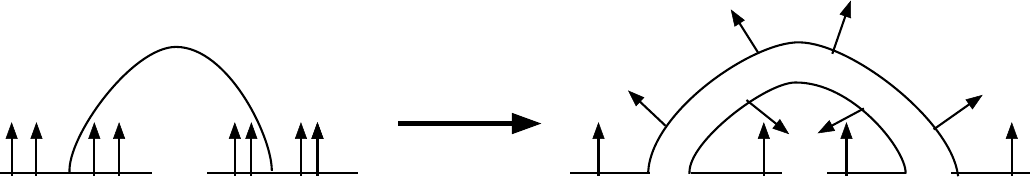}

\put(-194,25){\small$\alpha$}
\put(-240,-8){$\tau_{2}$}
\put(-190,-8){$\tau_{1}$}

\end{picture}
\caption{A tunneling cohomology}\label{tunneling}
\end{center}
\end{figure}

Let $\sigma\ss \intr F$ be an embedded,  oriented circle.  Suppose that $\sigma$ intersects $J_{\kappa}$ transversely and in finitely many points.  Let $\sigma\cdot J_{\kappa}$ denote the algebraic intersection number. Let the transversely oriented components of $|J_{\kappa}|$ be $\tau_{j}$ with respective weights $w_{j}$, $1\le j\le r$. 

By tunneling we will prove the following.

\begin{prop}\label{tunnel}
For any  $\sigma\ss \intr F$, an embedded,  oriented circle  that intersects a $\kappa$-juncture  $J_{\kappa}$ transversely and in finitely many points, there is a new choice of $\kappa$-juncture $J_{\kappa}$, also intersecting $\sigma$ transversely and in finitely many points, such that 
\begin{enumerate}
\item $|\sigma\cdot J_{\kappa}|=\sum_{j=1}^{r}w_{j}\card(\sigma\cap\tau_{j})$,
\item $J_{\kappa}$ is weakly groomed, 
\item $|J_{\kappa}|$ does not separate $F$.
\end{enumerate}
\end{prop}

\begin{proof}
 Realize $J_{\kappa}$ not as a union of weighted, transversely oriented arcs and circles, but of packets of parallel, coherently transversely oriented arcs and circles.  If the asserted equality fails, there is a subarc $\alpha$ of $\sigma$ with endpoints $x_{1},x_{2}$ on respective component(s) $\tau_{1}, \tau_{2}$ of $J_{\kappa}$, not intersecting $J_{\kappa}$ in any other points, such that the  orientation of $\alpha$ at, say, $x_{1}$ disagrees with the transverse orientation of $\tau_{1}$ at that point and, at $x_{2}$, the orientation of $\alpha$ and the transverse orientation of $\tau_{2}$ agree.

 Let $A$ be a rectangle with a pair of opposite sides $\beta_{i}\ss\tau_{i}$ containing the points $x_{i}$ in their interiors, $i=1,2$, and a pair of opposite sides $\alpha_{i}$ parallel to $\alpha$ with $\intr\alpha\ss\intr A$. Clearly $A$ can be chosen to be disjoint from $\sigma\sm\alpha$. Transversely orient the circle $\bd A$ so that the orientation points outward from $A$ if $\alpha$ meets $\tau_{1}$ and $\tau_{2}$ on the positive sides and, in the alternative case, let $\bd A$ be transversely oriented into $A$.  Let $\tau$ denote the cocycle which is the transversely oriented $\bd A$.  This cocycle is clearly a coboundary and so the cocycle $J_{\kappa}+\tau$ is cohomologous to $J_{\kappa}$.  The overlap  of $\tau$ and $\tau_{i}$ is two copies of $\beta_{i}$ with opposite transverse orientation, hence can be ``erased'', $i=1,2$.  The resulting transversely oriented, properly embedded 1-manifold  $J^{*}_{\kappa}$ is as in Figure~\ref{tunneling} and the intersection points $x_{1},x_{2}$ which introduced cancelling intersection numbers $\pm1$ have been eliminated without changing the other intersections of $\sigma_{k}$ with $J^{*}_{\kappa}$, $1\le k\le n$.  Finite repetition produces $J_{\kappa}^{*}$ such that the algebraic intersection numbers at each point of $\sigma_{i}\cap J^{*}_{\kappa}$ all have the same sign. 
 
 Finally, the modifications in Lemma~\ref{groomed'} and Lemma~\ref{nosep} only involve throwing away components of $\kappa$-juncture, hence do not affect the intersection properties already established.
\end{proof}

 As an example of tunneling, the reader might try proving the following result which is not consequential for this paper but might be useful in studying examples.
 
 \begin{lemma}\label{arcsorcircle}
 The $\kappa$-juncture, chosen with all the above properties, can be assumed to be represented entirely by weighted, transversely oriented, properly embedded arcs, or one weighted, transversely oriented circle.
 \end{lemma}

\subsection{Translations}\label{translations}

The simplest endperiodic automorphisms are the translations, although even these can be surprisingly complicated (check out Example~\ref{weird}).

\begin{defn}[translation]\label{U'=L}
The endperiodic automorphism is a \emph{translation} if (in the notation of Section~\ref{jnctrs}) $L$ has an end $e$  such that $\UU_{e}=L$.
\end{defn}

\begin{rem}
In the case of a translation, $L$ has one positive end and one negative end.
\end{rem}

Suppose that $f$ is a translation.  It is evident that $L$ has two ends, an attracting end $e$ and a repelling end $e'$ and $p_{e}=1$.  Lemma~\ref{F'embedded} gives  a sequence $\{F_{n}\}$ of copies of $F'$ embedded in $\UU_{e} = L$ with $f(F_{n}) = F_{n+1}$ for $n\in\Z$ and   $L=\bigcup_{n=-\infty}^{\infty}F_{n}$. It is natural to think of the $F_{n}$'s as ``fundamental domains'' for the homeomorphism  $f$.   Example~\ref{weird} shows that,  in general, $F_{n}$ does not separate $L$ and can be attached to more than two of the $F_{j}$'s along common boundary components.  

We make the following definition only for use in the next two  lemmas 

\begin{defn}
If the $F_{n}$'s can be chosen   to separate $L$ and to be attached to exactly two of the $F_{j}$'s along common boundary components, we say that the translation $f$ is a \emph{simple translation}.
\end{defn}

An immediate consequence of Lemmas~\ref{disjointequal} and~\ref{ne}  is the following.

\begin{lemma}
If $f$ is a translation and $k$ is sufficiently large, then  $f^{k}$ is a simple translation.
\end{lemma}

In case $\bd L=\0$, hence $\bd F=\0$, and $f$ is a translation, the $\kappa$-juncture can always be chosen to be a simple closed curve $s$ (Lemma~\ref{arcsorcircle}).   As a consequence, we note the following.

\begin{lemma}
If $\bd L=\0$ and $f:L\to L$ is a translation, then $f$ is a simple translation.
\end{lemma}

\subsection{Avoiding the juncture intersection property} It was pointed out to us by the anonymous referee of~\cite{cc:almostnohol} that, by a small isotopy of the $f$-junctures in $\UU_{c}$, it can be assumed that every two $f$-junctures are disjoint.  This has the effect of allowing part~$(1)$ of Definition~\ref{perends} to be strengthened to read:
$$
f^{p_{e}}(U_{e})\ss \intr U_{e}.
$$
While this may seem desirable, Example~\ref{weird} suggests that it is a bit contrived and in the context of depth one foliations it is decidedly unnatural. Our asymptotic construction of $f$-junctures is entirely motivated by depth one foliations. We record the fact for its possible usefulness, but continue to stick with Definition~\ref{perends} as stated and the juncture intersection property.

 
 \section{Preliminaries to the Handel-Miller Theory}\label{hyppre}

We present here some material that will be appealed to repeatedly.

\begin{*hyp}

\textbf{In Section~\ref{hyppre}, we assume that $L$ is a noncompact, connected  surface   with finitely or infinitely many ends and possibly with boundary}.
 
\end{*hyp}

 \subsection{Some remarks on isotopies}\label{epsteinbaer}

In constructing isotopies in this paper, we will frequently use the Epstein-Baer theorems~\cite[Theorem~2.1 and~3.1]{Epstein:isotopy} about homotopies and isotopies of curves.  For convenient reference we state them here and refer the reader to~\cite{Epstein:isotopy} for the proofs.  In what follows, $L$ is a connected surface, compact or noncompact, with or without boundary and orientable or nonorientable. Recall that an ambient isotopy on $M$ is a continuous map $\Phi:M\x I\to M$, written $\Phi(x,t)=\Phi^{t}(x)$, where $\Phi^{0}=\id$,   with $\Phi^{t}:M\to M$ a homeomorphism, $0\le t\le1$.  

\begin{rem}
Whether we say so or not, all of our isotopies will be ambient.\label{allisotamb}

\end{rem}

\begin{theorem}[Epstein-Baer]\label{2.1}
Let $\alpha,\beta:\SI\to\intr L$ be freely homotopic, embedded, $2$-sided, essential circles.  Then there is  an ambient isotopy $\Phi:L\x I\to L$, compactly supported in $(\intr L)\x I$, such that $\Phi^{1}\o\beta=\alpha$.
\end{theorem}

\begin{theorem}[Epstein]\label{3.1}
Let $\alpha,\beta:[0,1]\to L$ be properly embedded arcs with the same endpoints which are homotopic modulo the endpoints.   Then there is  an ambient isotopy $\Phi:L\x I\to L$, compactly supported in $L\x I$ with $\Phi(x,t)=x$  for $(x,t) \in\bd L\x I$, such that $\Phi^{1}\o\beta=\alpha$.
\end{theorem}

Epstein proves these first in the PL category, then extends them to the category TOP by an approximation argument.  The PL argument adapts very well to the smooth category DIFF, giving the following.

\begin{theorem}\label{epsteinsmooth}
In the above theorems, if the curves in question are smooth embeddings, then the isotopies can be chosen to be smooth.
\end{theorem}

\begin{rem}\label{trick}
One will frequently want to perform these isotopies sequentially on a possibly infinite sequence of disjoint curves.  In order to make sure that already completed isotopies are not undone by a subsequent one, one resorts to the following inductive trick.  If $\sigma_{1},\sigma_{2},\dots,\sigma_{n}$ are the resulting curves from isotopies already performed and $\tau$ is a curve to be isotoped to $\sigma_{n+1}$, both of these curves being disjoint from $S=\sigma_{1}\cup\sigma_{2}\cup\cdots\cup\sigma_{n}$, then cut $L$ apart temporarily along $S$ and perform the Epstein-Baer isotopies in the resulting surface $L'$.  Of course, we are assuming that $\tau$ and $\sigma_{n+1}$ lie in a common component of $L'$.  
\end{rem}

\subsection{Standard hyperbolic metrics}

In order to construct and analyze the Handel-Miller laminations, it is necessary to introduce a hyperbolic metric on $L$.  Here we give necessary definitions and prove a key result (Theorem~\ref{essc}). In~\cite{cc:epstein} we generalized to arbitrary noncompact  surfaces, with suitable hyperbolic metric, theorems we need that are well-known for hyperbolic surfaces of  finite area. 

Let  $L$  be a  connected surface, compact or not, with or without boundary and orientable or not. If $L$ has boundary and a complete hyperbolic metric making all components of $\bd L$ geodesics, it is well known that the double $2L$ has a canonical hyperbolic metric which is complete and agrees with the given one on $L\ss2L$. 

If $L$ has empty boundary,  the open unit disk $\Delta$, with its canonical hyperbolic metric, is the universal cover $\wt L$.  If $\bd L\ne\0$, view $L\ss2L$, identify $\wt{2L}=\Delta$ and choose a lift $\wt L\ss\Delta$ of $L$.  The projection $\pi:\Delta\to2L$ restricts to give the universal cover $\pi:\wt L\to L$.  The lifts  $\wt L$ are permuted transitively by the group of deck transformations of $2L$, giving all the choices of  embeddings of the universal cover of $L$ in $\Delta$.  Those deck transformations leaving a given $\wt L$ invariant restrict to define the group of deck transformations for $\pi:\wt L\to L$.

Let $\D^{2}=\Delta\cup\Si$, the closed unit disk with boundary $\Si$\label{s1inf} the unit circle, known as the ``circle at infinity''.   If $A\ss\D^{2}$, let $\wh A$ denote the closure of $A$ in $\D^{2}$. In particular, if  $\sigma\ss\Delta$ is a geodesic, $\wh\sigma$ will be the compact, properly embedded arc in $\D^{2}$ obtained by adjoining to $\sigma$ its endpoints in $\Si$.  This is called the completion of $\sigma$. We also let $\wh L$ denote the closure of $\wt L$ in $\D^{2}$.  This is called the completion of $\wt L$.

\begin{defn}[ideal boundary $E$]\label{idbd}
The set $E=\wh{L}\cap\Si$ is called the \emph{ideal boundary} of $L$ or of $\wt L$. 
\end{defn}

Either $\bd L=\0$ and $E=\Si$, or $\bd L\ne\0$ and $E$ is a compact subset of $\Si$.  Note that the completions  $\wh\gamma$ of the components $\gamma$ of $\bd\wt L$ have endpoints in $E$.

\begin{defn}[standard metric, standard surface]\label{standard}
A Riemannian metric $\mu$ on $L$ is \emph{standard} if it is a complete hyperbolic metric making all components of $\bd L$ geodesics  and admitting  no isometrically embedded, hyperbolic  half-planes (for short, no half-planes).  The pair $(L,\mu)$ is called a \emph{standard hyperbolic surface}. If $L$ is a surface admitting a standard metric $\mu$, we say that $L$ is a \emph{standard surface}.
\end{defn}

This is nonstandard terminology.  

By a hyperbolic half-plane, we mean, of course, the union in the Poincar\'e disk $\Delta$ of a geodesic $\gamma$ and one of the components of $\Delta\sm\gamma$.  In~\cite[Theorem~8]{cc:epstein}, we prove that, up to homeomorphism, there are exactly   13 nonstandard surfaces.  None of these are interesting from the point of view of endperiodic theory.  

In~\cite[Lemma~1]{cc:epstein} we show that the hyperbolic metric on  $2L$ is standard  if and only if the hyperbolic metric on $L$ is standard.

\begin{hyp}\label{hypstan}
\textbf{In this paper, all hyperbolic metrics are standard, save mention to the contrary.}
\end{hyp}

\subsection{Escaping curves}

It is well known that if $\sigma\ss L$ is an essential closed curve  which does not bound a cusp, then it is freely homotopic to a unique closed geodesic $\sigma^{\g}$. Similarly, if $\sigma$ is a boundary incompressible, properly embedded arc, it is homotopic to a unique, properly embedded geodesic arc $\sigma^{\g}$, where the homotopy keeps  the endpoints fixed.

 \begin{defn}[geodesic tightening]\label{geodtite}
 In either of the above cases, we say that $\sigma^{\g}$ is the geodesic tightening of $\sigma$. 
 \end{defn}
 
\begin{defn}[escapes]\label{seqesc}
Let $\{A_{k}\}$, be a sequence of subsets of $L$, indexed either by the nonnegative integers $k\ge0$, the nonpositive integers $k\le0$, or all integers $k\in\Z$.  We say that $\{A_{k}\}$   \emph{escapes} if, for every compact subset $K\ss L$,  $A_{k}\cap K=\0$ for all but finitely many values of $k$.  
  \end{defn}
 
   \begin{example}
If $e$ is a positive  end and  $J$ is an $f$-juncture corresponding to $e$,   then $\{f^{k}(J)\}_{k\ge0}$ escapes, but, except in trivial cases, $\{f^{k}(J)\}_{k\le0}$ does not escape.  A similar remark, with opposite signs, applies to negative ends.  
\end{example}

 Note that, by Hypothesis~\ref{noperpt}, if an $f$-juncture $J$ is an arc, then the countable set of endpoints of arc components of the set of $f$-junctures $\{f^{k}(J)\}_{k\in\Z}$ escapes.

\begin{defn}[virtually escapes]\label{vescapes}
Let $\{\sigma_{k}\}$ be a sequence of essential  simple loops not bounding a cusp or  a sequence of properly embedded, boundary incompressible arcs,  indexed either by the nonnegative integers $k\ge0$, the nonpositive integers $k\le0$, or all integers $k\in\Z$. We say that $\{\sigma_{k}\}$ \emph{virtually escapes}  if the sequence  of geodesic tightenings $\{\sigma^{\g}_{k}\}$ escapes.   
\end{defn}

\begin{example}
A hyperbolic half-plane  causes behavior we need to exclude.  In Figure~\ref{hypsemiplane}, suppose $f$ is a translation as indicated. The sequence of positive iterates under $f$ of an $f$-juncture escapes but does not virtually escape. Their geodesic tightenings (see Definition~\ref{geodtite}), which cannot intersect the hyperbolic half-plane $H$, accumulate on $\bd H$.  The $f$-junctures are depicted by dashed circles and the geodesic tightening of one of them is depicted by the boldfaced dashed circle.
\end{example}

\begin{figure}[ht]
\begin{center}
\begin{picture}(300,200)(-70,-10)

\includegraphics[width=160pt]{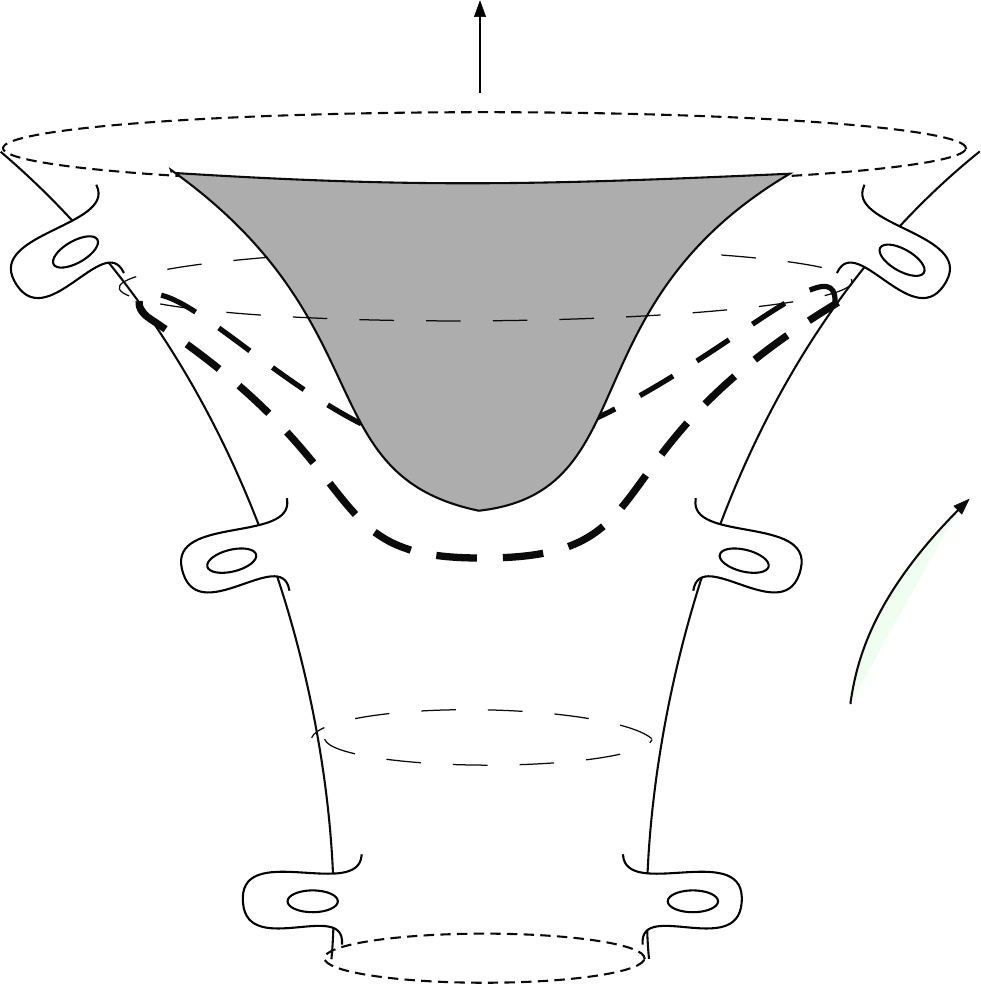}
\put(-85,90){\small$H$}
\put(-83,167){\small$e$}
\put(-20,70){\small$f$}

\end{picture}
\caption{A half-plane $H$ prevents the escaping sequence $\{f^{k}(J)\}_{k\ge0}$ of $f$-junctures from virtually escaping}\label{hypsemiplane}
\end{center}
\end{figure}

\begin{theorem}\label{essc}
In a standard hyperbolic surface $L$, if $\{\sigma_{k}\}$ is a sequence of essential  simple loops not bounding a cusp or a sequence of properly embedded, boundary incompressible arcs and $\{\sigma_{k}\}$  escapes, then $\{\sigma_{k}\}$ virtually escapes.
\end{theorem}

\begin{rem}
The converse is obviously false.   
\end{rem}

Evidently, it will be enough to carry out the proof for the case of a sequence $\{\sigma_{k}\}_{k\ge0}$.

Theorem~\ref{essc} will be proven in a series of lemmas.  We will suppose that the assertion fails and show that $L$ then contains an isometrically embedded hyperbolic half-plane.  This will contradict our assumption that $(L,\mu)$ is a standard hyperbolic surface. 

Fix a realization of the universal cover $\wt L\sseq\Delta$.  

Suppose that  the sequence $\{\sigma_{n}\}$ escapes but  does not virtually escape.  Let $\sigma^{\g}_{n}$ be the geodesic tightening of $\sigma_{n}$. Then there is a compact set $K$ such that infinitely many of the  $\sigma^{\g}_{n}$'s meet $K$. Passing to a subsequence, assume that they all do.  

If only finitely many of the $\sigma^{\g}_{n}$ are distinct, then there exists a subsequence of the sequence $\{\sigma_{n}\}$ consisting of freely homotopic loops.  It follows that $L$ has an end $e$ with a neighborhood homeomorphic to an open annulus. Then $e$ is either a  cusp bounded by all the $\sigma_{n}$ in the subsequence contrary to the hypotheses of the theorem or a flaring end  contrary to Hypothesis~\ref{hypstan}. Thus, infinitely many of the $\sigma^{\g}_{n}$ are distinct and, by passing to a subsequence, we assume they all are.  Since the sequence $\{\sigma_{n}\}$ escapes and each $\sigma_{n}$ is compact, by passing to a subsequence, we can further assume that the $\sigma_{n}$ are pairwise disjoint.

Pick $y_{n}\in \sigma^{\g}_{n}\cap K$. Then $\{y_{n}\}$ has a subsequence that converges to $x\in K$ and we reindex so that this subsequence is indexed by $n\in\Z^{+}$.  Let $D$ be a disk neighborhood of $x$. 
  
  Let $s$ be a geodesic arc through $x$ in $D$ which meets every $\sigma^{\g}_{n}$ transversely.  On at least one side of $x$ in $s$, infinitely many $\sigma^{\g}_{n}$ intersect $s$ in a closest point $x_{n}$ to $x$ and, again passing to a subsequence if necessary, we can assume that $x_{n}\to x$ monotonically as $n\to\infty$.     Fix a lift $\wt{x}\in\wt{L}$ of $x$. This fixes lifts $\wt D$, $\wt s$, $\wt x_{n}$ and  $\wt\sigma^{\g}_{n}$. If $\bd L=\0$, then $\bd\wh{L}=\Si$, and if $\bd L\ne\0$, then $\bd\wh{L}$ is a fractal curve homeomorphic to a circle. 

 \begin{defn}[$\Se$]\label{se}
  Whether or not $\bd L=\0$, the circle $\bd\wh L$ will be denoted by $\Se$.
  \end{defn}
  
 Each $\wt\sigma^{\g}_{n}$ cuts off two regions in $\wt L$, one of which contains  $\wt x$ and will be denoted by $\wt H_{n}$.  The $\sigma_{n}$ were chosen above to be disjoint. It follows that the $\wt\sigma^{\g}_{n}$ are disjoint and have no common endpoints on $\Se$. If the $\sigma^{\g}_{n}$ are arcs this is clear and if the $\sigma^{\g}_{n}$ are closed loops this follows since distinct closed geodesics in hyperbolic geometry can not share ideal points at infinity.  Therefore, $\wt H_{n+1}\ss\wt H_{n}$, a proper inclusion for all $n\ge0$.  The following, then, is evident.

\begin{lemma}\label{largen}
The endpoints $a_{n},b_{n}$ of\, $\wh\sigma^{\g}_{n}$ lie in $\bd\wh{L}$ and are the endpoints of a nested sequence of arcs,
$$\Se\supset[a_{0},b_{0}]\supset[a_{1},b_{1}]\supset\cdots\supset[a_{n},b_{n}]\supset\cdots,$$ where each $[a_{n},b_{n}]\ss(a_{n-1},b_{n-1})$. Consequently, the sequences $\{a_{n}\}, \{b_{n}\}$ converge strictly monotonically to points $a, b\in\bd\wh{L}$. 
\end{lemma}

\begin{lemma}\label{anotb}
The points $a,b$ are in $E$ and distinct.
\end{lemma}

\begin{proof}
The $\wt\sigma^{\g}_{n}$ are geodesics in $\wt L$ passing through an arbitrarily small neighborhood of $\wt x$.  By elementary properties of the Poincar\'e model, the endpoints cannot be converging to the same point on $\Se$.  If all $\sigma^{\g}_{n}$ are closed, all $a_{n},b_{n}\in E$, hence $a,b\in E$.  If they are properly embedded arcs, the sequences of iterates of their endpoints escape, and so neither $a_{n}$ nor $b_{n}$ can be approaching a point on $\bd\wt L$.
\end{proof}

The points $a,b\in E$ are endpoints of a geodesic $\wt\gamma\ss\wt{L}$ which projects to a geodesic $\gamma\ss L$ containing the point $x$.

\begin{lemma}\label{gammasimple}
The geodesic $\gamma$ is simple.
\end{lemma}

\begin{proof}
If $\gamma$ is not simple, another lift $\ol\gamma$ properly intersects $\wt\gamma$.  Then the sequence of  lifts $\{\wt\sigma_{n}\}$ converging on $\wt\gamma$ and corresponding sequence of lifts $\{\ol\sigma_{n}\}$ converging on $\ol\gamma$ will force proper intersections $\wt\sigma_{n}\cap\ol\sigma_{n}$. This forces $\sigma_{n}$ to self intersect for $n$ large enough.
\end{proof}

\begin{lemma}\label{gammanotclosed}
The geodesic $\gamma$ is not closed.
\end{lemma}

\begin{proof}
If $\gamma$ is closed, then $\wt{\gamma}$ is the axis of a deck transformation $T$. Passing to $T^{2}$, if necessary, assume that $T$ is orientation preserving.  Recall that the endpoints of $\wt{\gamma}$ are $a,b\in\Si$. Assume that $b$ attracts and $a$ repels. Let $\wt{\sigma}^{\g}_{n}$ be the lift of the geodesic tightening of $\sigma_{n}$ with endpoints $a_{n},b_{n}\in\Se$. The sequence   $\{\wt{\sigma}^{\g}_{n}\}$    converges uniformly to $\wt{\gamma}$ in the Euclidean metric on $\Delta$. The point $T(a_{n})$ is further from $a$ on the circle $\Se$ and the  point $T(b_{n})$ is closer to $b$ on the circle $\Se$. Thus, for $n$ sufficiently large, $\wt{\sigma}^{\g}_{n}\cap T(\wt{\sigma}^{\g}_{n})\ne\0$. This implies that $\sigma^{\g}_{n}$ is not a simple curve contrary to assumption. 
\end{proof}

Let $H\ss\Delta$ be the half-plane with boundary $\wt\gamma$ that is disjoint from every $\wt \sigma^{\g}_{k}$.
The intersection $\wh H\cap\Si$ is a compact, nondegenerate subarc $A\ss\Si$. We will denote this subarc $A$ by $[a,b]_{\infty}$ to distinguish it from the arc $[a,b]\ss \bd\wh L$ which are different if $\bd L\ne\0$.  

\begin{lemma}\label{HL}
$H\ss\wt{L}$
\end{lemma}

\begin{proof}
The lemma is trivial if $\bd L = \0$ so we assume that $\bd L\ne\0$. Let $\tau$ be the geodesic in $\Delta$ which contains the geodesic $\wt\sigma_{0}^{\g}\ss\wh L$. Let $H^{*}$ be the half-plane in $\Delta$ with boundary $\tau$ containing the half-plane $H$. We will work entirely in $H^{*}$. Since the sequence $\{\sigma_{n}\}$ escapes and $\sigma_{0}^{\g}$ is compact, for $n$ sufficiently large, $\sigma_{n}\cap\sigma_{0}^{\g}=\0$ so $\wt\sigma_{n}\cap\wt\sigma_{0}^{\g}=\0$. Since the endpoints $a_{n},b_{n}$ of $\wt\sigma_{n}$ lie in $H^{*}$ it follows that $\wt\sigma_{n}$ lies in $H^{*}$ for $n$ sufficiently large. We restrict ourselves to such sufficiently large $n$.

Let $S_{\epsilon}\ss\Delta$ be the Euclidean circle, concentric with $\Si$, of Euclidean radius $1-\epsilon$, $0<\epsilon<1$.  This circle and $\Si$ cobound an annulus $V_{\epsilon}$ in $\D^{2}$. Let $D_{\epsilon}\ss\Delta$ denote the closed disk of radius $1-\epsilon$ so $\bd D_{\epsilon}=S_{\epsilon}$. For a fixed $\epsilon>0$, if infinitely many $\wt\sigma_{n}$'s intersect the disk $D_{\epsilon}$, then projecting down into $ L$ produces infinitely many $\sigma_{n}$'s meeting the compact set $p(D_{\epsilon})\cap L$, contradicting the fact that the sequence $\{\sigma_{n}\}$ escapes. Thus, for $n$ sufficiently large,  $\wt\sigma_{n}\cap D_{\epsilon}=\0$ so $\wt\sigma_{n}\ss V_{\epsilon}$.

Recall that $[a,b]_{\infty}$ denotes the interval $\ol H\cap\Si$ in $\Si$ with endpoints $a,b$. We are required to show that no   boundary component of $L$ can   lift to an arc $\alpha$ issuing transversely from $(a,b)_{\infty}$. If such an arc $\alpha$ exists then there is an $\epsilon>0$ such that $S_{\epsilon}\cap H\cap\alpha\ne\0$. By the previous paragraph there is a $\wt\sigma_{n}$ (in fact infinitely many) meeting $\alpha$ which is a contradiction. Since no lift of a boundary component of $L$ can have endpoint in $(a,b)_{\infty}$ it follows that $H\ss\wh L$.
\end{proof}

\begin{lemma}\label{Lproj}
The projection $\pi:\wt{L}\to L$  embeds $\intr H$ isometrically in $L$. Thus, $L$ contains an embedded  half-plane.
\end{lemma}

\begin{proof}
If an endpoint of an axis of a deck transformation lies in $[a,b]_{\infty}$, then, as in the proof of Lemma~\ref{HL}, that axis meets infinitely many $\wt\sigma_{n}$. This is a contradiction since the sequence $\{\sigma_{n}\}$ escapes and the axis projects to a compact subset of $L$.  Thus the axis of every deck transformation $\psi$ has both endpoints outside of $[a,b]_{\infty}$. If $(a,b)_{\infty}$ properly overlaps its $\psi$-image for some deck transformation $\psi$, then $\gamma$ has two distinct lifts that intersect. Thus $\gamma$  intersect itself contradicitng Lemma~\ref{gammasimple}.  If there is a deck transformation that takes $(a,b)_{\infty}$ to itself then $\wh\gamma$ is that axis. Lemma~\ref{gammanotclosed} rules out  that possibility.  Thus, the images of $(a,b)_{\infty}$ under the deck transformations are disjoint and the projection $\pi:\wt{L}\to L$ embeds $\intr H$ isometrically in $L$ so  $L$ contains a half-planes.
\end{proof}

Theorem~\ref{essc} is proven. 

Our definition of virtually escaping is metric dependent.  We show that, for standard hyperbolic metrics, the notion of virtually escaping is independent of choice of metric.

\begin{cor}\label{indepofchoice}
If  $\mu_{i}$ are standard hyperbolic metrics on $L$, $i=1,2$, then the sequence $\{\sigma_{k}\}$ virtually escapes relative to $\mu_{1}$ if and only if it virtually escapes  relative to $\mu_{2}$.
\end{cor}

\begin{proof}
Let $\sigma_{k}^{\g}$ be the tightening of $\sigma_{k}$ to a $\mu_{1}$-geodesic, $\sigma_{k}^{\g*}$  its $\mu_{2}$-tightening.  Then $\sigma_{k}^{\g*}$ is the $\mu_{2}$-tightening of $\sigma_{k}^{\g}$ which, in turn, is the $\mu_{1}$-tightening of  $\sigma_{k}^{\g*}$. By Theorem~\ref{essc}, the sequence $\{\sigma_{k}^{\g}\}$ escapes  if and only if the sequence $\{\sigma_{k}^{\g*}\}$  escapes.
\end{proof}

\begin{rem}
Theorem~\ref{essc} is used in the proof of Theorem~\ref{esctoe} whose most important application is to ensure there are no spurious leaves in the geodesic laminations, constructed in Section~\ref{HMconstruct}, that are associated to the endperiodic automorphism $f$.

\end{rem}

 
\section{The Laminations and  endperiodic automorphism preserving them}\label{constr}

In the Handel-Miller theory, given an endperiodic automorphism $f:L\ra L$, one tightens the $f$-junctures to geodesic junctures, and uses these to construct a pair of transverse geodesic laminations and an endperiodic automorphism $h:L\ra L$, preserving the geodesic laminations  and  the geodesic junctures. This is reminiscent of the Nielsen-Thurston theory for automorphisms of compact surfaces. 

\subsection{Laminations}\label{LAMS}

Roughly speaking, a $p$-dimensional lamination of an  $n$-manifold $M$ is a  foliated subset of leaf dimension~$p$. The codimension is $q=n-p$. The leaves are one-one immersed $p$-dimensional submanifolds, but transverse $q$-disks intersect the lamination in relatively closed  subsets which may  be quite messy and typically are totally disconnected.  

In this paper we are interested in laminations of surfaces by curves.  We give careful definitions here in that setting, but everything works for laminations of  arbitrary dimension~$p$ in manifolds of arbitrary dimension $n>p$. (cf.~\cite[pp.~404-405]{ms} for codimension~$1$ laminations of $n$-manifolds.  This easily adapts to codimension $n-p$.) We are interested in surfaces with boundary and in laminations with some leaves transverse to the boundary.  We will use ``laminated charts'' in analogy with ``foliated charts'' in foliation theory.  Since we allow boundary, these charts will either be homeomorphic to $(a,b)\x(c,d)\ss\R^{2}$ or to $(a,0]\x(c,d) \ss\R^{2}_{-}$, where $\R^{2}_{-}=(-\infty,0]\x\R$ is the Euclidean half plane.  We let $\F^{2}$ denote either $\R^{2}$ or $\R^{2}_{-}$ and we denote by $I$ either $(a,b)$ or $(a,0]$ and $J=(c,d)$.  There is no reason to  require $I$ and $J$ to be finite intervals but there is also no good reason not to, so we assume $a,b,c,d\in \R$.

\begin{defn}[laminated chart]\label{lamchart}
A \emph{laminated chart} in a surface $L$ (possibly with boundary)  is a triple $(U,Y,\phi)$, where $U\ss L$ is open, $\phi$ is a homeomorphism of $U$ onto $I\x J\ss\F^{2}$ and $Y\sseq J$ is a relatively closed subset.
\end{defn}

Let $N\sseq L$ be a subset which is the union of a disjoint set $\Lambda=\{\lambda_{\beta}\}_{\beta\in\B}$ of one-one immersed, connected $1$-manifolds in $L$.

\begin{defn}[lamination]\label{lamination}
We say that $\Lambda$ is a \emph{lamination} of $L$ by curves if there is a set $\{(U_{\alpha},Y_{\alpha},\phi_{\alpha})\}_{\alpha\in\A}$ of  laminated charts such that $\{U_{\alpha}\}_{\alpha\in\A}$ covers $N$ and, for each $\alpha\in\A$, the path connected components of $\lambda\cap U_{\alpha}$, where $\lambda$ ranges over $\Lambda$, called \emph{plaques} of the laminated chart $(U_{\alpha},Y_{\alpha},\phi_{\alpha})$, are carried by $\phi_{\alpha}$ exactly onto the sets $I\x\{y\}$, $y\in Y_{\alpha}$.  The \emph{support} of the lamination is $|\Lambda|=N$ and each $\lambda\in\Lambda$ is called a \emph{leaf} of the lamination.  If $|\Lambda|$ is closed in $L$ we will say that the lamination is \emph{closed}.
\end{defn}

\begin{defn}[laminated (partial) atlas]\label{lamatlas}
The set $\AAA_{\Lambda}=\{(U_{\alpha},Y_{\alpha},\phi_{\alpha})\}_{\alpha\in\A}$ is called a \emph{laminated partial atlas}.  If $\{U_{\alpha}\}_{\alpha\in\A}$ covers $L$, we say that $\AAA_{\Lambda}=\{(U_{\alpha},Y_{\alpha},\phi_{\alpha})\}_{\alpha\in\A}$ is a \emph{laminated atlas}.
\end{defn}

\begin{rem}
$\AAA_{\Lambda}$ is a \emph{partial} atlas because it may not cover $L$.  If, $\Lambda$ is a \emph{closed} lamination, one can produce a laminated atlas.  For each point $x\not\in|\Lambda|$, choose a trivially laminated chart $(U,\0,\phi)$ about $x$ such that $U\cap|\Lambda|=\0$.  In this paper, we work mainly with closed laminations and associated laminated atlases.
\end{rem}

\begin{rem}
If a plaque $P$ of $U_{\alpha}$ meets a plaque $Q$ of $U_{\alpha'}$, then it is clear that $P\cap Q$ is an open subset of $P$ and of $Q$.  Without being given $N$ and $\Lambda$, one can define an abstract  ``laminated atlas'', by this property.  Then, as in the case of foliations, one can recover the leaves via chains of overlapping plaques.  A lamination can then be defined to be  an (abstract) laminated atlas. 
Note that, in our definition, two laminated atlases for the  same lamination have union a laminated atlas for that lamination.

Two (abstract) laminated atlases are said to be \emph{coherent} if their union is a laminated atlas.  This is an equivalence relation,  hence a lamination can be defined as a coherence class of laminated atlases.  The union of all atlases in the coherence class is the maximal element of the class and can also be identified as the lamination.

All of this is closely analogous to foliations and foliated atlases, a very detailed treatment of which will be found in~\cite[pp.~19-31]{condel1}.
\end{rem}

\begin{rem}
Note that, for a laminated chart $(U,Y,\phi)$, $\phi$ restricts to a homeomorphism $\psi:N\cap U\to I\x Y$.  Thus $N$, equipped with an ``atlas'' $\{N\cap U_{\alpha}, \psi_{\alpha}\}_{\alpha\in\A}$, satisfies the standard definition of a ``foliated space''~\cite[Chapter~11]{condel1}.  A lamination, therefore, is a topological embedding of a foliated space, but a very special embedding admitting continuous sets of local transversals. The lamination community is typically vague on this point, commonly defining a lamination as a foliated space embedded in a manifold as a closed subset, each leaf of which is complete in its path metric (cf.~\cite[Definition~6.9]{calegari}),  but for some delicate arguments in this paper that definition seems to be inadequate. Note also that laminations in our sense locally extend to foliations.
\end{rem}

In this paper we study a special class of $1$-dimensional laminations of surfaces $L$.  The transverse sets $Y_{\alpha}\ss \R^{1}$ will be totally disconnected.  As a result, the support $\left|\Lambda\right|$ determines the lamination and it is fairly customary to make no distinction between the lamination and its support.  We prefer to maintain the distinction throughout this paper.

\begin{defn}[transversely totally disconnected]\label{trtotdis} 
A lamination $\Lambda$ is \emph{transversely totally disconnected} if, for each of its laminated charts $(U_{\alpha},Y_{\alpha},\phi_{\alpha})$, the space $Y_{\alpha}$ is totally disconnected.
\end{defn}

All the  laminations we study will be transversely totally disconnected. 

\begin{rem}\label{longchart}
By abuse of notation, we may denote a laminated chart containing a plaque $P$ by $P\x (-1,1)$.  We can then denote by $P\x(-\delta,\epsilon)$ subcharts with $-\delta<0$ and $0<\epsilon$, identifying $P\x\{0\}$ with $P$.  The idea is that by making $\delta$ and $\epsilon$ sufficiently ``small'', we get arbitrarily thin normal neighborhoods of $P$.  By passing to the maximal laminated partial atlas, the plaque $P$ may be as ``long'' as desired, generally requiring that $\delta,\epsilon$ be sufficiently small.
\end{rem}

If $x\in \lambda\in\Lambda$ and if $\lambda$ is not isolated on at least one side, then points of $|\Lambda|$ accumulate on $x$ from a non-isolated side of $\lambda$.  The leaves containing these points accumulate \emph{locally uniformly} on $\lambda$ in the following sense.  

\begin{defn}[locally uniform accumulation]\label{locunif}
Let $\lambda\in\Lambda$.   We say that $\{\lambda_{\alpha}\}_{\alpha\in\A}\ss\Lambda$ \emph{accumulates locally uniformly} on $\lambda$ if, for every bounded subarc $P\ss\lambda$, there is a lamination chart $V_{P}\cong P\x(-\delta,\epsilon)$, having $P=P\x\{0\}\}$, such that  the set of plaques of $V_{P}\cap\bigcup_{\alpha\in\A}\lambda_{\alpha} = \{P\x\{t_{\beta}\}\}$  is a set of plaques accumulating uniformly on  $P$.  
\end{defn}

Remark that $\A$ might be a singleton, but the single leaf $\lambda_{\alpha}$ might accumulate locally uniformly on $\lambda$.  Indeed, $\lambda$ might accumulate locally uniformly on itself.  The laminations that we study in this paper typically exhibit such behavior.

\subsubsection{Bilaminations}

We will be studying a pair of mutally transverse laminations.  In the theory of foliations of $n$-manifolds, a pair of mutually transverse foliations $\FF$ and $\FF'$, of respective leaf dimensions $p$ and $q$, $p+q=n$, gives rise to a ``bifoliated atlas'', but for $\CO$ laminations, a proof of this eludes the authors.  We will need such a ``bilaminated (partial) atlas'', so we will define the term ``bilamination'' accordingly.  This will not get us into trouble since we are generalizing the Handel-Miller theory where the pair of mutually transverse geodesic laminations will be shown to define a bilamination.

In what follows, $I$ may be of the form $(a,b)$ or $(a,0]$ and $J$ may be of the form $(c,d)$ or $[0,d)$.  Thus $L$ might have boundary and corners.

\begin{defn}[bilaminated chart]\label{bichart}
A \emph{bilaminated chart} $(U,X,Y,\phi)$ is an open subset $U$ of $L$ together with a homeomorphism $\phi:U\to I\x J$, where $X\sseq I$ and $Y\sseq J$ are relatively closed subsets.
\end{defn}

Let $\Lambda,\Lambda'$ be a pair of  mutually transverse sets of disjoint, connected, one-one immersed  $1$-manifolds.  Denote by $|\Lambda,\Lambda'|$ the union $|\Lambda|\cup|\Lambda'|$.

\begin{defn}[bilamination]\label{bilam}
The pair $(\Lambda,\Lambda')$ is a  \emph{bilamination} if there is given a set $\AAA_{(\Lambda,\Lambda')}=\{(U_{\alpha},X_{\alpha},Y_{\alpha},\phi_{\alpha})\}_{\alpha\in\A}$ of bilaminated charts such that $\{U_{\alpha}\}_{\alpha\in\A}$ covers $|\Lambda,\Lambda'|$ and, for each $\alpha\in\A$, $(U_{\alpha},X_{\alpha},\phi_{\alpha})$ is a laminated chart for $\Lambda'$ and $(U_{\alpha},Y_{\alpha},\phi_{\alpha})$ is a laminated chart for $\Lambda$. We call $\AAA$ a \emph{bilaminated partial atlas}. If it covers $L$, we call it a \emph{bilaminated atlas}. The \emph{support} of the bilamination is $|\Lambda,\Lambda'|$.
\end{defn}

\begin{rem}
In particular, the bilaminated partial atlas can serve as a laminated partial atlas for each of $\Lambda$ and $\Lambda'$, hence these are laminations.
The laminations that make up a bilamination are mutually transverse and extend locally to a pair of transverse foliations.  If $\Lambda$ and $\Lambda'$ are each closed we say that the bilamination is closed. In this case,  the support $|\Lambda,\Lambda'|$ is closed and the bilaminated partial atlas can be extended to a bilaminated atlas by assigning bilaminated charts of the form $(U,\0,\0,\phi)$ as neighborhoods of $x\in\CC|\Lambda,\Lambda'|$ (here $\CC$ denotes ``complement'') where $U\cap|\Lambda,\Lambda'|=\0$.  
\end{rem}

It is sometimes useful to use closed bilaminated charts.  For these, one takes $I=[a,b]$ and $J=[c,d]$.  It is obvious that, for every $x\in |\Lambda,\Lambda'|$ and every bilaminated chart  containing $x$, there is a closed bilaminated subchart containing $x$. The proof of the following is elementary.

\begin{lemma}\label{Rchart}
Let $(\Lambda,\Lambda')$ be a bilamination and $R\ss L$ a compact, simply connected ``rectangle'' with top and bottom edges subarcs   of leaves of $\Lambda$ and right and left edges subarcs   of leaves of $\Lambda'$.  Then $R$ has a natural structure of a closed, bilaminated chart for $(\Lambda,\Lambda')$.
\end{lemma}

\begin{example}

Important examples of  bilaminations are given by a pair of mutually transverse, closed geodesic laminations  on a complete hyperbolic surface $L$.    Such an example is given by a pair $(\Lambda^{s},\Lambda^{u})$ of stable and unstable geodesic laminations in the Neilsen-Thurston theory (see Casson and Bleiler~\cite{bca}). Another example is given by the geodesic bilamination $(\Lambda_{+},\Lambda_{-})$ associated to an endperiodic automorphism $f$ constructed in  Section~\ref{lamconstr}.

\end{example}

\subsubsection{Smoothness}\label{smthnss}

The material in this subsection  is technical and not immediately important to our development  and can be skipped on first reading.  The authors are not aware of a generally agreed upon definition of ``smoothness'' for laminations.  We propose a definition which generalizes the standard definition for foliations.

The surface $L$ has a $C^{r}$ structure, $1\le r\le\infty$ or $r=\omega$ (a real analytic structure), namely a maximal $C^{r}$ atlas $\AAA_{r}$.  

\begin{defn}[smooth lamination, smooth bilamination]\label{Cilam}
The  lamination $\Lambda$ (respectively bilamination $(\Lambda,\Lambda')$) is \emph{smooth of class $C^{r}$} if it admits a laminated partial atlas $\AAA_{\Lambda}\ss\AAA_{r}$ (respectively a bilaminated partial atlas $\AAA_{\Lambda,\Lambda'}\ss\AAA_{r}$). If $r=\infty$, the lamination is just said to be \emph{smooth}.
\end{defn}

As usual, if the lamination or bilamination is closed, a partial laminated atlas of class $C^{r}$ extends to a laminated atlas on $L$ of class $C^{r}$.

\begin{rem}
It would not be equivalent to require only that the leaves be $C^{r}$ immersed curves. This is well known for foliations, but is also true for transversely totally disconnected laminations.   The problem is that gaps in a lamination may accumulate on a leaf, but because the lengths of the gaps at one end of a laminated chart may not compare uniformly to the lengths at the other, the mean value theorem can make it impossible to extend the lamination over the gaps to a foliation smooth  on the whole chart.  It is not  obvious whether the geodesic bilaminations $(\Lambda_{+},\Lambda_{-})$ associated to an endperiodic automorphism $f$ constructed in Section~\ref{HMconstruct} are smooth, but in general we think not. As the following example shows, a geodesic lamination of a hyperbolic surface need not be smooth of class $\CII$.
\end{rem}

\begin{example}\label{Denjoy}
Consider the Denjoy exceptional minimal set $\Lambda$ in $T^{2}$, a transversely totally disconnected lamination.  It is constructed by suspending a $\CI$ diffeomorphism of the circle which has  a minimal invariant Cantor set~\cite[Appendix]{sch:seifert}. While $\Lambda$  is $\CI$, it is not even  homeomorphic to a $\CII$ lamination.  This is classically known. To get this lamination on a hyperbolic surface, proceed as follows.
In the open subset of $T^{2}$ complementary to $|\Lambda|$, attach a handle, making the surface a $2$-holed torus $\Sigma_{2}$ and leaving $\Lambda$ intact.  One can then put a hyperbolic metric $\g$ on $\Sigma_{2}$.  Relative to this metric, each leaf $\lambda\in\Lambda$ is a pseudo-geodesic (Definition~\ref{pseudogeodesic}).  This is seen by noting that an essential circle $\sigma$ in $T^{2}$ which the oriented leaves of the Denjoy lamination periodically cross, always in the same direction, survives as such a circle $\sigma\ss\Sigma_{2}$. We can even take $\sigma$ to be a closed geodesic. For each lift $\wt\lambda$ of a leaf  $\lambda\in\Lambda$, suitable lifts of $\sigma$ crossed by $\wt\lambda$ form a nested family of geodesic arcs defining a unique point on $\Si$,  This gives one ideal endpoint of $\wt\lambda$ and the other endpoint is found similarly by reversing the orientation of $\lambda$.

Thus $\lambda\in\Lambda$ is homotopic to a unique geodesic $\lambda^{\g}$.  The collection of these geodesics is readily shown to be a lamination $\Lambda^{\g}$ of $\Sigma_{2}$.  It is highly doubtful that $\Lambda^{\g}$ is still $\CI$, but we claim that it certainly cannot be $\CII$.  Indeed, if it were $\CII$, the positively directed unit velocity field tangent to $\Lambda^{\g}$ extends locally to a $\CII$ vector field of unit length and these local extensions can be assembled by a $\CII$ partition of unity to get a nowhere zero vector field $v$ defined on an open set $U$ containing $|\Lambda^{\g}|$ and agreeing with the unit velocity field along the leaves of $\Lambda^{\g}$.  This defines a $\CII$ foliation $\FF$ on $U$  which incorporates $|\Lambda^{\g}|$ as a compact $\FF$-saturated subset.  The holonomy pseudogroup of $\FF$ is a (possibly infinitely generated) $\CII$ pseudogroup on $\R$ defined by a foliated atlas $\AAA_{\FF}$ for $\FF$ (e.g., see~\cite[Section~2.2]{condel1} for details).  A finitely generated sub-pseudogroup $\Gamma$ is obtained from a finite subcover of $|\Lambda^{\g}|$ by charts of $\AAA_{\FF}$ and $|\Lambda^{\g}|$ defines a $\Gamma$-invariant Cantor set $X\ss\R$.  A deep result of R.~Sacksteder~\cite[Theorem~1]{sa:pseudo} provides a point $x\in X$ and an element $\gamma\in \Gamma$ such that $\gamma(x)=x$ and $0<\gamma'(x)<1$.  But the leaves of $\Lambda^{\g}$ are simply connected, implying the germinal holonomy group at $x$ (cf.~\cite[Section~2.3]{condel1}) is trivial.  Thus $\gamma'(x)=1$ and this  contradiction proves the claim that $\Lambda^{\g}$ is not $\CII$.
\end{example}

Without smoothness, it is not clear that a pair of transverse laminations is a bilamination.  However, we have the following.

\begin{lemma}\label{cibilam}
If $\Lambda$ and $\Lambda'$ are closed, mutually transverse, $\CI$ laminations, then $(\Lambda,\Lambda')$ is  a bilamination.
\end{lemma}

\begin{proof}
If $x\in L\sm(|\Lambda|\cup|\Lambda'|)$, then there is a $\CI$ coordinate chart $(U,\phi)$ about $x$ such that $U\cap(|\Lambda|\cup|\Lambda'|)=\0$.  We can view $(U,\0,\0,\phi)$ as a bilaminated chart for $(\Lambda,\Lambda')$.  If $x\in|\Lambda|\sm|\Lambda'|$, there is a $\CI$ laminated chart $(U,Y,\phi)$ for $\Lambda$ about $x$ such that $U\cap|\Lambda'|=\0$, hence we can view $(U,Y,\0,\phi)$ as a bilaminated chart for $(\Lambda,\Lambda')$. We argue similarly for $x\in|\Lambda'|\sm|\Lambda|$.  For the case that $x\in|\Lambda|\cap|\Lambda'|$, the $\CI$  hypothesis becomes important.    There are two $\CI$ foliations defined in a rectangular neighborhood $R$ of $x$, $\FF$ having  plaques of $\Lambda$ among its leaves, and $\FF'$ incorporating the plaques of $\Lambda'$.  They may not be transverse throughout $R$.  Let $v$ be a continuous, nowhere vanishing vector field tangent to $\FF$, $v'$ such a field tangent to $\FF'$.  At $x$, the vectors $v_{x}$ and $v'_{x}$ are linearly independent, hence the fields, being continuous, are linearly independent at every point of a small enough neighborhood of $x$.  Thus, the foliations are transverse in a smaller rectangular neighborhood of $x$.  These foliations can be taken as a smooth coordinate grid about $x$ defining a bilaminated chart.
\end{proof}

\subsection{A heuristic example}\label{heur}

 Before presenting the formal construction of the geodesic laminations associated to an endperiodic automorphism, we present an example at an intuitive level.

\begin{figure}[h]
\begin{center}

\begin{picture}(300,120)(-15,-90)
\rotatebox{270}{\includegraphics[width=80pt]{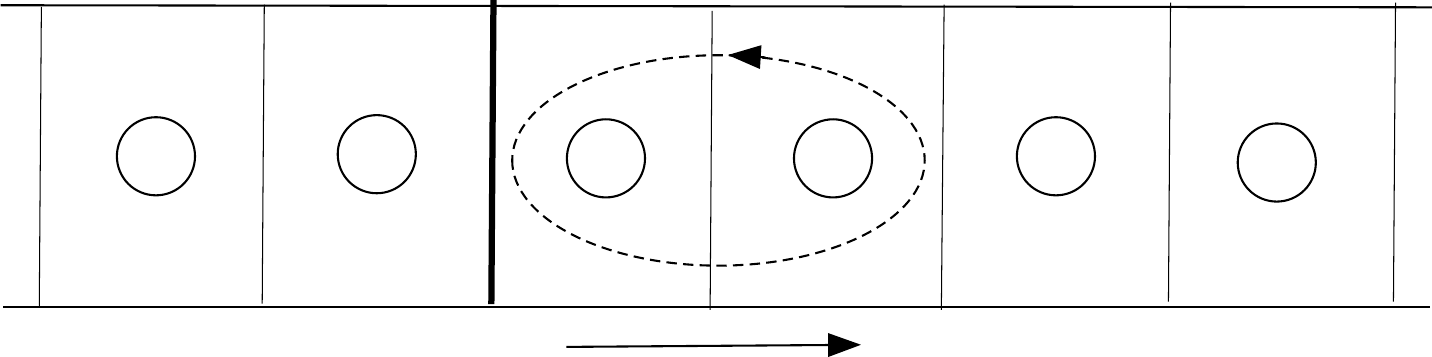}}\break
\put(-208,-60){\small$J$}
\put(-200,-20){\small$\tau$}
\put(-180,-90){\small$f=\tau\o g$}
\end{picture}

\begin{picture}(300,100)(-12,-90)
\rotatebox{270}{\includegraphics[width=78pt]{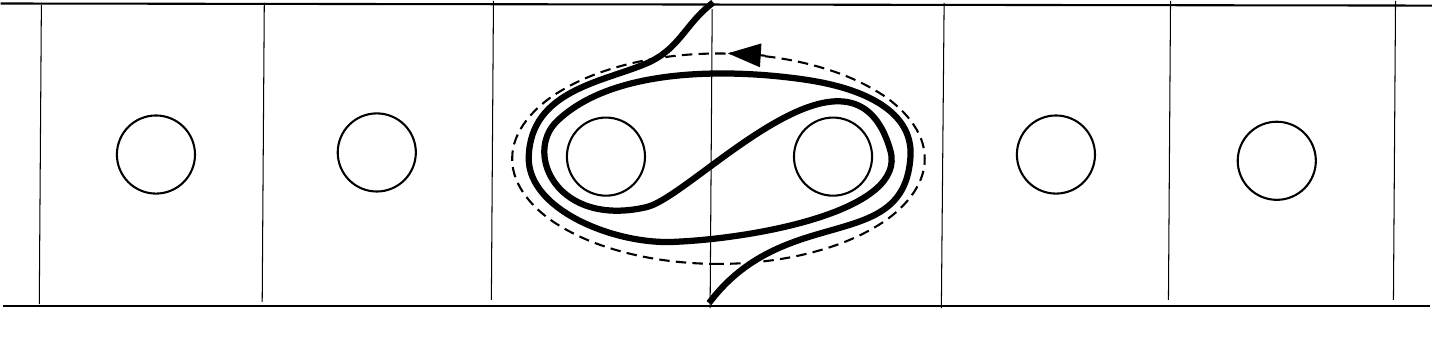}}\break
\put(-150,-65){\small$f(J)$}
\end{picture}

\begin{picture}(300,100)(-15,-90)
\rotatebox{270}{\includegraphics[width=70pt]{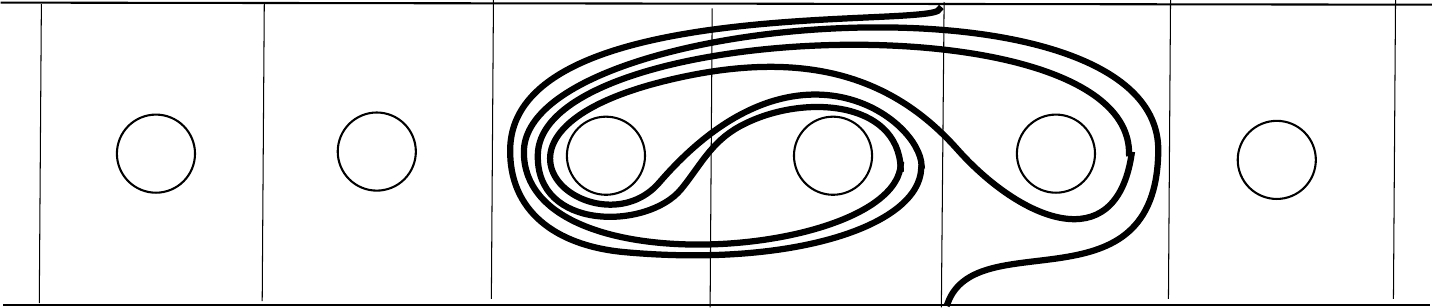}}\break
\put(-80,-65){\small$f^{2}(J)$}
\end{picture}

\caption{Juncture $J$ and geodesic tightenings of  $f(J)$, $f^{2}(J)$}\label{ItJunc}
\end{center}
\end{figure}

\begin{example}\label{simpex0}
 Let $L$ be a two-ended strip with disks removed approaching both ends, $g$ a translation of the strip from left to right which is an isometry relative to a standard hyperbolic metric quasi-isometric to the Euclidean metric implicit in Figure~\ref{ItJunc}.  Let $\tau$ be a Dehn twist in the dashed curve as indicated   in  Figure~\ref{ItJunc} (top). Let the endperiodic automorphism $f=\tau\o g$. A geodesic juncture $J$ is drawn on the left in  Figure~\ref{ItJunc} (top). The Handel-Miller theory studies endperiodic automorphisms by studying the limit laminations of the geodesic tightenings of the distorted $f$-junctures $f(J), f^{2}(J),\ldots$.  We call them ``distorted'' because under iterated applications of $f$ they become longer and longer without a finite upper bound, ``looking'' less and less like junctures and more and more (to a myopic observer) like noncompact leaves of a lamination. The first two distorted geodesic junctures are drawn in  Figure~\ref{ItJunc} (middle and bottom).  The sequence of geodesic tightenings of distorted $f$-junctures accumulates  on  the positive geodesic limit lamination $\Lambda_{+}$ carried by the traintrack in Figure~\ref{ItTT} (top) and disjoint from every  $f$-neighborhood of the negative end. 
 By using $f^{-1}$ instead of $f$, an analogous negative geodesic lamination $\Lambda_{-}$, transverse to $\Lambda_{+}$, is constructed, disjoint from every   $f$-neighborhood of the positive end.
 
 A rigorous presentation of the Handel-Miller construction is given in Section~\ref{HMconstruct}.  
 
 We will replace the endperiodic automorphism $f$ with an isotopic one $h$ which preserves the laminations. 
Figure~\ref{ItTT} (top) represents a traintrack $T$ which carries the positive geodesic lamination $\Lambda_{+}$ for this endperiodic automorphism. Figure~\ref{ItTT} (bottom) represents $h(T)\ss T$.   We leave it as an exercise for the reader to verify that these figures are correct. Note that $h(T)$ is obtained from $T$ by blowing air from $A$ to $A$. Similarly $h^{2}(T)\ss h(T)\ss T$  is obtained by blowing air from $B$ to $B$.  Once one sees $h(T)$, it is much easier to see $h^{2}(T), h^{3}(T),\ldots$. The positive lamination itself is $\Lambda_{+}= \bigcap_{n=1}^{\infty}h^{n}(T)$.  This example exhibits behavior that is typical and will be visited again in Example~\ref{simpex}.

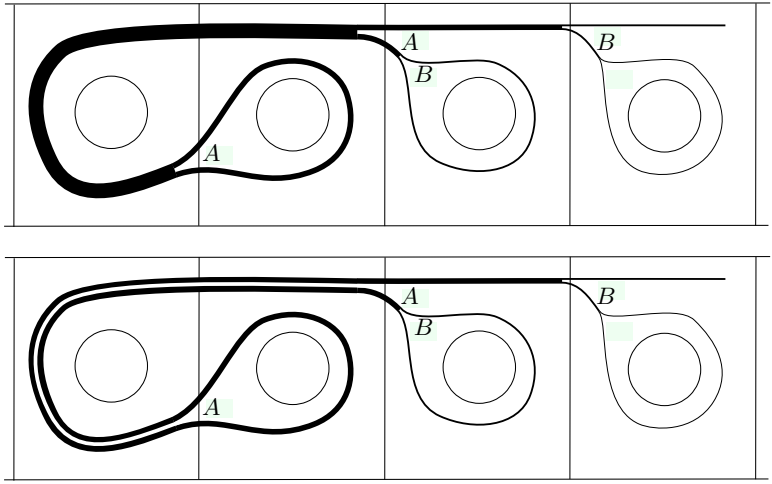
\begin{figure}[ht]
\begin{center}

\begin{picture}(300,90)(-5,-85)
\rotatebox{270}{\includegraphics[width=85pt]{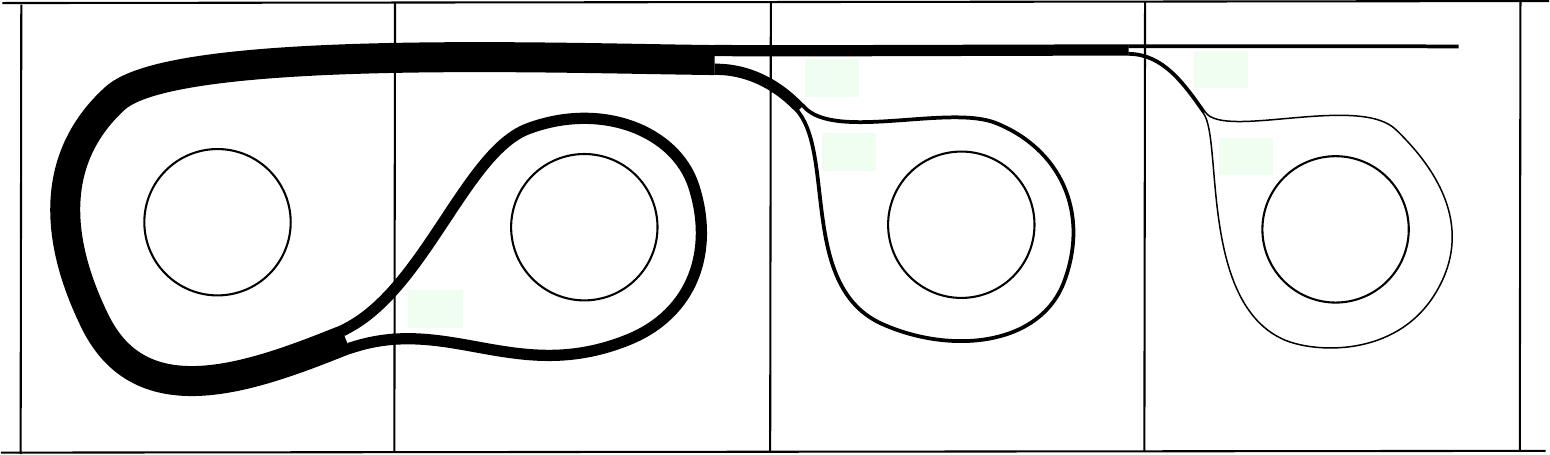}}
\put(-215,-60){\small$A$}
\put(-140,-18){\small$A$}
\put(-135,-30){\small$B$}
\put(-66,-18){\small$B$}
\end{picture}

\begin{picture}(300,90)(-5,-80)
\rotatebox{270}{\includegraphics[width=85pt]{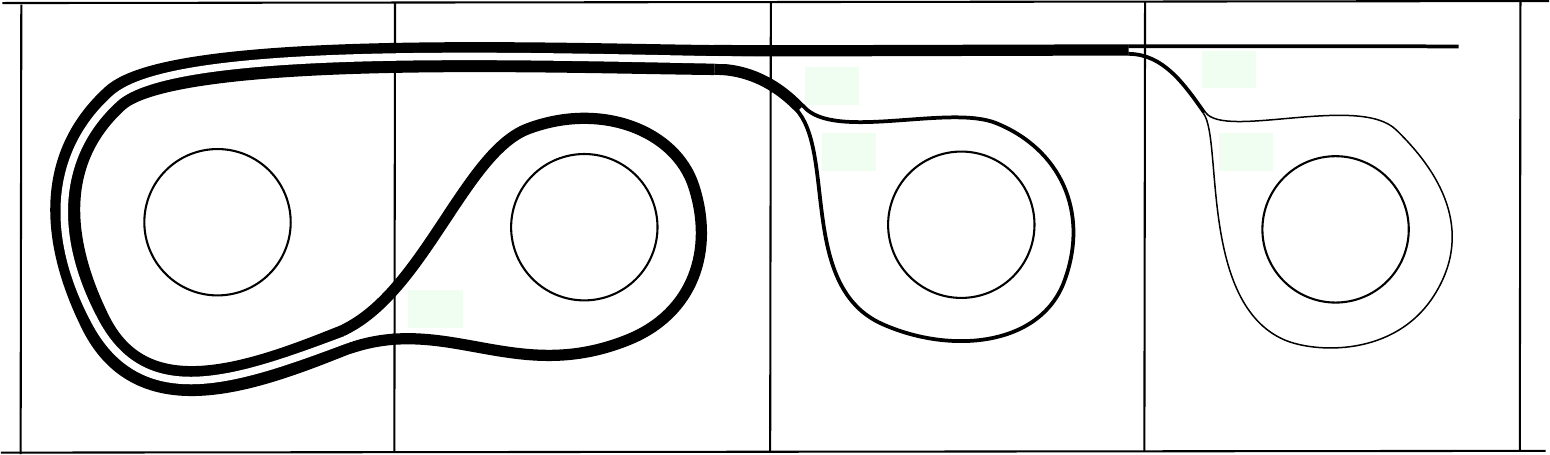}}
\put(-215,-60){\small$A$}
\put(-140,-18){\small$A$}
\put(-135,-30){\small$B$}
\put(-66,-18){\small$B$}

\end{picture}

\caption{The traintracks $T$, $h(T)\ss T$, etc.}\label{ItTT}
\end{center}
\end{figure}

\end{example}

\subsection{The Handel-Miller construction}\label{lamconstr}\label{HMconstruct}

Here we give the simple, classical and unpublished construction of Handel and Miller of the geodesic laminations associated to a given endperiodic automorphism $f:L\to L$. 

\begin{defn}[admissible]\label{admi}
 An \emph{admissible surface} is a standard hyperbolic surface with finitely many ends, none of which are simple (Definition~\ref{simpend}).
 \end{defn}
 
 For an admissible surface $L$, relative to an endperiodic automorphism, $\EE(L) = \EE_{-}(L)\cup\EE_{+}(L)$.
  
 \begin{lemma}\label{densinE}
 If $L$ is an admissible surface and $e\in E$, then the orbit of $e$ under the group of covering transformations    is dense in $E$. 
 \end{lemma}
 
 Indeed, it is standard that  a covering transformation $\psi$ induces a homeomorphism $\wh\psi:E\to E$.   Admissible surfaces satisfy the hypotheses of~\cite[Corollary~1]{cc:epstein}, proving the lemma.

\begin{hyp}\label{hypadmiss}
 \textbf{Hereafter, $L$ is an admissible surface.} 

\end{hyp}

Remark that an admissible surface has no cusps, hence there are no parabolic deck transformations of the universal cover.  All deck transformations are either hyperbolic or are orientation reversing and  have square a hyperbolic transformation.  In any case, they have a unique geodesic axis in $\wt L$.

\begin{defn}[geodesic lamination]\label{geolam}
A lamination of a hyperbolic surface in which all the leaves are geodesics is a \emph{geodesic lamination}.
\end{defn}

\subsubsection{Choosing and fixing the $f$-junctures}\label{definjunct}

We first choose and fix a countable set of $f$-junctures in the uncountable set of all $f$-junctures as follows. Let $e$ be an end of $L$ and set 
$$c=\{e_{0}=e,e_{1}=f(e),e_{2}=f^{2}(e),\dots,e_{p_{e}-1}=f^{p_{e}-1}(e)\},$$
 the  $f$-cycle  of ends containing $e$.  Let $N_{c} = J = \fr U_{e}$ be the $f$-juncture associated to the end $e$ defined in Proposition~\ref{comcomp} (where $U_{e} = V_{0}$ is defined in the proof of Proposition~\ref{comcomp} in the paragraph before Corollary~\ref{Je0}) and consider the set of $f$-junctures $\{f^{n}(N_{c})\ |\ n\in\Z\}$.   Choose and fix such an $N_{c}$ for each of the finitely many $f$-cycles $c$ of ends and  fix as our countable set of $f$-junctures the union of the sets $\{f^n(N_c)| n\in\Z\}$ as $c$ ranges over the set of f-cyclesÊ of ends.
 
\begin{rem}
From now on we will use the symbol $N$ to denote one of these countably many  $f$-junctures  and use the symbol $J$ to denote a juncture which we define below.

\end{rem}

 \begin{defn}[fixed set of $f$-junctures, $\NN$, $\NN_{+}$, $\NN_{-}$]\label{famNNg}
Fix the set of $f$-junctures constructed above.  The   set of all components of the fixed set of $f$-junctures  will be denoted by $\NN$. The subset of $\NN$ consisting of components of positive (respectively negative) $f$-junctures will be denoted by  $\NN_{+}$ (respectively $\NN_{-}$).
\end{defn}

We extend the definition of pseudo-geodesic given in~\cite[Definition~5]{cc:epstein} to include properly embedded, boundary incompressible  compact arc.

\begin{defn}[pseudo-geodesic]\label{pseudogeodesic}
A curve  $\gamma\ss L$ is a  \emph{pseudo-geodesic} if either some (hence every) lift $\wt\gamma$ has two distinct, well defined endpoints on $\Si$  or $\gamma$ is a properly embedded, boundary incompressible  compact arc.
\end{defn}

Remark that  essential embedded circles in $L$ that do not bound cusps and geodesics are pseudo-geodesics. Note that since the surface is admissible, every essential closed curve is a pseudo-geodesic.

In Definition~\ref{geodtite}, we defined the geodesic tightening of  an essential closed curve or a boundary incompressible, properly embedded arc. Here we extend the definition to an arbitrary pseudo-geodesic. 

\begin{defn}[geodesic tightening, $\gamma^{\g}$]\label{geotightpg}

  If $\gamma$ is a pseudo-geodesic, then the \emph{geodesic tightening}   of $\gamma$ is the unique geodesic $\gamma^{\g}$ whose lifts have the same endpoints on $\Se$ as the lifts of $\gamma$. 
  
\end{defn}

\begin{defn}[geodesic tightening map, $\iota$]\label{gtm}

We will call the map that sends a pseudo-geodesic to its geodesic tightening the \emph{geodesic tightening map} and denote it by $\iota$. Thus, if $\gamma$ is a pseudo-geodesic, then $\iota(\gamma)$ is the geodesic whose lifts have the same endpoints on $\Se$ as $\gamma$.

\end{defn}

\begin{defn}[juncture, $\JJ$, $\JJ_{+}$, $\JJ_{-}$]\label{junctdefn}

The set $\JJ$ of \emph{juncture components} consists of the geodesic tightenings of the $f$-juncture components in the set $\NN$. Let  $\iota:\NN\to\JJ$ be the geodesic tightening map. Then $\JJ_{+} = \iota(\NN_{+})$ is the set of \emph{positive juncture components} and $\JJ_{-} = \iota(\NN_{-})$ is the set of \emph{negative juncture components}. Further the map $\iota$ extends in a natural way to the fixed set of $f$-junctures to define a fixed set of \emph{junctures} of the form $J = \iota(N)$ where $N$ is one of the fixed $f$-junctures.

\end{defn}

\begin{rem}
It is important to keep clear the distinction between $f$-junctures and junctures.  The set of $f$-junctures is $f$-invariant, but generally not geodesic.  The set of junctures is geodesic but generally not $f$-invariant.
\end{rem}

We assume  the set of $f$-junctures, the set of junctures, the set $\NN$ of $f$-juncture components, and the set $\JJ$ of juncture components  has been constructed as above and fixed. 

\begin{defn}[$J_{n}$]\label{Jndef}
Given a juncture $J = \iota(N)$, let $J_{n}$ denote the geodesic tightening of $f^{n}(J)$ which is the same as the geodesic tightening of $f^{n}(N)$, $n\in\Z$.  
\end{defn}

 Since the positive (respectively negative) $f$-junctures  are constructed only to intersect in common components, the same is true after tightening. We state this formally for future reference.

\begin{prop}\label{comcomp'}
The set of junctures has the juncture intersection property.
\end{prop}

By Theorem~\ref{essc}, the junctures have the following critically important property.

\begin{theorem}\label{esctoe}
If $J=\iota(N)$ where the $f$-juncture $N$ cuts off a neighborhood of the positive end $e$ of period $p$, then $\{J_{n}\}_{n\ge0}$ escapes.  If the end $e$ is negative, then $\{J_{-n}\}_{n\ge0}$ escapes.
\end{theorem}

\begin{defn}[juncture escapes]\label{escapes}
A component $\sigma$ of a juncture $J$ \emph{escapes} if $\{i(f^{n}(\sigma))\}_{n\in\Z}$ escapes.  The juncture $J$ \emph{escapes} if each of its components escapes or, equivalenly, if $\{J_{n}\}_{n\in\Z}$ escapes.
\end{defn}

\begin{rem}
We will see that a juncture escapes if and only if $f$ is isotopic to a translation (Proposition~\ref{total}).  But escaping components of junctures can easily arise, as illustrated in Example~\ref{itj}.  They do not accumulate anywhere.  We will mainly be interested in the nonescaping components.
\end{rem}

\begin{defn}[$\XX_{\pm}$]\label{nonesc}
We denote by $\XX_{+}$ the set of nonescaping components of positive junctures and by $\XX_{-}$ the corresponding set of nonescaping components of negative junctures.
\end{defn}

\begin{rem}
$\XX_{+}\ss\JJ_{+}$ and $\XX_{-}\ss\JJ_{-}$.

\end{rem}

\subsubsection{The Handel-Miller bilaminations}

Define  $\G_{\pm} = \overline{|\XX_{\mp}|}$ and  set $\L_{\pm} = \G_{\pm} \sm |\XX_{\mp}|$. The goal of this section is to show that $\G_{\pm}$ and $\L_{\pm}$ are the supports of closed geodesic laminations $\Gamma_{\pm}$ and $\Lambda_{\pm}$, respectively.

\begin{lemma}

The space $\G_{\pm}$ consists of the disjoint union of one-one immersed, complete geodesics lines or compact geodesic arcs or circles which are the path components of $\G_{\pm}$.  

\end{lemma}

\begin{proof}
We consider $\G_{+}$. The proof for $\G_{-}$ is analogous. We have already seen that  the union of the positive junctures is a disjoint union of
isolated geodesic circles and/or properly embedded geodesic arcs.  Therefore it remains to show that $\L_{+}$ is a disjoint union of geodesics which are the path components of $|\L_{+}|$.  It will be enough to prove this for the lift $\wt\L_{+}\ss\wt L$.

Choose $x\in\wt\L_{+}$ and $x_{n}\in|\wt\XX_{-}|$ such that $x_{n}\to x$ as $n\to\infty$.  Let $\sigma_{n}$ be the unique element of $\wt\XX_{-}$ containing $x_{n}$.
 Let $v_{n}$ be a vector tangent to $\sigma_{n}$ at $ x_{n}$ of hyperbolic length $1$.   Passing to a subsequence, if necessary, we can assume that $v_{n}\to v$ as $n\to\infty$, where $v$ is a vector at $x$ of hyperbolic length $1$.  Then, there is a unique geodesic $\sigma$ through $x$ tangent to $v$ such that the sequence  $\{\wh\sigma_{n}\}$ converges uniformly to $\wh\sigma$ in the Euclidean metric on $\D^{2}$.  In particular, the endpoints (finite or ideal) on $\Se$ of the $\sigma_{n}$ converge to the endpoints of $\sigma$ on $\Si$. 
 
 Since $x\in\wt\L_{+}$ is arbitrary, we see that $\wt\L_{+}$ is a union of such geodesics and $\wt\XX_{-}$ accumulates on each point of $\wt\L_{+}$.  Thus, $\wt\L_{+}$ has empty interior.  In fact, if $x\in \wt\L_{+}$ and $U$ is a neighborhood of $x$, $\XX_{-}$ meets $U$ and each point of $\XX_{-}$ has a neighborhood disjoint from $\wt\L_{+}$.  If any two of these geodesics in $\wt\L_{+}$ intersect and are not identical, they intersect transversely, implying that two geodesics in $\wt\XX_{-}$ also intersect transversely.  We know this is false, so it follows that $\wt\L_{+}$ is a disjoint union of geodesics with endpoints on $\Si$ and has  empty interior. These geodesics, then, are path components of $\wt\L_{+}$.
 
 Projecting to $L$, we see that $\L$ has no interior and is also a  union of nonintersecting geodesics.  None of these geodesics  can self intersect since this would imply  that two geodesics in $\XX_{-}$ intersect transversely.  Thus, the path components of $\L_{+}$ are one-one immersed geodesics which are complete since the endpoints of their lifts are on $\Si$. 
\end{proof}

From now on we will write the set of geodesics forming the path components of $\L_{\pm}$ as $\Lambda_{\pm}$ and  the set of geodesics forming the path components of $\G_{\pm}$ as $\Gamma_{\pm}$. Thus, $\L_{\pm} = |\Lambda_{\pm}|$ and $\G_{\pm} = |\Gamma_{\pm}|$. These path components will be called ``leaves'' of $\Lambda_{\pm}$ and $\Gamma_{\pm}$, respectively, even before the proof of the following proposition is completed.

\begin{prop}\label{geobilam}
The pairs $(\Gamma_{+},\Gamma_{-})$ and $(\Lambda_{+},\Lambda_{-})$ are bilaminations. 
\end{prop}

\begin{proof}
For $p$ an arbitrary point of $L$ we construct a bilaminated chart $(U,X,Y,\phi)$ for $(\Gamma_{+},\Gamma_{-})$ with $p\in U$.    Since $p$ is arbitrary, we obtain a bilaminated atlas for       $(\Gamma_{+},\Gamma_{-})$.  By ignoring $\XX_{\pm}$, we see that this is also a bilaminated atlas for $(\Lambda_{+},\Lambda_{-})$.

Since $|\Gamma_{-}|\cup|\Gamma_{+}|$ is a closed set, points in its complement have coordinate charts that do not meet it.  These are trivially bifoliated charts for $(\Gamma_{+},\Gamma_{-})$.

The case in which     $p\in|\XX_{\pm}|\sm |\Lambda_{\mp}|$ is also easy. The point $p$ will only lie on the intersection of a positive and negative juncture or on one juncture.  Thus, the bifoliated chart $(U,X,Y,\phi)$ will have $X$ or $Y$ both singletons, or one a singleton and the other empty.

Suppose  $p\in|\Lambda_{+}|\sm|\Lambda_{-}|$.   Then $p\in\lambda\in\Lambda_{+}$ with $\lambda$ approached on one or both sides by geodesics in $\Gamma_{+}$. Thus, there is a convex geodesic quadrilateral $Q$ containing $p$ in its interior and having as one pair of opposite sides geodesic arcs $A$ and $B$ chosen as follows. If both  sides of $\lambda$ are appoached   by geodesics in $\Gamma_{+}$,  choose $A$ and  $B$ to be a  subarc  of $\Gamma_{+}$. Otherwise, choose $A$  on the side of $\lambda$ not approached by elements of $\Gamma_{+}$  in a small neighborhood of $p$ so that no points of $|\Gamma_{+}|$ lie between  $A$   and $\lambda$ and choose $B$ to be a subarc of $\Gamma_{+}$. Take the other pair of opposite sides of $Q$ to be geodesic arcs $\sigma$ and $\tau$  joining the endpoints of $A$ to the corresponding endpoints of $B$.  Choosing the arcs $A$ and $B$ suitably guarantees that $Q\cap|\Gamma_{-}|=\0$ if $p\notin|\XX_{+}|$ or $Q\cap|\Gamma_{-}|$ is one geodesic arc in a positive juncture if $p\in|\XX_{+}|$.  Let  $Y' = \tau\cap(|\Gamma_{+}|\cup A)$ (note that the $A$ in this definition of $Y'$ is redundant in the case that $\Gamma_{+}$ approaches $\lambda$ on both sides). Remark that the compact set  $Y'$ projects continuously and one-one onto $\sigma\cap(|\Gamma_{+}|\cup A)$ along arcs in leaves of $\Gamma_{+}$ (and the arc $A$ if necessary).  The continuity is a consequence of the fact that the geodesic arcs which are path components of $Q\cap|\Gamma_{+}|$ depend continuously on their endpoints.   This map extends linearly  over the gaps in $Y'$ to produce a homeomorphism $f:\tau\to\sigma$.  The geodesics with endpoints $y,f(y)$, $y\in\tau$, foliate $Q$ and the foliation contains among its leaves the path components of $Q\cap|\Gamma_{+}|$. (These leaves cannot intersect since this would produce geodesic digons in $Q$.)  It is then easy to use a linear map (or in the case that $p$ is in a positive juncture a piecewise linear map) from $A$ to $B$ to produce another geodesic foliation of $Q$ transverse to the first.  Coordinatizing $A$ with a coordinate $x$ and $\tau$ with a coordinate $y$ and using these transverse geodesic foliations as a coordinate grid, we obtain a closed coordinate chart $(Q,x,y)$ with $p$ in its interior  which, together with $Y = \tau\cap|\Gamma_{+}|$, defines a closed laminated chart for $\Gamma_{+}$.  The chart is trivially bilaminated by taking $X=\0$ if $p\notin|\XX_{+}|$ or one point if $p\in|\XX_{+}|$.  Similarly, if $p\in|\Lambda_{-}|\sm|\Lambda_{+}|$, we obtain a bilaminated chart with $p$ in its interior.  

Finally, if $p\in|\Lambda_{+}|\cap|\Lambda_{-}|$, the reader can adapt the above construction to obtain a pair of transverse geodesic foliations of $Q$  where, as above, $A,B$ will be arcs in leaves of $\Gamma_{+}$ or disjoint from $|\Gamma_{+}|$ and similarly $\sigma,\tau$ will be arcs in leaves of $\Gamma_{-}$ or disjoint from $|\Gamma_{-}|$.  One of these foliations incorporates the path components of $|\Gamma_{+}|\cap Q$ among its leaves, the other the path components of $|\Gamma_{-}|\cap Q$.  This gives a closed bilaminated chart about $p$, completing our construction of a bilaminated atlas for $(\Gamma_{+},\Gamma_{-})$.
\end{proof}

\begin{cor}\label{totdisc}

The laminations $\Lambda_{\pm}$ are transversely totally disconnected \emph{(Definition~\ref{trtotdis})}.

\end{cor}

\begin{proof}
This follows since the leaves of $\Gamma_{\pm}$ are the path components of $|\Gamma_{\pm}|$.
\end{proof}

\begin{defn}[Handel-Miller bilamination]\label{HMgeobilamin}
The bilamination $(\Lambda_{+},\Lambda_{-})$ is called the \emph{Handel-Miller \upn{(}geodesic\upn{)} bilamination} associated to the endperiodic automorphism $f$. The individual  laminations $\Lambda_{\pm}$ will be called the \emph{Handel-Miller \upn{(}geodesic\upn{)} laminations} associated to $f$. 

The bilamination $(\Gamma_{+},\Gamma_{-})$ is \emph{an extended Handel-Miller bilamination} associated to $f$.
\end{defn}

\begin{rem}
The use of the definite article ``the'' for the Handel-Miller bilamination and the indefinite article ``an'' for an extended one is important.  In Corollary~\ref{indep} we will show that the bilamination $(\Lambda_{+},\Lambda_{-})$ depends only on $f$ (and, of course, on the choice of standard hyperbolic metric). On the other hand $(\Gamma_{+},\Gamma_{-})$ also depends on the choice of $f$-junctures, hence of their geodesic tightenings. 
\end{rem}

\begin{rem}
The analogy of the construction of the geodesic laminations in the Handel-Miller theory and in the Nielsen-Thurston theory is evident, but there are some very significant differences.  The sequence of closed geodesics limiting on a geodesic lamination $\Lambda$ in the Nielsen-Thurston case is not pairwise disjoint nor disjoint from $\Lambda$.  The endperiodic case allows us to choose the connected geodesic $1$-manifolds in $\XX_{\pm}$ to be pairwise disjoint and disjoint from the lamination $\Lambda_{\mp}$ on which they limit.  This is a major simplification and also introduces new laminations $\Gamma_{\pm}$ which will be very useful in developing the theory and for which there are no analogues in the Nielsen-Thurston theory.
\end{rem}

\begin{rem}
The pair $(\Lambda_{+}\cup\JJ_{-},\Lambda_{-}\cup\JJ_{+})$ may not be a bilamination. In fact, an escaping component of a positive juncture may coincide with an  escaping component of a negative juncture. However, it is obvious that $\Lambda_{+}\cup\JJ_{-}$ and $\Lambda_{-}\cup\JJ_{+}$ are each laminations.

\end{rem}

\subsubsection{The strongly closed property}

\begin{defn}[pseudo-geodesic lamination]\label{psgeolam}
A lamination of the surface $L$ in which all the leaves are pseudo-geodesics is a \emph{pseudo-geodesic lamination}.
\end{defn}

\begin{defn}[converges strongly]\label{convstrly}

If $\Lambda$ is a pseudo-geodesic lamination and $\{\wt\lambda_{n}\}$ is a sequence of leaves of $\wt\Lambda$, then $\{\wt\lambda_{n}\}$ \emph{converges strongly} to $\wt\lambda\in\wt\Lambda$ if,

\begin{enumerate}

\item For every bounded subarc $P\ss\wt\lambda$ there is a laminated chart $V_{p}\cong P\times (-\epsilon, \epsilon)$, with $P = P\times\{0\}$, such that the sequence of plaques $V_{P}\cap\wt\lambda_{n} = P\times\{t_{n}\}$ converge to $P$ as $n\to\infty$\upn{;}

\item The endpoints of  the $\wt\lambda_{n}$ in $\SI$ converge to the endpoints of $\wt\lambda$ in $\Si $ in the Euclidean metric of $\D^{2}$ as $n\to\infty$.

\end{enumerate}

\end{defn}

\begin{defn}[strongly closed property]\label{strnglyclsed}
A lamination $\Lambda$ has the \emph{strongly closed property} if,
\begin{enumerate}

\item Whenever a sequence of points $x_{n}\in\ell_{n}\in\wt\Lambda$ converges to $x\in\wt{L}$ and $\ell$ is the leaf, of $\wt\Lambda$ through $x$,  the  sequence  $\{\ell_{n}\}$ converges strongly to   $\ell$;

\item Whenever the endpoints in $\Si$ of a sequence $\{\ell_{n}\}\ss\wt\Lambda$ converge in the Euclidean metric  to a pair of distinct points $a,a'\in \Si$, then these are the endpoints of a  leaf $\ell$ of $\wt\Lambda_{\pm}$ and the sequence $\{\ell_{n}\}$ converges strongly to $\ell$. 

\end{enumerate} 
\end{defn}

The next lemma is obvious from the construction of  the laminations $\wt\Gamma_{\pm}$ and $\wt\Lambda_{\pm}$.

\begin{lemma}\label{LGstrcl}

The  laminations $\Gamma_{\pm}$ and $\Lambda_{\pm}$ are strongly closed.

\end{lemma}

\begin{rem}
 In Section~\ref{uniq} we give an axiomatic treatment of the Handel-Miller theory.  However, the corresponding \emph{pseudo-geodesic} laminations satisfying the axioms are not obviously strongly closed, but, by the isotopy theorem (Theorem~\ref{isotlams}), the strongly closed property  will follow from the geodesic case.  In turn, this property will be needed at a key point of the proof of the transfer theorem (Theorem~\ref{transfer}). 
\end{rem}

\subsubsection{Examples}

\begin{example}\label{simpex}
We continue Example~\ref{simpex0}. The vertical lines in Figure~\ref{PlanarEnd} are  junctures, the ones to the left of the core $K$ being negative junctures and those to the right of the core being positive junctures. The juncture immediately to the  left of the core is the juncture $J$ of Figure~\ref{ItJunc}.   The traintracks representing both $\Lambda_{+}$ and $\Lambda_{-}$ (dashed) have been drawn.   Along the segments of the traintrack joining two consecutive switches, the lamination looks like an uncountable, but totally disconnected packet of parallel arcs. At the switches, the packet splits along a gap,  ``half'' of the curves veering to the left and half to the right. The intersection $\KK=|\Lambda_{+}|\cap|\Lambda_{-}|\ss K$ is a totally disconnected set, homeomorphic to the Cantor set and living in the regions indicated by the two small  circles in the figure.  After $f$ is isotoped to a homeomorphism $h$ preserving the laminations, the set $\KK$ will be invariant and the dynamics of $h|\KK$ will be isomorphic to that of a 2-ended Markov chain (cf.~Section~\ref{cordyn}).  
\end{example}

\begin{figure}[ht]
\begin{center}
\begin{picture}(300,90)(20,-90)
\rotatebox{270.2}{\includegraphics[width=85pt]{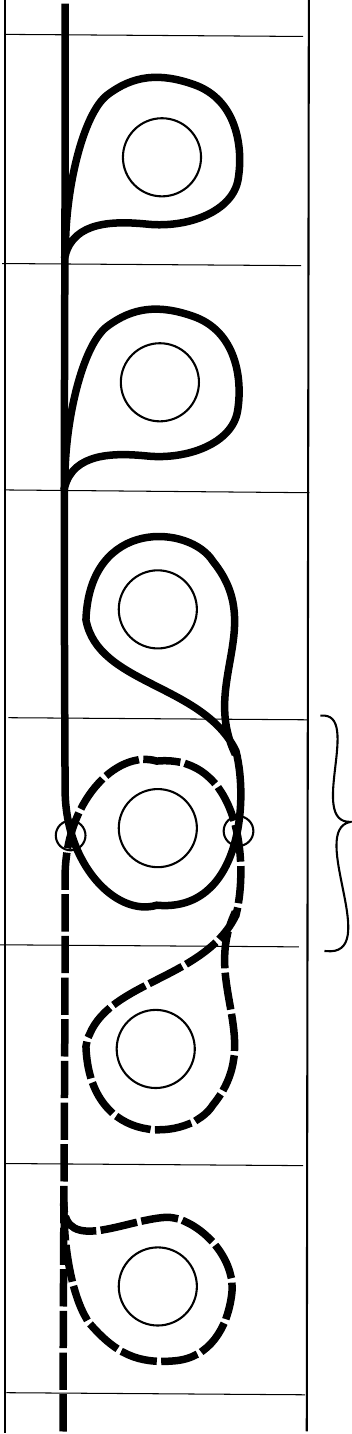}}
\put(-200,-93){\small{$K$}}
\put(-135, -10){\small{$\zeta$}}
\put(-350,-40){\small $e_{-}$}
\put(-5,-40){\small $e_{+}$}
\put(3,-5){\small$\ell$}
\end{picture}
\caption{A simple example}\label{PlanarEnd}
\end{center}
\end{figure}
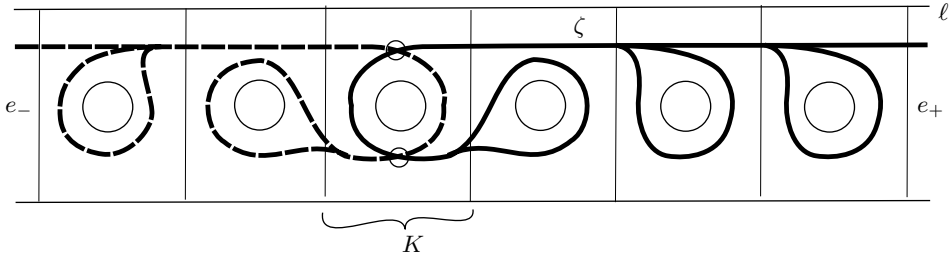

One easily proves that $\Lambda_{+}$ has a single leaf $\lambda_{+}$ that is isolated on one side. A ray $\zeta$ of $\lambda_{+}$ (the ``top'' ray on the right hand side of the core) makes a ``beeline'' for $e_{+}$, never veering back, but the complementary ray repeatedly delves arbitrarily deeply into the neighborhood of $e_{+}$ and then turns around to revisit the core before going even more deeply into the neighborhood of $e_{+}$.  Every other leaf of $\Lambda_{+}$ behaves like this in both directions, always returning to the core.  The lamination $\Lambda_{-}$ behaves similarly. 

This illustrates typical behavior that will be proven as theorems in this paper.  However not everything in this example is typical.  Generally, the laminations may not be transversely Cantor as they are here. The laminations may also have isolated leaves as well as limit leaves, be transversely countable, or even finite.  In Example~\ref{2e}, each lamination has a single leaf indicated in Figure~\ref{TwoEnds} in boldface.

\begin{figure}[ht]
\begin{center}
\begin{picture}(300,190)(0,-160)

\rotatebox{270.2}{\includegraphics[width=150pt]{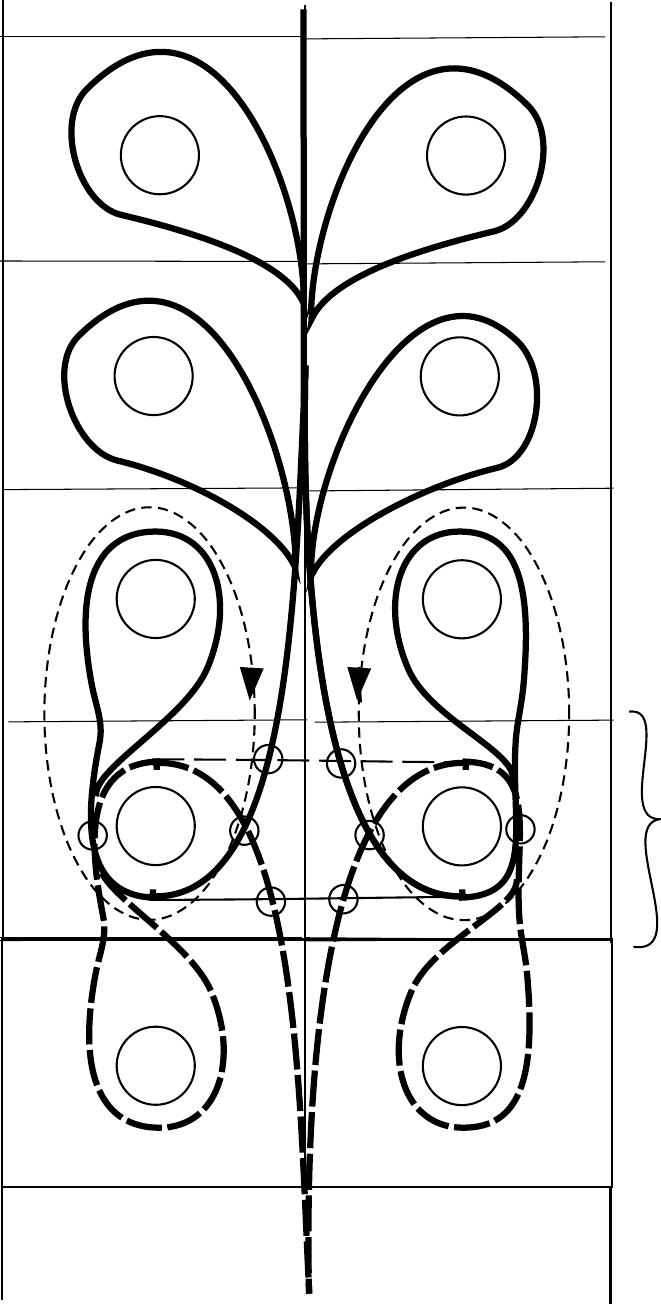}}

\put(2,-70){\small{$\ell$}}
\put(-190,-158){\small{$2K$}}
\end{picture}
\caption{A doubled example}\label{DoubledExample}
\end{center}
\end{figure}

\begin{example}\label{simpex2}
One might think that, if $\bd L\ne\0$, the Handel-Miller laminations for the double $2f:2L\to2L$ would just be two copies of the Handel-Miller laminations on $L$.  But doubling Example~\ref{simpex} shows this to be false. In Figure~\ref{DoubledExample} we draw the train tracks for the laminations $\Lambda_{\pm}$ for the double, where the dashed tracks carry $\Lambda_{-}$ and the solid carry $\Lambda_{+}$. In this figure, the top and bottom boundary lines are to be identified and each boundary circle in the top row should be identified with the corresponding boundary circle in the bottom row. In each lamination, there is an isolated leaf $\lambda_{\pm}$ (not labeled, but easy to spot) that crosses the line $\ell$ which was the top boundary line in Figure~\ref{PlanarEnd}.  This leaf is dense in $\Lambda_{\pm}$. The bold part of the traintracks carries uncountably many leaves, the lighter segments each only carry a segment of $\lambda_{\pm}$. The reason that this new isolated leaf appears  is that a component $\sigma$ of (distorted) juncture which is a properly embedded geodesic arc has double $2\sigma$ which is not generally a geodesic circle.   Tightening it to one distorts $2\sigma $ enormously, pulling it back to a closed geodesic which crosses $\ell$ in a point  in the compact core $2K$.  This sequence of crossings accumulates monotonically on a point $x_{\pm}$, the junctures accumulating locally uniformly near $x_{\pm}$ on the  segment of the isolated leaf $\lambda_{\pm}$ through $x_{\pm}$.  In the figure we have drawn circles where the two laminations intersect.  Remark that $\Lambda_{\pm}$ is equal to the double of its counterpart in Figure~\ref{PlanarEnd} together with the isolated leaf $\lambda_{\pm}$.  Here we have narrowed the gap between $\ell$ and the two copies of $\zeta$ to the point of invisibility.
\end{example}

\subsection{Distinguished  neighborhoods}\label{remarknote}

Suppose $e\in\EE(L)$ and the juncture $J = \iota(N)$ where $N$ is the frontier of a closed $f$-neighborhood $V$ of $e$. Applications of Theorem~\ref{2.1} and/or~\ref{3.1} to the components of $N$, provide an isotopy $\Phi$ such that $\Phi^{1}(N) = J$. Thus, $\Phi^{1}(V)$ is a closed neighborhood $U_{e}$ of $e$ and $J = \fr U_{e}$.

\begin{defn}[distinguished  neighborhood, $U_{e}$]\label{hnbhd}

The set $U_{e}$ as above is called a \emph{distinguished neighborhood of the end $e$}. The set of \emph{distinguished neighborhoods} is the set of all such sets $U_{e}$ with $\fr U_{e}$ a juncture.

\end{defn}

\begin{rem}
The fixed countable  set of $f$-junctures (Definition~\ref{famNNg})  is in one-one correspondence with the set of junctures  (Definition~\ref{junctdefn}) which is in one-one correspondence with the set of distinguished neighborhoods.

\end{rem}

\begin{lemma}\label{intofdom}

If $e\ne e'$ are any two positive \upn{(}respectively  negative\upn{)}  ends and $U_{e}$ and $U_{e'}$ are distinguished neighborhoods of $e$ and $e'$, then $U_{e}\cap U_{e'}=\0$.

\end{lemma}

\begin{proof}
Consider the junctures $J = \fr U_e$ and $J' = \fr U_{e'}$. Let $N$ and $N'$ be the $f$-junctures such that $\iota(N) = J$ and $\iota(N') = J'$. Let $V_e$ and $V_{e'}$ be $f$-neighborhoods of $e$ and $e'$ respectively such that $N = \fr V_{e}$ and $N' = \fr V_{e'}$. By Lemma~\ref{emptint}, $V_{e}\cap V_{e'} = \0$. It follows that $U_{e}\cap U_{e'} = \0$.
\end{proof}

Recall the notation, $J_{n}$ equals the geodesic tightening of $f^{n}(J)$ if $J$ is a juncture, $n\in\Z$.

\begin{defn}[$U^{n}_{e}$]\label{disnbhdsys}

Suppose $e\in\EE_{+}(L)$ \upn{(}respectively $e\in\EE_{-}(L)$\upn{)}, $U_{e}$ is a distinguished neighborhood of $e$, and $J = \fr U_{e}$. For $n\in\Z$, define $U_{e}^{n}$ to be the distinguished neighborhood such that $J_{np_{e}} = \fr U_{e}^{n}$ \upn{(}respectively  $J_{-np_{e}} = \fr U_{e}^{n}$\upn{)}.

\end{defn}

Theorem~\ref{esctoe} immediately implies,

\begin{lemma}\label{disfund}

For an end $e$,   $\{U_{e}^{i}\}_{i=0}^{\infty}$ is a fundamental  neighborhood system of $e$. 

\end{lemma}

\begin{rem}
By reindexing, we can choose the  distinguished neighborhoods $U_{e}$ so that, as $e$ ranges over all ends, positive and negative, the distinguished neighborhoods $U_{e}$ will be pairwise disjoint.  We fix such a choice.

\end{rem}

\subsection{The cores $K,K_{i}$ and the sets $W^{\pm},W^{\pm}_{i}$}\label{thecore}

\begin{defn}[core and $i^{\thh}$ core, $K,K_{i},W^{\pm},W^{\pm}_{i}$]\label{coredef}

For each $i\ge0$,  denote by $W_{i}^{+}$ the union of the $U_{e}^{i}$'s as $e$ ranges over the positive ends.  Similarly define $W_{i}^{-}$, for each $i\ge0$. These choices of the distinguished neighborhoods $U_{e}$ have been made so that $W_{i}^{+}\cap W_{i}^{-}=\0$, $i\ge0$.   The compact submanifold $K_{i}$ complementary to $\intr (W_{i}^{+}\cup W_{i}^{-})$ is called the \emph{$i^{\thh}$ core} of $L$.  We will write $K=K_{0}$ and call it simply the \emph{core}. We will also write $W^{+}$ for $W_{0}^{+}$ and $W^{-}$ for $W_{0}^{-}$.\label{KW}

\end{defn}

\begin{rem}
We will often change the indexing and relabel $K_{i}$ by $K_{0}$, thus choosing a larger \emph{core} $K=K_{0}$. 

\end{rem}

\begin{rem}
Note that $\fr K$ is a finite union of positive and negative junctures.  Accordingly, we will write $\fr K=\fr_{+}K\cup\fr_{-}K$.

\end{rem}

The surface $L$ has decomposition $L = W^{-}\cup K\cup W^{+}$ where,
$$W^{-} = \bigcup_{e\in\EE_{-}(L)}U_{e}\quad{\rm and}\quad W^{+} = \bigcup_{e\in\EE_{+}(L)}U_{e}.$$

\begin{defn}[$W,\JJ_{W}$]\label{jw}

Let $W = W^{-}\cup W^{+}$. Let $\JJ_{W}$ be the set consisting of all components of negative  junctures in $W^{-}$ and all components of positive junctures in $W^{+}$. 

\end{defn}

\begin{lemma}\label{imagejunct}

There exists an endperiodic automorphism $g$ isotopic to $f$ such that $g(\alpha)$ is a geodesic for every $\alpha\in\JJ_{W}$.

\end{lemma}

\begin{proof}
Enumerate the elements of $\JJ_W $ as $\{\tau^{i}\}_{i=1}^{\infty}$ in such a way that those meeting $K_{i}$ are listed before those not meeting $K_{i}$, $i\ge 0$.  Let $\psi_{0} = \id$. Using Theorem~\ref{2.1}, Theorem~\ref{3.1} and the   remark after Theorem~\ref{epsteinsmooth}, inductively find  sequences $\{\Phi_{i}\}$ of isotopies and $\{\psi_{i}\}$ of homeomorphisms with $\psi_{i} = \Phi^{1}_{i}\circ\psi_{i-1}$, $i\ge 1$,     fixing $\bd L$ pointwise and  with $\Phi_{i}$ fixing $(f(\tau^{j}))^{\g}$, $1\le j\le i-1$,  pointwise and moving $\psi_{i-1}\circ f(\tau^{i})$ to its geodesic tightening. That is, $\psi_{i}(f(\tau^{i})) =  (f(\tau^{i}))^{\g}$.  By  the   remark after Theorem~\ref{epsteinsmooth},  the supports of at most finitely many $\Phi_{i}$ meet any compact set. Thus, $\psi_{i}\to\psi$, a well defined homeomorphism isotopic to the identity by an isotopy $\Phi$ and $g=\psi\o f$ is such that $g(\alpha)$ is a geodesic for every $\alpha\in\JJ_{W}$. 
\end{proof}

\begin{lemma}\label{theyesc}
For each $i\ge0$ and each choice $J$ of  negative juncture, all but finitely many components of  $\bigcup_{n\ge0}J_{n}$ meet $W_{i}^{+}$ and, similarly, if $J$ is a positive juncture, all but finitely many components of~$\bigcup_{n\ge0}J_{-n}$ meet $W_{i}^{-}$.  
\end{lemma}

\begin{proof}
If a component of the negative juncture $J$ is a properly embedded geodesic arc, the assertion  is immediate  since  the positive iterates of its endpoints escape.  If the component is a simple closed geodesic $\sigma$, denote the corresponding component of $J_{n}$ by $\sigma_{n}$.  Suppose, by contradiction,  that $\sigma_{n}\ss K_{i}$, for infinitely many  values of $n>0$.  Since $K_{i}$ is a compact surface, there exists a positive integer $k$ such that if there exist more than $k$ disjoint, simple closed curves in $K_{i}$ two of them must bound an annulus. If $\sigma_{n}$ and $\sigma_{m}$ cobound an annulus, then, as homotopic geodesics, they must coincide. It follows that infinitely many of the $\sigma_{n}$ coincide which is a contradiction.
\end{proof}

\begin{lemma}\label{junctureK}

\begin{enumerate}

\item  The arcs of $\JJ_{\mp}|K_{i}$ with endpoints on $\fr_{\pm} K_{i}$ are   boundary incompressible and fall into finitely many isotopy classes in $(K_{i},\fr _{\pm}K_{i})$\upn{;}\label{1item}

\item For every such isotopy class $\AAA$, there exists a rectangle $R\ss K_i$ with a pair of opposite edges each  intervals in not necessarily distinct components $\alpha$ and  $\beta$ of $\fr_{\pm} K_{i}$ such that for every pair of arcs $\sigma,\tau\in\AAA$, the track of the isotopy between $\sigma$ and $\tau$ lies in $R$.\label{2item}. 

\end{enumerate}

\end{lemma}

\begin{proof}
Item (\ref{1item})  is an elementary topological consequence of the fact that $K_{i}$ is a compact surface. Let $\AAA$ be such an isotopy class.  There exist not necessarily distinct components $\alpha,\beta$ of $\fr_{\pm} K_{i}$ such that the endpoints of each arc in $\AAA$ lie in $\alpha$ and $\beta$. If $\sigma,\tau\in\AAA$, then there exists a rectangle $R_{\sigma\tau}$ with one pair of opposite edges $\sigma$ and $\tau$ and the other pair of opposite edges subintervals of $\alpha$ and $\beta$ respectively such that the track of the isotopy between $\sigma$ and $\tau$ lies in $R_{\sigma\tau}$.

If $\AAA$ has one element, then  (\ref{2item}) is obviously true. If $\AAA$ has two elements $\sigma, \tau$, then $R = R_{\sigma\tau}$ is the desired rectangle. Therefore we can assume $\AAA$ has at least 3 elements $\tau_-, \tau_0, \tau_+$ such that $\tau_0\in R_{\tau_-\tau_+}$. Let $R_-$ (respectively $R_+$) be the closure of $\bigcup\{R_{\tau_0\tau}\ |\ \tau\in\AAA, R_{\tau_0\tau_-}\ss R_{\tau_0\tau}\}$ (respectively $\bigcup\{R_{\tau_0\tau}\ |\ \tau\in\AAA, R_{\tau_0\tau_+}\ss R_{\tau_0\tau}\}$). If either $R_-$ or $R_+$ is not a rectangle, then we have the contradiction that $K = R_-$ or $K = R_+$ is an annulus with boundary $\alpha$ and $\beta$. Similarly if $R = R_-\cup R_+$ is not a rectangle, we have the contradiction that $K = R$ is an annulus with boundary $\alpha\cup\beta$. Thus, $R$ is the desired rectangle.
\end{proof}

\begin{lemma}\label{nbounded}
Every neighborhood of each end of a leaf of $ \Lambda_{\pm}$ meets $W_{i}^{\pm}$, $i\ge0$.
\end{lemma}

\begin{proof}
For definiteness, assume $ \lambda\in \Lambda_{+}$.    Choose $K=K_{0}$ large enough that $\lambda$ meets $\intr K$, choose    $x\in\lambda\cap\intr K$ and let $\{\sigma_{k}\}_{k=1}^{\infty}$ be a sequence of  negative juncture components with $x_{k}\in\sigma_{k}$ such that $x_{k}\to x$ as $k\to\infty$.  Fix $i\ge 0$.  Using Lemma~\ref{theyesc} and passing to a subsequence, we can assume that a subarc $\tau_{k}\ss\sigma_{k}\cap K_{i}$ contains $x_{k}$ and has endpoints $a_{k},b_{k}\in\fr_{+} K_{i}$.  By passing to a subsequence and using Lemma~\ref{junctureK}, assume that the arcs $\tau_{k}$ are all isotopic. Thus the points  $a_{k}$ belong to the same component $\alpha$ of $\fr_{+} K_{i}$ and converge to $a\in\alpha$ and  the points $b_{k}$ belong to the same component $\beta$ of $\fr_{+} K_{i}$ and converge to $b\in\beta$.  Without loss we can assume $a_{k}\to a$ and $b_{k}\to b$ monotonically. By Lemma~\ref{junctureK} there exists a rectangle $R$ with a pair of opposite edges in $\alpha$ and $\beta$ and containing the sequence of arcs $\{\tau_k\}_{k\ge 0}$.

Choosing a lift $\wt R$ of $R$, where $\wt R$ is a rectangle with a pair of opposite edges in $\wt\alpha$ a lift of $\alpha$ and $\wt\beta$ a lift of $\beta$, determines lifts   $\wt\tau_k$ of $\tau_k$ for all $k\ge 0$ such that  $\wt\tau_k$ has endpoints $\wt a_k\in[\wt a,\wt a_0]\ss\wt\alpha$ and $\wt b_k\in[\wt b,\wt b_0]\ss\wt\beta$ where $\wt a$ is a lift of $a$, $\wt b$ is a lift of $b$, $\wt a_k$ is a lift of $a_k$, and $\wt b_k$ is a lift of $b_k$. The geodesic arcs $\wt\tau_{k}$ nest in $\Delta$  on the geodesic arc $[\wt a,\wt b]\ss\ell\in\wt \Lambda_{+}$ and $\wt a_{k}\to \wt a$, and $\wt b_{k}\to \wt b$.  Since the lift $\wt x_{k}\in\wt\tau_{k}$ converges to a lift $\wt x\in[\wt a,\wt b]$ of $x\in\lambda$, it follows that $\ell = \wt\lambda$. The arc $[\wt a,\wt b]$ projects to a subarc $[a,b] = \gamma_{i}\ss\lambda$ containing the point $x$.

Then $\lambda$ exits $K_{i}$ at $a$ and $b$ and, varying $i$, $x\in\gamma_{0}\ss\gamma_{1}\ss\cdots\ss\gamma_{i}\ss\cdots$ forms an exhaustion of $\lambda$. The lemma follows.
\end{proof}

\begin{cor}\label{passesnear}
 No leaf of  $\Lambda_{\pm}$ is contained in a bounded region of $L$.
\end{cor}

\begin{cor}\label{lambdaline}
Every leaf of $\Lambda_{\pm}$ is a one-one immersed copy of $\R$.
\end{cor}

\begin{proof}
By Corollary~\ref{passesnear}, a leaf of $\Lambda_{\pm}$ can not be homeomorphic to a circle. Hypothesis~\ref{noperpt}   assures that no leaf of $\Lambda_{\pm}$ has an endpoint on $\bd L$.
\end{proof}

\begin{defn}[passes arbitrarily near]\label{arbitrarilynear}
An end $\epsilon$ of a curve $s$ in $L$ \emph{passes arbitrarily near} an end $e$ of $L$ if every neighborhood of $e$ meets every neighborhood of $\epsilon$.  In this case we also say that $s$ passes arbitrarily near $e$.
\end{defn}  

Notice that a curve can pass arbitrarily near $e$ and still return repeatedly to some compact subset $X\ss L$.  A careful analysis of Example~\ref{simpex0} shows that only one end of only one leaf of $\Lambda_{+}$ escapes (cf.~Definition~\ref{defnescpend}) to the positive end.  All other ends of all leaves repeatedly return to the same compact  ``core''.   We will see that this second behavior  is the typical behavior of the leaves of $\Lambda_{\pm}$.  

The following corollary follows immediately from Lemma~\ref{nbounded} since $L$ has finitely many ends.

 \begin{cor}\label{noncpt}
Both ends of every leaf of $\Lambda_{+}$ \upn{(}respectively $\Lambda_{-}$\upn{)} pass arbitrarily near at least one positive \upn{(}respectively  negative\upn{)} end of $L$. 
\end{cor}

\begin{lemma}\label{escesc}
No escaping component of a juncture meets $|\Lambda_{+}|\cup|\Lambda_{-}|$. 
\end{lemma}

\begin{proof}
For definiteness, suppose that $\sigma$ is a negative escaping juncture component.  Since each point of $\sigma$ has a neighborhood that meets no other negative juncture component, it follows that $\sigma\cap|\Lambda_{+}| = \0$. It remains to show that $\sigma\cap|\Lambda_{-}| = \0$. Suppose the contrary that there exists $\lambda\in\Lambda_{-}$ meeting $\sigma$.  Consider  lifts $\wt\sigma$ of $\sigma$ with endpoints (finite or ideal) $a,b\in\Se$ and $\wt\lambda$ of $\lambda$ with endpoints $x,y\in E\ss\Se$ such that the pair $\{a,b\}$  separates the pair $\{x,y\}$ in $\Se$.  The extension of a lift $\wt f$ to a homeomorphism $\wh f:\Se\to\Se$~\cite[Theorem~2]{cc:epstein} either preserves or reverses cyclic order, hence $\{\wh f^{k}(a),\wh f^{k}(b)\}$ separates $\{\wh f^{k}(x),\wh f^{k}(y)\}$,  $k\ge 0$. Thus, $\sigma_{k} = f^{k}(\sigma)$ meets a leaf of $\Lambda_{-}$, $k\ge 0$.   For $k$ sufficiently large, $\sigma_{k}$ is disjoint from $W^{-}$.  Since no leaf of $\Lambda_{-}$ meets $W^{+}$, it follows that, for all $k$ is sufficiently large,  $\sigma_{k}$ meets the compact set $K$, contradicting the hypothesis that $\sigma$ is an escaping juncture component.  An analogous proof shows that a positive escaping juncture component is disjoint from $|\Lambda_{+}|\cup|\Lambda_{-}|$.
\end{proof}

Define laminations $\Lambda_{\pm}|K$ by taking as leaves the path components of $|\Lambda_{\pm}|\cap K$.

\begin{lemma}\label{LambdaK}

\begin{enumerate}

\item  The leaves of $\Lambda_{\pm}|K$ are boundary incompressible arcs which have endpoints on $\fr_{\pm} K$ and which fall into finitely many isotopy classes in $(K,\fr _{\pm}K)$\upn{;}\label{11item}

\item For every such isotopy class $\AAA$, there exists a rectangle $R\ss K_i$ with a pair of opposite edges each  intervals in not necessarily distinct components $\alpha$ and  $\beta$ of $\fr_{\pm} K_{i}$ such that for every pair of arcs $\sigma,\tau\in\AAA$, the track of the isotopy between $\sigma$ and $\tau$ lies in $R$.\label{22item}

\end{enumerate}

\end{lemma}

\begin{proof}
The leaves of $\Lambda_{\pm}|K$  have endpoints on $\fr_{\pm} K$ by Lemma~\ref{nbounded}. The other assertions of the lemma are proven exactly as in the proof of Lemma~\ref{nbounded}.
\end{proof}

\begin{lemma}\label{extremals}
An isotopy class of leaves of $\Lambda_{\pm}|K$ either contains one arc or else contains two extreme arcs $\tau_{1},\tau_{2}$ which, together with two arcs in $\fr_{\pm}K$, form a quadrilateral bounding a simply connected region in $K$ containing all arcs in the isotopy class. The extreme arcs do not cut off another simply connected quadrilateral.
\end{lemma}

\begin{proof}
Consider an  isotopy class $\AAA$ of leaves of $\Lambda_{-}|K$ containing more than one arc.  The case of $\Lambda_{+}|K$ is similar.  Let $\alpha$ and $\beta$ be the components of $\fr_{-}K$ containing the endpoints of leaves in the isotopy class $\AAA$.  By Lemma~\ref{LambdaK}~(\ref{22item}), there exists a rectangle $R$ which contains every arc in $\AAA$ and has  a pair of opposite edges in $\alpha$ and $\beta$. Choosing a lift $\wt R$ of $R$, where $\wt R$ is a rectangle with a pair of opposite edges in $\wt\alpha$ a lift of $\alpha$ and $\wt\beta$ a lift of $\beta$, determines lifts $\wt\tau$ lifts of $\tau$ for every $\tau\in\AAA$ such that $\wt\tau$ has endpoints in $\wt\alpha$ and $\wt\beta$.

Because the set $\wt\Lambda_+$ is closed, the set of arcs $\wt\AAA = \{\wt\tau\ |\ \tau\in\AAA\}$ in $\wt R$ contains two  arcs $\sigma_1, \sigma_2$  which are  extreme arcs of $\wt\AAA$  and are thus lifts of arcs $\tau_1,\tau_2\in\AAA$ which are extreme arcs of $\AAA$.

If the extreme arcs cut off another simply connected quadrilateral the union of the two quadrilateral would form an annulus bounded by $\alpha$ and $\beta$ which would equal $K$ contradicting the fact that $K$ is not an annulus.
\end{proof}

\subsection{An endperiodic automorphism $h$ preserving the laminations}\label{defineh}

In this subsection we prove the following theorem.

\begin{theorem}\label{geodext}

If $f:L\to L$ is an endperiodic automorphism, then there exists  an endperiodic automorphism $h:L\to L$, isotopic to $f$  permuting the elements of each of the sets $\Lambda_{+}$,  $\Lambda_{-}$,  $\JJ_{+}$, and  $\JJ_{-}$. 

\end{theorem}

\begin{rem}
Since $h$ is  isotopic to $f$  and permutes the elements of each of the sets $\Lambda_{+}$,  $\Lambda_{-}$,  $\JJ_{+}$, and  $\JJ_{-}$, it follows that $h(\gamma) = (f(\gamma))^{\g} = (h(\gamma))^{\g}$ if $\gamma\in\Lambda_{+}\cup\Lambda_{-}\cup\JJ_{+}\cup\JJ_{-}$. In particular $J_{n} = h^{n}(J)$ if $J$ is a juncture.

\end{rem}

\begin{rem}
This homeomorphism $h$ is not uniquely determined, but 
$$h:|\Lambda_{+}|\cap|\Lambda_{-}|\to|\Lambda_{+}|\cap|\Lambda_{-}|$$
 is unique. This will be called the core dynamical system and will be analyzed in Section~\ref{cordyn}. 
 
 \end{rem}

We prove Theorems~\ref{geodext} in Section~\ref{geodextpf} after some preliminaries in Sections~\ref{slideisot} and~\ref{prelisotf}.

\subsubsection{Sliding isotopies and other isotopies}\label{slideisot}

 Recall from page~\pageref{allisotamb} that all isotopies are ambient isotopies. By Definition~\ref{geotightpg},  the geodesic tightening $\tau^{\g}$ of a pseudo-geodesic $\tau$ is the geodesic whose lifts are the geodesics in $\wt L$ sharing endpoints on $\Se$ with the lifts of $\tau$.

Suppose $\sigma$ is a complete geodesic with ordered lift $\wt\sigma$ and $\gamma_{1},\ldots,\gamma_{k}$ are   pseudo-geodesics which are either disjoint or coincide, form no digons with $\sigma$, and meet $\sigma$  at the points $x_{1},\ldots,x_{k}$ with lifts $\wt\gamma_{1},\ldots,\wt\gamma_{k}$ meeting $\wt\sigma$  at the points $\wt x_{1}<\cdots<\wt x_{k}$ so that the interval $[\wt x_1,\wt x_k]\ss\wt\sigma$ meets no other lift of the $\wt\gamma_i$. Let $\wt\gamma_i^{\g}$ be the lift of the geodesic tightening of $\gamma_i$ sharing endpoints on $\Se$ with $\wt\gamma_i$, $1\le i\le k$.

\begin{lemma}\label{slide}

Under the above conditions, there exists an isotopy supported in a small neighborhod of $\sigma$ with lift sliding $\wt x_{i}$ along $\wt\sigma$ to the point $\wt x_i^{\g} = \wt\sigma\cap\wt\gamma_i^{\g}$, $1\le i\le k$.

\end{lemma}

\begin{defn}[sliding isotopy]\label{slidingisotopy}

We will refer to an isotopy  as in Lemma~\ref{slide}  as a \emph{sliding isotopy}.

\end{defn}

Suppose $\sigma_i$, $i=1,2$, are   complete geodesics which are either disjoint or coincide and have  disjoint lifts $\wt\sigma_i$,   $\gamma$ is a pseudo-geodesics forming no digons with the $\sigma_i$, and $\alpha =[x_1,x_2]\ss\gamma$ is an arc meeting the $\sigma_i$ only  at the points $x_{1}\in\sigma_1$ and $x_{2}\in\sigma_2$ with lift $\wt\alpha = [\wt x_1,\wt x_2]\ss\wt\gamma$ meeting  $\wt\sigma_1$  at the point $\wt x_{1}$ and $\wt\sigma_2$ at the point $\wt x_{2}$. Let $\wt\gamma^{\g}$ be the lift of the geodesic tightening of $\gamma$ sharing endpoints on $\Se$ with $\wt\gamma$. The next lemma follows from Lemma~\ref{slide} and Theorem~\ref{2.1}.

\begin{lemma}\label{arctight}

Under the above conditions, there is an isotopy, fixing each $\sigma_i$, $i=1,2$, with lift moving the arc $\wt\alpha = [\wt x_1,\wt x_2]\ss\wt\gamma$ to the geodesic arc $\wt\alpha^{\g} = [\wt x_1^{\g},\wt x_2^{\g}]\ss\wt\gamma^{\g}$, $\wt x_i$ sliding along $\wt\sigma_i$ to $\wt{x_i^{\g}}$, $i=1,2$, as in \emph{Lemma~\ref{slide}}.

\end{lemma}

\subsubsection{The tilings used in inductive proofs}\label{prelisotf}

Cutting $L$ apart along the juncture components in $\JJ_{W}$ (Definition~\ref{jw}) decomposes $L$ into a set $\mathfrak T^{\g}$ of compact surfaces. Similarly, cutting $L$ apart along the juncture components in $\{g(\gamma)\ |\ \gamma\in\JJ_{W}\}$ decomposes $L$ into a set $\mathfrak T^{\g}_{*}$ of compact surfaces where $g$ is the endperiodic automorphism of Lemma~\ref{imagejunct}.

\begin{defn} [$\mathfrak T^{\g}$, $\mathfrak T^{\g}_{*}$, tile, tiling]\label{defntile}

The sets $\mathfrak T^{\g}$ and $\mathfrak T^{\g}_{*}$ will be called \emph{tilings} of $L$. The surfaces in $\mathfrak T^{\g}$ and $\mathfrak T^{\g}_{*}$ will be called \emph{tiles}.

\end{defn}

\begin{rem}
The superscript $\g$ on the symbols $\mathfrak T^{\g}, \mathfrak T^{\g}_{*}$ emphasizes that the juncture components which are  frontiers of the tiles are geodesics.

\end{rem}

\begin{rem}
The tilings $\mathfrak T^{\g}, \mathfrak T^{\g}_{*}$ will be used  in the proofs of Theorems~\ref{geodext} and~\ref{geodext1}. The tiling $\mathfrak T^{\g}$ will be used in the proof of Theorem~\ref{isotlams}.

\end{rem}

The following is  immediate.

\begin{lemma}\label{hstartile}
 The tiles of $\mathfrak T^{\g}$ and of $\mathfrak T^{\g}_{*}$ have boundary consisting of arcs and circles in $\bd L$ and juncture components.   
\end{lemma}

The next lemma is an immediate consequence of Lemma~\ref{imagejunct}.

\begin{lemma}\label{preservestiles}
For $P\in\mathfrak T^{\g}$, $g(P)\in\mathfrak T^{\g}_{*}$ where $g$ is the endperiodic automorphism of \emph{Lemma~\ref{imagejunct}}.
\end{lemma}

\subsubsection{The construction of $h$}\label{geodextpf}

Let $g$ be the endperiodic automorphism of Lemma~\ref{imagejunct}.

\medskip
\ni\textbf{Strategy.} We first define $h$ on $\ol U$, where $U$ is any component  of $L\sm(|\Lambda_{+}|\cup|\Lambda_{-}|\cup|\JJ_{+}|\cup|\JJ_{-}|)$  which    is not a rectangle, in such a way that $h(\bd U)$ is  contained in the union of $|g(\JJ_W)|$ and  the circles and extreme arcs of Lemma~\ref{extremeP}. To define $h$ on all such $\ol U$, we modify $g$ by isotopies (Proposition~\ref{hextreme} and Lemma~\ref{hlinear}) so that if we define $h=g$ on $\ol U$ for all such components $U$, then the procedure of Casson-Bleiler~\cite[pp.~89-90]{bca} can be used  to extend $h$ over the rest of $L$.
\medskip

The  following lemma is proven like  Lemmas~\ref{LambdaK} and~\ref{extremals}.

\begin{lemma}\label{extremeP}

If $P\in\mathfrak T^{\g}_*$, then there are finitely many isotopy classes of leaves of each of the laminations $(\Lambda_{+}\cup(\JJ_{-}\sm g(\JJ_W))|P$ and $(\Lambda_{-}\cup(\JJ_{+}\sm g(\JJ_W))|P$. Each isotopy class either contains one circle in $\JJ_{\pm}$, one arc, or  two extreme arcs $\alpha_{1},\alpha_{2}$ which, together with two arcs in $\bd P$, form a quadrilateral bounding a simply connected region in $P$ containing all arcs in the isotopy class. 

\end{lemma}

\begin{rem}
If an isotopy class contains only one arc we will also refer to that arc as an extreme arc.
\end{rem}

\begin{nota}
For $P\in\mathfrak T^{\g}_*$ denote by  
\begin{eqnarray*}
\GG_{+}(P)&=&g(\Lambda_{+}\cup(\JJ_{-}\sm\JJ_W))|P\\
\GG_{-}(P)&=&g(\Lambda_{-}\cup(\JJ_{+}\sm\JJ_W))|P
\end{eqnarray*}
the laminations consisting of the arcs and circles which are the path components of $|g(\Lambda_{+}\cup(\JJ_{-}\sm\JJ_W))|\cap P$ and $|g(\Lambda_{-}\cup(\JJ_{+}\sm\JJ_W))|\cap P$  and by
\begin{eqnarray*}
\GG^{\g}_{+}(P)&=&(\Lambda_{+}\cup(\JJ_{-}\sm g(\JJ_W)))|P\\ 
\GG^{\g}_{-}(P)&=&(\Lambda_{-}\cup(\JJ_{+}\sm g(\JJ_W)))|P, 
\end{eqnarray*}
the  laminations consisting of the geodesic arcs and circles which are the  path components of $|\Lambda_{+}\cup(\JJ_{-}\sm g(\JJ_W))|\cap P$ and $|\Lambda_{-}\cup(\JJ_{+}\sm g(\JJ_W))|\cap P$.

\end{nota}

Choose  lifts $\wt L$ of $L$ and  $\wt f:\wt L\to \wt L$ of $f$. This determines a lift $\wt g:\wt L\to \wt L$. 

A leaf $\alpha\in\GG_{\pm}(P)$ is either a circle in $g(\JJ_{\pm})$ or an arc.  If $\alpha = [x_1.x_2]\in\GG_{\pm}(P)$ is an arc, then $\alpha$ is contained in a leaf $\gamma_{\alpha}\in g(\Lambda_{+}\cup\JJ_-\cup\Lambda_-\cup\JJ_+)$ and  has endpoints $x_i$ in the geodesics $\sigma_i$ which are either juncture components in $\fr P$ or components of $\bd L$, $i=1,2$.  Let $\wt\alpha = [\wt x_1,\wt x_2]\ss\wt{\gamma_{\alpha}}$ be  lifts of $\alpha = [ x_1,x_2]\ss\gamma_{\alpha}$ with $\wt x_i$ in the lift $\wt\sigma_i$ of $\sigma_i$, $i=1,2$. The  geodesic tightening $\gamma_{\alpha}^{\g}$ of $\gamma_{\alpha}$ has lift $\wt{\gamma_{\alpha}^{\g}}$ sharing endpoints on $\Se$ with $\wt{\gamma_{\alpha}}$ and is a leaf of the lamination $\Lambda_{\pm}\cup\JJ_{\mp}$. There is a unique arc $\wt{\alpha^{\g}} = [\wt{x_1^{\g}},\wt{x_2^{\g}}]\ss\wt{\gamma_{\alpha}^{\g}}$ with $\wt{x_i^{\g}}\in\wt\sigma_i$, $ i=1,2$.

The projection of the arc $\wt{\alpha^{\g}}$ is in $\GG^{\g}_{\pm}(P)$ and will be denoted $\alpha^{\g} = [x_1^{\g},x_2^{\g}]$.  If $\alpha$ is a circle, let $\alpha^{\g}$ be the geodesic tightening of $\alpha$. 

\begin{rem}
The correspondence $\alpha\lra\alpha^{\g}$ induces a one-one correspondence between $\GG_{\pm}(P)$ and $\GG_{\pm}^{\g}(P)$.
\end{rem}

\begin{rem}

Suppose $P\in\mathfrak T^{\g}_*$. 

\begin{enumerate}

\item If $P\ss g(W_-)$, then $\GG_{+}(P) = \0$ and $\GG^{\g}_{+}(P) = \0$\upn{;}  

\item If $P\ss g(W_+)$, then $\GG_{-}(P) = \0$ and $\GG^{\g}_{-}(P) = \0$\upn{;}

\item  If $P = g(K)$, $\GG_{-}(P) \ne \0$, $\GG^{\g}_{-}(P) \ne \0$, $\GG_{+}(P) \ne \0$, $\GG^{\g}_{+}(P) \ne \0$. 

\end{enumerate}
It follows that  an element of $\GG_{-}(P)$ can meet an element of $\GG_{+}(P)$ only if $P=g(K)$.

\end{rem}

We will call $\alpha\in\GG_{\pm}(P)$ an extreme arc  of $\GG_{\pm}(P)$ if the corresponding arc $\alpha^{\g}\in\GG_{\pm}^{\g}(P))$  is an extreme arc of $\GG_{\pm}^{\g}(P))$.

\begin{prop}\label{hextreme}

There exists an  isotopy $\Phi$  preserving each $P\in\mathfrak T^{\g}_*$ and  such that $\Phi^1(\alpha) = \alpha^{\g}$ for each circle or extreme arc  $\alpha\in\GG_{\pm}(P)$, all $P\in\mathfrak T^{\g}_*$.

\end{prop}

\begin{proof}
We define $\Phi$  inductively. Enumerate the circles and  extreme arcs of $\GG_{\pm}(P)$, $P\in\mathfrak T_*^{\g}$, in a sequence $\{\alpha_n\}_{n=0}^{\infty}$ in such a way that every such arc or circle lying in the $i^{\rm th}$-core $K_i$ is listed before every such arc or circle meeting $L\sm K_i$, $i\ge 0$. Let $\psi_0 = \id$ and inductively find sequences $\{\Phi_n\}$ of isotopies and $\{\psi_n\}$ of homeomorphisms such that $\psi_n = \Phi_n^1\circ\psi_{n-1}$, $\Phi_n$ preserves each $P\in\mathfrak T^{\g}_*$,  $\Phi_n$ fixes $\psi_{n-1}(\alpha_i)$ for $1\le i\le n-1$, and $\psi_n(\alpha_n) = \alpha_n^{\g}$.

In defining $\Phi_n$ there are three cases to consider. If $\alpha_n\in\JJ_-\cup\JJ_+$ is a circle in $P$ not meeting $\alpha_1\cup\cdots\cup\alpha_{n-1}$, define $\Phi_n$ using Theorem~\ref{2.1}.  

If $\alpha_n\in\GG_\pm(P)$, $P\in\mathfrak T^{\g}_*$,  is an extreme arc not meeting $\alpha_1\cup\cdots\cup\alpha_{n-1}$,   then  $\alpha_n = [x_1,x_2]$ has endpoints $x_1,x_2$ in geodesics  $\sigma_1,\sigma_2$ which are juncture components in $\fr P$ or components of $\bd L$  and  $\alpha_n\ss\gamma_n\in g(\Lambda_{+}\cup\JJ_-\cup\Lambda_-\cup\JJ_+)$.  By Lemma~\ref{arctight}, there is an isotopy, fixing each $\sigma_i$, $i=1,2$, with lift moving the arc $\wt\alpha = [\wt x_1,\wt x_2]\ss\wt{\gamma_n}$ to the geodesic  arc $\wt{\alpha^{\g}} = [\wt{x_1^{\g}},\wt{x_2^{\g}}]\ss\wt{\gamma_n^{\g}}$, $\wt x_i$ sliding along $\wt{\sigma_i}$ to $\wt{x_i^{\g}}$, $i=1,2$, as in Lemma~\ref{slide}. Here $\wt{\gamma_n^{\g}}$ is the lift of $\gamma_n^{\g}$ sharing endpoints on $\Se$ with $\wt{\gamma_n}$.

If $\alpha_n$ is a circle in $\JJ_-\cup\JJ_+$ or extreme arc meeting $\alpha_1\cup\cdots\cup\alpha_{n-1}$ (which, by the Remark after the introduction of the notation $\GG_{\pm}(P),\GG^{\g}_{\pm}(P)$, can happen only when $P = g(K)$ where $K$ is the core), then use Lemma~\ref{arctight} on each of the subarcs $\alpha_n$ is divided into by $\alpha_1,\ldots,\alpha_{n-1}$. In all three cases apply the remark after Theorem~\ref{epsteinsmooth} to the surface $P$.

Thus, $\psi_{i}\to\psi$, a well defined homeomorphism isotopic to the identity by an isotopy $\Phi$ and $\Phi^1(\alpha) = \alpha^{\g}$ for each $\alpha\in\GG_{\pm}(P)$.
\end{proof}

Redefine $g$ as $\Phi^1\circ g$. 

Since there are finitely many  extreme arcs from Lemma~\ref{extremeP} in each tile that can be edges of $\bd U$  for some component $U$ of $K\sm(|\Lambda_{+}|\cup|\JJ_{-}|\cup|\Lambda_{-}|\cup|\JJ_{+}|)$ which   is not a rectangle, these  extreme arcs  do not accumulate. The following is elementary.

\begin{lemma}\label{hlinear}

After a further isotopy of $g$, we can assume that $g$ is linear in the hyperbolic metric on each arc  that is an edge of $\bd U$  for each component $U$ of $L\sm(|\Lambda_{+}|\cup|\JJ_{-}|\cup|\Lambda_{-}|\cup|\JJ_{+}|)$ which   is not a rectangle.

\end{lemma}

At this stage we define $h|\ol U = g|\ol U$ if $U$ is a component  of $L\sm(|\Lambda_{+}|\cup|\JJ_{-}|\cup|\Lambda_{-}|\cup|\JJ_{+}|)$  which   is not a rectangle.

To finish  the proof of Theorem~\ref{geodext}, we must extend $h$ over the rest of $L$. To do this, we mimic the proof of Lemma~6.1 of Casson-Bleiler~\cite[pp.~89-90]{bca}.  Since Casson-Bleiler deal with an irreducible endperiodic automorphism of a closed surface they do not have to handle regions that are not simply connected. 

Choose  lifts $\wt L$ of $L$ and  $\wt g:\wt L\to \wt L$ of $g$. This determines an extension $\wh g:\wh L\to\wh L$.  Let 
$$X = \Bigl((|\Lambda_{+}|\cup|\JJ_{-}|)\cap(|\Lambda_{-}|\cup|\JJ_{+}|)\Bigr)\cup\Bigl((|\JJ_{-}|\cup|\JJ_{+}|)\cap\bd L\Bigr)$$ 
 and $\wt X$ the set of lifts of the points of $X$ to $\wt L$. Define $\wt h:\wt X\to \wt X$ by,
 \begin{enumerate}
 
 \item $\wt h(\wt x) = (\wt g(\wt{\gamma_1}))^{\g}\cap(\wt g(\wt{\gamma_2}))^{\g}$ if $\wt x = \wt{\gamma_1}\cap\wt{\gamma_2}$ with $\wt{\gamma_1}\in\wt{\Lambda_+}\cup\wt{\JJ_-}$ and $\wt{\gamma_2}\in\wt{\Lambda_-}\cup\wt{\JJ_+}$;
 
 \item $\wt h(\wt x) = \wt g(\wt x)$ if $\wt x\in(|\wt{\JJ_{-}}|\cup|\wt{\JJ_{+}}|)\cap\wt{\bd L}$.
 
 \end{enumerate}
The fact that $\wh g:\wh L\to\wh L$ is continuous implies that the map $\wt h:\wt X\to \wt X$ is a homeomorphism.     As in   Casson-Bleiler~\cite[pp.~89-90]{bca}, we extend $\wt h$ linearly and equivariantly over any lift of an arc of $|\Lambda_{+}|\cup|\JJ_{-}|\cup|\Lambda_{-}|\cup|\JJ_{+}|$ with both endpoints in $X$ and interior disjoint from $X$ and equivariantly over the  lifts of rectangular components of $L\sm(|\Lambda_{+}|\cup|\JJ_{-}|\cup|\Lambda_{-}|\cup|\JJ_{+}|)$   using the technique of  Casson-Bleiler~\cite[pp.~90]{bca}. We have already defined $h$ on the nonrectangular components of $L\sm(|\Lambda_{+}|\cup|\JJ_{-}|\cup|\Lambda_{-}|\cup|\JJ_{+}|)$  to match the extensions over the rectangular components on shared boundary edges. These extensions lift to give $\wt h:\wt L\to \wt L$.
 
Both $\wh f,\wh h:\wh L\to \wh L$ are defined and agree on $\Se$. Thus,  Corollary~5 of~\cite{cc:epstein} implies that the maps $h$ and $f$ are isotopic on $L$.  

Theorem~\ref{geodext} is proven.

\subsection{The escaping sets}\label{furthconax}

From now on, $h:L\to L$ is an endperiodic automorphism, isotopic to $f$, and  permuting the elements of each of the sets $\Lambda_{+}$,  $\Lambda_{-}$,  $\JJ_{+}$, and  $\JJ_{-}$.

\begin{defn}[positive/negative escaping set, $\UU_{e}$, $\UU_{\pm}$]\label{pmesc}
For $e\in\EE(L)$, set $\UU_{e}=\bigcup_{n=-\infty}^{\infty}U_{e}^{n}$. The union of the sets $\UU_{e}$ as $e$ ranges over the negative (respectively  positive) ends will be denoted by $\UU_{-}$ (respectively  $\UU_{+}$).  We will call $\UU_{-}$ the \emph{negative escaping set} and $\UU_{+}$ the \emph{positive escaping set}.
\end{defn}

\begin{lemma}

The set $\UU_{+}$ \upn{(}respectively $\UU_{-}$\upn{)} consists of the set of points $x\in L$ such that the sequence $\{h^{n}(x)\}_{n\ge 0}$  \upn{(}respectively $\{h^{n}(x)\}_{n\le 0}$\upn{)} escapes \upn{(}{\rm Definition~\ref{seqesc}}\upn{)}.

\end{lemma}

\begin{lemma}\label{samesc}
The positive escaping set $\UU_{+}$ and the negative escaping set $\UU_{-}$ are each  independent of the choice of the set  of $f$-junctures.
\end{lemma}

\begin{proof}
By the definition of the sets $\UU_{\pm}$ and the remark after Definition~\ref{hnbhd}, it suffices to prove that the set $\UU_{e}=\bigcup_{n=-\infty}^{\infty}U_{e}^{n}$ is independent of the choice of distinguished neighborhood $U_e$ for every end $e\in\EE(L)$. Let $V_e$ be another choice of distinguished neighborhood of $e$ and $\VV_{e}=\bigcup_{n=-\infty}^{\infty}V_{e}^{n}$. By Lemma~\ref{disfund}, $\{U_{e}^{i}\}_{i=0}^{\infty}$ is a fundamental system of neighborhoods of $e$. Therefore, there exists $i\ge 0$ such that $U_e^i\ss V_e$. It follows that $\UU_e\ss\VV_e$. The reverse inequality is proven in the same way. Thus, $\UU_e = \VV_e$ as desired.
\end{proof}

\begin{rem}
The set $\UU_{e}$ is clearly open and connected. By Lemma~\ref{intofdom}, if $e\ne e'$ are both positive ends or both negative ends, $\UU_{e}\cap\UU_{e'}=\0$, hence the $\UU_{e}$'s are the connected components of $\UU_{-}$, as $e$ ranges over the negative ends, and the connected components of $\UU_{+}$, as $e$ ranges over the positive ends. Evidently, every negative juncture lies in $\UU_{-}$ and every positive one in $\UU_{+}$.
\end{rem}

\begin{rem}
$\UU_{-}\cap\UU_{+}\ne\0$.
 
 \end{rem}
 
\begin{defn}[escaping set, $\UU$]\label{espset}
The \emph{escaping set} is $\UU = \UU_-\cap\UU_+$. 
\end{defn}

\begin{lemma}

The set $\UU$ consists of the set of points $x\in L$ such that the sequence $\{h^{n}(x)\}_{n\in\Z}$ escapes.

\end{lemma}

 \begin{lemma}\label{LambdadoesnotmeeUU'}
 The leaves of $\Lambda_{\pm}$ do not meet $\UU_{\mp}$.
 \end{lemma}
 
 \begin{proof}
  If $x\in\UU_{-}$ then, by Definition~\ref{pmesc},  $x\in\UU_{e}=\bigcup_{n=-\infty}^{\infty}U_{e}^{n}$ for some negative end $e$. Thus,  $x$ has a neighborhood that meets at most one negative juncture component. Thus, $x\notin|\Lambda_{+}|$. A parallel argument shows $\UU_{+}\cap|\Lambda_{-}| = \0$.
  \end{proof}
 
 \begin{cor}\label{lambdacapcore}
 $\Lambda_{+}\cap\Lambda_{-}\ss\intr K$.
\end{cor}

\begin{proof}
Since $W^{-}\ss\UU_{-}$, $W^{+}\ss\UU_{+}$, and $\intr K$ is the complement of $W^{-}\cup W^{+}$, the corollary follows immediately from Lemma~\ref{LambdadoesnotmeeUU'}.
\end{proof}

\begin{lemma}\label{front'}
The frontier of $\UU_{\mp}$ is $|\Lambda_{\pm}|$.
\end{lemma}

\begin{proof}
Let $x\in|\Lambda_{\pm}|$.  By construction, there is a sequence $\{x_{n}\}_{n\ge0}\ss|\XX_{\mp}|$ which converges to $x$.  This sequence lies in $\UU_{\mp}$ and, by Lemma~\ref{LambdadoesnotmeeUU'}, $x\not\in\UU_{\mp}$, hence $x\in\fr\UU_{\mp}$. For the reverse inclusion, let $x\in\fr\UU_{\mp}$.  Then  $x\not\in\UU_{\mp}$ but every connected neighborhood of $x$ meets $\UU_{\mp}$, hence by Definition~\ref{pmesc} meets a distinguished neighborhood, hence meets a  juncture.   It follows that $x\in|\Lambda_{\pm}|$.
\end{proof}

Lemma~\ref{front'} and Corollary~\ref{samesc} imply,

\begin{cor}\label{indep}
The laminations $\Lambda_{\pm}$ are independent of the choice of the set  of $f$-junctures.
\end{cor}

  \begin{prop}\label{notcommon}
Leaves $\wt{\lambda}_{-}\in\wt{\Lambda}_{-}$ and $\wt{\lambda}_{+}\in\wt{\Lambda}_{+}$ cannot have a common ideal endpoint.
\end{prop}

\begin{proof}
Suppose $a\in E$ is a common ideal endpoint of $\wt{\lambda}_{-}\in\wt{\Lambda}_{-}$ and $\wt{\lambda}_{+}\in\wt{\Lambda}_{+}$. Let the end $\epsilon$ of $\lambda_{+}$ have a neighborhood $[x,\epsilon)\ss\lambda_{+}$, whose lift approaches $a$. By Corollary~\ref{noncpt} $[x,\epsilon)\ss\lambda_{+}$ must pass arbitrarily near some positive end $e$.  Let $\{U_{e}^{n}\}_{n=0}^{\infty}$ be a fundamental system of distinguished neighborhoods of $e$ (Lemma~\ref{disfund}). Thus, the neighborhood $[x,\epsilon)$ of $\epsilon$ must cross $J_{e}^{n} = \fr U_{e}^{n}$ for all $n\ge 0$.  Let $\sigma_{n}$ be a component of $J_{e}^{n}$ crossed by $[x,\epsilon)$  and $\wt\sigma_{n}$ be a lift of $\sigma_{n}$ that meets $\wt{\lambda}_{+}$.  By Theorem~\ref{esctoe}, the $\wt\sigma_{n}$ have endpoints on $\Se$ which nest on $a$. Since $\wt{\lambda}_{-}$ has ideal endpoint $a$, it follows that $\lambda_{-}$ meets a positive juncture which  violates Lemma~\ref{LambdadoesnotmeeUU'}.
\end{proof}

\begin{prop}\label{atleastonce}

Every leaf of $\Lambda_{\pm}$ meets at least one leaf of $\Lambda_{\mp}$.

\end{prop}

\begin{proof}
Let $\lambda$ be a leaf of $\Lambda_{-}$.  The proof is parallel for $\lambda\in\Lambda_{+}$.  By Lemma~\ref{nbounded}, $\lambda$ meets some negative juncture component $\sigma\in\XX_{-}$. Choose a point $x\in\sigma\cap\lambda$. Let $\lambda_{n} = h^{n}(\lambda)$ and $\sigma_{n} = h^{n}(\sigma)$.  Let  $x_{n}   = \lambda_{n}\cap\sigma_{n}$.  The sequence $\{x_{n}\}$ accumulates at a point $y\in\intr K$. The point $y$ lies in $|\Lambda_{-}|$ since the sequence $\{x_{n}\}$ lie in the closed set $|\Lambda_{-}|$. Since $x_{n}\in h^{n}(\sigma)$,  the sequence $\{x_{n}\}\ss|\XX_{-}|$. Thus, by the construction of $\Lambda_{+}$, the  point $y$ lies in a leaf $\lambda_{+}\in\Lambda_{+}$.

Let $(V,X,Y,\phi)$ be a bilamination chart for the bilamination $(\Lambda_{+},\Lambda_{-})$ where $V$ is an open neighborhood of $y$.  Choose $N$ so that $x_{N}\in V$.  Then the plaque $P\ss \lambda_{N}\cap V$ containing $x_{N}$ meets $\lambda_{+}$. Thus, $\lambda = h^{-N}(\lambda_{N})$ meets $h^{-N}(\lambda_{+})$ which  is a leaf of $\Lambda_{+}$.  The proposition is proven.
\end{proof}

 \begin{cor}\label{leafmeetsK}
 
 Each leaf of $\Lambda_{\pm}$ meets $\intr K$.

\end{cor}

\begin{proof}
The corollary follows since, by Corollary~\ref{lambdacapcore}, $|\Lambda_{-}|\cap|\Lambda_{+}|\ss\intr K$.
\end{proof}

\subsection{Translations}

We next consider the possibility, not yet excluded, that the laminations $\Lambda_{\pm}$ and $\Gamma_{\pm}$ might be empty.

\begin{prop}\label{total}The following are equivalent,
\begin{enumerate}
\item $\Lambda_{+} = \0=\Lambda_{-}$.\label{ttone}
\item All junctures $J$ escape.\label{tttwo}
\item Some juncture $J$ escapes.\label{ttthree}
\item $f$ is isotopic to a translation $g$ \emph{(Definition~\ref{U'=L})}.\label{ttfour}
\end{enumerate}

\end{prop}

\begin{proof}
Clearly, $(\ref{ttone})\Ra(\ref{tttwo})\Ra(\ref{ttthree})$.
We prove $(\ref{ttthree})\Ra(\ref{ttfour})\Ra(\ref{ttone})$.
\vspace{.1in}

\ni$(\ref{ttthree})\Ra(\ref{ttfour})$. Suppose that $J = \fr U_{e_{-}}$ is a  juncture that escapes where $U_{e_{-}}$ is a  distinguished neighborhood (Definition~\ref{hnbhd}) of a negative end $e_{-}$. The proof in the other case is analogous.  By Definition~\ref{escapes}, the juncture $J$ escapes if the set $\{J_{n}\}_{n\in\Z}$ escapes.  Exactly as in the proof of Lemma~\ref{imagejunct}, $f$ is isotopic to an endperiodic automorphism $g$ such that $g^{n}(J) = J_{n}$, $n\in\Z$.    Then there exists an integer $N\ge 0$ such that $J_{N}\ss W_{+}$.      Since $L\sm g^{N}(U_{e_{-}})$ is connected (by Definition~\ref{perends}) and $J_{N} = \fr g^{N}(U_{e_{-}})\ss W_{+}$, it follows that $L\sm g^{N}(U_{e_{-}})$ is contained in a component of $W_{+}$. That is, $L\sm g^{N}(U_{e_{-}})\ss U_{e_{+}}$ where $U_{e_{+}}$ is a distinguished neighborhood of a positive end $e'_{+}$. It follows that $L$ has just two ends $e_{-},e_{+}$. 

Let $V$ be an arbitrarily small  neighborhood of $e_{+}$ such that $\fr V$ separates $L$. Since $J$ escapes,  there exists an integer $n\ge 0$ such that $J_{n}\ss V$.  Since $L\sm g^{n}(U_{e_{-}})$ is  connected and $\fr V$ separates $L$, it follows that $L\sm g^{n}(U_{e_{-}})\ss V$. Since $V$ was an arbitrarily small neighborhood of $e_{+}$, it follows that $\bigcup_{n=0}^{\infty}g^{n}(U_{e_{-}}) = L$.  Since $\{J_{n}\}_{n\le 0}$ escapes and $g^{n}(J) = J_{n}$, $n\le 0$, it follows that  $U_{e_{-}}$ is $g$-neighborhood of $e_{-}$ and that $g$ is a translation.
\vspace{.1in}

\ni$(\ref{ttfour})\Ra(\ref{ttone})$.  Let $g$ be a translation isotopic to $f$.  By a preliminary isotopy of $g$, we can assume that $f|\bd L=g|\bd L$.  Since $\wh f|E=\wh g|E$ for suitable choices of completed lifts~\cite[Corollary 5]{cc:epstein}, we see that $\wh f|\Se=\wh g|\Se$.   Suppose $\sigma\in\JJ_{-}$. Since $g$ is a translation, $\{g^{n}(\sigma)\}_{n\ge 0}$ escapes. Denote by $\sigma_{n}$, the geodesic tightening of $g^{n}(\sigma)$, $n\ge 0$. By Theorem~\ref{essc}, since $\{g^{n}(\sigma)\}_{n\ge 0}$ escapes, then $\{\sigma_{n}\}_{n\ge 0}$ escapes. Since $\wh f|\Se=\wh g|\Se$, $\sigma_{n}$ is also the geodesic tightening of $f^{n}(\sigma)$, $n\ge 0$. Therefore, an arbitrary negative juncture component $\sigma$ escapes under $f$.  It follows that $\Lambda_{+} = \0$. Similarly $\Lambda_{-} = \0$.
\end{proof}

\begin{cor}\label{everymeets}
If $f$ is not isotopic to a translation and $e$ is a positive \upn{(}respectively negative\upn{)} end of $L$, then every neighborhood  of  $e$ meets $|\Lambda_{+}|$ \upn{(}respectively $|\Lambda_{-}|$\upn{)}.
\end{cor}

\begin{proof}
For definiteness, let $e$ be a positive end of $L$ and consider a juncture $J$ which is the frontier of a distinguished neighborhood of $e$.  The sequence $\{h^{k}(J)\}_{k\ge 0}$ escapes, but by Proposition~\ref{total},  for  some component $\sigma$ of $J$, the sequence  $\{h^{k}(\sigma)\}_{k\le 0}$  does not escape and therefore accumulates locally uniformly on at least one leaf $\lambda_{-}\in\Lambda_{-}$. By Proposition~\ref{atleastonce}, $\lambda_{-}$ meets at least one leaf $\lambda_{+}\in\Lambda_{+}$ and  does so transversely. Thus, $\lambda_{+}$ meets  $h^{k}(J)$, for some  $k<0$.  Applying suitable arbitrarily large positive powers of $h$ produces leaves of $\Lambda_{+}$ meeting $h^{np_{e}}(J)$ for arbitrarily large values of $n$ and  the assertion follows.
\end{proof}

\begin{cor}
$\Lambda_{+}=\0$  if and only if $f$ is isotopic to a translation  if and only if $\Lambda_{-}=\0$.
\end{cor}

\begin{proof}
If either lamination is empty, some juncture escapes. Since~(\ref{ttthree}) implies~(\ref{ttfour}) in Lemma~\ref{total}, $f$ is isotopic to a translation.  If $f$ is isotopic to a translation, the implication~(\ref{ttfour}) $\Ra$~(\ref{ttone}) implies $\Lambda_{+} = \0=\Lambda_{-}$. 
\end{proof}

Thus, endperiodic automorphisms that are isotopic to translations are uninteresting from the point of view of Handel-Miller theory.  

\begin{hyp}\label{hyp5}
\textbf{Hereafter, we assume that $f$ is not isotopic to a translation.}

\end{hyp}


\section{The Complementary Regions to the Laminations}\label{L-lams}

We identify the components of $L\sm|\Lambda_{\pm}|$.  They fall into two essentially different types:  the positive and negative escaping sets (Definition~\ref{pmesc}) and the principal regions. First we need some technicalities about open sets, their internal completions and borders.

\subsection{Internal completion and the border of open sets}\label{section51}

Let $U\sseq L$ be an open, connected subset and define the ``path metric''  $d$ as follows.  Given two points $x,y$ in this open, connected set, there are piecewise geodesic paths  connecting these points.  Each such path has a well-defined length and we define $d(x,y)$ to be the greatest lower bound of these lengths.  It is standard that this defines a topological metric on $U$  compatible with the topology of this open set.   Notice that two points can be very far apart in the metric $d$ and very close together in the hyperbolic metric of $L$.   

\begin{defn}[internal completion, $\ddot U$, $\ddot\iota$]\label{intcompl}
Define the \emph{internal completion} $\ddot{U}$ to be the completion of $U$  in the metric $d$ and define the \emph{boundary} of $\ddot{U}$  to be 
$$
\bd\ddot{ U}=(\ddot{U}\sm U)\cup (U\cap\bd L).
$$
\end{defn}

As is standard, the metric $d$ extends to a metric on $\ddot{U}$ which we will again denote by $d$.

While $\ddot{U}$ and $\bd\ddot {U}$ are not generally subsets of $L$, the inclusion map $\iota:U\hra L$ is a topological embedding. It extends canonically to a continuous map $\ddot{\iota}:\ddot{U}\to L$.  This map may be very pathological on $\ddot{U}\sm U$, but in the cases occurring in this paper, it will be an immersion  on this set.

\begin{defn}[border component]\label{bordcomp}
The image under $\ddot{\iota}$ of a component of $\ddot U\sm U$ will be called a 
\emph{border component} of $U$.   The \emph{set }$\delta U$ of border components of $U$ will be called the \emph{border} of $U$. The \emph{support} $|\delta U|\ss L$ of the border is the union of the elements of $\delta U$ as subsets of L.
\end{defn}

\begin{rem}
Distinct components of $\delta U$ can intersect as subsets of $L$. 

\end{rem}

\begin{rem}
We must be careful to distinguish $|\delta U|$ from the set-theoretic \emph{frontier} $\fr U$. Clearly, $|\delta U|\ss\fr U$.

\end{rem}

\begin{example}
Quite often a component $U$ of the escaping set $\UU$ (Definition~\ref{espset}) has $|\delta U|=\fr U$, but not always. 
In Example~\ref{simpex}, if $U$ is the component of $\UU$ containing the bottom boundary component of $L$, then $\fr U$ properly contains $|\delta U|$ and $\ol U$ properly contains $\ddot\iota(\ddot U)$. That is,  $\ddot\iota(\ddot U)$ is not closed in $L$. For all other components $U$  of $\UU$ in this example, $|\delta U| = \fr U$ and $\ol U = \ddot\iota(\ddot U)$.
\end{example}

\begin{rem}
For the open sets $U$ encountered in this paper, $\delta U$ is a set of lines and circles, often pieced together from arcs in leaves of $\Gamma_{+}$ and $\Gamma_{-}$. 
\end{rem}

These notions extend to open sets that are not connected.  If $U$ is such, $\ddot{U}$ denotes the disjoint union of the internal completions of each component.  Correspondingly, $\bd\ddot{U}$ is defined separately for each component, $\ddot{\iota}:\ddot{U}\to L$ is defined componentwise, and $\delta U$ is the set of images under $\ddot{\iota}$ of components of $\ddot{U}\sm U$.

\begin{lemma}
For each point $x\in\bd\ddot U$, there exists a  continuous map $s:[0,1]\to\ddot{U}$  such that $s(1)=x$ and $s(t)\in U$, $0\le t<1$.
\end{lemma}

\begin{proof}
If $x\in\bd\ddot{U}$, there is a Cauchy sequence (relative to the metric $d$) $\{x_{n}\}_{n=1}^{\infty}\ss U$ that converges to $x$.  Let $\epsilon_{n}>d(x_{n},x_{n+1})$ such that $\lim_{n\to\infty}\epsilon_{n}=0$.  Then there is a piecewise geodesic path $s_{n}\ss U$ joining $x_{n}$ to $x_{n+1}$ of length $<\epsilon_{n}$.  Every point in this path is within $\epsilon_{n}$ of $x_{n+1}$.  Joining these paths sequentially and suitably parametrizing produces a continuous map $s:[0,1)\to U$ such that $\lim_{t\to1}s(t)=x$.  Thus, $s$ is extended continuously to $[0,1]$ so that $s(1)=x$.  
\end{proof}

Applying $\ddot{\iota}$ to the above picture gives the following.

\begin{cor}\label{sd}
If $x$ is a point of an element $\ell\in\delta U$, there exists a  continuous map $s:[0,1]\to L$  such that $s(1)=x$ and $s(t)\in U$, $0\le t<1$.
\end{cor}

A leaf $\lambda_{\pm}$ of $\Lambda_{\pm}$ will either be \emph{isolated} in $\Lambda_{\pm}$ (i.e. approached by points of $|\Lambda_{\pm}|$ on neither side) or approached by points of $|\Lambda_{\pm}|$ on  one or both sides.

\begin{defn}[semi-isolated]\label{defnsemiis}
If a leaf $\lambda_{\pm}\in\Lambda_{\pm}$ is approached by points of  $|\Lambda_{\pm}|$ on at most one side, we say $\lambda_{\pm}$ is \emph{semi-isolated}.
\end{defn}

\begin{rem}
Note that our definition of semi-isolated includes all isolated leaves.  Recall that $\UU$ denotes the escaping set (Definition~\ref{pmesc}).

\end{rem}

\begin{cor}\label{bdUU}
An element of $\delta\UU$ is a subset of the union of the semi-isolated leaves in $\Lambda_+\cup\Lambda_-$.
\end{cor}

\begin{proof}
If $x\in|\delta\UU|$, then $x\notin\UU$, but $x$ is the limit of a sequence $\{x_{n}\}$  of points of $\UU$. Thus, either $x\not\in\UU_+$ or $x\not\in\UU_-$.  In either case, the sequence $\{x_{n}\}\ss\UU_+\cap\UU_-$ and so either $x\in\fr\UU_+$ or $x\in\fr\UU_-$.  By Lemma~\ref{front'}, $x\in|\Lambda_+|\cup|\Lambda_-|$.  Then, by Corollary~\ref{sd},  the points of $|\delta\UU|$ must lie on semi-isolated leaves of $\Lambda_\pm$.
\end{proof}

The sets $U$ that we will be considering have   $|\delta U|$ made up piecewise of  subarcs of  $|\Gamma_{+}|\cup|\Gamma_{-}|$.  Since each  such subarc of $|\delta U|$ has two sides, the following lemma is obvious.

\begin{lemma}

If the set $U$ is such that  $|\delta U|$  is made up piecewise of  subarcs of  $|\Gamma_{+}|\cup|\Gamma_{-}|$, then each such subarc is the image of at most two such subarcs in $\ddot{U}\sm U$.  

\end{lemma}

\begin{rem}
Thus a point $x\in|\delta U|$ may be viewed as (one of possibly two preimages) living in $\ddot {U}\sm U$.  
\end{rem}

\subsection{The positive and negative escaping sets}

Recall  that 
 $|\Lambda_\pm | = \fr\UU_\mp$ (Lemma~\ref{front'}).

\begin{lemma}\label{borderUUpm}
An element of $\delta\UU_{\pm}$ is a  semi-isolated leaf in $\Lambda_{\mp}$. 
\end{lemma}

\begin{proof}
We prove the lemma for $\delta\UU_{+}$. The proof for $\delta\UU_{-}$ is analogous.  If $\gamma\in\delta\UU_{+}$, then $\gamma\ss|\delta\UU_{+}|\ss\fr\UU_{+} = |\Lambda_{-}|$. Since the components of $|\Lambda_{-}|$ are the leaves of $\Lambda_{-}$ and $\gamma$ is connected, it follows that $\gamma$ is a subset of a leaf  $\lambda\in\Lambda_{-}$. Let $U_{\gamma}$ be the component of $\UU_{+}$ such that $\gamma\in\delta U_{\gamma}$. By Corollary~\ref{sd} applied to $x\in\gamma\ss\lambda$, it follows that $\lambda$ borders $U_{\gamma}$ and is semi-isolated on the side bordering $U_{\gamma}$. Thus, $\gamma = \lambda$.
\end{proof}

\begin{lemma}
If $e$ is a negative \upn{(}respectively, positive\upn{)} end, then $\UU_{e}$ \emph{(Definition~\ref{pmesc})} is a component of $L\sm|\Lambda_{+}|$ \upn{(}respectively, of $L\sm|\Lambda_{-}|$\upn{)}.
\end{lemma}

\begin{proof}
Indeed, $\UU_{e}$ is connected and lies in the complement of $|\Lambda_{+}|$.  But $\fr \UU_{e}\sseq\fr\UU_{-}=|\Lambda_{+}|$.
\end{proof}

\begin{prop}\label{side}
If $\lambda\in\Lambda_+$ \upn{(}respectively $\lambda\in\Lambda_{-}$\upn{)} is a border leaf of a component $\UU_e$ of $\UU_-$ \upn{(}respectively of $\UU_+$\upn{)}, then the set of negative junctures   \upn{(}respectively  positive junctures\upn{)} accumulates locally uniformly on $ \lambda$ on any side that borders $\UU_e$.
\end{prop}

\begin{proof}
Suppose $\lambda\in\Lambda_+$ is a border leaf of $\UU_e$ where $\UU_e $ is a component of $\UU_-$.  The proof in the alternate case is analogous. If $x\in\lambda$, then there exists a sequence $\{x_{n}\}_{n\ge 0}\ss\UU_{e}\cap|\XX_{-}|$ which converges to $x$. The proposition then follows since $\Gamma_{+}$ is strongly closed.
\end{proof}

\begin{prop}\label{borderdense}

$|\delta\UU_{\mp}|$ is dense in $|\Lambda_{\pm}|$.

\end{prop}

\begin{proof}
We show that $|\delta\UU_{-}|$ is dense in $|\Lambda_{+}|$. The proof that  $|\delta\UU_{+}|$ is dense in $|\Lambda_{-}|$ is analogous. Let $x$ be a point of a leaf of $\Lambda_{+}$ and $\alpha$ a transverse arc containing $x$ in its interior. Since $|\XX_{-}|$ accumulates on every point of $|\Lambda_{+}|$, there exists a point $y\in\alpha$ lying on a negative juncture and thus in $\UU_{e}$ for some negative end $e$. Thus   $\alpha$ meets a border component of $\UU_{-}$. Since $\alpha$ was arbitrary it follows that $|\delta\UU_{\mp}|$ is dense in $|\Lambda_{\pm}|$.
\end{proof}

\subsection{Principal regions}

In the Nielsen-Thurston theory, the connected components of the complements of the laminations are called principal regions.  In the Handel-Miller theory, there are escaping regions which should not be thought of as principal regions.  In fact, the appropriate analogues of principal regions in our situation do  not even exist in many cases.  When they do, they have a nucleus and finitely many arms, analogous to the principal regions of Nielsen-Thurston.

\begin{defn}[positive/negative principal regions]\label{pnprinreg}
The set $\PP_+$ is the union of the components of the complement of  $|\Lambda_+|$ that contain no points in $|\XX_-|$. We define a \emph{positive principal region} to be a component of $\PP_+$. Similarly, the set $\PP_-$  is the union of the components of the complement of  $|\Lambda_-|$ that contain no points in $|\XX_+|$ and we define a \emph{negative principal region} to be a component of $\PP_-$. 
\end{defn}

\begin{lemma}\label{unions}

$L=\UU_-\cup|\Lambda_+|\cup\PP_+$ and $L=\UU_+\cup|\Lambda_-|\cup\PP_-$ where the unions are disjoint.

\end{lemma}

\begin{proof}
Suppose $V$ is a component of $L\sm|\Lambda_{+}|$ that contains a point $x\in|\XX_{-}|$. Since some component $\UU_{e}$ of $L\sm|\Lambda_{+}|$ contains $x$,  $V = \UU_{e}$.
\end{proof}

\begin{lemma}\label{escsetprop}
The escaping set  $\UU = L\sm (|\Lambda_-|\cup\PP_-\cup|\Lambda_+|\cup\PP_+).$ 
\end{lemma}

\begin{proof}
\begin{eqnarray*}
\UU &=& \UU_-\cap\UU_+\\
  &=& (L\sm(|\Lambda_+|\cup\PP_+)) \cap (L\sm(|\Lambda_-|\cup\PP_-))\\
  &=& L\sm (|\Lambda_-|\cup\PP_-\cup|\Lambda_+|\cup\PP_+).
\end{eqnarray*}
\end{proof}

\begin{defn}[invariant set]\label{invset}
The \emph{invariant set} is $\II =L\sm(\UU_-\cup\UU_+)$.
\end{defn}

Of course, $\II$ will be the set of points $x$ such that niether the positive nor negative iterates under $h$ of $x$ escape.

\begin{lemma}\label{ii}
$\II = 
(|\Lambda_+|\cap|\Lambda_-|) \cup (|\Lambda_+|\cap\PP_-) \cup (|\Lambda_-|\cap\PP_+) \cup (\PP_+\cap\PP_-)$
\end{lemma}

\begin{proof}
\begin{eqnarray*}
\II  &=& L\sm(\UU_-\cup\UU_+)\\
  &=& (L\sm\UU_-)\cap(L\sm\UU_+) \\
  &=& (|\Lambda_+|\cup\PP_+)\cap(|\Lambda_-|\cup\PP_-) \\
  &=& (|\Lambda_+|\cap|\Lambda_-|) \cup (|\Lambda_+|\cap\PP_-) \cup (|\Lambda_-|\cap\PP_+) \cup (\PP_+\cap\PP_-)
\end{eqnarray*}
\end{proof}

\begin{rem}\label{nonescbd}
If there is a finite set of compact components of $\bd L$ which is permuted by $h$, then each of these components lies in $\PP_{+}\cap\PP_{-}$.
\end{rem}

\begin{lemma}
$\PP_+\subset \II \cup\UU_+$ and $\PP_-\subset \II \cup\UU_-$.
\end{lemma}

\begin{proof}
If $x\in\PP_+$ and $x\notin \UU_+$, then $x\in L\sm(\UU_-\cup\UU_+) = \II$. Thus $\PP_+\subset \II \cup\UU_+$. Similarly, $\PP_-\subset \II \cup\UU_-$.
\end{proof}

\subsection{Action of $h$ on semi-isolated leaves of $\Lambda_{\pm}$ and components of $\UU_{\pm}$}

\begin{lemma}
The map $h:\Lambda_{\pm}\to\Lambda_{\pm}$   carries semi-isolated leaves to semi-isolated leaves.
\end{lemma}

This is a consequence of the fact that  $h:L\to L$ is a homeomorphism.

Since $\UU_{\mp}\cup\PP_{\pm}$ consists of all complementary regions of $|\Lambda_{\pm}|$, the semi-isolated leaves of $\Lambda_{\pm}$ are exactly the border leaves of these regions.  In the case of border leaves of $\UU_{\pm}$, it is possible that both sides of the leaf borders this set, in which case the semi-isolated leaf is actually isolated and the natural map $\ddot{\iota}:\bd\ddot{\UU}_{\pm}\to L$ will not be  one-one, but rather will identify some boundary components pairwise.  In the case of $\delta{\PP}_{\pm}$, this cannot happen.

\begin{lemma}
Each leaf $\ell$ of $\delta\PP_{\pm}$ borders $\PP_{\pm}$ on only one side.  Thus, $\ddot{\iota}:\bd\ddot{\PP}_{\pm}\to L$ is one-one.
\end{lemma}

\begin{proof}
Indeed, in the case of $\PP_{+}$, negative $h$-junctures cluster on $\ell$ on one side. An analogous argument holds for  $\PP_{-}$.
\end{proof}

\begin{lemma}\label{hinv}
 The sets of border leaves $\delta\UU_{\pm}$ and $\delta\PP_{\pm}$ are $h$-invariant.
\end{lemma}

\begin{proof}
 An isolated side of a leaf $\lambda$ borders $\UU_{\pm}$ if and only if a sequence of leaves of $\XX_{\mp}$ accumulates on $\lambda$ from that side. The lemma then follows since  $h:L\to L$ is a homeomorphism.
\end{proof}

\begin{cor}

The set $\delta\UU$ is $h$-invariant.

\end{cor}

\begin{lemma}

$h(\UU_{e}) = \UU_{f(e)}$. 

\end{lemma}

\begin{proof}
Recall that $\UU_{e} = \bigcup_{n=-\infty}^{\infty}U_{e}^{n}$ where $U_{e}^{n} = h^{np_{e}}(U_{e})$ for $U_{e}$ a distinguished neighborhood of $e$ and  that $\UU_{f(e)} = \bigcup_{n=-\infty}^{\infty}U_{f(e)}^{n}$ where $U_{f(e)} = h(U_{e})$ is a distinguished neighborhood of $f(e)$ and $U_{f(e)}^{n} = h^{np_{e}}(U_{f(e)}) = h^{np_{e}}(h(U_{e})) = h(U_{e}^{n})$ (clearly $p_{e} = p_{f(e)}$).  Thus, 
$$h(\UU_{e}) = h\bigl(\bigcup_{n=-\infty}^{\infty}U_{e}^{n}\bigr) = \bigcup_{n=-\infty}^{\infty}h(U_{e}^{n})  = \bigcup_{n=-\infty}^{\infty}U_{f(e)}^{n} = \UU_{f(e)}.$$
\end{proof}

\begin{rem}
We will show in Section~\ref{crown} that there are finitely many principal regions and that the map $h$ permutes the principal regions.

\end{rem}


\section{Semi-isolated Leaves}\label{negpos}

It will turn out that there are only finitely many semi-isolated leaves and that each contains an $h$-periodic point.  Our analysis will cast more light on $\UU_{\pm}$ and $\PP_{\pm}$.

\subsection{Counting the semi-isolated leaves}\label{finsemisolvs}

Recall that the core $K$ is the complement of the union of the interiors of the disjoint distinguished neighborhoods $U_{e}$ as $e$ ranges over the set $\EE(L)$ of ends of $L$, and that the choice of core is not unique. If $J_{e} = \fr U_{e}$, then $\fr_{\pm}K=\bigcup_{e\in\EE_{\pm}(L)}J_{e}$.

\begin{defn}[rectangle, rectangular]\label{rctnglr}

A \emph{rectangle} is a four sided, simply connected set $R\ss L$ with one pair of opposite edges in $|\Gamma_{+}|$ and the other pair of opposite edges in $|\Gamma_{-}|$.  The set $R$ will also be described as  ``rectangular''.
\end{defn}

Note that a rectangle can be open, closed, or neither.

\begin{rem}\label{notgeomrect}
We are not using ``rectangle'' in the geometric sense that the edges meet in right angles.  In hyperbolic geometry, there are no such rectangles.  Here, we are following Casson and Bleiler~\cite[p.~99]{bca} where rectangles had opposite edges in a pair of transverse foliations.  Our rectangles are convex geodesic quadrilaterals.
\end{rem}

Consider an arc $\alpha$ of $ |\Lambda_{\pm}|\cap K$  with endpoints on components of $\fr K$ and consider the isotopy classes of such arcs \emph{where the isotopy is to be through arcs with endpoints on components of $\fr K$}.   By Lemma~\ref{extremals}, this isotopy class contains two extreme arcs. These two extreme arcs together with a pair of arcs in $\fr K$ bound a rectangle $R_{\alpha}$ containing all the arcs in the isotopy class.  Remark that the  rectangle $R_{ \alpha}$ may degenerate into a single arc.

\begin{defn}[extreme rectangle, $R_{\alpha}$]\label{extquad}

We will call $R_{\alpha}$ the \emph{extreme rectangle} associated to the arc $\alpha$.

\end{defn}

\begin{lemma}\label{only}
Let $\alpha$ be an arc of $ |\Lambda_{+}|\cap K$ \upn{(}respectively $ |\Lambda_{-}|\cap K$\upn{)}. The extreme rectangle $R_{\alpha}$ meets $| \Lambda_{+}|$ \upn{(}respectively $ |\Lambda_{-}|$\upn{)} only in arcs isotopic to $ \alpha$.
\end{lemma}

\begin{proof}
We consider the case $\alpha$  an arc of $ |\Lambda_{+}|\cap K$. The case $\alpha$ an arc of $ |\Lambda_{-}|\cap K$ is analogous. The leaves of $ \Lambda_{+}$ cannot properly intersect one another, so any arc  $ \beta$ of $ |\Lambda_{+}|\cap K$ that meets $R_{ \alpha}$ must be contained in that extreme rectangle.  The ends of $ \beta$ cannot lie in the same edge of $R_{ \alpha}$ because then $\beta$ would form a digon with a component of $\fr_{+}K$. But there are no geodesic digons in hyperbolic geometry.  Thus $ \beta$ is isotopic to $ \alpha$.
\end{proof}

The following lemma follows immediately from Lemmas~\ref{LambdaK} and~\ref{extremals}.

\begin{lemma}\label{finman}
The arcs $ \alpha$ of $| \Lambda_{+}|\cap K$ \upn{(}respectively $| \Lambda_{-}|\cap K$\upn{)} fall into finitely many isotopy classes, determining finitely many disjoint extreme rectangles $R_{ \alpha}$.
\end{lemma}

\begin{theorem}\label{finmany}
There are only finitely many semi-isolated leaves. Every semi-isolated leaf is $h$-periodic.
\end{theorem}

\begin{proof}
We will prove the theorem for semi-isolated leaves in $\Lambda_{-}$. The proof in the other case is analogous. Let $\AAA$ be the set of semi-isolated leaves $\lambda\in\Lambda_{-}$ such that there exists  a component $\alpha$ of $\lambda\cap K$  which is an edge of a component $U$ of $\UU_+\cap K$ which is not a rectangle.  Since $U$ is not a rectangle, it follows that $\alpha$ is an edge of the extreme rectangle $R_{\alpha}$. Thus, by Lemma~\ref{finman}, the set $\AAA$ is finite.

Note that if $\lambda\in\AAA$, then $h(\lambda)\in\AAA$. In fact, if $\lambda\in\AAA$, there exists a component $\alpha$ of $\lambda\cap K$ which is the edge of a component $U$ of $\UU_+\cap K$ which is not a rectangle. If $h(\lambda)\notin\AAA$, then   any component  of $\UU_{+}\cap K$ that has $h(\alpha)$ on its boundary  is a rectangular component $R$  of $\UU_{+}\cap K$.  If $\lambda$ is an isolated leaf, there would be two such $R$. Otherwise $R$ is unique.  In either case it follows that the component $U$ of $\UU_{+}\cap K$ with $\alpha$ an edge of $U$ is a  component of $h^{-1}(R)\cap K$ and thus a rectangle. This contradiction implies $h(\lambda)\in\AAA$.

Let $\lambda$ be a semi-isolated leaf of $\Lambda_{-}$. We will prove that there are finitely many semi-isolated leaves by showing that $\lambda$ lies in the finite set $\AAA$, proving that each semi-isolated leaf is $h$-periodic along the way. Let $\lambda_{n} = h^{n}(\lambda)$, $n\in\Z$. We first show that $\lambda_{N}\in\AAA$ for some $N\ge 0$ and therefore by the previous paragraph that $\lambda_{n}\in\AAA$ for all $n\ge N$. Assume on the contrary that $\lambda_{N}\notin\AAA$ for any $N\ge 0$. Since the leaf $\lambda\notin\AAA$ and meets the core $K$, it follows that $\lambda$ contains an edge of a rectangular component $R$ of $\UU_{+}\cap K$.   Again since $\lambda_{1}\notin\AAA$, the component of $\UU_{+}$ meeting $h(R)$ contains a rectangular component $R_{1}$ of $\UU_{+}\cap K$ so that $h(R)\ss R_{1}$.  By iterating this argument, we obtain an infinite increasing nest 
$$
R\ss h^{-1}(R_{1})\ss h^{-2}(R_{2})\ss\cdots\ss h^{-r}(R_{r})\ss\cdots,
$$
where each $R_{r}$ is a rectangular component of $\UU_{+}\cap K$.  Since $R_{r}$ is rectangular with two edges  segments of negative junctures, $h^{-r}(R_{r})$ is rectangular with two edges segments of negative junctures that lie arbitrarily deep in neighborhoods of negative ends as $r\to\infty$.  The increasing union of this nest of rectangles gives an infinite rectangle whose sides are leaves of $\Lambda_{-}$.  By Theorem~\ref{esctoe}, a lift of this infinite rectangle has sides with the same ideal endpoints.   But geodesics with lifts having the same endpoints coincide, so our rectangles all degenerate to arcs in $\lambda$, a contradiction.

Thus, $\lambda_{n}\in\AAA$ for $n\ge N$. Since the set $\AAA$ is finite, so there exists integers $n\ge N$ and $p>0$ such that  $\lambda_{n} = \lambda_{n+p}$. Therefore,
$$\lambda_{k} = h^{k-n}(\lambda_{n}) = h^{k-n}(\lambda_{n+p}) = \lambda_{k+p},\quad\rm{all}\quad k\in\Z.$$
Thus $\lambda$ is $h$-periodic. It follows that $\lambda = \lambda_{ip}\in\AAA$ if $i\in\Z$ is chosen so that $ip\ge N$. 
\end{proof}

\begin{cor}
 For every negative \upn{(}respectively, positive\upn{)} end $e$ of $L$, there is a leaf of $\Lambda_{-}$  \upn{(}respectively, of $\Lambda_{+}$\upn{)} with an end that passes arbitrarily near $e$ \upn{(}see \emph{Definition~\ref{arbitrarilynear})}.
\end{cor}

\begin{proof}
Consider the case that $e$ is a negative end, the other case is analogous.  By Corollary~\ref{everymeets}, every neighborhood of $e$ meets $|\Lambda_{-}|$, hence meets semi-isolated leaves of $\Lambda_{-}$.  Since there are only finitely many of these leaves, the assertion follows.
\end{proof}

\begin{rem}
Of course, every end of every leaf of the laminations passes arbitrarily near an end of $L$.  The above corollary only points out that no end of $L$ is left out.
\end{rem}

\subsection{Periodic leaves}\label{periodicleaves}

We consider leaves $\lambda_{\pm}\in\Lambda_{\pm}$ which are periodic under $h$. For definiteness, assume that $\lambda_-\in\Lambda_-$ is $h$-periodic of period $p$. If  $h^{p}$ reverses orientation  of $\lambda_-$, replace $p$ by $2p$. Thus, we assume $h^{p}(\lambda_{-})=\lambda_{-}$ and preserves the orientation of $\lambda_{-}$.  

\begin{nota}
Set $g=h^{p}$. 

\end{nota}

We will study this situation in the universal cover $\wt{L}\ss\Delta$. Recall that the ideal boundary of $\wt L$ is $E = \wt L\cap\Si$ (Definition~\ref{idbd}). In the universal cover $\wt L$, we fix a lift $\wt{\lambda}_{-}$ and choose the lift $\wt{g}:\wt{\Gamma}_{\pm}\to\wt{\Gamma}_{\pm}$  so that $\wt{g}(\wt{\lambda}_{-})=\wt{\lambda}_{-}$.   Set $\wh g=\wh h^{p}$ and remark that $\wh{g}$ fixes each endpoint $c$ and $d$  of $\wh\lambda_{-}$ in $E$.

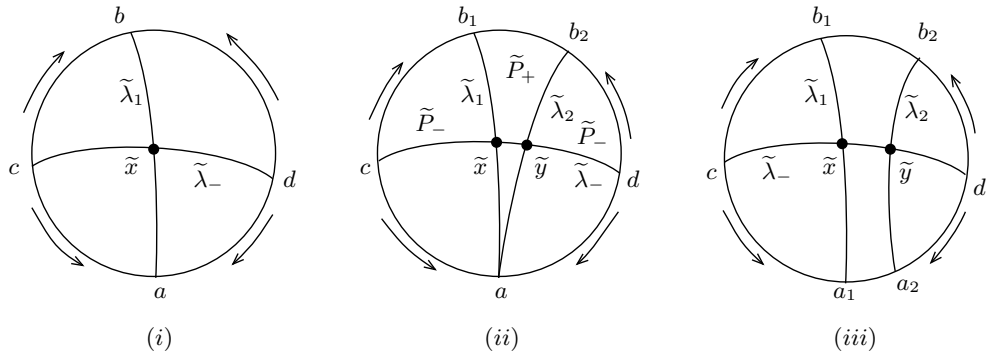
\begin{figure}[b]
\begin{center}
\begin{picture}(300,160)(42,-180)
\rotatebox{270}{\includegraphics[width=300pt]{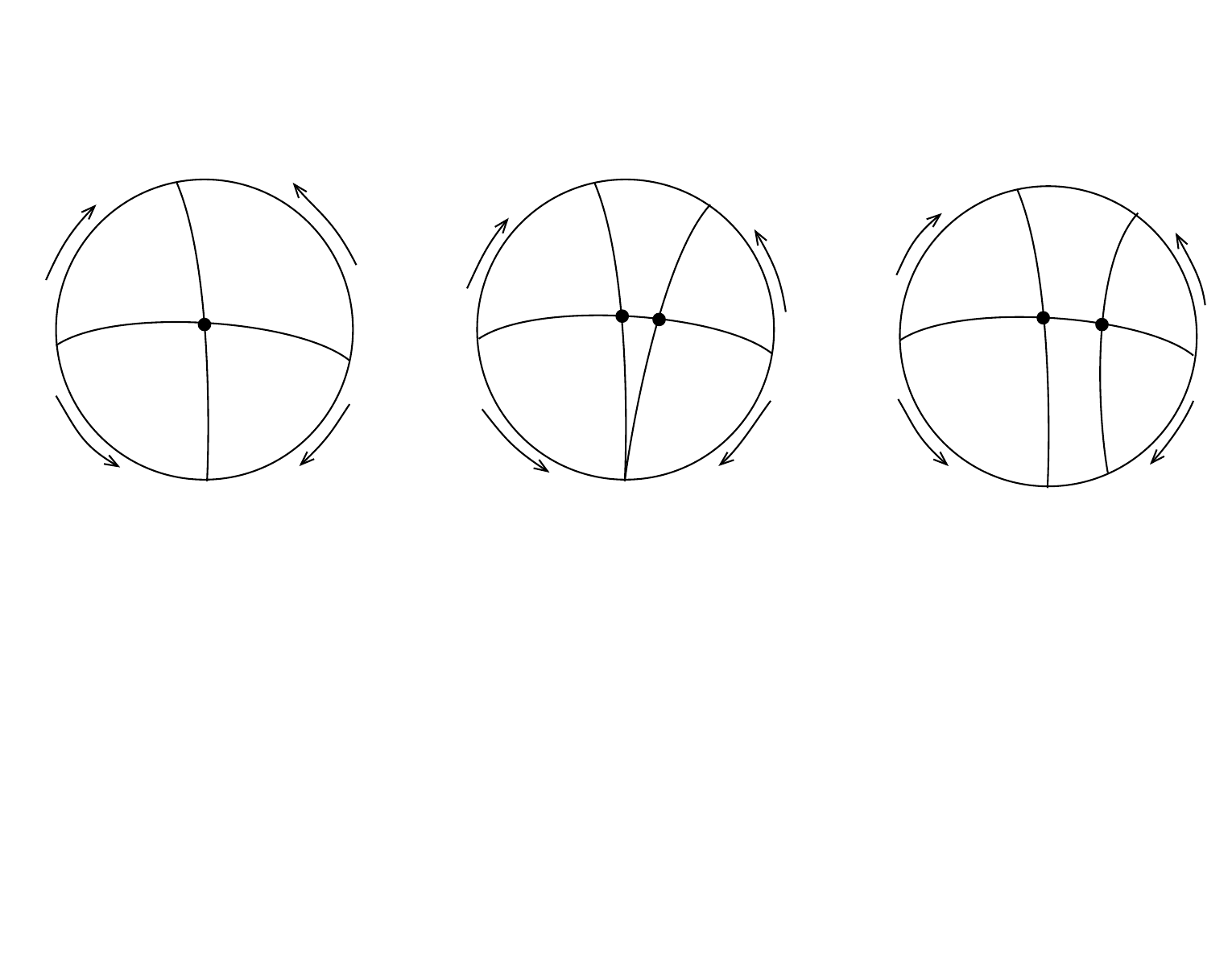}}

\put (-332,-83){\small$\tilde{\lambda}_1$}
\put (-304,-115){\small$\tilde{\lambda}_-$}
\put (-330,-110){\small$\tilde{x}$}
\put (-334,-53){\small$b$}
\put (-319,-157){\small$a$}
\put (-374,-110){\small$c$}
\put (-270,-115){\small$d$}
\put(-322,-175){\small$(i)$}

\put (-203,-83){\small$\tilde{\lambda}_1$}
\put (-169,-89){\small$\tilde{\lambda}_2$}
\put (-160,-115){\small$\tilde{\lambda}_-$}
\put (-175,-110){\small$\tilde{y}$}
\put (-198,-110){\small$\tilde{x}$}
\put (-220,-94){\small$\tilde{P}_-$}
\put (-186,-74){\small$\tilde{P}_+$}
\put (-159,-99){\small$\tilde{P}_-$}
\put (-204,-53){\small$b_{1}$}
\put (-162,-59){\small$b_{2}$}
\put (-190,-157){\small$a$}
\put (-241,-110){\small$c$}
\put (-140,-115){\small$d$}
\put(-194,-175){\small$(ii)$}

\put (-73,-83){\small$\tilde{\lambda}_1$}
\put (-35,-89){\small$\tilde{\lambda}_2$}
\put (-89,-112){\small$\tilde{\lambda}_-$}
\put (-37,-112){\small$\tilde{y}$}
\put (-66,-110){\small$\tilde{x}$}
\put (-70,-53){\small$b_{1}$}
\put (-30,-59){\small$b_{2}$}
\put (-62,-157){\small$a_{1}$}
\put (-38,-154){\small$a_{2}$}
\put (-110,-112){\small$c$}
\put (-9,-118){\small$d$}
\put(-62,-175){\small$(iii)$}

\end{picture}
\caption{Three cases}\label{PR}
\end{center}

\end{figure}

We get  three possibilities by taking the completion $\wh{ \sigma}$ of a lift $\wt \sigma$ of a component $\sigma$ of $\XX_{-}$ which meets  $\tilde{\lambda}_-$ and iterating using $\wt{g}$. Remark that $\sigma$ cannot be an escaping component since it meets a leaf of $\Lambda_{-}$. Thus, depending on where we start, the strongly closed property (Definition~\ref{strnglyclsed}) implies that  the sequence $\{\wh g^{n}(\wh{\sigma})\}_{n=1}^{\infty}$  either strongly converges to (1) the completion  $\wh\lambda_{1}$ of a leaf $\wt\lambda_{1}$ of $\wt{\Lambda}_+$ from the left or (2) the completion $\wh{ \lambda}_{2}$ of a leaf  $\wt{\lambda}_2\in\wt{\Lambda}_+$ from the right. (See Figure~\ref{PR}.  It is drawn explicitly for the case that $\wt{L}=\Delta$, but works just as well when $\wt{L}$ is a proper subset of the unit disk and $E=\wh{L}\cap\Si$ is a Cantor set. In the latter case, the circles in the figure represent $\Se$.) 

\begin{lemma}
The sequence $\{\wt g^{-n}(\wt\sigma)\}_{n=1}^{\infty}$ converges  to  $c$ in case~\emph{(1)} and to $d$ in case~\emph{(2)}, hence the endpoints of $\wh\lambda_{-}$  are repelling under $\wh g:\Se\to \Se$.
\end{lemma}

Indeed, $\{h^{k}(\sigma)\}_{k\le 0}$ escapes (cf. Theorem~\ref{esctoe}). By abuse, we will often say that $\wt g$ fixes the endpoints $c$ and $d$.  Following the endpoints (ideal or finite) of $\wt g^{-n}(\wt\sigma)$, we see that the action of $\wh g$ on $\Se$ is as indicated by the arrows in Figure~\ref{PR}.

The leaf $\tilde{\lambda}_-\in\tilde{\Lambda}_-$ must be approached on at least one side by  lifts of  components of $\XX_{+}$. Without loss, we may assume there exists a component $\tau$ of $\XX_{+}$ so that  the sequence $\{\wh g^{n}(\wh{\tau})\}_{n\le 0}$ strongly converges to  $\wh{\lambda}_-$ from below (in Figure~\ref{PR} as drawn).  

\begin{lemma}\label{caseiiinot}
Case $(iii)$ in \emph{Figure~\ref{PR}} cannot occur.
\end{lemma}

Again, since $\{h^{k}(\tau)\}_{k\ge 0}$ escapes, Theorem~\ref{esctoe}  implies that the sequence $\{\wh g^{n}(\wh{\tau})\}_{n=1}^{\infty}$ must converge to  a point  $a\in E$ as in Figure~\ref{PR}~$(i)$ or~$(ii)$. This proves the impossibility of Figure~\ref{PR}~$(iii)$.  We now analyze Figure~\ref{PR}~$(i)$~and~$(ii)$.

\begin{prop}\label{3cases}

Let $\wt{\lambda}_-$ and its ends be fixed by $\wt{g}$. Then,

\begin{enumerate}
\item[$(i)$] If $\tilde{\lambda}_1 = \tilde{\lambda}_2$, the endpoints $a,b$ of $\tilde{\lambda}_1$ in $E$ are attracting fixed points of $\wt{g}$ and the endpoints $c,d$ of $\tilde{\lambda}_-$ in $E$ are repelling fixed points and no point of $E\sm\{a,b,c,d\}$ is fixed by $\wt{g}$. The point $\wt{x}=\wt{\lambda}_{1}\cap\wt{\lambda}_{-}$ is fixed by $\wt{g}$.
\item[$(ii)$] If $\tilde{\lambda}_1 \ne \tilde{\lambda}_2$,  the endpoints $c,d$ of $\tilde{\lambda}_-$ are repelling fixed points of $\wt{g}$, there is a common endpoint $a$ of $\tilde{\lambda}_1,\tilde{\lambda}_2$ which is is an attracting fixed point, and the other endpoints $b_{1},b_{2}$ of $\tilde{\lambda}_1,\tilde{\lambda}_2$ are attracting on the sides facing the endpoints of $\tilde{\lambda}_-$. No point in the intervals $(c,b_{1}),(c,a),(d,b_{2}),(d,a)$ is fixed by $\wt{g}$.  The points $\wt{x}=\wt{\lambda}_{1}\cap\wt{\lambda}_{-}$ and $\wt{y}=\wt{\lambda}_{2}\cap\wt{\lambda}_{-}$ are each fixed by $\wt{g}$.  
\end{enumerate}
Parallel assertions hold when $\wt{\lambda}_{-}$ is replaced by a leaf $\wt{\lambda}_{+}$ of $\wt{\Lambda}_{+}$.

\end{prop}

\begin{proof}
If $\tilde{\lambda}_1 = \tilde{\lambda}_2$, we let $\tilde{x} = \tilde{\lambda}_{-}\cap \tilde{\lambda}_{1}$. If $\tilde{\lambda}_1 \ne \tilde{\lambda}_2$, we let $\tilde{x} = \tilde{\lambda}_{-}\cap \tilde{\lambda}_{1}$ and  $\tilde{y} = \tilde{\lambda}_{-}\cap \tilde{\lambda}_{2}$. If $\tilde{\lambda}_1 = \tilde{\lambda}_2$ we have case $(i)$. If $\tilde{\lambda}_1 \ne \tilde{\lambda}_2$ but $\tilde{\lambda}_1,\tilde{\lambda}_2$ share the endpoint $a\in E$, then we have case $(ii)$. 

Parallel considerations hold when $\wt{\lambda}_{-}$ is replaced by a leaf $\wt{\lambda}_{+}$ of $\wt{\Lambda}_{+}$.
\end{proof}

\begin{cor}\label{cor722}
Every semi-isolated leaf  has an $h$-periodic point.
\end{cor}

\begin{proof}
By Theorem~\ref{finmany} every semi-isolated leaf is $h$-periodic and so  Proposition~\ref{3cases} implies that every semi-isolated leaf has an $h$-periodic point.
\end{proof}

The set $\wt P_{+}$ (and a corresponding set  $\wt P_{-}$ when $\wt\lambda_{-}$ is replaced by $\wt\lambda_{+}$) in Figure~\ref{PR}~$(ii)$ are lifts of  \emph{principal regions} and will be explained shortly.

\begin{cor}\label{attracting}
 In case~$(i)$, $\wt{x}$ is an attracting \upn{(}in $\wt\lambda_{-}$\upn{)} $\wt{g}$-fixed point on both sides and $\wt{\lambda}_{1}$ can not border the lift of a principal region on either side. In case~$(ii)$, $\wt{x}$ and $\wt{y}$ are attracting  in $\wt\lambda_{-}$  on the sides not meeting $(\wt{x},\wt{y})$.  If $\wt{\lambda}_{-}$ is replaced by $\wt{\lambda}_{+}\in\wt{\Lambda}_{+}$, these points are repelling  in $\wt\lambda_{+}$.
\end{cor}

Interchanging the roles of $\lambda_{-}$ and $\lambda_{1}$ in the previous corollary one sees,

\begin{cor}
  In case~$(i)$,  $\wt{\lambda}_{-}$ can not border the lift of a principal region on either side.
\end{cor}

The proof of the following corollary  is similar to the proof of Lemma~\ref{caseiiinot}.

\begin{cor}\label{cor714}
 In case~$(ii)$,  $\wt{\lambda}_{-}$ is semi-isolated and $\wt\lambda_{-}$ is bordered above by the lift of a negative principal region $P_{-}$.
\end{cor}

\begin{proof}
 In case~$(ii)$ of Figure~\ref{PR}, if there exists a sequence $\{x_{n}\}$ of lifts of  points of $|\XX_{+}|$ converging to a point of  $\wt{\lambda}_{-}$ from above, then, by the strongly closed property, there exists a juncture component $\tau\in\XX_{+}$ with lift $\wt\tau$ above $\wt{\lambda}_{-}$ with endpoints on $\Se$ separated by both $b_{1}$ and $b_{2}$.  Since $\{h^{k}(\tau)\}_{k\ge 0}$ escapes, Theorem~\ref{esctoe}  implies that the sequence $\{\wh g^{n}(\wh{\tau})\}_{n\ge 0}$ must converge to  one point  in $E$  implying $b_{1}=b_{2}$ and contradicitng that we are in case~$(ii)$. The corollary follows.
\end{proof}

\begin{cor}
If $\lambda_{-}\in\Lambda_{-}$ and $p$ are as in the second paragraph of \emph{Section~\ref{periodicleaves}}, then $\lambda_{-}$ contains a periodic point of period $p$ or $p/2$.
\end{cor}

Further, since the segment $[\wt{x},\wt{y}]$ is contained in the invariant set, no lifts of components of positive junctures can meet the segment $[\wt{x},\wt{y}]$. It follows that,

\begin{cor}\label{cor715}
 In case~$(ii)$,  $\wt{\lambda}_{1}, \wt{\lambda}_{2}\in\wt{\Lambda}_{+}$ are semi-isolated and lifts of  border leaves of a positive principal region $P_{+}$.
\end{cor}

\begin{lemma}\label{left/right}
The leaf $\wt{\lambda}_{-}$ can either be approached from below by leaves of $\wt{\Lambda}_{-}$ or be  bordered  from below by a lift of $\UU_{+}$. Similarly, $\wt{\lambda}_{1}$ \upn{(}respectively $\wt{\lambda}_{2}$\upn{)} can be approached from the left \upn{(}respectively from the right\upn{)} by leaves of $\wt{\Lambda}_{+}$ or be the border from the left \upn{(}respectively from the right\upn{)} of a lift of $\UU_{-}$.
\end{lemma}

Let $\wt{A}$  be the portion of the cusp in case~$(ii)$ of Figure~\ref{PR} bordered by the half-infinite segments $[\wt{x},a)\ss\wt{\lambda}_{1}$ and $[\wt{y},a)\ss\wt{\lambda}_{2}$ and the segment  $[\wt{x},\wt{y}]$.

\begin{defn}[arm]\label{defnarm}

The projection $A\ss L$ of  $\wt{A}$ under the covering map is called an \emph{arm} of the principal region $P_{+}$.

\end{defn}

\begin{lemma}\label{simpconnarms}

The projection of $\wt{A}$ onto $A$ is one--one. Thus any arm $A$ of a principal region is simply connected.

\end{lemma}

\begin{proof}
We must show that if $T$ is any nontrivial covering transformation, then $T(\wt{A})\cap\wt{A}=\emptyset$. If $T(a)\ne a$ and if $T(\wt{A})\cap\wt{A}\ne\emptyset$ then one of $T(\wt{\lambda}_{1}),T(\wt{\lambda}_{2})$ must meet one of $\wt{\lambda}_{1},\wt{\lambda}_{2}$ which is a contradiction. Thus we can assume $T(a) = a$ but in that case, if $T(\wt{A})\cap\wt{A}\ne\emptyset$, we must have $T(\wt{\lambda}_{1}) = \wt{\lambda}_{1}$, $T(\wt{\lambda}_{2}) = \wt{\lambda}_{2}$. This is not possible since $T$ must have exactly two fixed points on $E$.
\end{proof}

\subsection{Escaping ends}\label{ee}

Let $ \lambda$ be a leaf of $ \Lambda_{\pm}$ and $ \epsilon$ an end of $ \lambda$.

\begin{defn}[escaping end of leaf]\label{defnescpend}
 A ray $[x,\epsilon)\ss\lambda$ represents an escaping end  $\epsilon$ of $\lambda$ if there is an end $e$ of $L$ such that, for every neighborhood $U$ of $e$ in $L$, there is $z\in(x,\epsilon)$ such that $(z,\epsilon)\ss U$. 
\end{defn}

\begin{rem} By abuse of language, we often call the ray $[x, \epsilon)$ itself an escaping end.  Note that this is a much stronger property than passing arbitrarily near an end $e$ of $L$.  The latter allows  return to the core infinitely often  which the above definition does not.  By Proposition~\ref{atleastonce}, we can and do require that $x\in|\Lambda_{+}|\cap|\Lambda_{-}|$.
\end{rem}

Let $e$ be a positive end of $L$ and consider the component $\UU_{e}$ of the positive escaping set. Since $L$ has  only finitely many ends and, since there are only finitely many border leaves of $\UU_{e}$ (Theorem~\ref{finmany}), we can assume that there is  $p\ge1$  such that $h^{p}$ takes each such border leaf  to itself, preserving its ends.   

\begin{rem}
As before, the use of the terms ``below''  and ``above'' in this subsection always refers to Figure~\ref{PR}.
\end{rem}

Fix a connected lift $\tilde{\UU}_{e}$  of $\UU_{e}$. The lift $\wt{g}$  of $h^{p}$  that we consider can be chosen to take $\wt{\UU}_{e}$, hence $\delta\tilde{\UU}_{e}$, to itself. In additon, we can assume that $\wt{g}$  fixes a lift $\tilde{\lambda}_{-}\ss\delta\wt{\UU}_{e}$ of a specific border leaf $\lambda_{-}$ of $\UU_{e}$.  View $\wt{\lambda}_{-}$ as in Figure~\ref{PR}, $(i)$ or $(ii)$.  In case~$(i)$, interchange the roles of $a$  and $b$, if necessary, to assume that $\wt{\UU}_{e}$ lies below $\wt{ \lambda}_{-}$.  In case~$(ii)$, $\wt{\UU}_{e}$ also borders $\wt{ \lambda}_{-}$ from below (Corollary~\ref{cor714}).  Thus every point  $z\in\wt{\lambda}_{-}$ has the property that a small enough transverse arc $[z,\eta)$ issuing  from $z$ into the region below $\wt{\lambda}_{-}$ in Figure~\ref{PR} has $(z,\eta)\ss\wt{\UU}_{e}$.  Thus,

\begin{lemma}\label{notbelow}
Under the above assumptions, leaves of $\wt{\Lambda}_{-}$ cannot accumulate on $\wt{\lambda}_{-}$ from below.
\end{lemma}

\begin{lemma}
The ray\upn{(}s\upn{)} $(\wt{x},a)$ \upn{(}and $(\wt{y},a)$\upn{)}  lie in $\wt{\UU}_{e}$.
\end{lemma}

\begin{proof}
Indeed, no lift of a leaf of  $\wt{\Lambda}_{-}$ can have one endpoint in the arc of $\Se$ between $a$ and $c$ and the other in the arc between $a$ and $d$.  Otherwise leaves of $\wt{\Lambda}_{-}$ would accumulate on $\wt{\lambda}_{-}$ from below, contradicting Lemma~\ref{notbelow}.  Since the rays in question start out in $\wt{\UU}_{e}$, they can never exit.
\end{proof}
 
  By Proposition~\ref{side} and the strongly closed property, we conclude the following.

\begin{lemma}
Completions of components of $\wt{\XX}_{+}$ accumulate on $\wh{ \lambda}_{-}$ from below, becoming uniformly close in the Euclidean metric on $\wh L\sseq \D^{2}$.
\end{lemma}

By this lemma, let a component $\wt{\sigma}$  of $\wt{\XX}_{+}$ meet $(\wt{x},a)$ in the single point $u$.  In case~$(ii)$,  $\wt{\sigma}$ also meets $(\wt{y},a)$ in a singleton $u'$.  Set  $v,v'=\wt{g}(u),\wt{g}(u')$.

\begin{lemma}
Only finitely many  components of $\wt{\XX}_{+}$   meet the arc $[u,v]$ \upn{(}and $[u',v']$\upn{)}. 
\end{lemma}

\begin{proof}
Otherwise the intersections of these curves with $[u,v]$ or $[u',v']$ cluster there, implying that  $(\tilde{x}, a)\cap|\wt{\Lambda}_{-}|\ne\0$. This contradicts $(\wt{x},a)\ss \wt{\UU}_{e}$.
\end{proof}

\begin{cor}\label{bddk}
There exist an integer $k>0$ so that exactly $k$ components of $\wt{\XX}_{+}$ meet $\tilde{h}^{np}\bigl([u,v]\bigr)$  \upn{(}and, if pertinent, $\tilde{h}^{np}\bigl([u',v']\bigr)$\upn{)}, $n\ge0$.
\end{cor}

\begin{cor}\label{esc}
If ${\lambda}_{-}$ is a  border leaf  of $\UU_{+}$, then one of the following holds:
\begin{enumerate}
\item There is a unique $h$-periodic point $x\in\lambda_{-}$ and the ray $[x,\epsilon)\ss|\Lambda_{+}|$ issuing from $x$ into $\UU_{+}$ represents an escaping end $\epsilon$ of a leaf $\lambda\in\Lambda_{+}$;
\item There is a unique maximal, compact, $h$-periodic interval  $I\ss\lambda_{-}$, with $x$ either endpoint of $I$, and the ray $[x,\epsilon)\ss|\Lambda_{+}|$ issuing from $x$ into $\UU_{+}$ represents an escaping end $\epsilon$ of a leaf $\lambda\in\Lambda_{+}$.
\end{enumerate}
In case~\emph{(2)}, the rays issuing from the endpoints of $I$ determine an escaping cusp.  The corresponding assertions hold for a border leaf $\lambda_{+}$ of $\UU_{-}$.
\end{cor}

\begin{proof}
Let $\UU_{e}=\bigcup_{n=-\infty}^{\infty}U_{e}^{n}$ be the component of $\UU_{+}$ bordered by $\lambda_{-}$ where $U_{e}$ is any distinguished neighborhood of $e$ (Definition~\ref{pmesc}).  As usual set $g=h^{p}$ where $p = kp_{e}$. By Corollary~\ref{noncpt},  the ray $[x,\epsilon)$ meets infinitely many positive junctures, necessarily in nonescaping components.  
Choose a point $s\in[x,\epsilon)\cap|\XX_{+}|$ and let $t=g(s)$. Only finitely many components of positive junctures meet the interval $[s,t]$. Let $N$ be the least integer such that  a component of the juncture $J_{e}^{N} = \fr U_{e}^{N}$ meets the interval $[s,t)$. Since $[t,\epsilon) = \bigcup_{n=1}^{\infty}g^{n}([s,t))$, the interval $[t,\epsilon)$ does not meet the juncture $J_{e}^{N}$ and thus lies in $U_{e}^{N}$. Since the interval $[t,\epsilon)$ does not meet the juncture $J_{e}^{N}$, it follows that the interval $g^{i}([t,\epsilon))$ does not meet the juncture $J_{e}^{N+ik} = \fr U_{e}^{N+ik}$ and thus lies in $U_{e}^{N+ik}$.  Thus, $[x,\epsilon)$ ultimately enters and remains in any neighborhood of $e$ in $L$.
\end{proof}

This corollary has the following converse.

\begin{lemma}\label{doesnotreturn}
 If the end  $\epsilon$ of $\ell\in\Lambda_{+}$ is escaping, it has a neighborhood $[x,\epsilon)$  satisfying either $(1)$ or $(2)$ in \emph{Corollary~\ref{esc}}.  The analogous assertion holds for escaping cusps and the corresponding assertions hold for negative escaping ends and cusps.
\end{lemma}

\begin{proof}  
Let $\lambda_{-}\in\delta\UU_{e}$ be the leaf such that there is a point $z\in\lambda_{-}\cap|\Lambda_{+}|$ with the open ray $(z,\epsilon)\ss\UU_{e}$. The leaf $\lambda_{-}$ is semi-isolated and thus has a periodic point. If $z$ is the unique $h^{p}$-fixed point $x\in\lambda_{-}$ or an endpoint of the unique maximal $h^{p}$-invariant interval $[x,y]\ss\lambda_{-}$, we are done.  If not, assume that, in the lifted picture in Figure~\ref{PR}, Case~(i), $\wt {z}\in(\wt{x},d)$ and deduce a contradiction. (A similar contradiction occurs if $\wt{z}\in(c,\wt x)$ in Figure~\ref{PR}, Case~(i) or if $\wt{z}\in(c,\wt x)$ or $\wt{z}\in(\wt{y},d)$ in Figure~\ref{PR}, Case~(ii).) Note that by Corollary~\ref{cor715}, in the case of an escaping cusp, $(\wt x,\wt y)$ is contained in a principal region. Then the points $\wt{z}_{n}=\wt{h}^{-np}(z)$ converge to $d$ as $n\to\infty$ and the lift of a positive escaping ray issues from each $\wt{z}_{n}$.  Projecting down to $L$ by the covering map, we obtain escaping rays $[z_{n},\epsilon_{n})$ issuing into $\UU_{e}$ from the points $z_{n}\in\lambda_{-}$.  These points do not converge in the intrinsic real line topology of $\lambda_{-}$.  But, as points of intersection $z_{n}\in|\Lambda_{+}|\cap|\Lambda_{-}|$, these are points of the compact invariant set and cluster in $L$ at a point $z_{*}$ of that set.  Fix a neighborhood $V$ of $z_{*}$ in $L$ which is a product neighborhood for both laminations.  A subsequence $z_{n_{k}}$ consists of points that lie on distinct components (plaques) $Q_{k}$ of $V\cap\lambda_{-}$ and converge to $z_{*}$.  Clearly, for all but at most one $Q_{k}$, the ray $[z_{n_{k}},\epsilon_{n_{k}})$ must cross at least one other $Q_{k'}$, contrary to hypothesis.
\end{proof}

We have completely characterized the escaping ends and escaping cusps.

\begin{theorem}\label{escends}
 The escaping ends of leaves $\lambda_{\pm}\in\Lambda_{\pm}$ are exactly those represented by rays $(z,\epsilon)$ lying in $\UU_{\pm}$ where $z$ is either the unique periodic point on a leaf of $\delta\UU_{\pm}$ or an endpoint of the unique maximal periodic compact interval on such a leaf.  The escaping cusps are similarly characterized  where the unique periodic point  is replaced by the unique maximal compact periodic interval.
\end{theorem}

\begin{rem}
There are only finitely many escaping ends and escaping cusps. Some ends and cusps may escape and some may not. 

\end{rem} 

\begin{cor}
If a leaf $\lambda$ of $\Lambda_{\pm}$ has an escaping end, it is a periodic leaf.
\end{cor}

The converse, of course, is false as there are generally infinitely many periodic leaves.

\subsection{The structure of principal regions and their crown sets}\label{crown}
We consider $\PP_{+}$ and its components $P_{+}$, but all arguments and results have parallels for $\PP_{-}$ and $P_{-}$.  These components are the principal regions and, by Theorem~\ref{finmany}, there are only finitely many of them.

Fix a choice of $P\ss\PP_{+}$.   The components of $P\cap K$   will be  rectangles   or   regions with frontier finitely many simple closed curves 
$$
\epsilon_{1}\cup\beta_{1}\cup\epsilon_{2}\cup\cdots\cup\epsilon_{r}\cup\beta_{r},
$$
where the $\beta_{i}\ss\ell_{i}\in\Lambda_{+}$ are 
extreme arcs of $|\Lambda_{+}|\cap K$, alternating with proper subarcs $\epsilon_{i}$ of positive junctures in $\fr K$.  There is a least integer $p>0$ such that $h^{-p}(\beta_{i})\ss\beta_{i}$.    Then $h^{-p}(\epsilon_{i})$ is a segment of positive juncture with endpoints in $\beta_{i-1}$ and $\beta_{i}$, respectively.  Infinite iteration gives a sequence of segments of positive junctures converging to a segment $\alpha_{i}$ of a leaf $\ell'_{i}$ of $\Lambda_{-}$ having endpoints $x_{i}$ and $y_{i}$ on $\beta_{i-1}$ and $\beta_{i}$, respectively. Shorten the arcs $\beta_{i}$ to have endpoints $y_{i}$ and $x_{i+1}$,  defining a simple closed curve 
$$\gamma=\alpha_{1}\cup\beta_{1}\cup\alpha_{2}\cup\cdots\cup\alpha_{r}\cup\beta_{r}.$$  Note that the indices are taken mod~$r$. The lifts of each of the curves $\ell_{i}$ (resp. $\ell'_{i}$) can play the role of $\lambda_{-}$ in case~$(ii)$ of Proposition~\ref{3cases} with lifts of the arcs $\beta_{i}$ (resp. $\alpha_{i}$) playing the role of the segments $[\wt{x},\wt{y}]$ in that proposition. Thus, the $\beta_{i}$ lie in the invariant set and cut off $r\ge 1$ arms $A_{i}$ of the principal region $P$. Similarly, the $\alpha_{i}$ lie in the invariant set and cut off $r$ arms $A'_{i}$ of the principal region $P'\ss\PP_{-}$.

\begin{defn}[dual principal regions]\label{defndualPR}
The  principal regions $P$ and $P'$ are called dual principal regions.
\end{defn}

\begin{defn}[nucleus]\label{defnnucleus}
The closure of the intersection $P\cap P'$ is the \emph{nucleus} of the principal region $P$ and of its dual $P'$
\end{defn}

Thus a principal region is the union of the interior of its nucleus and arms.
This nucleus may be bounded by several polygonal curves $\gamma_{i}$ as above and to each is attached a set of arms for $P$ and a set of arms for $P'$.

\begin{lemma}
The nucleus $N$ of a principal region is compact.
\end{lemma}

\begin{proof}
Indeed, $\bd N\sm\bd L$ is the union of all the curves $\gamma$ associated to the principal region $P$ and the dual $P'$ and $N$ lies in both $P$ and $P'$.  If $N$ were noncompact then, since its boundary is compact, it would be a neighborhood of at least one end of $L$, hence would contain positive and/or negative junctures, contradicting the fact that $N= P\cap P'$.
\end{proof}

\begin{rem}
There are four cases to consider:
\begin{enumerate}
\item $N$ has negative Euler characteristic.  In this case, $\gamma$ can be tightened in its homotopy class to a unique simple closed geodesic $\rho_{\gamma}\ss \intr N$ and the correspondence $\gamma\lra\rho_{\gamma}$ is one-one.
\item  $N$ is an annulus.  If there are two piecewise geodesic  boundary curves $\gamma_{1}$ and $\gamma_{2}$, both will be homotopic to the same geodesic $\rho_{\gamma_{1}}=\rho_{\gamma_{2}}\ss\intr N$.   It may happen that one boundary curve $\delta$ of the annulus is a component of $\bd L$.  In this case, denote the other component of $\bd N$ by $\gamma$ and note that $\rho_{\gamma}=\delta$.
\item $N$  is a M\"obius strip with one boundary curve $\gamma$ and center circle a geodesic $\sigma\ss\intr N$. The curve $\gamma$ is homotopic to an immersed  geodesic $\rho_{\gamma}$ that is a two-to-one cover of $\sigma$.
\item\label{disknucleus} $N$ is a disk.  There is one boundary curve $\gamma$, but it is not homotopic to a closed geodesic.  In this case, we set $\rho_{\gamma}=\0$.
\end{enumerate}
In any event, the arcs making up $\gamma$ are isolated on the side facing the nucleus of the principal region $P$.
\end{rem}

\begin{defn}[dual crown sets]\label{defndualSC}
The \emph{crown set} $C_{\gamma}$ is the closure of the  component  of $P\sm\rho_{\gamma}$ that contains the curve $\gamma$. The \emph{crown set} $C'_{\gamma}$ is the closure of the  component of $P'\sm\rho_{\gamma}$ that contains the curve $\gamma$. The crown sets $C_{\gamma}$ and $C'_{\gamma}$ are called \emph{dual crown sets}.
\end{defn}

Thus, in all cases except~(\ref{disknucleus}) in the above remark, the crown sets are annuli with finitely many cusps.  In case~(\ref{disknucleus}), the crown sets are disks with finitely many cusps.

\begin{defn}[rim]\label{prreduce}
The closed curve $\rho_{\gamma}\ss P$ that cuts off a crown set  is    called the \emph{rim} of the crown set. The closed curves $\gamma$ and $\rho_{\gamma}$  will be said to be associated. 
\end{defn}

Since $h$ permutes the borders of the principal  regions, there is a corresponding permutation of the crown sets themselves. Thus,  we get cycles $C=C_{0}, C_{1}=h(C),\dots,C_{n}=h^{n}(C)=C_{0}$, and a corresponding cycle $\rho=\rho_{0},\rho_{1},\dots\rho_{n}=\rho_{0}$ of rims.  
If $n$ is the minimal period, then $h^{n} $ induces a permutation in the arms of each crown set.

\subsection{The set $\mathfrak S$ of reducing curves}\label{redcurv}

As in Nielsen-Thurston theory~\cite{bca}, we split $L$ into simpler  pieces  using \emph{reducing curves}~\cite[page 5]{fe:endp}. The reducing curves will be geodesics, either   homeomorphic to $\SI$ or $\R$ and will lie in $\intr L$.   They will be constructed   as we develop our theory. The reducing curves will be \emph{nonperipheral} in the sense that none cobounds  an annulus or infinite strip with a component of $\bd L$ and they will be disjoint from $\Lambda_{\pm}$.

\begin{defn}[$\S$]\label{famredc}
We will let $\S$ denote the set of reducing curves.
\end{defn} 

The first set of reducing curves we construct are the geodesic  rims $\rho_{\gamma}$ of crown sets for the case of  nuclei which are neither discs nor M\"obius strips. (The rim $\rho_{\gamma}$ for the M\"obius strip case could be counted as a reducing curve, but we choose not to because doing so would make some statements awkward.)  If the nucleus is peripheral,  $\rho_{\gamma}$  is a component of $\bd L$ and is not taken as a reducing curve. If the nucleus $N$ is an annulus with boundary $\gamma_{1}\cup\gamma_{2}$,  the two rims $\rho_{\gamma_{1}}$ and $\rho_{\gamma_{2}}$ are identical and this will be a reducing curve. Otherwise,   the  rims $\rho_{\gamma}$  correspond one-one with $\gamma$, are disjoint from each other and disjoint from the geodesic laminations $\Lambda_{\pm}$ and $\bd L$. 

 Every  $\gamma$ associated to a reducing curve in $\mathfrak S$ will be  the union of segments of semi-isolated leaves (Definition~\ref{defnsemiis}) of $\Lambda_{\pm}$.   As usual, we will distinguish $\mathfrak S$ from the \emph{support} (i.e., union)  $|\mathfrak S|$.

In Section~\ref{reducing}, we will construct an endperiodic automorphism $g$ which is isotopic to $f$ and permutes both the sets $\Lambda_{+}$ and $\Lambda_{-}$ and  the set $\mathfrak S$ of reducing curves.

\subsection{Examples of principal regions}

Principal regions occur for an endperiodic automorphism  when there is some topology that remains in the core or when there are more than two repelling or attracting ends. There can not be simply connected principal regions with one or two arms.

\begin{example}\label{onearmprincep}
In this example, the principal regions have only one arm.  Let $L$ be a  surface with two nonplanar ends and one disk removed as in Figure~\ref{onearm}. Then $L$ has one circle boundary component $C$. If $g$ is the endperiodic automorphism that moves each handle to the right  one unit near both ends but leaves $C$ invariant,  then the laminations $\Lambda_{\pm}$ each contain one leaf as in Figure~\ref{onearm}.  Both the positive and negative principal regions have one arm. If one composes $g$ with Dehn twists in the three dotted  curves in Figure~\ref{onearm} to get an endperiodic automorphism $f$, the new laminations $\Lambda_{\pm}$ both have uncountably many leaves. The one arm of the positive principal region is bordered by an isolated leaf $\lambda_{+}$ and no longer is escaping but returns infinitely often to the core. The negative principal region is bordered by a semi-isolated but not isolated leaf $\lambda_{-}$ and has one escaping arm. 
\end{example}

\begin{figure}[h]
\begin{center}
\begin{picture}(300,70)(40,-150)
\rotatebox{270}{\includegraphics[width=300pt]{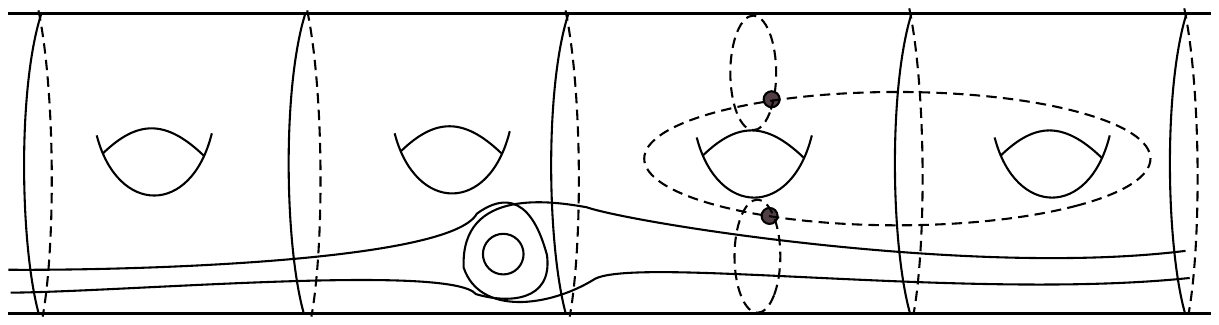}}
\put (-341,-137){$\lambda_{-}$}
\put (-39,-141){$\lambda_{+}$}
\end{picture}
\caption{An example whose principal regions have one arm}\label{onearm}
\end{center}
\end{figure}

\begin{example}
The double of the surface $L$ of Example~\ref{onearmprincep} along the circle boundary component $C$ of $L$ has dual positive and negative principal regions whose nucleus is an annulus. Each principal region has two crown sets.
\end{example}

\begin{example}
Let $S$ be a pair of pants with boundary $C_{0}$, $C_{1}$,  $C_{2}$.  Attach three copies of $L$ along $C$ to each of $C_{0}$, $C_{1}$,  $C_{2}$. This example has dual positive and negative principal region whose nucleus is a pair of pants. The rims of the three pairs of dual crown sets are $C_{0}$, $C_{1}$,  $C_{2}$.
\end{example}

\begin{example}
We give examples in which the principal regions have $r\ge3$ arms.
Let $L$ be a surface similar to the surface in  Figure~\ref{TwoEnds}  but with $r\ge3$ negative ends and one positive end and let $f$ be an  endperiodic automorphism that permutes the negative ends and moves the handles along in a way similar to the example illustrated in Figure~\ref{TwoEnds}. The laminations $\Lambda_{\pm}$ each contain $r$ leaves and there are positive and negative principal regions, each with $r$ arms.  By composing $f$ with suitable Dehn twists the laminations can be made more complicated. The principal regions are simply connected but topology can be added to them as in  Example~\ref{exwithtop}.
\end{example}

\begin{example}\label{exwithtop}
Let $L$ be the surface in Figure~\ref{TwoEnds} but with a disk centered at the saddle point  removed. Thus, $L$ has one circle boundary component $C$. Let $f$ be the map indicated in Figure~\ref{TwoEnds} modified to send  $C$ to itself.  The laminations $\Lambda_{\pm}$ each contain two leaves. There are positive and negative principal regions each with two arms. By composing $f$ with suitable Dehn twists one can obtain similar examples with more complicated laminations.
\end{example}

\subsection{The induced laminations on the compact surface $F$}\label{inducedlam}

Let $e$ be an end of $L$ and denote its full $h$-cycle by $c=\{e=e_{1},e_{2},\dots,e_{p_{e}}\}$.  Recall that, 
$$\UU_{c}=\bigcup_{n=-\infty}^{\infty}h^{n}(U_{e})=\UU_{e_{1}}\cup\UU_{e_{2}}\cup\cdots\cup\UU_{e_{p_{e}}},$$  
where $U_{e}$ is any $h$-neighborhood (Definition~\ref{fnbhd}) of $e$. Remark that $\UU_{c}$ is an open, $h$-invariant set with no periodic points.  The connected components $\UU_{e_{i}}$ of $\UU_{c}$ are permuted cyclically by $h$.      As in Section~\ref{jnctrs}, one has an infinite cyclic covering
$$q:\UU_{c}\to F.$$
The group of deck transformations for $q$ is generated by $h|\UU_{c}$. 

For definiteness, we consider the case that $c$ is an $h$-cycle of negative ends of $L$.
Since   $\UU_{c}$ and $\Lambda_{-}$ are    $h$-invariant, it follows that the induced lamination   $\Lambda_{-}|\UU_{c}$ is invariant under the group of deck transformations of   $q:\UU_{c}\to F$. Similarly, the set $\JJ_{-}$ of negative juncture components is $h$-invariant and transverse to $\Lambda_{-}$. Thus,

\begin{lemma}[$\Lambda_{F},J_{\kappa}$]\label{descends}
The lamination   $\Lambda_{-}|\UU_{c}$ descends to a well defined closed lamination $\Lambda_{F}$ of  $F$ and the set of negative juncture components in $\UU_{c}$ descends to a compact, transversely oriented, properly embedded $1$-manifold $J_{\kappa}$ that is transverse to $\Lambda_{F}$. 
\end{lemma} 

\begin{rem}
The $1$-manifold $J_{\kappa}$ is the $1$-manifold $J_{\kappa}$ of Definition~\ref{kapJunct}.

\end{rem}

We analyze the structure of $\Lambda_{F}$ using the   properties of $\Lambda_{-}|\UU_{c}$.  

 The border $\delta\UU_{-}$ consists of  semi-isolated leaves in $\Lambda_{+}$ and is invariant under $h$.  Thus, $\delta\UU_{-}$ and $\delta\UU_{c}$  each consist of $h$-cycles of semi-isolated leaves of $\Lambda_{+}$.  Recall that there are only finitely many semi-isolated leaves of $\Lambda_{\pm}$ (Theorem~\ref{finmany}) and that each contains either a unique $h$-periodic point or a unique maximal, compact, nondegenerate $h$-periodic interval.  For the semi-isolated leaves of $\Lambda_{+}$, the isolated periodic point is repelling under applications of $h$ and the endpoints of the periodic interval are each repelling on the side not meeting the interval  (Corollary~\ref{attracting}). For $\Lambda_{-}$, these points are attracting.

Evidently,  if  $[a,\infty)$ is the neighborhood of an escaping end issuing from  a periodic point $a$ on the isolated side of a leaf $\lambda$ of $\Lambda_{+}$ in $\delta\UU_{c}$, the leaf $(a,\infty)$ of $\Lambda_{-}|\UU_{c}$ descends to a circle leaf $C_{a}\ss F$ of $\Lambda_{F}$.  Either $a\in\lambda$ is an isolated, repelling,  $h$-periodic point, or it is an endpoint of a compact, nondegenerate $h$-periodic arc $[a,a']\ss\lambda$.  In this case there are two escaping ends $[a,\infty)$ and $[a',\infty)$ and, by the structure theory of principal regions, these cobound an arm $A$ of a principal region.  Since $A$ is simply connected (Lemma~\ref{simpconnarms}), $(a,\infty)$ and $(a',\infty)$      descend to a pair of circle leaves $C_{a}$ and $C_{a'}$ of $\Lambda_{F}$,  cobounding an annulus in $F$ which meets no other leaves of $\Lambda_{F}$.  The points $a,a'$ are  both repelling  on the sides opposite to $[a,a']$.  In all cases, the natural orientation of the escaping ends toward $\infty$ induces an orientation on these circle leaves of $\Lambda_{F}$.

\begin{lemma}\label{1or2}
Every leaf of $\Lambda_{-}|\UU_{c}$ issues from either one or two points of $|\delta\UU_{-}|$.  \end{lemma}

\begin{proof}
We need to show that no leaf  of $\Lambda_{-}$ lies entirely in $\UU_{c}$.  Otherwise, by Lemma~\ref{LambdadoesnotmeeUU'}, that leaf would meet no leaf of $\Lambda_{+}$, contrary to Proposition~\ref{atleastonce}.  
\end{proof}

\begin{lemma}\label{onlynoncpt}
The only leaves of $\Lambda_{-}|\UU_{c}$ which do not complete to compact arcs with endpoints in $\delta\UU_{c}$ are the escaping ends with completion $[a,\infty)$
\end{lemma}

\begin{proof}
Suppose there is a leaf $\ell=(b,\infty)$ of $\Lambda_{-}|\UU_{c}$ with completion  $\wh\ell=[b,\infty)$, $b\in|\delta\UU_{c}|$, which is not an escaping end.  Then $\ell$ either crosses some fixed juncture $J$ infinitely often or eventually remains in some compact part of $\UU_{c}$. In either case, since $\Lambda_{-}$ is closed, there is a leaf $\ell'\in\Lambda_-$ in the asymptote of the end of $\ell=(b,\infty)$. Since $\ell\ss\UU_{c}$, Lemma~\ref{LambdadoesnotmeeUU'} implies that $\ell\cap|\Lambda_+| = \0$.  Therefore, $\ell'\cap|\Lambda_+| = \0$  contradicting   Proposition~\ref{atleastonce}. 
\end{proof}

\begin{defn}[parallel packets]\label{parpackets}
    The compact completions $\wh\ell_{1}$ and $\wh\ell_{2}$ of two bounded leaves of $\Lambda_{-}|\UU_{c}$  are parallel if they are a pair of opposite sides of a rectangle in $\ddot\UU_{c}$, the other two sides being compact arcs in $|\delta\UU_{c}|$.  This is an equivalence relation on the set of compact completions of leaves of $\Lambda_{-}|\UU_{c}$ and the equivalence classes will be called \emph{parallel packets}. Each parallel packet contains two \emph{extreme leaves}.
\end{defn}

Evidently, the parallel packets are permuted by $h$.

\begin{lemma}\label{finhorb}
There are only finitely many $h$-orbits of parallel packets.
\end{lemma}

\begin{proof}
Since $\delta\UU_{c}$ has only finitely many elements, there is a least integer $p\ge0$ such that $h^{-p}$ carries each onto itself, preserving orientation.    It is enough to show that there are only finitely many $h^{-p}$-orbits of parallel packets.  If not, there is a leaf $\lambda$ of $\delta\UU_{c}$ and a compact subarc $[x,h^{-p}(x)]\ss\lambda$ containing no periodic point, and infinitely many points $x_{n}\in[x,h^{-p}(x)]$ out of which issue completions pertaining to distinct packets.  Let $y\in[x,h^{-p}(x)]$ be a cluster point of $\{x_{n}\}$. Since $\Lambda_{-}$ is a closed lamination, there exists a leaf of $\Lambda_{-}$ issuing from $y$. Since $y\in[x,h^{-p}(x)]$, $y$ is not a periodic point. Thus, by Theorem~\ref{escends} and Lemma~\ref{onlynoncpt}, the completion of the leaf of $\Lambda_{-}|\UU_{c}$ issuing from $y$ is a compact arc $\ddot\ell$.  Let $\ddot\ell_{n}$ be the completion of the leaf of $\Lambda_{-}|\UU_{c}$ issuing from $x_{n}$.  Then the $\ddot\ell_{n}$'s cluster locally uniformly on the compact arc $\ddot\ell$, proving that infinitely many of them are parallel to $\ddot\ell$.  This is contrary to hypothesis, completing the proof.
\end{proof}

\begin{rem}[Properties of the lamination $\Lambda_F$]

Putting these lemmas together, we see that the escaping rays descend under the covering projection $q:\UU_{c}\to F$  to finitely many circle leaves of $\Lambda_F$ and the parallel packets descend under the covering projection  to finitely many packets, homeomorphic to $X\x\R$, of parallel noncompact leaves of   $\Lambda_{F}$.  Here $X$ is compact and totally disconnected.  One end of such a  packet spirals in on a circle leaf, as does the other end.  Indeed,  if an end of a leaf of $\Lambda_{-}|\UU_{c}$ issues from a point $x_0\in\delta\UU_c$, then the image of the that end under $h^{np}$ issues from a point $x_n\in\delta\UU_c$,   the sequence $\{x_n\}_{n\ge 0}$ converges to a periodic point $x\in\delta\UU_c$, and the leaf of $\Lambda_{-}|\UU_{c}$ issuing from $x$ is an escaping ray. The leaves (and packets) upstairs accumulate locally uniformly on two (not necessarily distinct) $h$-orbits of escaping rays which descend to (one or two) circle leaves in $F$. 

\end{rem}

\begin{rem}[continued]
The circle leaves of $\Lambda_F$ have a preferred orientation obtained from orienting the escaping ends towards the end. The noncompact leaves of $\Lambda_F$ can not be oriented.

\end{rem}

\begin{rem}[continued]\label{noreebs}

By Lemma~\ref{simpconnarms} and Theorem~\ref{escends}, an arm of a principal region  is an infinite strip bounded by escaping rays $(a,\infty)$ and $(a',\infty)$, where $[a,a']$ is an $h$-periodic interval and thus descends under the covering projection $q:\UU_{c}\to F$ to an annulus bounded by two circle leaves of $\Lambda_F$. This is the only way two parallel circles leaves can occur in $\Lambda_F$.   As a subset of a principal region, an arm contains no leaves of $\Lambda|\UU_{c}$. Thus, it is not possible to have three mutually parallel circle leaves in $\Lambda_F$ nor is it possible  that $\Lambda_{F}$ contains a ``Reeb'' annuli.

\end{rem}

The above development can be used to define  a traintrack $\T$ carrying the lamination $\Lambda_{F}$. We have,

\begin{prop}[$\T$]\label{T'}
The traintrack $\T$ in $F$ carrying the lamination $\Lambda_{F}$ consists of finitely many oriented circles, one for each single oriented circle in $\Lambda_{F}$ or pair of oriented circles in $\Lambda_{F}$ bounding an annulus, and finitely many compact arcs, one for each packets of parallel noncompact leaves of $\Lambda_{F}$. Each end of each arc meets a circle in a point called a  switch so that the arc makes an acute angle with the outgoing arc of the circle.  
\end{prop}

\begin{rem}
Arcs can not have a preferred orientation.  Either choice of orientation will be coherent with that of the circle at the one switch and will be opposed at the other. 

\end{rem}

  We view $\T$ as a branched $1$-manifold. As a graph $\T$ is $3$-valent.

 \begin{figure}[h]
\begin{center}
\begin{picture}(210,120)(-8,0)

\rotatebox{90}{\includegraphics[width=120pt]{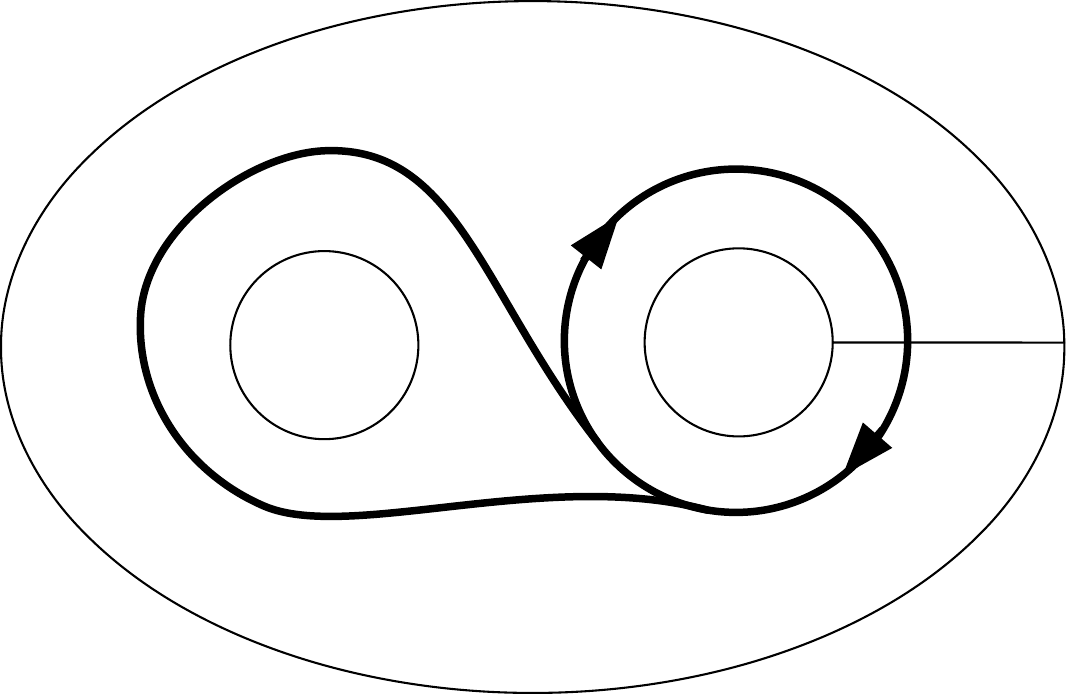}}

\put(-95,15){\small$F$}
\put(-19,65){\small$J_{\kappa}$}

\end{picture}
\caption{Traintrack $\T$ for negative end of Example~\ref{simpex}}\label{ttsimple}
\end{center}
\end{figure}

\begin{example}
Figure~\ref{ttsimple} gives the traintrack $\T$, juncture $J_{\kappa}$,  and surface  $F$ for the negative end of Example~\ref{simpex}.\end{example}

\begin{rem}
In~\cite{fe:endp}, the third author induces a hyperbolic structure and geodesic laminations on $F$ with only finitely many leaves.  The laminations in~\cite{fe:endp} are the ``geodesic tightening'' of ours and are carried by  the same traintrack.   Our packets $X\x\R$ of leaves winding in on the circle leaves generally contain infinitely many leaves, but each packet lifts to the universal cover $\wt F$ to curves having the same ideal endpoints.  Hence, tightening our circle leaves to geodesics and each of our finitely many packets to a single geodesic, each end of which winds in on one of the circles, gives  the geodesic lamination of $F$ produced in~\cite{fe:endp}.
\end{rem}


\section{The Escaping Set $\UU$}\label{escUU}

The main purpose of this section is to construct reducing curves associated to certain border components of $\UU$. This involves a detailed, technical study of $\delta\UU$. The reader should review in Section~\ref{section51} the definition of the metric completion $\ddot U$ of a component $U$ of $\UU$, the metric completion of the set $\UU$ (which is not generally connected), the map $\ddot\iota:\ddot\UU\to L$, and the definiton of the set $\delta\UU$ of border components of $\UU$. 

Recall (Lemma~\ref{escsetprop}) that the escaping set satisfies,
   $$\UU=\UU_{+}\cap\UU_{-} = L\sm (|\Lambda_-|\cup\PP_-\cup|\Lambda_+|\cup\PP_+)$$
and that $\UU$ is thus disjoint from the principal regions.  Further, the set $|\XX_{\pm}|$ will meet each component $U$ of $\UU$ but will cluster in $U$ only on points of $|\delta U|$ (Definition~\ref{bordcomp}). An element of  $\XX_{\pm}$ can not meet   $\PP_{\mp}$ but will meet $\PP_{\pm}$ in the arms.

\begin{lemma}
If $U$ is a component of $\UU$, then $\delta U\ne\0$.
\end{lemma}

\begin{proof}
By  our  assumption that $f$ is not isotopic to a translation (Hypothesis~\ref{hyp5}) and Proposition~\ref{total}, $U$ is a proper subset of $L$.  Choose $x\notin U$ and $y\in U$ and denote by $[x,y]$  a closed geodesic arc  with endpoints $x$ and $y$.  There is a point $z\in[x,y)$ such that $(z,y]\ss U$ and $z\notin U$. Then clearly $z\in|\delta U|$ so $\delta U\ne\0$.
\end{proof}

\begin{rem}
The  components $U$ of $\UU$ that are rectangles play a special role.  If a component $U$ of $\UU$ is a rectangle, the sequence of iterates  under $h$ of a point of $U$ escapes. The sequence of iterates under $h$ of a vertex of $U$ remains in the core and cannot escape while only the sequence of positive (respectively negative) iterates of an interior point of the two edges in $|\Lambda_{+}|$ (respectively $|\Lambda_{-}|$) escape.

\end{rem}

\subsection{The border of the escaping set}\label{structure}

The next lemma is an immediate consequence of Proposition~\ref{atleastonce}.

\begin{lemma}\label{noentireleaf}
An entire leaf $\lambda\in\Lambda_{\pm}$ cannot be an element of $\delta\UU$.
\end{lemma}

 The following lemma formally defines what we mean by vertices and edges and is clear by Lemma~\ref{bdUU}.

\begin{lemma}[vertex/edge]\label{veredg}

If $x\in\gamma\in\delta \UU$, then $x$ lies in some semi-isolated leaf $\lambda$ of one of the laminations and either,

\begin{enumerate}

\item $x\in|\Lambda_{+}|\cap|\Lambda_{-}|$ and $x$ also lies in a semi-isolated leaf $\lambda'$ of the other lamination and  there are maximal, nondegenerate  subarcs $[x,y)$ and $[x,z)$ of $\lambda$ and $\lambda'$, respectively, that meet no other points of $|\Lambda_{+}|\cap|\Lambda_{-}|$.  If $y$ is an end of $\lambda$, then $[x,y)\ss\gamma$ and otherwise $[x,y]\ss\gamma$ and $y\in|\Lambda_{+}|\cap|\Lambda_{-}|$.  In either case, the resulting arc is called an \emph{edge} of $\gamma$ and $x$ is called a \emph{vertex} of $\gamma$.  Similar considerations hold for $[x,z)$.  

\item $x\not\in|\Lambda_{+}|\cap|\Lambda_{-}|$ and there is a maximal open subarc  $(y,z)\ss\lambda$ containing $x$ and not  meeting $|\Lambda_{+}|\cap|\Lambda_{-}|$.  By \emph{Lemma~\ref{noentireleaf}}, one or both of $y,z$ is finite,  lies in $|\Lambda_{+}|\cap|\Lambda_{-}|$, and is again called a \emph{vertex} of $\gamma$, the resulting closed subarc or infinite ray in  $\lambda$ being an \emph{edge}  of $\gamma$.  

\end{enumerate}

\end{lemma}

\begin{lemma}\label{RorS}
Each  $\gamma\in\delta\UU$ is either an immersed copy of the real line or an immersed circle with an even number of edges.  
\end{lemma}

\begin{proof}
The lemma follows since the components of $\ddot U\sm U$ are homeomorphic to a line or circle for every component $U$ of $\UU$ and every $\gamma\in\delta\UU$ is the image under $\ddot\iota$ of a component of $\ddot U_{\gamma}\sm U_{\gamma}$ for some component $U_{\gamma}$ of $\UU$.
\end{proof}

\begin{rem}
We will show in Corollary~\ref{gamembed} that every border component of $\UU$ is embedded.

\end{rem}

\begin{lemma}\label{firstkind}

If  $\gamma\in\delta\UU$ is an immersed line then, either its sequence of vertices $\dots,x_{i},x_{i+1},\dots$ is bi-infinite or $\gamma$ has only one vertex $x_{0}$ connecting two unbounded edges $\alpha_{1}$ and $\beta_{1}$.  In the latter case,  $x_{0}$ is the unique $h$-periodic point on the semi-isolated leaves of $\Lambda_{\pm}$ passing through it.
\end{lemma}

\begin{proof}
Suppose, for definiteness, that the sequence is not infinite to the left and denote its initial vertex by $x_{0}$.  Thus, its initial edge $\alpha_{1}$ must be a ray in a semi-isolated leaf $\lambda$ of one of the laminations.  By Theorem~\ref{finmany}, we know that, for some integer $k\ge1$, $h^{k}(\lambda)=\lambda$.  Without loss of generality, we can suppose that $h^{k}$ fixes the ends of $\lambda$.   Thus, orienting $ \lambda$ so that its initial end is the end of $\alpha_{1}$, we see that either $h^{k}(x_{0})<x_{0}$, $h^{-k}(x_{0})<x_{0}$, or $h^{k}(x_{0})=x_{0}$.  The first two cases imply that $\intr\alpha_{1}$ meets the other lamination, hence contains a vertex.  Hence the third case holds.  Now suppose that the edge $\beta_{1}=[x_{0},x_{1}]$ is bounded. It must be fixed by $h^{k}$  and so by (2) of Corollary~\ref{cor715}, $(x_{0},x_{1})$  lies in a principal region $P$ which then meets $\UU$ which is a contradiction.     The final assertion is  clear.
\end{proof}

\begin{defn}[first/second kind]\label{fskind}
A real line border component $\gamma$ of $\UU$ with just one vertex is said to be of the \emph{first kind}.  Otherwise, the border component is of the {second kind}.
\end{defn}

\begin{nota}\label{bordernota}
The border components $\gamma$ of $\UU$ of the first kind are of the form $\gamma = \alpha_{0}\cup\beta_{0}$ with $\alpha_{0}\ss\Lambda_{-}$ and $\beta_{0}\ss\Lambda_{+}$ and with  one vertex $y_{0} = \alpha_{0}\cap\beta_{0}$. Otherwise a border $\gamma$ of $\UU$ is of the form $\gamma= \bigcup_{i = -\infty}^{\infty} \alpha_i\cup\beta_i$ with $\alpha_i\ss|\Lambda_-|$, $\beta_i\ss|\Lambda_+|$, $\beta_{i-1}\cap\alpha_{i}=x_{i}$, and $\alpha_{i}\cap\beta_{i}=y_{i}$, $i\in\Z$. If $\gamma$ is of the second kind, then the $\alpha_{i},\beta_{i}$ form a bi-infinite sequence. If $\gamma$ is compact, then there exists an integer $r>0$ such that $\alpha_{i+r} = \alpha_{i}$ and $\beta_{i+r}= \beta _{i}$, all $i\in\Z$.

\end{nota}

The next lemma follows immediately from the definitions of $\UU$ and $\delta\UU$.

\begin{lemma}\label{endsee}

If $\gamma\in\delta\UU$, then   there exist a component $U_{\gamma}$ of $\UU$, $e_{-}\in\EE_{-}(L)$ and  $e_{+}\in\EE_{+}(L)$ such that $\gamma\in\delta U_{\gamma}$ and $U_{\gamma}\ss\UU_{e_{-}}\cap\UU_{e_{+}}$.  

\end{lemma}

From now on, we focus on  the end $e_{-}\in\EE_{-}(L)$. The discussion for $e_{+}\in\EE_{+}(L)$ is analogous. Let  $c=\{e_{-}=e_{1},e_{2},\dots,e_{p_{e_{-}}}\}$ be the full $h$-cycle containing $e_{-}$. Recall the projection $q:\UU_{c}\to F$, the lamination $\Lambda_{F}$, and the traintrack $\T$ for $\Lambda_{F}$ (Section~\ref{inducedlam}).  Note that the surface $F$ depends on the $h$-cycle of ends containing $e_-$.

\begin{lemma}\label{notint}

If $\gamma\in\delta\UU$ is not of the first kind and is not  the border component of a rectangle and $\alpha\ss|\Lambda_{-}|$ is an edge of $\gamma$, then $\alpha$ is an extreme leaf of a parallel packet of  leaves  of $\Lambda_{-}|\ddot\UU_{e_{-}}$.

\end{lemma} 

\begin{proof}
If  $\alpha$ is an interior leaf of a parallel packet of  leaves, then $\gamma$ is the border component  of a rectangle contrary to  assumption.
\end{proof}

 \begin{nota}
Denote by $T^{*}$ the compact surface obtained by fattening up $T$ in $F$. 

\end{nota}

\begin{figure}[h]
\begin{center}
\begin{picture}(210,70)(-8,0)

\includegraphics[width=200pt]{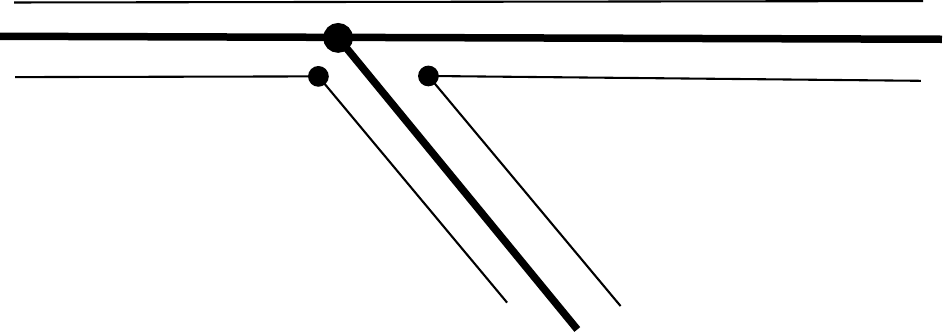}
\put (-62,75){$\bd\T^{*}$}
\put (-162,40){$\bd\T^{*}$}
\put (-82,40){$\bd\T^{*}$}
\put (-72,-10){$\E$}
\put (9,58){${\rm C}$}

\end{picture}
\caption{Inserting vertices into  $\bd\T^{*}$ at a switch}\label{atswitch}
\end{center}
\end{figure}
 
 A switch (Proposition~\ref{T'}) on the train track $T$ comes from the end of a  parallel packet of leaves winding in on a circle leaf. If one considers $T$ as a $3$-valent graph, a switch   is a vertex of degree three. Two of the edges coming out of a switch are subsets of the circle leaf. The third edge represents the parallel packet of leaves and is thus distinguished from the other two edges. In a natural way, the structure of $\T$ as a $3$-valent graph divides certain  components of $\bd\T^{*}$ into vertices and edges by inserting two vertices into  $\bd\T^{*}$ at each switch  as indicated in Figure~\ref{atswitch}. In Figure~\ref{atswitch}, the two horizontal edges at the switch belong to a circle leaf $C$ while the third edge $E$ represents an end of a parallel packet of  leaves winding in on $C$.

\begin{rem}

Some components of $\bd\T^{*}$ may contain no vertices. In Theorem~\ref{traintrack}, we show that these components correspond to border components of $\UU$ of the first kind.

\end{rem}

\begin{rem}

The surface $T^*$ may be non-orientable.

\end{rem}

\begin{theorem}\label{traintrack}

\begin{enumerate}

\item There is a $\iota^*/p_{e_-}$-to-one correspondence between the set of border components $\gamma$ of $\UU$ of the first kind that meet $\UU_{e_{-}}$ and the  set of  components $C^*$ of $\bd\T^{*}$ not containing a vertex where $\iota^*$ is the intersection number of $C^*$ and $J_{\kappa}$\upn{;}

\item There is a one-one correspondence between the set of  infinite families of compact border components $\gamma$ of $\UU$  that meet $\UU_{e_{-}}$ and are not rectangles and the  set of   components $C^*$ of $\bd\T^{*}$  containing a vertex with zero intersection number with the juncture $J_{\kappa}$.

\item There is a $\iota^*/p_{e_-}$-to-one correspondence between the set of border components $\gamma$ of $\UU$ of the second kind that meet $\UU_{e_{-}}$ and the  set of  components $C^*$ of $\bd\T^{*}$ containing a vertex such that $C^*$ has non-zero intersection number $\iota^*$ with the juncture $J_{\kappa}$\upn{;}

\end{enumerate}

\end{theorem}

\begin{proof}
First, suppose $\gamma = \alpha_{0}\cup\beta_{0}$ with $\alpha_{0}\ss\Lambda_{-}$ and $\beta_{0}\ss\Lambda_{+}$ and with  one vertex $y_{0} = \alpha_{0}\cap\beta_{0}$ is a border component  of $\UU$ of the first kind  that meets $\UU_{e_{-}}$, then $C_{\gamma} = q(\intr\alpha_{0})$ is a circle leaf of $\Lambda_{F}$ such that there are no edges of the traintrack on one side of the circle $C_{\gamma}$. In $\T^{*}$, $C_{\gamma}$ gives rise to a   component $C_{\gamma}^{*}$  of $\bd\T^{*}$ containing no vertices. Thus, every border component $\gamma$ of the first kind that meets $\UU_{e_{-}}$ corresponds to a unique component $C^{*}_{\gamma}$  of $\bd\T^{*}$. 

Conversely, if $C^{*}$ is a component of $\bd\T^{*}$ containing no vertices, then there exists a corresponding circle leaf $C$ of $\Lambda_{F}$  such that there are no edges of the traintrack  on one side of the circle $C$. It is then easy to see that $C = q(\intr\alpha_{0})$ where $\gamma = \alpha_{0}\cup\beta_{0}$ is  a border leaf of $\UU$ of the first kind  that meets $\UU_{e_{-}}$. Thus, every component $C^{*}\ss\bd\T^{*}$ containing no vertices  corresponds to to at least one border componet of $\UU$ of the first kind that meets $\UU_{e_{-}}$. 

Let $\iota^*$ be the intersection number of $C^*$ and $J_{\kappa}$. Since $h^{\iota^*}(\gamma) = \gamma$ and the curves $h^{j}(\gamma)$, $0\le j<\iota^*$, are distinct and there are $p_{e_-}$ ends in the cycle of ends containing $e_-$, it follows that there are $\iota^*/p_{e_-}$ border components  of $\UU$  of the first kind that meet $\UU_{e_{-}}$ corresponding to  the  component $C^*$.

Next, suppose $\gamma= \bigcup_{i = -\infty}^{\infty} \alpha_i\cup\beta_i$ (in the notation of page~\pageref{bordernota}) is a border component of a component $U_{\gamma}$ of  $\UU$ such that $U_{\gamma}\ss\UU_{e_{-}}$ is not a rectangle and $\gamma$ is not a border component  of $\UU$ of the first kind. The leaf $\alpha_{i}$, $i\in\Z$,  of $\Lambda_{-}|\ddot\UU$ is extreme in a parallel packet of leaves (Lemma~\ref{notint}) and, since $q(\alpha_i)\ss q(U_{\gamma})$, determines an edge of a component $C^{*}_{\gamma}$ of $\bd\T^{*}$ which we denote by $\alpha_{i}^{*}$.  The edge $\alpha_{i}^{*}$ inherits an orientation from $\alpha_{i}$ and we denote the initial vertex of $\alpha_{i}^{*}$ by $x_{i}^{*}$ and the terminal vertex by $y_{i}^{*}$. 

The  terminal vertex $y_{i-1}^{*}$ of $\alpha_{i-1}^{*}$ and the initial vertex $x_{i}^{*}$ of $\alpha_{i}^{*}$ determine a unique interval $(y_{i-1}^{*},x_{i}^{*})\ss C^{*}_{\gamma}$ which contains no vertices. This follows since there exists an edge $\beta_{i}\ss|\Lambda_{+}|$ such that $\alpha_{i-1}\cap\beta_{i}=y_{i}$,   $\beta_{i}\cap\alpha_{i}=x_{i+1}$, and $\alpha_{i-1}\cup\beta_{i}\cup\alpha_{i}\ss\delta U_{\gamma}$ for some component $U_{\gamma}$ of $\UU$.

We denote the interval $[y_{i-1}^{*},x_{i}^{*}]$ by $\beta_{i}^{*}$. Then $\bigcup_{i = -\infty}^{\infty} \alpha_i^{*}\cup\beta_i^{*}$ where $\alpha_{i-1}^{*}\cap\beta_{i}^{*}=y_{i}^{*}$ and   $\beta_{i}^{*}\cap\alpha_{i}^{*}=x_{i+1}^{*}$ equals $C^{*}_{\gamma}$. The union is, in fact, a finite union $\bigcup_{i = 0}^{p-1} \alpha_i^{*}\cup\beta_i^{*}$ with $\alpha_{0}^{*} = \alpha_{p}^{*}$, for a least integer $p>0$,  since each component of $\bd\T^{*}$ has finitely many edges. Thus $\gamma$ corresponds to $C^{*}_{\gamma}\ss\bd\T^{*}$.

If $\gamma$ is compact, it follows that $\alpha_{p} = \alpha_{0}$ and $\alpha_{j}\ne\alpha_{0}$ for $0<j<p$.  Thus the intersection number of $C^{*}_{\gamma}$ and the juncture $J_{\kappa}$ is zero. Similarly, if $\gamma$ is of the second kind, then $\alpha_0\ne\alpha_p$. In this case $h^{\iota^*}(\alpha_{0}) = \alpha_{p}$ where $\iota^*$ is the intersection number of $C^{*}_{\gamma}$ and $J_{\kappa}$.

Conversely, if $C^{*}$ is a component of $\bd\T^{*}$ containing vertices, choose an edge $\alpha_{0}^{*} = [x_{0}^{*},y_{0}^{*}]$ corresponding to an extreme leaf $\alpha_{0}$ of a parallel packet of leaves of $\Lambda_{-}|\ddot\UU_{e_{-}}$. Then $\alpha_{0}\ss\gamma$ where $\gamma$  is a border component of $\UU$ that meets $\UU_{e_{-}}$ which is not a rectangle or a border component  of $\UU$ of the first kind.  Exactly as above, $\gamma$ corresponds to $C^{*}$. Thus for each $C^{*}\ss\bd\T^{*}$ containing a vertex there is at least one $\gamma$ not a rectangle and not of the first kind that corresponds to it.

If $\gamma$ is compact,  $\{h^{np_{e_-}}(\gamma)\ |\ n\in\Z\}$ is a bi-infinite family of disjoint compact border components that meet $\UU_{e_{-}}$. Thus, in this case, there is  one bi-infinite family of border components corresponding to $C^{*}_{\gamma}$.

If $\gamma$ is not compact, then $\gamma$ must be of the second kind. Let $\iota^*$ be the intersection number of $C^*$ and $J_{\kappa}$. Since $h^{\iota^*}(\gamma) = \gamma$ and the curves $h^{j}(\gamma)$, $0\le j<\iota^*$, are distinct and there are $p_{e_-}$ ends in the cycle of ends containing $e_-$, it follows that there are $\iota^*/p_{e_-}$ border components  of $\UU$  of the second kind that meet $\UU_{e_{-}}$ corresponding to  the  component $C^*$.
\end{proof}

In the case that $\gamma$ is not a rectangle, the next two corollaries follow because a component of $\bd\T^{*}$ is   a simple curve. They are clear if $\gamma$ is a rectangle.

\begin{cor}

If $\gamma\in\delta\UU$, then $\gamma$ is embedded in $L$.

\end{cor}

\begin{cor}\label{gamembed}
If $\gamma\in\delta\UU$ is compact, then $h^{m}(\gamma)$ is disjoint from $h^{n}(\gamma)$ for all $m\ne n\in\Z$.
\end{cor}

The next corollary follows because $\bd\T^{*}$ has finitely many components and $L$ has finitely many ends.

\begin{cor}\label{finnumbord}

\begin{enumerate}

\item There are finitely many infinite families $\{h^{n}(\gamma)\ |\ n\in\Z\}$ of disjoint compact border components of $\UU$ that are not rectangles where $\gamma$ is a compact border component\upn{;}

\item There are finitely many noncompact border components of $\UU$.

\end{enumerate}

\end{cor}

The next corollary follows because there are finitely many   noncompact border components of $\UU$.

\begin{cor}

If $\gamma\in\delta\UU$ is noncompact, then there exists a least integer $r>0$ such that $h^{r}(\gamma) = \gamma$.

\begin{rem}
The integer $r$ is necessarily a multiple of both $p_{e_{-}}$ and $p_{e_{+}}$.

\end{rem}

\begin{rem}
If $\gamma\in\delta\UU$ is of the first kind, then $h^{r}$ has one fixed point on $\gamma$. If $\gamma\in\delta\UU$ is of the second kind, then $h^{r}$ is fixed point  free on $\gamma$.

\end{rem}

\end{cor}

Denote by $\epsilon_{-}$ the negative end of $\gamma$. That is, if $\gamma = \alpha_{0}\cup\beta_{0}$ is of the first kind then $\alpha_{0}= (\epsilon_{-},y_{0}]$ and if $\gamma$ is of the second kind and $x\in\gamma$,  then  $h^{nr}(x)\to \epsilon_{-}$ as $n\to-\infty$.

\begin{lemma}\label{disfr}

If $\gamma\in\delta\UU$ is noncompact and $(\epsilon_{-},x_{0}]\ss\gamma$ is a neighborhood of $\epsilon_{-}$, then there exists a neighborhood of $e_{+}$ \emph{(see Lemma~\ref{endsee})} disjoint from $(\epsilon_{-},x_{0}]$.

\end{lemma}

\begin{proof}
The lemma is clear if $\gamma$ is of the first kind. Suppose $\gamma$ is of the second kind. Choose $z\in\gamma$ and let $z_{n} = h^{nr}(z)$, $n\in\Z$. Since $[z_{-1},z_{0}]\ss\gamma$ is compact, there exists a distinguished neighborhood $U_{e_{+}}$ of $e_{+}$ disjoint from $[z_{-1},z_{0}]$. Since $U_{e_{+}} \supset U_{e_{+}}^{k} = h^{kp_{e_{+}}}(U_{e_{+}})$, all $k\ge 0$, it follows that $U_{e_{+}}^{k}$ is disjoint from $[z_{-1},z_{0}]$, all $k\ge 0$. Thus $U_{e_{+}} = h^{-kr}(U_{e_{+}}^{kr/p_{e_{+}}})$ is disjoint from $h^{-kr}([z_{-1},z_{0}]) = [z_{-(k+1)},z_{-k}]$. Thus, $U_{e_{+}}$ is disjoint from $\bigcup_{k=0}^{\infty}[z_{-(k+1)},z_{-k}] = (\epsilon_{-}, z_{0}]$.
\end{proof}

\begin{cor}\label{aneb}

If $\gamma\in\delta\UU$ is noncompact, then any lift $\wt\gamma$ of $\gamma$ has two endpoints $a\ne b\in\Si$.

\end{cor}

\begin{proof}
Let $\sigma$ be a positive  juncture component meeting $\gamma$ at a point $z$. Let $\sigma_{n} = h^{nr}(\sigma)$ and $z_{n} = h^{nr}(z)$. Let $\wt\gamma$ be a lift of $\gamma$, $\wt z_{n}\in\wt\gamma$ lifts of $z_{n}$, and $\wt\sigma_{n}$ lifts of $\sigma_{n}$ containing $\wt z_{n}$. Then the geodesics $\wt\sigma_{n}$ nest on a point $b\in\Si$ as $n\to\infty$. An analogous argument shows that the negative  end of $\wt\gamma$ limits on a well defined point $a\in\Si$.

If $a=b$ then every neighborhood of the negative end of $\wt\gamma$ meets $\wt\sigma_{n}$ for $n$ sufficiently large. Since the sequence $\{\sigma_{n}\}_{n\ge 0}$ escapes to $e_{+}$, this contradicts Lemma~\ref{disfr}.
\end{proof}

\begin{example}\label{frstsndknd}
In Example~\ref{simpex}, there is one semi-isolated leaf $\lambda_{-}\in\Lambda_{-}$ and one semi-isolated leaf $\lambda_{+}\in\Lambda_{+}$. This example has one real line border component   of the first kind and one  real line border component  of the second kind,  both clearly visible in Figure~\ref{PlanarEnd}.
\end{example}

\begin{example}\label{exam719}
Example~\ref{simpex} has one bi-infinite sequence $\CC$ of compact border components of $\UU$ that are not borders of  rectangles.   Each $\gamma_{n}\in\CC$ has one edge in the semi-isolated leaf  $\lambda_{-}\in\Lambda_{-}$ and one edge in the semi-isolated leaf  $\lambda_{+}\in\Lambda_{+}$.  The component $U_{n}\ss\UU$ with this $\gamma_{n}$ as border is a (stretched out) annulus and  has   one of the boundary circles of $L$ in its boundary. In this example (and most examples) there are infinitely many bi-infinite sequences of compact border components of $\UU$ that are the borders of rectangles.  Example~\ref{itj} is similar, each $\gamma_{n}$ again having two edges and two vertices and all belonging to the same bi-infinite sequence $\CC$, but now they are the elements of $\delta U$, where $U$ is a single unbounded component of $\UU$.  
\end{example}

\begin{defn}[peripheral border component]\label{perif}
A border component $\gamma$ of $\UU$ is \emph{peripheral} if there is a component $C\ss\UU$ of $\bd L$ such that $\gamma$ and $C$ cobound a component $U_{\gamma}$ of $\UU$ homeomorphic either to an open  annulus or an open infinite strip $\R\x(0,1)$.   
\end{defn}

\subsection{Reducing circles}\label{redcircs}

Let $\gamma\in\delta U_{\gamma}$ be a circle border  component of $\UU$, where $U_{\gamma}$ is a component of $\UU$ that is not an open disk nor  an open M\"obius strip.  Then there is a unique simple closed geodesic  $\sigma_{\gamma}$  in the free homotopy class of $\gamma$. By the convexity of $U_{\gamma}$, $\sigma_{\gamma}\ss U_{\gamma}$. This geodesic $\sigma_{\gamma}$ is 2-sided and cobounds an open annulus $A\ss U_{\gamma}$ with $\gamma$.  If $\gamma$ is peripheral, then $\sigma_{\gamma}$ is a component of $\bd L$ and is not taken as a reducing curve. Otherwise $\sigma_{\gamma}$ will be included in the set $\S$ of reducing curves.     If the  set $U_{\gamma}$ is an open annulus with boundary  the two curves $\gamma_{1}$ and $\gamma_{2}$, then $\sigma_{\gamma_{1}}=\sigma_{\gamma_{2}}$. Otherwise,   the various $\sigma_{\gamma}$'s will be disjoint and disjoint from $\Lambda_{\pm}$, $\bd L$, and the previously constructed reducing curves.

\begin{lemma}\label{fincov}
The sequence of reducing circles $\{\sigma_{h^{n}(\gamma)}\}_{n\in\Z}$ escapes. 
\end{lemma}

\begin{proof}
Note that $\sigma_{h^{n}(\gamma)} = (h^{n}(\sigma_{\gamma}))^{\g}$. Since the circle $\sigma_{\gamma}$ is compact and contained in $\UU_{e_{-}}\cap\UU_{e_{+}}$, it follows that $\sigma_{\gamma}\ss U_{e_{-}}\cap U_{e_{+}}$ for a distinguished neighborhood $U_{e_{-}}$ of $e_{-}$ and  distinguished neighborhood $U_{e_{+}}$ of $e_{+}$. We will use the fact that $\sigma_{\gamma}\ss U_{e_{+}}$ to show that $\{\sigma_{h^{n}(\gamma)}\}_{n\ge 0}$ escapes. In an analogous way the fact that $\sigma_{\gamma}\ss U_{e_{-}}$ shows that $\{\sigma_{h^{n}(\gamma)}\}_{n\le 0}$ escapes.

Let $J = \fr U_{e_{+}}$. By the construction of $h$ in Section~\ref{defineh}, $J_{n} = h^{n}(J)$, $n\in\Z$. Thus $J_{n} = \fr h^{n}(U_{e_{+}})$ and $h^{n}(U_{e_{+}})$ is the distinguished neighborhood of a positive end in the cycle of ends containg the end $e_{+}$. Since $\fr h^{n}(U_{e_{+}})$ consists of the geodesics forming $J_{n}$ and $h^{n}(\sigma_{\gamma})\ss h^{n}(U_{e_{+}})$, it follows that $\sigma_{h^{n}(\gamma)} = (h^{n}(\sigma_{\gamma}))^{\g}\ss h^{n}(U_{e_{+}})$, $n\in\Z$. Since only finitely many of the distinguished neighborhoods $h^{n}(U_{e_{+}})$, $n\ge 0$, meet any compact set, it follows that $\{\sigma_{h^{n}(\gamma)}\}_{n\ge 0}$ escapes. 
 \end{proof}

\begin{rem}
 As one forwardly iterates applications of $h$ to $\gamma$, the vertices remain in $K$ and the edges $\beta_i\ss|\Lambda_+|$ stretch without bound.  Similarly, under iterates of $h^{-1}$, the edges $\alpha_i\ss|\Lambda_-|$ become unbounded. The open annulus cobounded by  $\sigma_{\gamma}$  and $\gamma$ also stretches without bound.  A simple modification of Example~\ref{simpex} illustrates this behavior.  Alter $L$ by gluing a punctured torus to each boundary circle.  Those circles now become reducing curves $\sigma_{\gamma}$ and one easily sees $\gamma$ and the annulus that it cobounds with $\sigma_{\gamma}$.
\end{rem}

\begin{figure}[htb]
\begin{center}
\begin{picture}(275,130)(40,-10)
\includegraphics[width=350pt]{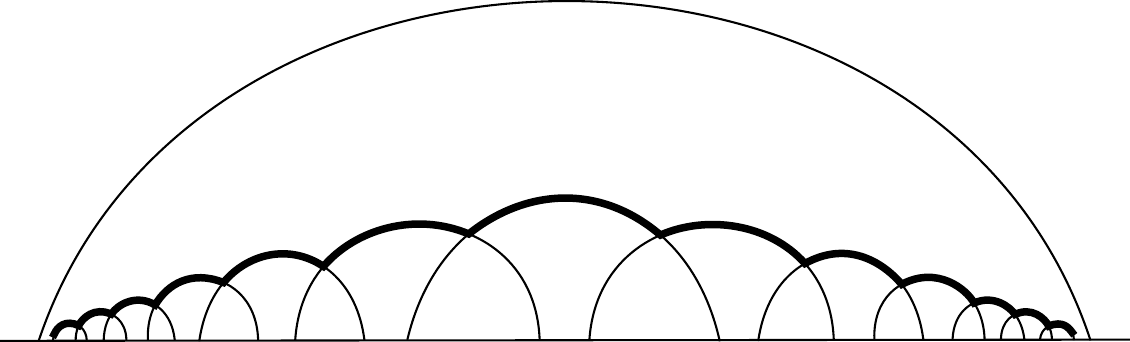}

\put (-340,-10){$a$}
\put (-13,-10){$b$}
\put (-228,42){$\tilde{\alpha}_i$}
\put (-180,50){$\tilde{\beta}_i$}
\put (-135,42){$\tilde{\alpha}_{i+1}$}
\put (-73,25){$\tilde{\gamma}$}
\put (-66,75){$\tilde{\sigma}_{\gamma}$}

\end{picture}
\caption{The lifts $\wt{\sigma}_{\gamma}$ and $\wt{\gamma}$ (boldface) for $\gamma\in\delta\UU$ of second kind}\label{firstposs}
\end{center}
\end{figure}

\subsection{Reducing lines}\label{redcrv2}

Let $\gamma\in\delta U_{\gamma}$ be a  border component of $\UU$ homeomorphic to the reals.  The associated curve $\sigma_{\gamma}$ is the geodesic whose lift has   endpoints $a\ne b$ of Corollary~\ref{aneb}.  Figure~\ref{firstposs} (respectively Figure~\ref{firstknd}) illustrates  $\wt\gamma$ and $\wt{\sigma}_{\gamma}$ when $\gamma$ is of the second kind (respectively first kind). Here   $\wt\gamma$ is given in boldface and  the bottom line  represents a subarc of $\Se$ (Definition~\ref{se}).

\begin{figure}[h]
\begin{center}
\begin{picture}(275,150)(40,-10)
\includegraphics[width=350pt]{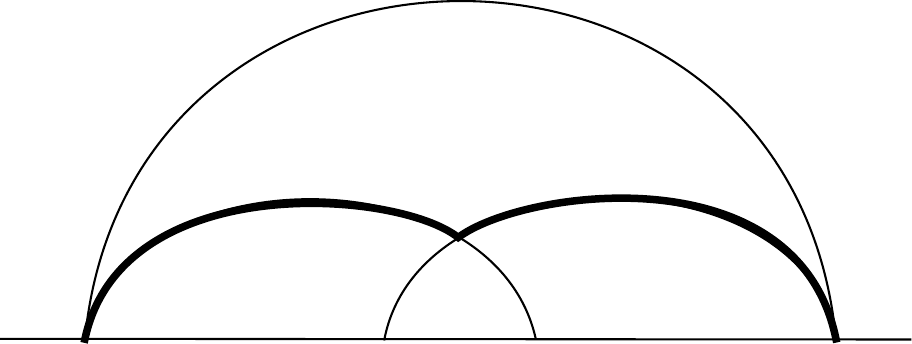}

\put (-70,53){$\tilde{\gamma}$}
\put(-77,100){$\wt{\sigma}_{\gamma}$}
\put(-176,30){$\wt y_{0}$}
\put(-260,58){$\wt\alpha_{0}$}
\put(-150,58){$\wt\beta_{0}$}
\put(-319,-9){$a$}
\put(-29,-9){$b$}

\end{picture}
\caption{The lifts $\wt{\sigma}_{\gamma}$ and $\wt{\gamma}$ (boldface) for $\gamma\in\delta\UU$ of first kind}\label{firstknd}
\end{center}
\end{figure}

\begin{lemma}\label{siggamprop}
The curve $\sigma_{\gamma}$ is simple, homeomorphic to $\R$, and disjoint from $\gamma$. 
\end{lemma}

\begin{proof}
If $\sigma_{\gamma}$ is not simple, then some lift of $\sigma_{\gamma}$ intersects the lift $\wt\sigma_{\gamma}$ with endpoints $a,b$. It follows that some lift of $\gamma$ intersects the lift $\wt\gamma$ with endpoints $a,b$ which contradicts the fact that $\gamma$ is simple. Thus, $\sigma_{\gamma}$ is simple.

As in the proof of Corollary~\ref{aneb}, there exists   lifts $\sigma_n$ of positive juncture components that lie in arbitrarily small neighborhoods of the end $e_+$ for $n$ sufficiently large such that the $\sigma_n$ nest on $b$. Since $\sigma_{\gamma}$ has a lift with  endpoint $b$, this implies that $\sigma_{\gamma}$ is unbounded. If $\sigma_{\gamma}$ is homeomorphic to a circle then it is compact and thus bounded. This contradiction implies that $\sigma_{\gamma}$ is homeomorphic to the reals.

If $\sigma_{\gamma}$ meets $\gamma$, then some lift $\wt\sigma'_{\gamma}$ of $\sigma_{\gamma}$ with endpoints $a',b'\in\Si$ meets the lift $\wt\gamma$ of $\gamma$ with endpoints $a,b$. If $\gamma$ is a border component of $\UU$ of the first kind, then $\wt\sigma'_{\gamma}$ meets one of the rays $\wt\alpha_0$ or $\wt\beta_0$ (see Figure~\ref{firstknd}). If $\gamma$ is a border component of $\UU$ of the second kind, then $\wt\sigma'_{\gamma}$ meets one of the compact arcs $\wt\alpha_i$ or $\wt\beta_i$, $-\infty<i<\infty$ (see Figure~\ref{firstposs}). Let $\alpha$ with lift $\wt\alpha$  be such a ray or compact arc  that $\wt\sigma'_{\gamma}$ meets $\wt\alpha$ and let $\lambda\in\Lambda_{\pm}$ be such that $\lambda\supset\alpha$ has lift $\wt\lambda\supset\wt\alpha$. Then the endpoints of $\wt\lambda$ on $\Si$ separate $a'$ and $b'$ so $\wt\lambda$ intersects the lift $\wt\gamma'$ of $\gamma$ with endpoints $a',b'$ on $\Si$  transversely contradicting the fact that $\gamma\in\delta\UU$. Thus, $\sigma_{\gamma}$ and $\gamma$ are disjoint.
\end{proof}

\begin{cor}\label{cor631}

The curves $\gamma$ and $\sigma_{\gamma}$ cobound a region $O_{\gamma}\ss L$ homeomorphic to $(0,1)\x\R$ and contained in the component $U_{\gamma}$ of $\UU$.

\end{cor}

\begin{proof}
Let $O_{\gamma}$ be the projection of the open region $\wt O_{\gamma}\ss\Delta$ bounded by the curves $\wt\gamma$ and and $\wt\sigma_{\gamma}$.  We must show that if $T$ is any nontrivial covering transformation, then $T(\wt O_{\gamma})\cap \wt O_{\gamma}=\emptyset$. Suppose to the contrary that $T(\wt O_{\gamma})\cap \wt O_{\gamma}\ne\emptyset$.  If $T(a)\ne a$ and/or $T(b)\ne b$,  then one of $T(\wt\gamma),T(\wt\sigma_{\gamma})$ must meet one of $\wt\gamma,\wt\sigma_{\gamma}$ which either contradicts the fact that $\gamma$ and $\sigma_{\gamma}$ are each simple or contradicts the fact that $\gamma$ and $\sigma_{\gamma}$ are disjoint.   Thus we can assume $T(a) = a$ and  $T(b) = b$ but in that case,  we must have $T(\wt O_{\gamma})\ = \wt O_{\gamma}$ and $T(\wt\gamma) = \wt\gamma$. Since $\gamma$ is homeomorphic to the reals, $T$ must be the identity transformation.
 
By the convexity of $\wt O_{\gamma}$, any geodesic in $\wt\Lambda_{\pm}$ that meets $\wt O_{\gamma}$ must also meet $\wt\gamma$. It follows that $\wt O_{\gamma}$ is disjoint from $\wt\Lambda_{\pm}$ and  that $O_{\gamma}\ss U_{\gamma}$.
\end{proof}

\begin{cor}\label{2ndconnects}

One end of the curve $\sigma_{\gamma}$ approaches the positive end $e_{+}$ while the other end of $\sigma_{\gamma}$ approaches the negative end $e_{-}$.

\end{cor}

If $\gamma$ is peripheral, then $\sigma_{\gamma}\ss\bd L$ and is not taken as a reducing curve. If the  set $U_{\gamma}$ is a doubly infinite open strip with boundary  two curves $\gamma_{1}$ and $\gamma_{2}$, then $\sigma_{\gamma_{1}}=\sigma_{\gamma_{2}}$. In this case,  both of the $\gamma_i$ may be of the second kind or one of the $\gamma_{i}$ may be of the first kind and   the  other of the second kind.  It can not happen that both of the $\gamma_i$ are of the first kind. 

Indeed, exactly as in Example~\ref{simpex2}, if the  set $U_{\gamma}$ is a doubly infinite open strip with boundary  two curves $\gamma_{1}$ and $\gamma_{2}$ both of the first kind, then there would exist a leaf of $\Lambda_+$ and a leaf of $\Lambda_-$ each meeting $U_{\gamma}$ contradicting the fact that $U_{\gamma}$ is contained in the escaping set.

 Otherwise,   the various $\sigma_{\gamma}$'s, associated to real line border components $\gamma$ of $\UU$, will be disjoint and disjoint from $\Lambda_{\pm}$, $\bd L$. The curves $\sigma_{\gamma}$ will be included in the set $\S$ of reducing curves.

\begin{rem}
If $\gamma$ is of the second kind, the edges  $\beta_i\ss|\Lambda_+|$ of $\gamma$ stretch unboundedly as $i\to\infty$ and  the edges $\alpha_i\ss|\Lambda_-|$ stretch unboundedly as $i\to -\infty$.  The vertices of $\gamma$ remain in the core. Thus,  these border components of the second kind appear quite bizarre in $L$ and do not directly connect a negative end of $L$ to a positive one but rather return infinitely often to the core.   By Corollary~\ref{2ndconnects},    $\sigma_{\gamma}$ does directly connect a negative end to a positive end.  
\end{rem}

The set $\S$ of reducing curves is now complete.

\begin{rem}
The reducing curves in $\S$ are of following kinds,

\begin{enumerate}

\item The rims of crown sets defined in Section~\ref{redcurv}.

\item The infinite families of reducing circles defined in Section~\ref{redcircs}.

\item The reducing lines corresponding to border components of $\UU$ of the first and  second kind defined in Section~\ref{redcrv2}.

\end{enumerate} 

\end{rem}

\begin{rem}
A reducing curve is always a geodesic.

\end{rem}

 
\section{Reduction}\label{reducing}

By a reduced piece, we mean the internal completion $Q=\ddot U$ of a component $U$ of $L\sm|\S|$. Recall that $\ddot\iota:Q\to\ol U$ is not necessarily one-one but possibly identifies border components of $Q$ (Section~\ref{section51}). As we define $g$ on $\ol U$, in abuse of notation, we will also consider $g$ to be defined on $Q$. If $A\ss L$ we will abuse notation and denote $\ddot\iota^{-1}(A)$ by $A\cap Q$ and say $A$ meets $Q$ if $A\cap Q\ne\0$.

\begin{theorem}\label{geodext1}
If $f:L\to L$ is an endperiodic automorphsm, then there exists   an endperiodic automorphism $g:L\to L$, isotopic to $f$ such that,

\begin{enumerate}

\item $g\bigl||\Lambda_{+}|\bigr. = h\bigl||\Lambda_{+}|\bigr.$ and $g\bigl||\Lambda_{-}|\bigr. = h\bigl||\Lambda_{-}|\bigr.$ where $h$ is the endperiodic automorphism of \emph{Theorem~\ref{geodext}}\upn{;}

\item $g$ permutes the elements of the set $\S$ of reducing curves\upn{;}  

\item There are at least one and at most finitely many components $U$  of $L\sm|\S|$ with $\ol U$ noncompact. For such  $U$, if $g^{m}(Q) = Q$, then either,

\begin{enumerate}

\item $g^{m}:Q\to g^{m}(Q)$ is isotopic to a translation\upn{;}

\item $g^{m}:Q\to g^{m}(Q)$ is a pseudo-anosov automorphism \upn{(}Definition~\ref{psdansv}\upn{)}.

\end{enumerate}

\end{enumerate}

\end{theorem}

\begin{rem}
The dynamics of $g$ on the compact reduced pieces is described in Propositions~\ref{hnisNT} and~\ref{hnistriv}. Briefly, if $Q$ is a compact reduced piece in a principal region, the dynamics of $g$ on $Q$ is given by Nielsen-Thurston theory. If $Q$ is a compact reduced piece contained in the escaping set $\UU$, then the dynamics of $g$ on $Q$ is trivial.

\end{rem}

\begin{rem}
If $\sigma_{\gamma}$ is a reducing curve associated to the border component $\gamma\in\delta\UU$ then $g(\sigma_{\gamma}) = \sigma_{h(\gamma)}$. If $\rho_{\gamma}$ is a reducing curve which is a rim of a crown set associated to a boundary component $\gamma$ of the nucleus of a principal region, then $g(\rho_{\gamma}) = \rho_{h(\gamma)}$.

\end{rem}

\subsection{Proof of first part of Theorem~\ref{geodext1}}

 Let $h$ be the endperiodic automorphism  of Theorem~\ref{geodext}. Note that the isotopies in Lemma~\ref{preservesS} will move some juncture components but leave invariant  the juncture components in $\{h(\sigma)\ |\ \sigma\in\JJ_{W}\}$ ($\JJ_{W}$ is defined in Definition~\ref{jw}).  

\begin{lemma}\label{preservesS}
There exists an endperiodicautomorphism $g$ isotopic to $h$ by an isotopy with support disjoint from $\Lambda_{+}\cup\Lambda_{-}\cup\bd L$ such  that $g(\S)=\S$.
\end{lemma}

\begin{proof}[Sketch of proof]

Enumerate the elements of $h(\S)$ as $\{\gamma_i\}_{i=1}^{\infty}$ (or $\{\gamma_i\}_{i=1}^{N}$ if there are finitely many reducing curves). Let $\psi_0 = \id$. We inductively find  sequences $\{\Phi_{i}\}$ of isotopies and $\{\psi_{i}\}$ of homeopmorphisms with $\psi_{i} = \Phi^{1}_{i}\circ\psi_{i-1}$, $i\ge 1$,     with support disjoint from $\Lambda_{+}\cup\Lambda_{-}\cup\bd L$, and  with $\Phi_{i}$ fixing $\gamma_{j}^{\g}$ pointwise, $1\le j\le i-1$,   and moving $\psi_{i-1}\circ \gamma_{i}$ to its geodesic tightening. That is, $\psi_{i}(\gamma_{i}) =  \gamma_{i}^{\g}$.

 Suppose $\Phi_i, \psi_i$, $1\le i\le n-1$, have been defined satisfying these properties and let $\wt\gamma_n$ be a lift of $\gamma_n$ and $\wt\gamma^{\g}_n$ be the lift of $\gamma^{\g}_n$ sharing endpoints on $\Si$ with $\wt\gamma_n$. 
 
 If $\sigma$ is a juncture component in $h(\JJ_W)$ and  $x\in\sigma\cap\gamma_i$  has lift $\wt x\in\wt\sigma\cap\wt\gamma_i$ where $\wt\sigma$ is a lift of $\sigma$, then by Lemma~\ref{slide} there is a sliding isotopy supported in a small neighborhood of $\sigma$ with lift sliding $\wt x$ along $\wt\sigma$ to the point of intersection of $\wt\sigma$ and $\wt\gamma^{\g}_i$.  After a sequence of such sliding isotopies , we can assume that $\wt\sigma\cap\wt\gamma_i = \wt\sigma\cap\wt\gamma^{\g}_i$ for every juncture component $\sigma\in h(\JJ_W)$ amd lift $\wt\sigma$ meeting $\wt\gamma_i$.

Consider the set of all points of intersection of $\sigma\cap\gamma_n$, all $\sigma\in h(\JJ_W)$.  If $a,b\in\gamma_n$ are two such points of intersection so that the interval $[a,b]\ss\gamma_n$ contains no other such point of intersection, then, by Theorem~\ref{3.1}, there is an isotopy moving the interval $[a,b]\ss\gamma_n$ to the corresponding interval in $\gamma^{\g}_n$ with endpoints  $a,b$. After a sequence of such isotopues we can assume $\gamma_n$ has been moved to $\gamma^{\g}_n$.  Let $\Phi_n$ be the composition of the isotopies of this and  the previous  paragraph and $\psi_{n} = \Phi^{1}_{n}\circ\psi_{n-1}$.   By the remark after Theorem~\ref{epsteinsmooth}, we can assume $\Phi_{n}$ fixes $\gamma_{j}^{\g}$ pointwise, $1\le j\le n-1$.

The    supports of the $\Phi_{n}$ do not accumulate. Thus, $\psi_{n}\to\psi$, a well defined homeomorphism isotopic to the identity by an isotopy $\Phi$ and $g=\psi\o h$ is such that $g(\alpha)$ is a geodesic for every $\alpha\in\JJ_{W}$. Since $\gamma_n$ and $\gamma_n^{\g}$ both lie in the same component of $L\sm(|\Lambda_{+}|\cup|\Lambda_{-}|)$ for each $i\ge 1$, the  isotopy $\Phi$ can be defined to have support disjoint from $\Lambda_{+}\cup\Lambda_{-}$. The lemma follows.
\end{proof}

This $g$ is the required endperiodic automorphism of Theorem~\ref{geodext1}.

\begin{cor}\label{hstarf}

The endperiodic automorphism $g$ permutes the components  $U$ of $L\sm|\S|$.

\end{cor}

The endperiodic automorphism $g$ of Theorem~\ref{geodext1} permutes  the elements of each of the sets $\Lambda_{+}$,  $\Lambda_{-}$,  and $\S$ . It is not possible to construct $g$ isotopic to $f$ so that $g$ also  permutes the elements of each of the sets   $\JJ_{+}$ and  $\JJ_{-}$  without additional restriction on the choice  of $f$-junctures. The problem one encounters in trying to do this is illustrated in the following example.

\begin{figure}[h]
\begin{center}
\begin{picture}(300,200)(-50,0)
\includegraphics[width=200pt]{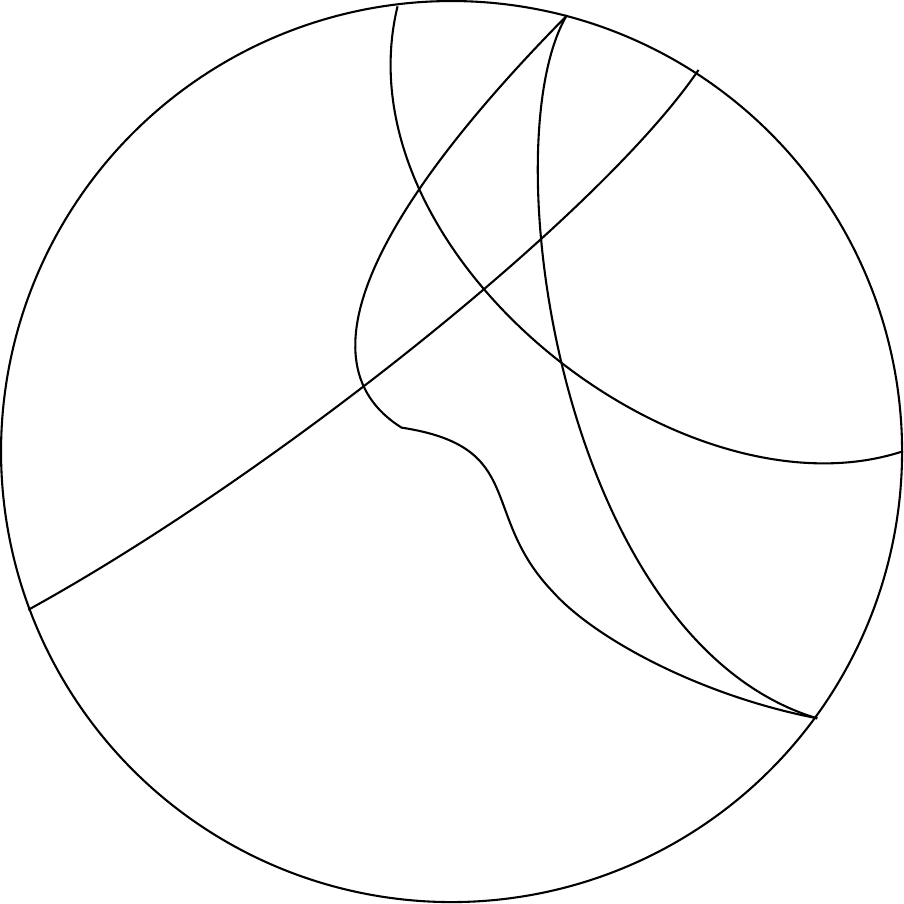}

\put(-90,60){$\sigma$}
\put(-55,75){$\sigma^{\g}$}
\put(-35,105){$\beta$}
\put(-150,80){$\alpha$}

\end{picture}
\caption{A bad configuration of three pseudo-geodesics in $\Delta$}\label{BermTri}
\end{center}
\end{figure}

\begin{example}\label{ExBermTri}

If three pseudo-geodesics in $\Delta$ are such that any two of them intersect, then there might not exist an ambient isotopy of $\Delta$ that moves each of them to a geodesic. Figure~\ref{BermTri} illustrates such a bad triple of pseudo-geodesics. In Figure~\ref{BermTri}, $\alpha$ and $\beta$ are geodesics and $\sigma$ is a pseudo-geodesic with geodesic tightening $\sigma^{\g}$. There clearly does not exist an ambient isotopy that moves each of $\alpha$, $\beta$, and $\sigma$ to geodesics.

It is possible to construct an endperiodic automorphism $f$ of a surface $L$ with positive juncture component $\tau^+$, negative juncture component $\tau^-$, and reducing curve $\gamma$ such that the lifts of the pseudo-geodesics $f(\tau^+), f(\tau^-),f(\gamma)$
 are configured as $\alpha,\beta,\sigma$ in Figure~\ref{BermTri}. For such an example it would be impossible to find an ambient isotopy of $L$ permutting each of the sets $\JJ_+$, $\JJ_-$, and $\S$.

\end{example}

\subsection{The reduced pieces}\label{rpand pf}

By a reduced piece $Q$, we mean the internal completion $Q=\ddot U$ of a component $U$ of $L\sm|\S|$.

\begin{lemma}\label{QolU}

Suppose $U$ is a component of $L\sm\S$.

\begin{enumerate}

\item If $U$ does not meet $|\Lambda_{+}|\cup|\Lambda_{-}|$, then $\ddot\iota:Q\to\ol U$ is a homeomorphism\upn{;}\label{QolU1}

\item If $U$ does  meets $|\Lambda_{+}|\cup|\Lambda_{-}|$, then either $\ddot\iota:Q\to\ol U$ is a homeomorphism or $Q$ is homeomorphic to $\ol U$ cut apart along some reducing curves.\label{QolU2}

\end{enumerate}

\end{lemma}

\begin{proof}
If $\sigma_{\gamma}$ is a reducing curve, then on one or both sides of $\sigma_{\gamma}$, there is a curve $\gamma\ss|\Lambda_{+}|\cup|\Lambda_{-}|$ such that $\gamma$ and $\sigma_{\gamma}$ cobound an annulus or doubly infinite strip with interior disjoint from $|\S|$. Thus, one or both sides of $\sigma_{\gamma}$ borders a component of $L\sm|\S|$ which meets $|\Lambda_{+}|\cup|\Lambda_{-}|$  and~(\ref{QolU1}) follows. In case (\ref{QolU2}), it is easy to construct examples in which $\ddot\iota:Q\to\ol U$ is not a homeomorphism to $\ol U$ but  rather identifies pairs of components of $\delta Q$ to form reducing curves.
\end{proof}

\subsubsection{The noncompact reduced pieces}\label{noncpieces}

\begin{lemma}\label{Qstand}
A noncompact reduced piece ${Q}$ is standard in the hyperbolic metric induced by the metric of $L$.
\end{lemma}

\begin{proof}
If $Q = \ol U$, then $Q$ is a standard hyperbolic surface as a subsurface with geodesic boundary of the standard hyperbolic surface $L$. Otherwise, $Q$ is $\ol U$ cut apart along embedded geodesics and so is a standard hyperbolic surface.
\end{proof}

Since there are finitely many reducing curves of the first and second kind, there are finitely many noncompact reduced pieces. Therefore, for any noncompact reduced piece $Q$, there exists an integer $m>0$ such that $g^{m}(Q) = Q$.

\begin{prop}\label{fmonQendp}

Suppose $Q$ is a noncompact reduced piece. Then,

\begin{enumerate}

\item The  immersion $\ddot \iota:{Q}\looparrowright L$ induces a map $\iota_{Q}:\EE({Q})\to\EE(L)$ and  $\EE({Q})$ is finite\upn{;}\label{Qone}

\item If $m>0$ is an integer such that $g^{m}(Q) = Q$, then $g^{m}:{Q}\to {Q}$ is endperiodic.\label{Qtwo}

\end{enumerate}

\end{prop}

\begin{proof}
 Fix an integer $k>0$ which is a multiple of $m$ and of $p_{e}$ for every end $e$ of $L$. Also fix an exhaustion $K=K_{0}\ss K_{1}\ss\cdots K_{n}\ss\cdots$ of $L$ by larger and larger choices of cores. Then $\{K_{n}\cap {Q}\}_{n\ge 0}$ is an exhaustion of $Q$. Let $e\in\EE({Q})$ and let  $U_{n}$ be the connected component of the complement in ${Q}$ of $K_{n}\cap {Q}$ which is a neighborhood of $e$. Then $e=[\{U_{n}\}]$ (see the  definition of ``end'' on page~\pageref{enddefn}). Let $V_{n}$ be the component of $L\sm K_{n}$ which contains $U_{n}$. Thus, $e' = [\{V_{n}\}]$ is a uniquely defined end in $\EE(L)$ and the map $\iota_{Q}:\EE({Q})\to\EE(L)$ defined by  $\iota_{Q}(e) = e'$ is well defined, proving the first assertion of (\ref{Qone}).
 
Next we prove that $\EE({Q})$ is finite. Since $k$ is divisible by $p_{e'}$, for every $e'\in\EE(L)$, $g^{k}:L\to L$ fixes  $\EE(L)$ pointwise.  The ends $e\in\EE({Q})$ can arise in one of two ways. Let $\iota_{Q}(e) = e'$. In the first case,  $\UU_{e'}\ss Q$ and $e$ is naturally identified to $e'$ as an attracting or repelling   end of ${Q}$ relative to $g^{k}$.  There can be at most finitely many  ends of $\EE({Q})$ obtained in this way. All ends of  $\EE({Q})$  obtained in this way are attracting or repelling ends for $g^{k}:{Q}\to {Q}$.

Otherwise, $\UU_{e'}$ meets and therefore contains a reducing curve $\gamma$ of the first or second kind or an infinite family of compact reducing curves which can be thought of as boundary component(s) of ${Q}$. Since there are only finitely many reducing curves of the first or second kind or inifinite familes of compact reducing curves, at most finitely many ends of ${Q}$ can be obtained in this way. 

Let $e\in\EE({Q})$ be of either of these types,  $e' = i_{Q}(e)$, and let $U_{e'}\ss L$ be a $g$-neighborhood of the end $e'$ as in Definition~\ref{fnbhd}.   Suppose $U_{e}\ss {Q}$ is the component $X$ of $(U_{e'}\cap {Q})$ which is a neighborhood of the end $e$ in ${Q}$ with any contiguous compact components of ${Q}\sm X$ added on. Then $U_{e}$ satisfies the conditions of Definition~\ref{perends} so $e$ is an attracting or repelling end for $g^{k}:{Q}\to {Q}$. Since all ends in $\EE({Q})$ are attracting or repelling, it follows that $g^{k}$ is endperiodic on ${Q}$. By Lemma~\ref{power}, it follows that $g^{m}$ is endperiodic on ${Q}$.
\end{proof}

Recall that the leaves of the laminations $\Lambda_{\pm}$ do not intersect the reducing curves in $\S$.

\begin{lemma}\label{lambdaQ}

Suppose $Q$ is a noncompact reduced piece and $g^{m}(Q) = Q$. 

\begin{enumerate}

\item $\UU_{\pm}\cap Q$ are the positive/negative escaping sets for the endperiodic automorphism $g^{m}:Q\to Q$.\label{1esc}

\item $(\Lambda_{-}\cap Q,\Lambda_{+}\cap Q)$ is the Handel-Miller bilamination for the endperiodic automorphism $g^{m}:Q\to Q$.\label{2lambda}

\end{enumerate}

\end{lemma}

\begin{proof}
For $x\in Q$, the sequence $\{g^{n}(x)\}_{n\ge 0}$ escapes in $L$ if and only if the sequence $\{(g|Q)^{km}(x)\}_{k\ge 0}$ escapes in $Q$ and~(\ref{1esc}) follows. By Lemma~\ref{front'}, $|\Lambda_{\pm}| = \fr\UU_{\mp}$. Let $\Lambda^{Q}_{\pm}$ denote  the positive/negative Handel-Miller laminations for the endperiodic automorphism $g^{m}:Q\to Q$.  By  Lemma~\ref{front'} and~(\ref{1esc}),   
$$|\Lambda^{Q}_{\pm}| = \fr(Q\cap \UU_{\mp}) =  Q\cap\fr\UU_{\mp} = Q\cap|\Lambda_{\pm}|$$ 
and~(\ref{2lambda}) follows.
\end{proof}

\subsection{Description of the action of $g$ on the reduced pieces}\label{actofh}

\subsubsection{The Nielsen-Thurston case}\label{NTcase}  
Suppose $U$ is a component of $L\sm|\S|$ in the nucleus of a principal region. Then $\ol U$ is compact and $(|\Lambda_{+}|\cup|\Lambda_{-}|)\cap U = \0$.     By Lemma~\ref{QolU}, $\ddot\iota:Q\to\ol U$ is a homeomorphism so $Q$ can be identified with $\ol U$. By Corollary~\ref{hstarf},  $g$ permutes the components $U$ of $L\sm|\mathfrak S|$. Since there are only finitely many such $U$, there exists a least integer $m$ so that $g^{m}(Q) = Q$. We have,

\begin{prop}\label{hnisNT}

Suppose $Q$ is a compact reduced piece lying in a principal region and $g^{m}(Q) = Q$.  Then the dynamics of the homeomorphism $g^{m}:Q\to h^{m}(Q)$ is given by Nielsen-Thurston theory.

\end{prop}

\subsubsection{The trivial case}\label{trivcase}  
Suppose $U\ss\UU$ is a component of $L\sm|\S|$ with $\ol U$ compact.    By Lemma~\ref{QolU}, $\ddot\iota:Q\to\ol U$ is a homeomorphism so $Q$ can be identified with $\ol U$.  By Corollary~\ref{hstarf},  $g$ permutes the components $U$ of $L\sm|\mathfrak S|$.  Since the sets $g^{n}(\ol U)$, $n\in\Z$, are disjoint, the dynamics of $g$ on $Q$ is trivial in this case.

\begin{prop}\label{hnistriv}

Suppose $Q$ is a compact reduced piece lying in $\UU$.   Then for all nonzero $n\in\Z$ the map  $g^{n}:Q\to g^{n}(Q)$ is is a homeomorphism between disjoint compact sets.

\end{prop}

\subsubsection{The translation case}\label{transcase}

Suppose $U$ is a   component of $L\sm|\S|$ with $\ol U$ not compact that does not meet  $|\Lambda_{+}|\cup|\Lambda_{-}|$. Then $U\ss\UU$. 

By Lemma~\ref{QolU}, $\ddot\iota:Q\to\ol U$ is a homeomorphism so $Q$ can be identified with $\ol U$.  By Corollary~\ref{hstarf},  $g$ permutes the components $U$ of $L\sm|\mathfrak S|$.  Since the number of noncompact reduced pieces is finite, it follows that there exists a least integer $m>0$ so that $g^{m}(Q) = Q$.  By Lemma~\ref{lambdaQ} (\ref{2lambda}), $\Lambda_{\pm}\cap Q$ are the Handel-Miller laminations for the endperiodic automorphism $g^{m}:Q\to Q$. Since $\Lambda_{\pm}\cap Q=\0$, by Lemma~\ref{total}, $g^{m}$ is isotopic to  a translation. We have, 

\begin{prop}\label{hnistransl}

If $Q$ is a noncompact reduced piece which does not meet  $|\Lambda_{+}|\cup|\Lambda_{-}|$ and $g^{m}(Q) = Q$,  then $g^{m}:Q\to Q$ is isotopic to  a translation.

\end{prop}

\subsubsection{The pseudo-anosov case}\label{pacase}

Suppose $U$ is a components of $L\sm|\S|$ with noncompact closure meeting $|\Lambda_{+}|\cup|\Lambda_{-}|$.    In this case the map $\ddot\iota:Q\to\ol U$ is not always one-one but rather sometimes identifies border components of $Q=\ddot U$. In addition, $\ol U$ may share border components in $\S$  with other components of $L\sm|\S|$.

Since the leaves of the laminations $\Lambda_{\pm}$ do not intersect the reducing curves in $\S$ we have,

\begin{lemma}

For a component $U$ of $L\sm|\S|$, the following two statments are equivalent,

\begin{enumerate}

\item $U$ meets $|\Lambda_{+}|\cup|\Lambda_{-}|$\upn{;}\label{psanone}

\item $U$ meets both $|\Lambda_{+}|$ and $|\Lambda_{-}|$. If $x\in U$ and $x\in\lambda\in\Lambda_{+}\cup\Lambda_{-}$, then $\lambda\ss U$.\label{psantwo}

\end{enumerate}

\end{lemma}

 \begin{defn}[pseudo-anosov]\label{psdansv}

An endperiodic automorphism $h:L\to L$ is a  \emph{pseudo-anosov automorphism} if,

\begin{enumerate}

\item  $h$ is not a translation\upn{;}

\item $L$ contains no reducing curves\upn{;}

\item $h$ preserves the Handel-Miller bilamination $(\Lambda_{-},\Lambda_{+})$.

\end{enumerate}
 .
\end{defn}

\begin{rem}
Recall that by Lemma~\ref{total}, an endperiodic automorphsm is a translation if and only if $\Lambda_{-} = \0 = \Lambda_{+}$.

\end{rem}

\begin{figure}[h]
\begin{center}
\begin{picture}(150,140)(0,0)

\includegraphics[width=150pt]{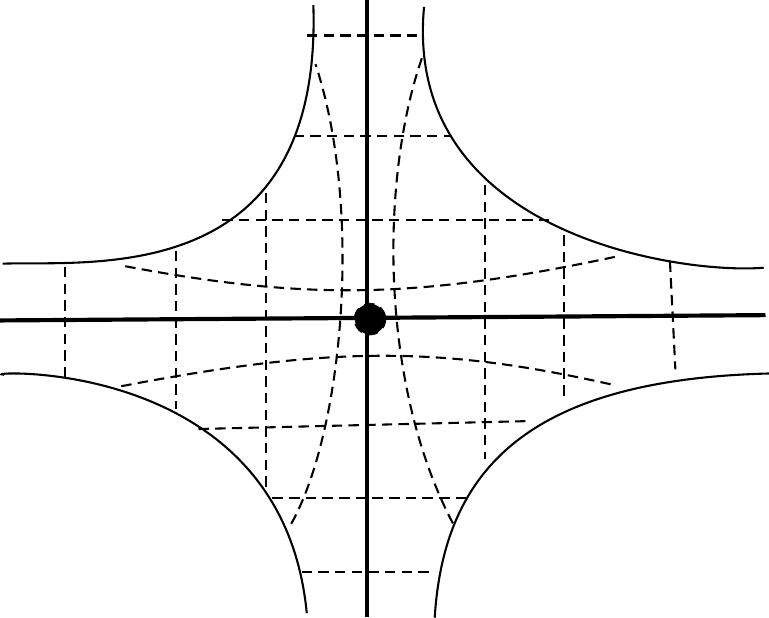}

\end{picture}
\caption{A pseudo-anosov component that is not admissible}\label{chairs1}
\end{center}
\end{figure}

\begin{rem}
We use the lower case ``a'' because the analogy with the pseudo-Anosov components for automorphisms of compact surfaces is weak.  Also, some of these reduced pieces might not even be admissible, although they will be standard. A pseudo-anosov  reduced piece as in Figure~\ref{chairs1} occurs, for instance, in Figure~\ref{TwoEnds}.  Here, because it inherits a complete hyperbolic metric from $L$, the simple ends are cusps. The laminations are still the locally uniform limits of the  geodesic junctures. This endperiodic automorphism occurs as the monodromy of the depth one leaf in Gabai's ``stack of chairs'' foliation of a sutured solid torus (\cite{ga0}, but also cf.~\cite[Section~11.1]{condel2}).
\end{rem}

By Corollary~\ref{hstarf},  $g$ permutes the components $U$ of $L\sm|\mathfrak S|$.  Since the number of  noncompact reduced pieces is finite, it follows that if   there exists a least integer $m>0$ so that $g^{m}(Q) = Q$.   We have,

\begin{prop}\label{hnispa}

If $Q$ is a noncompact reduced piece  which meets $|\Lambda_{+}|\cup|\Lambda_{-}|$ and $g^{m}(Q) = Q$,  then $g^{m}:Q\to Q$ is a pseudo-anosov automorphism.

\end{prop}

\section{Dynamics in the Core}\label{cordyn}

In the Nielsen-Thurston theory, the pseudo-Anosov homeomorphism is semi-conjugate to a two-ended Markov shift of finite type.  The analogous result here concerns the core dynamical system.

\begin{defn}[core dynamical system]\label{CoreDy}
Let $f:L\to L$ be an endperiodic automorphism and $h$ an endperiodic automorphism isotopic to $f$ and preserving the bilamination $(\Lambda_{+},\Lambda_{-})$.  Then the restriction of $h$ to $\KK=|\Lambda_{+}|\cap|\Lambda_{-}|$ defines the core dynamical system $h:\KK\to \KK$.
\end{defn}

The main theorem of this section is,

\begin{theorem}\label{basicmar}

The core dynamical system $h:\KK\to \KK$ is topologically conjugate to a two-ended Markov shift of finite type. 

\end{theorem}

We allow $\bd L\ne\0$ and we do not reduce.  Theorem~\ref{basicmar} will follow immediately from Proposition~\ref{prop1013}.

We fix a core $K$.  The pre-Markov and Markov rectangles $R\ss K$ we consider below will have a pair of opposite edges $\alpha_{R},\beta_{R}$ that are subarcs of $\Lambda_{+}$ (called the positive or vertical edges) and a pair of opposite edges $\delta_{R},\gamma_{R}$ that are subarcs of $\Lambda_{-}$ (called the negative or horizontal edges).  As remarked earlier (on page~\pageref{notgeomrect}), we are using the term ``rectangle'' as it is used in the literature on Markov partitions.  Our rectangles will be convex geodesic quadrilaterals, not geometric rectangles, as the latter are impossible in hyperbolic geometry.

\begin{rem}
We allow degenerate rectangles.  If $ \alpha_{R}= \beta_{R}$, then $R= \alpha_{R}$ degenerates to an arc in $ \Lambda_{+}$ and, if $ \delta_{R}= \gamma_{R}$, $R=\delta_{R}$ degenerates to an arc in $ \Lambda_{-}$.    These degenerate possibilities should be kept in mind in what follows.
\end{rem}

\begin{defn}[$\QQ^{+}$]\label{QQ+}
Let $Q^{\dagger}\ss K$ be the  \emph{extreme rectangle} (Definition~\ref{extquad}) with two  edges $\alpha_{Q^{\dagger}},\beta_{Q^{\dagger}}$  that are extreme arcs of $\Lambda_{+}\cap K$ in their isotopy class. The other two edges $\delta_{Q^{\dagger}},\gamma_{Q^{\dagger}}$ 
of $Q^{\dagger}$ are arcs in positive junctures in $\bd_{+} K$.  Let $Q\ss Q^{\dagger}$ be the largest  rectangular subset of $Q^{\dagger}$ with two edges subsets of $\alpha_{Q^{\dagger}},\beta_{Q^{\dagger}}\ss|\Lambda_{+}|$ and the other two edges subsets of leaves of $\Lambda_{-}$.   The finite set of all such rectangles will be denoted by $\QQ^{+}$.
\end{defn}

\begin{rem}
The rectangle $Q\ss Q^{\dagger}$ has as positive edges subarcs of the positive edges of $Q^{\dagger}$.  Its negative edges are arcs of $\Lambda_{-}\cap Q^{\dagger}$.  These arcs may be identical, in which case $Q$ degenerates to an arc in $\Lambda_{-}$.  Conceivably, $\Lambda_{-}\cap Q^{\dagger}=\0$, hence $Q=\0$. In this case $Q^{\dagger}$ lies in the positive escaping set $\UU_{+}$ and contributes nothing to the dynamical system $h:\KK\to\KK$.
\end{rem}

\begin{figure}[htb]
\begin{center}
\begin{picture}(100,100)(30,55)
\includegraphics[width=150pt]{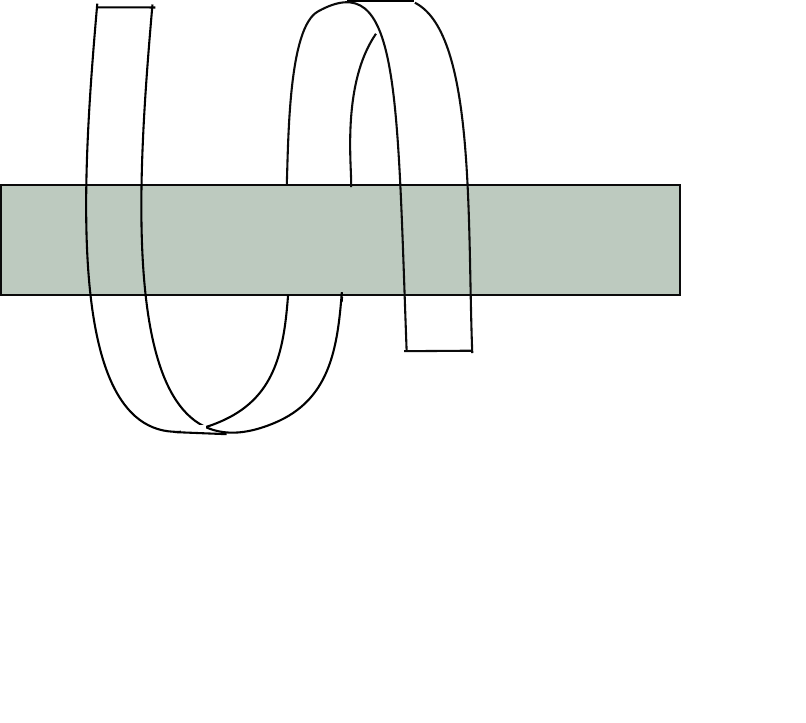}

\put(-50,90){\small $R'$}
\put(-77,120){\small $R$}
\end{picture}
\caption{$R$ completely crosses $R'$ twice}\label{compcross}
\end{center}
\end{figure}

\begin{defn}[completely crosses]\label{completelycrosses}
Let $R$ and $R'$ be nondegenerate rectangles with geodesic sides.
Then $R$ completely crosses   $R'$ in the positive direction if each component (if any) of $R\cap R'$ is a nondegenerate rectangle with horizontal boundary components subarcs of $|\Lambda_{-}|\cap R'$  and vertical boundary components subarcs of $|\Lambda_{+}|\cap R$. If $R$ and/or $R'$ degenerates, the definition is modified in the obvious way.  We say that $R$ completely crosses $R'$ in the negative direction precisely if $R'$ completely crosses $R$ in the positive direction.
\end{defn}

 This definition allows $R$ to cross $R'$ multiple times (see Figure~\ref{compcross}) or never.  It also allows the situation pictured in Figure~\ref{touch}.
 
\begin{figure}[htb]
\begin{center}
\begin{picture}(100,100)(10,0)
\includegraphics[width=90pt]{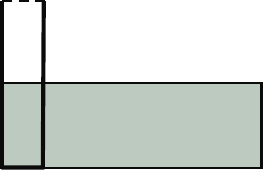}

\put(-50,10){\small $R'$}

\put(-86,40){\small $R$}
\end{picture}
\caption{$R$ completely crosses $R'$ with overlapping edges}\label{touch}
\end{center}
\end{figure}

\begin{prop}\label{oneone}
If   $Q\in\QQ^{+}$, then $h(Q)$ completely  crosses  \upn{(}in the positive direction\upn{)} any rectangle $Q'\in\QQ^{+}$ that it meets.
\end{prop}

\begin{proof}
Let $Q\ss Q^{\dagger}$  and $Q'\ss(Q')^{\dagger}$ as in Definition~\ref{QQ+}.  Suppose $h(Q)$ meets $Q'$. 
Since the vertical edges of $(Q')^{\dagger}$ are extremals of their isotopy class in $\Lambda_{+}\cap K$ and the horizontal edges of $h(Q^{\dagger})$ lie in  positive junctures disjoint from $K$, we see that $h(Q^{\dagger})$ completely crosses $(Q')^{\dagger}$, hence completely crosses $Q'$.  If $h(Q)$ does not completely cross $Q'$ then, since it meets $Q'$, at least one horizontal edge of $h(Q)$, say $\delta_{h(Q)}=h(\delta_{Q})$, lies in $Q'\sm\delta_{Q'}$.  By applying $h^{-1}$, we see that this contradicts the fact that $Q$ is the largest rectangular subset of $Q^{\dagger}$ with pairs of opposite edges in $\Lambda_{+}$ and $\Lambda_{-}$ respectively.
\end{proof}

\begin{defn}[Markov family, pre-Markov family]\label{marfam}

A finite family $\{R_{1},\ldots,R_{n}\}$ of disjoint rectangles is called a \emph{Markov family} (in the literature frequently called a \textit{Markov partition}) if it satisfies Properties I -- IV below,

\newcounter{bean}

\begin{list}{\Roman{bean}.}{\usecounter{bean}}
\item One pair of opposite edges of each $R_{i}$ is contained in  $|\Lambda_{+}|$ and the other pair of opposite edges of each $R_{i}$ is contained in $|\Lambda_{-}|$. 
\item $R_{i}\cap R_{j} = \0$ if $i\ne j$.
\item $h(R_{i})$ completely crosses each $R_{j}$.
\item $h(R_{i})\cap R_{j}$ has at most one component.
\item For each bi-infinite sequence,
$$
(\dots,i_{-k},\dots,i_{-1},i_{0},i_{1},\dots,i_{k},\dots) \in \{1,2,\dots,n\}^{\Z}
$$
the intersection,
$\bigcap_{k=-\infty}^{\infty} h^{k}(R_{i_{k}})$,
is either empty or exactly one point.
\end{list}
If the family of rectangles satisfies Properties  I -- III, we call it a \emph{pre-Markov family} (or partition).

\end{defn}

\begin{example}

We will show that the family
$$\MM^{+} = \{\text{components of }h(Q)\cap Q' \mid  Q,  Q'\in \QQ^{+}\} = \{R_{1}^{+},\ldots,R_{n}^{+}\}.
$$ is a Markov family. We refer to $\MM^{+}$ as the family of \emph{positive Markov rectangles}. The family $\QQ^{+}$ satisfying Proposition~\ref{oneone} is a pre-Markov family.

\end{example}

\begin{rem}
In the literature,  a pre-Markov partition is sometimes called a Markov partition.

\end{rem}

\begin{rem}
There are infinitely many different possible choices for the family of Markov rectangles. The family $\MM^{+}$ satisfies  Property~V if and only if there are no principal regions. Property~V   is required of a family used to prove Proposition~\ref{prop1013} and Theorem~\ref{basicmar}. Below we will give a Markov family $\MM$ that does satisfy Property~V.

\end{rem}

\begin{rem}
Often a Markov family is allowed to have contiguous elements as in Figure~\ref{touch'}.  Our families do not have contiguous elements.  This is because the rectangles $Q^{\dagger}$ that we started with could not be contiguous.
\end{rem}

\begin{rem}
It is not standard to allow degenerate rectangles in Markov families, but there is no real problem in doing so.  In our situation, Markov rectangles that degenerate to arcs are forced, for example, if there are principal regions with isolated border leaves.  It is possible that an isolated leaf of either lamination might contribute such a degenerate rectangle  and that the components of $h(Q)\cap Q'$ might degenerate to points. In general, we allow  Markov families to have some rectangles that degenerate to a point. The reader should keep such possible degeneracies in mind throughout the following discussion.
\end{rem}

\begin{figure}[htb]
\begin{center}
\begin{picture}(20,100)(100,0)
\includegraphics[width=200pt]{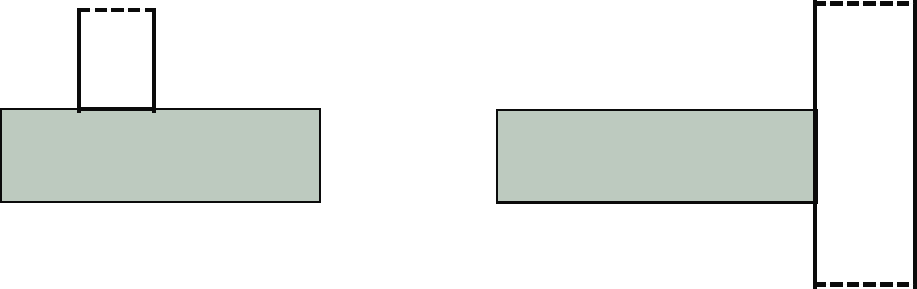}

\put(-50,25){\small $R_{i}^{+}$}
\put(-170,25){\small $R_{i}^{+}$}
\put(-15,45){\small $R_{j}^{+}$}
\put(-178,45){\small $R_{j}^{+}$}
\end{picture}
\caption{Contiguous rectangles}\label{touch'}
\end{center}
\end{figure}

\begin{cor}\label{twotwo}
If   $R_{i}^{+}\in\MM^{+}$ is a component of $h(Q)\cap Q'$,  then $R_{i}^{+}$  completely crosses $Q'$ in the positive direction.
\end{cor}

\begin{prop}\label{just1}
The family  $\MM^{+}$ has the property that  $h(R_{i}^{+})\cap R_{j}^{+}$  is either empty or $h(R_{i}^{+})$ completely crosses $R_{j}^{+}$ in the positive direction and their intersection  has a single component, $1\le i,j\le n$. Consequently, $\MM^{+}$ is a Markov family.
\end{prop}

\begin{proof}
Suppose that $R_{i}^{+}$ is a component of $h(Q)\cap Q'$ with $Q,Q'\in\QQ^{+}$ and suppose  $R_{j}^{+}\in\MM^{+}$. Since the sets in $\QQ^{+}$ are disjoint, if $h(R_{i}^{+})\ss h(Q')$ meets $R_{j}^{+}\in\MM^{+}$, it follows from the definition of $\MM^{+}$ that $R_{j}^{+}$ is a component of $h(Q')\cap Q''$ some $Q''\in\QQ^{+}$. By Corollary~\ref{twotwo}, $R_{i}^{+}$ completely crosses $Q'$ in the positive sense so $h(R_{i}^{+})$ completely crosses $R_{j}^{+}$ exactly once in the positive sense. ``Exactly once'' is due to the fact that $R_{j}^{+}$ is not $Q''$, but a component of $h(Q')\cap Q''$ and also $h(R_{i}^{+})\ss h(Q')$.
\end{proof}

Let $\{R_{1},R_{2},\dots,R_{n}\}$ be a Markov family.  One sets up an $n\x n$ \emph{Markov matrix} $A$  with entries
$$
A_{ij} = 
\begin{cases}
 1, & \text{if $R_{i}\cap h(R_{j})\ne\0$},\\
 0, & \text{if $R_{i}\cap h(R_{j})=\0$.}
\end{cases}
$$
The corresponding set $\SS$ of \emph{symbols} consists of all bi-infinite sequences
$$
\iota = (\dots,i_{-k},\dots,i_{-1},i_{0},i_{1},\dots,i_{k},\dots) \in \{1,2,\dots,n\}^{\Z}
$$
such that $A_{i_{k},i_{k+1}}=1$, $-\infty<k<\infty$.

 \begin{rem}
There are always periodic symbols in $\SS$, usually a countable infinity of them.
 \end{rem}

For each bi-infinite sequence in the symbol set $\SS$ for $\MM^{+}$ (and only for such), 
$$
\cdots\cap h^{-k}(R_{i_{-k}}^{+})\cap\cdots\cap h^{-1}(R_{i_{-1}}^{+})\cap R_{i_{0}}^{+}\cap h(R_{i_{1}}^{+})\cap\cdots\cap h^{k}(R_{i_{k}}^{+})\cap\cdots = \zeta_{ \iota} \ne\0.
$$
We would like to have that $ \zeta_{ \iota}$ ranges exactly over $\KK$ as $ \iota$ ranges over $\SS$, in which case the right-shift operator $ \sigma:\SS\to\SS$ will be exactly conjugate to $h:\KK\to \KK$.  This will be true if there are no principal regions, but generally many $ \zeta_{ \iota}$ may be whole arcs of intersection of $ \Lambda_{+}$ with arms of negative principal regions.  We will leave it to the reader to see this, remarking only that it is due to a basic asymmetry in the definition of $\MM^{+}$ which favors the role of $ \Lambda_{+}$.  In studying particular examples, and even for the applications of symbolic dynamics in~\cite{cc:cone,cc:almostnohol}, it is always adequate to use $\MM^{+}$.  The following discussion, which does not explicitly mention the principal regions, is motivated partly  by aesthetics, but is also useful.

\begin{example}
 We first define the negative pre-Markov rectangles. Let $Q^{\dagger}\ss K$ be the \emph{extreme rectangle}  (Definition~\ref{extquad}) with two  edges $\delta_{Q^{\dagger}},\gamma_{Q^{\dagger}}$  that are extreme arcs of $\Lambda_{-}\cap K$ in their isotopy class. The other two edges $\alpha_{Q^{\dagger}},\beta_{Q^{\dagger}}$ 
of $Q^{\dagger}$ are arcs in negative junctures in $\bd K$.  Let $Q\ss Q^{\dagger}$ be the largest  rectangular subset of $Q^{\dagger}$ with two edges subsets of $\delta_{Q^{\dagger}},\gamma_{Q^{\dagger}}\ss|\Lambda_{-}|$ and the other two edges subsets of leaves of $\Lambda_{+}$.   The finite set of all such rectangles will be denoted by $\QQ^{-}$.  Define
$$\MM^{-} = \{\text{components of }h^{-1}(Q)\cap Q' \mid  Q,  Q'\in \QQ^{-}\} = \{R_{1}^{-},\ldots,R_{m}^{-}\}.$$  In analogy with $\MM^{+}$, this is shown to be a Markov family.

\end{example}

\begin{rem}
 If there are no principal regions one expects $\MM^{-}$ to be closely related to $\MM^{+}$. For example, in Examples~\ref{simpex0} and~\ref{simpex}, if one takes the core $K$ to be as in Figure~\ref{DoubledExample}, then $\MM^{+} = h(\MM^{-})$. 

\end{rem}

\begin{example}

The symmetric family of \emph{Markov rectangles} is given by,
$$\MM=\{\text{components of }R_{i}^{+}\cap R_{j}^{-} \mid  R_{i}^{+}\in \MM^{+}, R_{j}^{-}\in \MM^{-}\} = \{R_{1},\dots,R_{q}\}$$
The top and bottom edges of $R_{\ell}$ will be arcs in $\Lambda_{-}$ and the left and right edges will be arcs in $\Lambda_{+}$. The family $\MM$ (again Markov) satisfies Property V and will be used to prove Theorem~\ref{basicmar}.

\end{example}

It should be remarked that some of these new Markov rectangles might degenerate to a point, a component of intersection of a degenerate rectangle $R_{i}^{+}$ with a degenerate $R_{i}^{-}$.  This finite set of points will be permuted by $h$.  There may also be arcs from one or another of the laminations among these rectangles.

As we will see, the symbol set for this Markov family will exactly encode $\KK$ and provide the desired conjugacy of $h|\KK$ to the resulting shift map.

We will leave it to the reader to check that $\MM$  is again a Markov family for $h$.  

\begin{rem}
Our situation in which distinct Markov rectangles cannot be contiguous  is a bit stronger than the usual requirement that they merely not overlap.  This eliminates the usual ambiguity in the coding, insuring that $ \sigma:\SS\to\SS$ is conjugate to $h:\KK\to \KK$ and not merely semi-conjugate. 
\end{rem}

We are ready to prove the  key result.

\begin{prop}\label{prop1013}
If $\iota = (\dots,i_{-k},\dots,i_{-1},i_{0},i_{1},\dots,i_{k},\dots) \in \{1,2,\dots,q\}^{\Z}$ is a symbol for $\MM$, then the infinite intersection 
$$
I_{ \iota}^{+}=R_{i_{0}}\cap h(R_{i_{1}})\cap h^{2}(R_{i_{2}})\cap\cdots\cap h^{k}(R_{i_{k}})\cap\cdots
$$  
is an arc of $ \Lambda_{+}\cap R_{i_{0}}$ and 
$$
I_{ \iota}^{-}=R_{i_{0}}\cap h^{-1}(R_{i_{-1}})\cap h^{-2}(R_{i_{-2}})\cap\cdots\cap h^{-k}(R_{i_{-k}})\cap\cdots
$$
is an arc of $ \Lambda_{-}\cap R_{i_{0}}$.
Furthermore, all such arcs are obtained in this way.  Consequently,
$ \zeta_{ \iota}=I_{ \iota}^{-}\cap I_{ \iota}^{+}\in \KK$ and every point of $\KK$ is of the form $ \zeta _{\iota}$ for a unique symbol $ \iota$.
\end{prop}

\begin{proof}
 One easily sees that $\MM^{+}$ and $\MM^{-}$ each covers $\KK$, hence so does $\MM$. Consequently, the assertions about $I_{ \iota}^{+}$ and $I_{ \iota}^{-}$ imply the assertion about $ \zeta_{ \iota}$. In fact, $ \zeta_{ \iota}=I_{ \iota}^{-}\cap I_{ \iota}^{+}$  will be one point as the intersection of two arcs, one in a leaf $\lambda_{-}\in\Lambda_{-}$ joining a pair of opposite edges of  $R_{i_{0}}$   and the other  in a leaf $\lambda_{+}\in\Lambda_{+}$ joining the other pair of opposite edges of  $R_{i_{0}}$.     As remarked above, the fact that each point of $\KK$ uniquely determines its symbol is due to the fact that our Markov rectangles are disjoint. Thus, it remains to prove the assertions about $I_{ \iota}^{+}$ and $I_{ \iota}^{-}$.

 If $I_{ \iota}^{+}$ is not as asserted, it must be a rectangle with nonempty interior.  Assume this and deduce a contradiction as follows.
 By the construction of $\MM$, the sides of $R_{i_{k}}$ in $ \Lambda_{+}$ extend to the sides in $ \Lambda_{+}$ of a rectangle $C_{i_{k}}\ss K$, the other two sides of which are arcs $ \delta_{k}, \gamma_{k}$ in positive junctures in $\bd K$.  Consider the set
 $$
 P_{k} = h^{k}(C_{i_{k}})\cap h^{k+1}(C_{i_{k+1}})\cap\cdots
 $$
 By our hypothesis, $P_{k}$ is a nondegenerate rectangle with two sides $ \delta'_{k}, \gamma'_{k}$ subarcs of $h^{k}( \delta_{k})$ and $h^{k}( \gamma_{k})$, respectively.  Thus $\delta'_{k},\gamma'_{k}$ are subarcs of  positive junctures in fundamental neighborhoods of positive ends.      Furthermore, 
 $$
 P_{0}\ss P_{1}\ss\cdots\ss P_{k}\ss\cdots
 $$ 
 and the edges in $ \Lambda_{+}$ of each $P_{k}$ are subarcs of the corresponding sides of $P_{k+1}$, $0\le k<\infty$.  The increasing union of these rectangles is an infinite strip $P$ bounded by \emph{distinct} leaves $ \lambda, \mu\in \Lambda_{+}$.  Any lift of this strip to the universal cover is a strip with distinct boundary components $\wt{ \lambda},\wt{ \mu}\in\wt{ \Lambda}_{+}$ covering $ \lambda, \mu\in \Lambda_{+}$ and limiting on two pairs $\{x,y\},\{z,w\}\ss\Si$. The notation is chosen so that the lifts $\wt\delta_{k}'$ lying in this strip have endpoints $x_{k},y_{k}$ on $\wt\lambda$ and $\wt\mu$, respectively, and $\{x_{k}\}_{k\ge0}$ converges to $x$ and $\{y_{k}\}_{k\ge0}$ to $y$.  If $x\ne y$, then $\{\wt\delta'_{k}\}_{k\ge0}$ converges to the geodesic in $\wt L$ with endpoints $x$ and $y$.  Consequently the sequence $\{\delta'_{k}\}_{k\ge0}$ accumulates locally uniformly on a geodesic in $L$. But by Theorem~\ref{esctoe},  this sequence escapes, hence $x=y$. Similarly,  $z=w$, hence $\wt{ \lambda}=\wt{ \mu}$.  This is the desired contradiction.
 
 The assertion about $I_{ \iota}^{-}$ is proven in the same way.  Finally, the fact that $\MM$ covers $\KK$ implies that the union of all $I_{ \iota}^{+}$'s also covers $\KK$, as does the union of all $I_{ \iota}^{-}$'s.  Thus, all arcs of $ \Lambda_{+}\cap R_{i_{0}}$ and all arcs of $ \Lambda_{-}\cap R_{i_{0}}$ are obtained as asserted, $i_{0}=1,2,\dots,q$.  
 \end{proof}

\begin{rem}
It is possible to use the pre-Markov family $\QQ^{+}=\{Q_{1},Q_{2},\dots Q_{p}\}$ of rectangles to produce projectively invariant measures for $h$ on $ \Lambda_{+}$ (and $ \Lambda_{-}$) much as in the case of pseudo-Anosov automorphisms of compact surfaces.  The following sketch follows the lead of~\cite[pages~95-102]{bca}  and we refer the reader there for more details.

Let  $B=(B_{k,j})$ be the incidence matrix, where  $B_{k,j}$ is  the number of components of $h(Q_{{k}})\cap Q_{{j}}$.  By the Brouwer fixed point theorem, this matrix 
has an eigenvector  $y\ne\mathbf{0}$ with all entries nonnegative and with eigenvalue  $\kappa\ge1$.  In the (typical) case that $\kappa>1$, one obtains a transverse, projectively invariant measure $\mu_{+}$ for $h$ on $ \Lambda_{+}$ with projective constant $\kappa$.

Now $\QQ^{+}$ is also pre-Markov for $h^{-1}$ with intersection matrix the transpose $B^{\T}$ and (left) eigenvector $y^{\T}$.  This gives a transverse, projectively invariant measure $\mu_{-}$ for $h^{-1}$ on  $ \Lambda_{-}$ with projective constant $\kappa$.  Viewed as a projectively invariant measure for $h$, it has projective constant $\kappa^{-1}<1$.

Since the eigenvectors $y$ and $y^{\T}$ may have some zero entries, these measures will not generally have full support.  See~\cite{fe:endp} for a simple example.  If, however, $ \Lambda_{+}$ is a minimal $h$-invariant lamination, the measures will evidently have full support.
\end{rem}


\section{Pseudo-geodesic laminations and the isotopy theorem}\label{uniq}

The theory developed so far is for geodesic laminations, but for applications to foliation theory this is much too restrictive.  In this section we give an axiomatic approach to endperiodic automorphisms using pseudo-geodesic laminations.

\subsection{The axioms}\label{axsys}

We fix a choice of standard hyperbolic metric $\g$ on the admissible surface $L$.

\begin{lemma}\label{samepseudos}
Any two standard hyperbolic metrics on $L$ have the same pseudo-geodesics.
\end{lemma}

\begin{proof}
Assume first that $\bd L=\0$ and let $g$ and $g'$ be two standard hyperbolic metrics on $L$.  Let $\pi_{0}:\Delta\to L$ and $\pi_{1}:\Delta\to L$ be the universal covers corresponding to these respective metrics.  It is well known that there is a commutative diagram
$$\begin{CD}
\Delta @>\phi>> \Delta\\
@V\pi_{0}VV     @VV\pi_{1}V\\
L      @>>\id>    L
\end{CD}$$
where $\phi$ is a homeomorphism completely determined by its value on a single point $x_{0}\in\Delta$.  Since the hyperbolic metrics are standard, as in the proof of~\cite[Theorem~2]{cc:epstein},  $\phi$ extends canonically to a homeomorphism $\wh\phi:D^{2}\to D^{2}$ which clearly takes geodesics to pseudo-geodesics.

If $\bd L\ne\0$, the proof is similar.  
\end{proof}

\begin{rem}
Since $L$ has a standard hyperbolic metric,  \cite[Corollary~4]{cc:epstein} implies that an arbitrary homeomorphism $\phi:L\to L$ takes pseudo-geodesics to pseudo-geodesics.
\end{rem}

Recall from Definition~\ref{geotightpg} that, if $\gamma$ is a pseudo-geodesic, then $\gamma^{\g}$ denotes its \emph{geodesic tightening}, the unique geodesic the lifts of which have the same ideal endpoints as the corresponding lifts of $\gamma$.  Note that  $\wt\gamma^{\g}\ss\wt L$ since $\bd L$, hence $\bd\wt L$, is geodesic.  If we allow the metric $\g$ to vary, we will refer to the $\g$-geodesic tightening.

Beginning with an endperiodic automorphism $f:L\to L$, we want to define an associated pair of pseudo-geodesic laminations satisfying four  axioms.  

The notations $\Lambda_{\pm}$ and $\JJ_{\pm}$, which have been used up to now for the geodesic Handel-Miller laminations and the geodesic juncture components, will from now on be used for the pseudo-geodesic laminations and juncture components given in the axioms.

\begin{axiom}\label{muttran} $(\Lambda_{+},\Lambda_{-})$ is a bilamination,   $\Lambda_{+}$ and $\Lambda_{-}$ being closed,   pseudo-geodesic laminations with all leaves   disjoint from $\bd L$. 
\end{axiom}

Hereafter, for the sake of economy, every assertion about ``the lamination $\Lambda_{\pm}$'' is really two assertions, one about $\Lambda_{+}$ and one about $\Lambda_{-}$.

\begin{axiom}\label{eachmeets}
A leaf of $\wt{\Lambda}_{\pm}$ can meet a  leaf of $\wt{\Lambda}_{\mp}$ in at most one  point.
\end{axiom}

Equivalently, the leaves of $\Lambda_{+}$ cannot intersect the leaves of $\Lambda_{-}$ so as to form digons.

\begin{defn}[endpoint correspondence property]\label{epcorrp}
The pseudo-geodesic lamination $\Lambda_{\pm}$ has the endpoint correspondence property with respect to $f$ if the correspondence $\lambda\mapsto\lambda^{\g}$ sends $\Lambda_{\pm}$ one-one onto the positive/negative geodesic Handel-Miller laminations associated to $f$  (Definition~\ref{HMgeobilamin}).
\end{defn}

When the context makes it clear, we sometimes shorten the language, saying that $\Lambda_{\pm}$ has the endpoint correspondence property.  We also say that the bilamination has the endpoint correspondence property (with respect to $f$).

The final two axioms will tie $(\Lambda_{+},\Lambda_{-})$ to the endperiodic automorphism $f$.

\begin{axiom}\label{ecorr}
The bilamination $(\Lambda_{+},\Lambda_{-})$ has the endpoint correspondence property with respect to $f$.
\end{axiom}

\begin{nota}
If $X$ is a set of pseudo-goedesics, we use the notation 
$$X^{\g} = \{\gamma^{\g}\ |\ \gamma\in X\}.$$ 

\end{nota}

\begin{rem}
By Axiom~\ref{ecorr}, $(\Lambda^{\g}_{+},\Lambda^{\g}_{-})$ are the Handel-Miller geodesic bilamination associated to $f$.

\end{rem}

The next two lemmas follow immediately from Axiom~\ref{ecorr}  since $\Lambda^{\g}_{\pm}$ has the same properties.

\begin{lemma}\label{everyleaf}

Every leaf of $\Lambda_{\pm}$ meets at least one leaf of $\Lambda_{\mp}$.

\end{lemma}

\begin{lemma}\label{uniquepseudo}
Each leaf of the lifted lamination $\wt{\Lambda}_{\pm}$ of $\wt L$ is  determined by its   endpoints in $\Si$. 
\end{lemma}

Fix a set of $f$-junctures  and corresponding set $\NN$  of $f$-juncture components (Definition~\ref{famNNg}). We define a set of pseudo-geodesic junctures and corresponding set $\JJ$ of juncture components in analogy to the sets of geodesic junctures and juncture components we introduced in Definition~\ref{junctdefn}. 

\begin{defn}[junctures,  $\JJ$, $\JJ_{+}$, $\JJ_{-}$, $\iota$]\label{jnctfam}

A set $\JJ$ of juncture components associated to a set $\NN$ of $f$-juncture components  is a set  of compact, properly embedded arcs and circles in $L$ such that
\begin{enumerate}
\item there is  a bijection $\iota:\NN\to\JJ$;
\item for each $\gamma\in\NN$, $\iota(\gamma)$ is homotopic to $\gamma$ with endpoints (if any) fixed;
\item if $\gamma\ne\gamma'$ are both in $\NN_{+}$ or both in $\NN_{-}$, then $\iota(\gamma)$ is disjoint from $\iota(\gamma')$.
\end{enumerate}
The set $\JJ_{+}=\iota(\NN_{+})$ is called the set of \emph{positive juncture components}. The set $\JJ_{-}=\iota(\NN_{-})$ is called the set of \emph{negative juncture components}.  Further the map $\iota$ extends in a natural way to the set of fixed $f$-junctures to define a fixed set of \emph{junctures} of the form $J = \iota(N)$ where $N$ is one of the fixed $f$-junctures.

\end{defn}

The map $\iota:\NN\to \JJ$  serves the role  of  the geodesic tightening map of Definition~\ref{gtm}  and can be extended to a map $\iota:\NN^{\dagger}\to \JJ$ where $\NN^{\dagger}$\label{gtm10NN} is the set of pseudo-geodesics $\gamma$ such that there exists an $f$-juncture component $\sigma\in\NN$ whose lifts have the same endpoints on $\Se$ as the lifts of $\gamma$. Then define $\iota(\gamma) = \iota(\sigma)$. This definition then extends in a natural way to a function $\iota$ with domain the finite unions of elements of $\NN^{\dagger}$. Then the last equality in the following definition makes sense. Compare this definition with Definition~\ref{Jndef}.

\begin{rem}
Notice that, by item~(3) in Definition~\ref{jnctfam}, the set  of junctures has the juncture intersection property.
\end{rem}

\begin{defn}[$J_{n}$]\label{Jndef10}
Given a juncture $J = \iota(N)$, let $J_{n} = \iota(f^{n}(N)) = \iota(f^{n}(J))$, $n\in\Z$.  
\end{defn}

In what follows, we assume a choice of $\NN$ and $\JJ$, $f$-juncture components and juncture components, and thus a choice of set of $f$-junctures and set of junctures.  Axiom~\ref{trnsvrs} will then claim the existence of choices with suitable properties.

\begin{axiom}\label{trnsvrs}
There is a choice of the families $\NN_{\pm}$ of positive/negative $f$-juncture components and $\JJ_{\pm}=\iota(\NN_{\pm})$ of positive/negative juncture components such that:
\begin{enumerate}
\item $\Lambda_{+}\cup\JJ_{-}$   and $\Lambda_{-}\cup\JJ_{+}$ are each sets of disjoint pseudo-geodesics;
\item $\Lambda_{+}$ is transverse to $\JJ_{+}$ and $\Lambda_{-}$ is transverse to $\JJ_{-}$;
\item no leaf of $\Lambda_{\pm}$ can meet an element of $\JJ_{\pm}$ so as to form digons.
\end{enumerate}
\end{axiom}

This completes the list of axioms.

\begin{defn}[Handel-Miller pseudo-geodesic bilamination associated to $f$]\label{HMbilam}
If $(\Lambda_{+},\Lambda_{-})$ satisfies the four axioms, it will be called a Handel-Miller pseudo-geodesic bilamination associated to $f$.  The individual  laminations $\Lambda_{\pm}$ will be called Handel-Miller pseudo-geodesic laminations associated to $f$.
\end{defn}

\begin{rem}
The image  of a Handel-Miller pseudo-geodesic bilamination associated to $f$ under an isotopy of $L$ will be a  Handel-Miller pseudo-geodesic bilamination associated to $f$. Thus,  a Handel-Miller pseudo-geodesic bilamination associated to $f$ is not unique but, by Theorem~\ref{uniqlams}, is unique up to isotopy.

\end{rem}

\begin{lemma}\label{totallydisc}
The Handel-Miller pseudo-geodesic  laminations associated to $f$  are transversely totally disconnected \emph{(Definition~\ref{trtotdis})}.
\end{lemma}

\begin{proof}
We will show that $\Lambda_{+}$ is transversely totally disconnected. The proof that $\Lambda_{-}$ is transversely  totally disconnected is similar. Suppose that $|\Lambda_{+}|$ has nonempty interior. Then $|\wt\Lambda_{+}|$ has nonempty interior. Let $x\in\lambda_{+}\in\wt\Lambda_{+}$ have a neighborhood contained in $|\wt\Lambda_{+}|$. Then there exists a leaf $\lambda_{-}\in\wt\Lambda_{-}$ and a point $y\in\lambda_{+}\cap\lambda_{-}$ such that $y$ has a neighborhood  $V\ss|\wt\Lambda_{+}|$ which is an open disk. This follows since one can find a bilamination chart of the form $P\times(-\epsilon,\epsilon)$ with $x\in P\times\{0\}$ and $P\ss\lambda_{+}$ as long as desired (see  Remark  page~\pageref{longchart}) and since, by Lemma~\ref{everyleaf}, every leaf in $\wt\Lambda_{+}$ meets some leaf in $\wt\Lambda_{-}$. Here, of necessity, $\epsilon>0$ will be small.

Thus, there exists an arc $[a,b]\ss\lambda_{-}\cap|\wt\Lambda_{+}|$ with $y\in(a,b)$. Let $\lambda_{a},\lambda_{b}\in\wt\Lambda_{+}$  be such that $\lambda_{a}\cap\lambda_{-} = \{a\}$ and $\lambda_{b}\cap\lambda_{-} = \{b\}$. Consider the geodesics $\lambda^{\g}_{a},\lambda^{\g}_{b},\lambda^{\g}_{-}$ in $\wt L$ with the same endpoints on $\Si$ as $\lambda_{a},\lambda_{b},\lambda_{-}$ respectively. By Axiom~\ref{ecorr}, $\lambda^{\g}_{a},\lambda^{\g}_{b}\in\wt\Lambda^{\g}_{+}$  and $\lambda^{\g}_{-}\in\wt\Lambda^{\g}_{-}$. Let $\{a^{\g}\} = \lambda^{\g}_{a}\cap\lambda^{\g}_{-}$, $\{b^{\g}\} = \lambda^{\g}_{b}\cap\lambda^{\g}_{-}$, and $[a^{\g},b^{\g}]\ss\lambda^{\g}_{-}$. Since the semi-proper leaves are dense in $\Lambda^{\g}_{+}$, there exists  a pair of semi-proper leaves  $\lambda^{\g}_{1},\lambda^{\g}_{2}\in\wt\Lambda^{\g}_{+}$  meeting the interval $[a^{\g},b^{\g}]$ at a pair of points $x^{\g}_{1},x^{\g}_{2}$ with  $(x^{\g}_{1},x^{\g}_{2})\cap|\wt\Lambda^{\g}_{+}| = \0$. Let $\lambda_{1},\lambda_{2}\in\wt\Lambda_{+}$ have the same endpoints on $\Si$ as $\lambda^{\g}_{1},\lambda^{\g}_{2}$ respectively (Axiom~\ref{ecorr}). By Axiom~\ref{ecorr}, if $x_{i} = \lambda_{i}\cap\lambda_{-}$, $i=1,2$, then   $(x_{1},x_{2})\cap|\wt\Lambda_{+}| = \0$ which is a contradiction. Thus, $\intr|\Lambda_{+}| = \0$. The lemma follows. 
\end{proof}

\begin{cor}\label{unsemiso}
The union of the semi-isolated leaves of $\Lambda_{\pm}$ is dense in $|\Lambda_{\pm}|$.
\end{cor}

\begin{lemma}

No  leaf  $\lambda\in\Lambda_{\pm}$  is contained  in a bounded region of $L$.

\end{lemma}

\begin{proof}
Let $X_{0}\ss X_{1}\ss\cdots\ss X_{i}\ss\cdots$ be an exhaustion of $L$ by compact subsets. Without loss we can assume $\fr X_{i}$ is a geodesic $1$-manifold. By Lemma~\ref{nbounded}, $\lambda^{\g}$ meets    $\fr X_{i}$, for  $i$ sufficiently large. Suppose $\alpha$ is a geodesic component of $\fr X_{i}$ with $\alpha\cap\lambda^{g}\ne\0$. Let $\wt\alpha$ and $\wt\lambda^{\g}$ be lifts of $\alpha$ and $\lambda^{\g}$ with $\wt\alpha\cap\wt\lambda^{\g}\ne\0$. If $\wt\lambda$ is lift of $\lambda$ sharing endpoints with $\wt\lambda^{\g}$, then $\wt\lambda\cap\wt\alpha\ne\0$. Thus, $\lambda\cap\fr X_{i}\ne\0$, for all $i$ sufficiently large, and the lemma follows.
\end{proof}

\begin{cor}\label{noleafbdd}

No neighborhood of any end of a  leaf of $\Lambda_{\pm}$ is contained in a bounded region of $L$.

\end{cor}

\begin{proof}
If a neighborhood of an end $\epsilon$ of  a leaf of $\Lambda_{\pm}$ is contained in a bounded region of $L$, then the asymptote of $\epsilon$ (see definition in the proof of Lemma~\ref{noncpt}) would contain a leaf of $\Lambda_{\pm}$ in a bounded region of $L$ which contradicts the lemma.
\end{proof}

In particular, the leaves are noncompact and one-one immersions of $\R$ in $\intr L$.  The endpoints of the lifted leaves are ideal. That is they are in  $E$.

\begin{theorem}\label{consistency}
The Handel-Miller geodesic bilamination $(\Lambda^{\g}_{+},\Lambda^{\g}_{-})$  associated to $f$ satisfies the four axioms, where   $\iota:\NN\to\JJ^{\g}$ is defined by $\iota(\gamma)=\gamma^{\g}$, $\forall \gamma\in\NN$.
\end{theorem}

\begin{proof}
\textbf{Axiom~\ref{muttran}.}  By Proposition~\ref{geobilam}, $(\Lambda^{\g}_{+},\Lambda^{\g}_{-})$ are bilaminations.  By construction they are closed. Clearly every geodesic is a pseudo-geodesic. If $x\in\bd L$, then $x$ can not lie in either $\Lambda^{\g}_{+}$ or $\Lambda^{\g}_{-}$  because $\{h^{n}(x)\}_{n\in\Z}$ escapes if $x\in\bd L$ but not if $x\in\Lambda^{g}_{\pm}$.

\textbf{Axiom~\ref{eachmeets}.} Since there cannot be geodesic digons in hyperbolic geometry, this axiom is immediate.

\textbf{Axiom~\ref{ecorr}.}  This axiom is tautologically true.

\textbf{Axiom~\ref{trnsvrs}.}  This follows from Proposition~\ref{geobilam} and the fact that, in hyperbolic surfaces, there are no geodesic digons.
\end{proof}

\subsection{Uniqueness}\label{uniq'}
An axiomatization needs to satisfy two conditions, consistency and completeness.  The consistency of our axioms is given by Theorem~\ref{consistency}.  For completeness, one needs to show that the system defined by the axioms is unique up to a reasonable equivalence relation.  In this section we will prove the following.

\begin{theorem}[Uniqueness Theorem]\label{uniqlams}
The Handel-Miller pseudo-geodesic bilamination  associated to an endperiodic automorphism $f$ is uniquely determined by $f$  up to ambient isotopy.
\end{theorem}

This is an immediate consequence of   Theorem~\ref{isotlams}.  The following is useful in applications.

\begin{cor}
The core dynamical system $h:\KK\to\KK$  is uniquely determined by the endperiodic automorphism $f$ up to topological conjugacy $\psi\o h|\KK\o\psi^{-1}$ by a homeomorphsim $\psi:L\to L$ isotopic to the identity.
\end{cor}

\begin{theorem}[Isotopy Theorem]\label{isotlams}
    If $(\Lambda_{+},\Lambda_{-})$ is  a   bilamination  satisfing the axioms,  then  there is a homeomorphism $\psi:L\to L$, isotopic to the identity by an isotopy fixing $\bd L$ pointwise,  such that $\psi(\lambda)=\lambda^{\g}$, for each $\lambda\in\Lambda_{\pm}$ and $(\psi(\Lambda_{+}),\psi(\Lambda_{-}))$ is the Handel-Miller geodesic bilamination. 
\end{theorem}

\begin{cor}\label{metricindep}
If $f:L\to L$ is an endperiodic automorphism, the  geodesic bilamination $(\Lambda^{\g}_{+},\Lambda^{\g}_{-})$ associated to $f$ is independent, up to ambient isotopy, of the choice of standard hyperbolic metric $\g$.
\end{cor}

\begin{proof}
Let $\g$ and $\g'$ be two standard hyperbolic metrics on $L$ and let $(\Lambda^{\g}_{+},\Lambda^{\g}_{-})$ be the Handel-Miller geodesic bilamination associated to $f$ and corresponding to $\g$, $(\Lambda^{\g'}_{+},\Lambda^{\g'}_{-})$ the one corresponding to $\g'$.  Relative to the metric $\g$, the lamination $\Lambda^{\g'}_{\pm}$ is pseudo-geodesic (Lemma~\ref{samepseudos}).  The pair $(\Lambda^{\g'}_{+},\Lambda^{\g'}_{-})$ forms a pseudo-geodesic bilamination which clearly satisfies Axioms~\ref{muttran},~\ref{eachmeets}, and~\ref{trnsvrs} as these are metric independent.  By Corollary~\ref{indep}, we can assume we are using the same choice of set $\NN$ of $f$-junctures in Axiom~\ref{trnsvrs} for both laminations $\Lambda^{\g}_{\pm}$ and $\Lambda^{\g'}_{\pm}$. 

It remains to prove Axiom~\ref{ecorr}. Assume that $\bd L = \0$. The proof in the case $\bd L \ne \0$ is similar. As in the proof of Lemma~\ref{samepseudos}, the identity map $\id:L\to L$ lifts to a map $\phi:\Delta\to \Delta$ uniquely determined by its value at one point which extends to a map $\wh\phi:\D^{2}\to\D^{2}$. Axiom~\ref{ecorr} follows immediately  since the endpoints on $\Si$ of the leaves of $\wh\NN$ and of $\wh\phi^{-1}(\wh\NN)$ coincide.
\end{proof}

\subsection{Proof of Theorem~\ref{isotlams}}\label{geodextpf10}

We will prove  Theorem~\ref{isotlams} in a sequence of steps, noting that by Corollary~\ref{unsemiso}, we only need to straighten the semi-isolated leaves. At each step we will modify the laminations $\Lambda_{\pm}$ by an isotopy. That is we will replace the laminations $\Lambda_{\pm}$ by $\phi(\Lambda_{\pm})$ where $\phi$ is a homeomorphsm isotopic to the identity.  We do not apply the homeomorphism $\phi$ to $\Lambda^{\g}_{\pm}$   which is to be thought of  simply as the target of the whole process.

\subsubsection{A preliminary isotopy to set up the tiling}

We use the  tiling $\mathfrak T^{\g}$ defined in Section~\ref{prelisotf}.   

\begin{lemma}\label{fg10}
Suppose the juncture $J_{e} = \fr V_{e}$ where $V_{e}$ is a closed neighborhood of the end $e\in\EE(L)$.  Then there exists a homeomorphism $\psi_{e}$, isotopic to the identity, such that $\psi_{e}(\tau) = \tau^{\g}$ for every juncture component $\tau\ss V_{e}$.  
\end{lemma}

\begin{proof}
Enumerate the juncture components in $V_{e}$ as $\{\tau_{i}\}_{i=1}^{\infty}$.  Let $\psi_{0} = \id$. Using   Theorem~\ref{2.1}, Theorem~\ref{3.1} and the   remark after Theorem~\ref{epsteinsmooth}, inductively find  sequences $\Phi_{i}$ of isotopies and $\psi_{i}$ of homeomorphisms isotopic to the identity with $\psi_{i} = \Phi^{1}_{i}\circ\psi_{i-1}$, $i\ge 1$,     fixing $\bd L$ and $\tau_{1}^{\g},\ldots,\tau_{i-1}^{\g}$ pointwise with $\Phi_{i}$ moving $\psi_{i-1}(\tau_{i})$ to its geodesic tightening.  That is, $\psi_{i}(\tau_{i}) = \Phi^{1}_{i}(\psi_{i-1}(\tau_{i})) = \tau_{i}^{\g}$.  Then, by the remark after Theroem~\ref{epsteinsmooth}, the supports of at most finitely many $\Phi_{i}$ meet any compact set so that $\psi_{i}\to\psi_{e}$, a well defined homeomorphism isotopic to the identity and $\psi_{e}(\tau_{i}) = \psi_{i}(\tau_{i})  = \tau_{i}^{\g}$, $i\ge 1$.
\end{proof}

\begin{cor}\label{imagejunct10}

Suppose the juncture $J_{e}= \fr V_{e}$ is chosen, each $e\in\EE(L)$, such that the $V_{e}$ are disjoint. Then there exists a homeomorphism $\psi$, isotopic to the identity, such that $\psi(\tau) = \tau^{\g}$ for every juncture component $\tau\ss \bigcup_{i=1}^{k}V_{e_{i}}$. 

\end{cor}

\begin{proof}
Take $\psi$ to be the composition of the $\psi_{e}$ each $e\in\EE(L)$.
\end{proof}

Note that the homeomorhism $\psi$ is such that $\psi(J_{e}) = J_{e}^{\g}$ and $\psi(V_{e}) = U_{e}$ where $J_{e}^{\g} = \fr U_{e}$, $e\in\EE(L)$, and the $U_{e}$ are disjoint. As usual $L = W^{-}\cup K\cup W^{+}$ where,
$$W^{-} = \bigcup_{e\in\EE_{-}(L)}U_{e}\quad{\rm and}\quad W^{+} = \bigcup_{e\in\EE_{+}(L)}U_{e}.$$

The set $\JJ^{\g}_{W}$   consists of all components of negative geodesic  junctures in $W^{-}$ and all components of positive geodesic junctures in $W^{+}$ (Definition~\ref{jw}). Cutting $L$ apart along the geodesic juncture components in $\JJ^{\g}_{W}$ decomposes $L$ into a set $\mathfrak T^{\g}$ of ``tiles'' (Definition~\ref{defntile}).  

\begin{lemma}

The laminations $\psi(\Lambda_{\pm})$ are transverse to the geodesic junctures in $\JJ_{W}^{\g}$.

\end{lemma}

\begin{proof}
This follows immediately from the fact that the laminations $\Lambda_{\pm}$ were transverse to the junctures in $\JJ_{\pm}$.
\end{proof}

We replace the laminations $\Lambda_{\pm}$ by the laminations $\psi(\Lambda_{\pm})$. This is our first step in straightening the laminations $\Lambda_{\pm}$ to the laminations $\Lambda_{\pm}^{\g}$.

\subsubsection{Preliminary isotopies in  the core and the laminations $\GG^{\g}_{\pm}$}\label{lamofcore} 

As above,  $\fr_{+}K=\bigcup_{e\in\EE_{+}(L)}J^{\g}_{e}$ and $\fr_{-}K=\bigcup_{e\in\EE_{-}(L)}J^{\g}_{e}$.

\begin{nota}
Denote by  
\begin{eqnarray*}
\GG_{\pm}&=&\Lambda_{\pm}|K\\
\GG^{\g}_{\pm}&=&\Lambda^{\g}_{\pm}|K, 
\end{eqnarray*}
the laminations induced  on $K$ by $\Lambda_{\pm}$ and $\Lambda^{\g}_{\pm}$.

\end{nota}

A leaf $\alpha = [x_1.x_2]\in\GG_{\pm}$ is contained in a leaf $\gamma_{\alpha}\in\Lambda_{\pm}$ and  has endpoints $x_i$ in geodesic juncture components $\sigma_i\ss\fr_{\pm} K$, $i=1,2$.  Let $\wt\alpha = [\wt x_1,\wt x_2]\ss\wt{\gamma_{\alpha}}$ be  lifts of $\alpha = [ x_1,x_2]\ss\gamma_{\alpha}$ with $\wt x_i$ in the lift $\wt\sigma_i$ of $\sigma_i$, $i=1,2$. The  geodesic tightening $\gamma_{\alpha}^{\g}$ of $\gamma_{\alpha}$ has lift $\wt{\gamma_{\alpha}^{\g}}$ sharing endpoints on $\Si$ with $\wt{\gamma_{\alpha}}$. There is a unique arc $\wt{\alpha^{\g}} = [\wt{x_1^{\g}},\wt{x_2^{\g}}]\ss\wt{\gamma_{\alpha}^{\g}}$ with $\wt{x_i^{\g}}\in\wt\sigma_i$, $ i=1,2$.

\begin{nota}
The covering projection of the arc $\wt{\alpha^{\g}}$ is in $\GG^{\g}_{\pm}$ and will be denoted $\alpha^{\g} = [x_1^{\g},x_2^{\g}]$.  The correspondence $\alpha\leftrightarrow\alpha^{\g}$ induces a one-one correspondence between $\GG_{\pm}$ and $\GG^{\g}_{\pm}$. 
\end{nota}

Recall from Lemma~\ref{extremals} that each isotopy class of leaves of $\GG^{\g}_{\pm}$ contains  two extremal leaves (in the degenerate case where the isotopy class has one leaf we consider that leaf to be an extremal leaf).  For reasons that will become obvious, we will refer to a leaf  $\alpha\in\GG_{\pm}$ as an extremal leaf  if the corresponding geodesic segment    $\alpha^{\g}\in\GG^{\g}_{\pm}$ is an extremal leaf of $\GG^{\g}_{\pm}$.

By Lemma~\ref{LambdaK}, the leaves of $\GG^{\g}_{\pm}$ fall into finitely many isotopy classes. Thus, $\GG_{\pm}$ has finitely many extremal leaves.

\begin{prop}\label{strextleaves}
There is a preliminary isotopy, supported arbitrarily near $K$ and leaving that surface invariant, that straightens each extremal leaf $\alpha\in\GG_{\pm}$  to its  corresponding leaf  $\alpha^{\g}\in\GG^{\g}_{\pm}$. 
\end{prop}

\begin{proof}
Each of the finitely many extremal leaves in $\GG_+$ has its endpoints in two (not necessarily distinct) geodesic juncture components in $\fr_+ K$. An application of Lemma~\ref{arctight} to each of these in turn will straighten each of them in turn as required.

Each of the finitely many extremal leaves in $\GG_-$ has its endpoints in two (not necessarily distinct) geodesic juncture components in $\fr_- K$ and will possibly be divided into finitely many segments by the geodesics containing the already straightened extremal leaves in $\GG_+$. Finitely many applications of Lemma~\ref{arctight} to each of these segments in turn will staighten each of the extremal leaves in $\GG_-$ in turn without moving the already straightened extremal leaves of $\GG_+$.
\end{proof}

Thus, it can be assumed that each of  the extremal leaves $\alpha\in\GG_{\pm}$ is identical with its corresponding leaf $\alpha^{\g}\in\GG_{\pm}^{\g}$. They each lie in the geodesic $\gamma^{\g}_{\alpha}$, the lifts of which have the same endpoints as the lifts of the leaf $\gamma_{\alpha}$ of $\Lambda_{\pm}$ containing $\alpha$. Note that at this stage it is not yet true that $\gamma_{\alpha} = \gamma^{\g}_{\alpha}$, only that $\alpha = \alpha^{\g}$.

Looking in the universal cover we see that,

\begin{cor}

Each leaf of $\GG_{+}$ \upn{(}respectively $\GG_{-}$\upn{)} lies in a rectangle $R$ with a pair of opposite edges extremal leaves of $\GG_{+}$ \upn{(}respectively $\GG_{-}$\upn{)}  and a pair of opposite edges contained in $\fr_{+}K$ \upn{(}respectively $\fr_{-}K$\upn{)} .  

\end{cor}

\begin{nota}
The set of rectangles $R\ss K$ bounded by extremal leaves of  $\GG_{+}$ will be denoted by $\RR_{+}$.  The analogous set of rectangles bounded by extremal leaves of  $\GG_{-}$ will be denoted by $\RR_{-}$. Let $\RR = \RR_{+}\cup\RR_{-}$.
\end{nota}

\begin{lemma}
If $R_{+}\in\RR_{+}$   and $R_{-}\in\RR_{-}$, then the components of $R_{+}\cap R_{-}$ are rectangles with one pair of opposite sides in opposite sides of $R_{+}$ and the other pair in opposite sides of $R_{-}$.
\end{lemma}

\begin{proof}
This follows from Axiom~\ref{eachmeets} which implies that a side of $R_{+}$ and a side of $R_{-}$ cannot intersect so as to form one or more digons.
\end{proof}

If $\alpha$ is a leaf of $\GG_{+}$ (respectively of $\GG_{-}$) lying in  $R\in\RR$, then $\alpha$ connects two positive (respectively negative) juncture components lying in $\fr_{+}K$ (respectively $\fr_{-}K$). 
Let $R\in\RR$.  The geodesic arcs $R\cap\fr K$, taken in either order, will be called the bottom and top edges of $R$ and the extremal geodesic arcs will be called the left and right edges.  Fix these choices.

Notice that $|\RR|$ has both concave and convex corners.   Of course a given $R_{-}\in\RR_{-}$ can cross an $R_{+}\in\RR_{+}$ more than once.  Since the extremals have already been straightened, both $\GG_{\pm}$ and $\GG^{\g}_{\pm}$ can be viewed as laminations of $|\RR|$.  If some of the rectangles $R$ degenerate to arcs, this language is a bit unorthodox, but harmlessly so.

\begin{nota}
Denote by $\bd_{*}\RR$   the closure of $\fr\RR\sm\fr K$.
\end{nota}

In other words, $\bd_{*}\RR$ is the part of $\bd\RR$ lying in $|\Lambda_{+}|\cup|\Lambda_{-}|$.

The leaves of $\GG_{+}$ are properly embedded arcs  having   both endpoints in $\fr_{+}K$.  All isotopies of these leaves are to  fix $\bd_{L}K\cup\fr_{-}K$ pointwise where  $\bd_{L}K = K\cap\bd L$. As already indicated, the isotopies leave $\fr_{+}K$ componentwise invariant, but not pointwise. Similar remarks  apply to $\GG_{-}$.

The leaves of $\GG_{\pm}$ fall into finitely many  isotopy classes.    There are extremal leaves in each isotopy class which, together with arcs in $\fr_{\pm}K$ cut off rectangles containing all the leaves in that isotopy class.  If the isotopy class has only one leaf, we consider that leaf as a degenerate rectangle.

\subsubsection{Straightening all leaves of $\GG_{\pm}$}\label{straightleaves}

This is the delicate part of the argument.

By Corollary~\ref{unsemiso}, we only need to straighten the semi-isolated leaves  of $\Lambda_{\pm}$. In this subsection, we will  construct a sequence $\{\Phi_{i}\}_{i=1}^{\infty}$ of  ambient isotopies which progressively ``locally straighten'' each leaf  $\alpha\in\GG_{\pm}$ which is a segment of a semi-isolated leaf $\lambda\in \Lambda_{+}\cup\Lambda_{-}$  to the corresponding geodesic segment $\alpha^{\g}\in\GG^{\g}_{\pm}$, taking care that once a leaf segment has been straightened by $\Phi_{i}$, it remains invariant under $\Phi_{j}$, $j>i$.   

We will need to prove that the infinite composition $\psi=\cdots\o\Phi^{1}_{i}\o\cdots\o\Phi^{1}_{2}\o\Phi^{1}_{1}$ is a well defined homeomorphism, isotopic to the identity.  This is a delicate point since we cannot ensure that the supports of these isotopies form a locally finite family of sets.  But then, since the semi-isolated leaves have been straightened, the homeomorphism $\psi$ performs as advertised.

Throughout this process the homeomorphism $\psi_{i}=\Phi^{1}_{i}\o\cdots\o\Phi^{1}_{2}\o\Phi^{1}_{1}$  distorts the laminations $\Lambda_{\pm}$ to laminations $\Lambda^{i}_{\pm}$.

The goal of this section is to  construct  an isotopy straightening all the   leaves of $\GG_{\pm}$.  Note that $|\GG_{\pm}|\ss|\RR|$.  Recall that there is a canonical correspondence between the leaves $\alpha$ of $\GG_{\pm}$ and leaves $\alpha^{\g}$ of $\GG^{\g}_{\pm}$,  We will progressively ``tighten''  a countable dense set of leaves $\alpha\in\GG_{\pm}$  one at a time to their corresponding geodesic leaves $\alpha^{\g}\in\GG^{\g}_{\pm}$  by ambient isotopies compactly supported in a neighborhood of $K$.   Already tightened leaves of $\GG_{\pm}$ and components of $\bd_{L}K\cup\fr_{\mp}K$ are pointwise fixed under the ambient isotopy  where  $\bd_{L}K = K\cap\bd L$.  The components of $\fr_{\pm}K$ and already tightened leaves of $\GG_{\mp}$ will be invariant but generally not fixed pointwise by the isotopy.

\begin{lemma}\label{compfolns}
There are mutually transverse $\CO$ foliations $\FF_{\pm}$ of $|\RR|$ such that $\GG_{+}\ss\FF_{+}$ and $\GG_{-}\ss\FF_{-}$.
\end{lemma}

\begin{proof}
By Axiom~\ref{muttran}, $(\Lambda_{+},\Lambda_{-})$ is a bilamination. By Lemma~\ref{Rchart},  applied to this bilamination,  each of the rectangles making up $|\RR_{+}|\cap|\RR_{-}|$ can be bifoliated as required. The rest of $|\RR|$ consists of rectangles that can be bifoliated as laminated charts for either $\Lambda_{+}$ or $\Lambda_{-}$.  The usual process of gluing foliations along transverse and tangential boundary components allows the foliations in these various rectangles to be matched up.
\end{proof}

\begin{lemma}\label{FF'}
There are mutually transverse $\CO$ foliations $\FF^{\g}_{\pm}$ of $|\RR|$ by geodesic arcs such that $\GG^{\g}_{+}\ss\FF^{\g}_{+}$ and $\GG^{\g}_{-}\ss\FF^{\g}_{-}$.
\end{lemma}

\begin{proof}
In each rectangle $R$ making up  $|\RR_{+}|$, the geodesic lamination induced by $\GG_{+}^{\g}$ extends to a geodesic foliation $\FF_{+}^{\g}$ by the method of proof of Proposition~\ref{geobilam}.  Similarly, we extend $\GG_{-}^{\g}$ to a geodesic foliation $\FF_{-}^{\g}$ of $|\RR_{-}|$.  On the intersection $|\RR_{+}|\cap|\RR_{-}|$ these foliations, being geodesic, are necessarily transverse.  $\FF_{\pm}^{\g}$ is easily extended to the remaining subrectangles of $|\RR_{\mp}|$ to be transverse to $\FF_{\mp}^{\g}$ there.
\end{proof}

\begin{lemma}

There exists  a natural map $\mu:(\FF_{+},\FF_{-}) \ra (\FF_{+}^{\g},\FF_{-}^{\g})$. The map $\mu$ takes leaves  $\alpha\in\GG_{\pm}$ to their corresponding geodesic leaves  $\alpha^{\g}\in\GG^{\g}_{\pm}$ and gaps of $\GG_{\pm}$ to the corresponding gaps of $\GG^{\g}_{\pm}$.  
\end{lemma}

\begin{proof}
We define the map $\mu$ on $\FF_{+}$. The definition of $\mu$ on $\FF_{-}$ is analogous. One first defines $\mu(\ell) = \ell^{\g}$ for $\ell\in\GG_{+}$ and then extends $\mu$ over the gaps of $\GG_{+}$. There are two  types of gaps to consider. First let $V$ be a gap of $\GG_{+}$  lying in a   rectangle $R\in\RR_{+}$, $\sigma$ the bottom edge of $V$ and $\tau$ the bottom edge of the corresponding gap $W$ of $\GG^{\g}_{+}$. There is a  map $\nu:\sigma\ra\tau$, linear in the hyperbolic metric. If $\lambda\ss V$ is a leaf of $\FF_{+}$ and $\lambda\cap\sigma = \{p\}$, let $\mu(\lambda)\ss W$ be the leaf of $\FF_{+}^{\g}$ with $\mu(\lambda)\cap\tau = \{\nu(p)\}$. Next,  let $V$ be a gap of $\GG_{+}$  lying entirely in a   rectangle $R\in\RR_{-}$, crossing it from left edge to right edge, $\sigma$ the left edge of $V$, and $\tau$ the left edge of the corresponding gap $W$ of $\GG^{\g}_{+}$. As in the first case, there is a  map $\nu:\sigma\ra\tau$, linear in the hyperbolic metric. If $\lambda\ss V$ is a leaf of $\FF_{+}$ and $\lambda\cap\sigma = \{p\}$, let $\mu(\lambda)\ss W$ be the leaf of $\FF_{+}^{\g}$ with $\mu(\lambda)\cap\tau = \{\nu(p)\}$. 
\end{proof}

For $\lambda\in\FF_{\pm}$, we will denote $\mu(\lambda)$ by $\lambda^{\g}\in\FF_{\pm}^{\g}$. The following lemma is clear.

\begin{lemma}

The map $\mu$ induces a homeomorphism we will denote by $$|\mu|:|\RR|\ra|\RR|.$$

\end{lemma}

\begin{lemma}

There exist  countable  sublaminations $\HH_{\pm}\ss\FF_{\pm}$ and  $\HH_{\pm}^{\g}\ss\FF_{\pm}^{\g}$ that correspond under the map $\mu$ and such that $|\HH_{\pm}|$ and $|\HH^{\g}_{\pm}|$ are each dense in $|\RR|$. The set $\HH_{\pm}$  contains all the semi-isolated leaves of $\GG_{\pm}$. The rest of the leaves of $\HH_{\pm}$ lie in gaps of $\GG_{\pm}$. .

\end{lemma}

\begin{proof}
We first construct $\HH_{+}$. The construction of $\HH_{-}$ is similar. Start by putting the semi-isolated leaves of $\GG_{+}$ into $\HH_{+}$. Then for each gap $V$ of $\GG_{+}$, choose a countable  subset of $\FF_{+}$ with union dense in $V$. Since there are countably many gaps, each bounded by two semi-isolated leaves, the set $\HH_{+}$ is countable. Define the sets $\HH_{\pm}^{\g}= \{\mu(\lambda)\mid\lambda\in\HH_{\pm}\}$.
\end{proof}

Enumerate the elements of the set $\HH_{+}\cup\HH_{-}$ as $\ell_{1},\ell_{2},\ldots$. Note that we are not distinguishing whether $\ell_{n}$ lies in $\HH_{+}$ or $\HH_{-}$. Working in the  rectangle $R\in\RR$ which contains $\ell_{1}$, we produce an isotopy $\Phi_{1}$, compactly supported in $R$ and leaving $\bd_{*}\RR$ fixed pointwise, which carries $\ell_{1}$ to $\ell_{1}^{\g}$.  As usual, this is a sliding isotopy along the bottom and top of $R$ and then an application of Theorem~\ref{3.1}.  Let $\psi_{0}=\id$ and $\psi_{1} = \Phi^{1}_{1}$. 

\begin{prop}\label{key}
For $n=1,2,\ldots$, there exist  isotopies $\Phi_{n}$ and homeomorphisms $\psi_{n}$ defined on $|\RR|$, fixing $\bd_{*}\RR$ pointwise, such that,
\begin{enumerate}
\item $\psi_{n} = \Phi^{1}_{n}\circ\psi_{n-1}$. 
\item $\psi_{n-1}$ takes $\ell_{k}$ to $\ell_{k}^{\g}$, $1\le k\le n-1$.
\item $\Phi_{n}$ straightens $\psi_{n-1}(\ell_{n})$ to $\ell_{n}^{\g}$ while leaving the already straightened $\ell_{k}^{\g}$ invariant, $1\le k\le n-1$.
\end{enumerate}
\end{prop}

\begin{proof}
We proceed inductively. The assertions of the proposition for $n=1$ are satisfied tautologically or vacuously. Inductively, suppose $\Phi_{n}$, $\psi_{n}$ and $\psi_{n-1}$ have been defined satisfying the conditions of the proposition, for some $n\ge1$. The next leaf to be straightened is $\ell_{n+1}$. To simplify the discussion, assume $\ell_{n+1}\in\GG_{+}$. Let $R\in\RR_{+}$ be the   rectangle containing $\ell_{n+1}$ and let $\beta_{1},\ldots,\beta_{r}$ be the other already straightened leaves of   $\GG_{+}$ in $R$. Let $X$ be the rectangular component of $R\sm\bigcup_{i=1}^{r}\beta_{i}$ that contains $\psi_{n}(\ell_{n+1})$.  Then $\ell^{\g}_{n+1}\ss X$ and the support of $\Phi_{n+1}$ will be $X$. Let $\alpha_{1},\ldots,\alpha_{q}$ be the arcs crossing $X$ in which the already straightened leaves of $\GG_{-}$ meet $X$. Note that the endpoints of $\psi_{n}(\ell_{n+1})$ lie in two geodesic segments $X\cap\fr_{+}K$ which we denote $\alpha_{0},\alpha_{q+1}$. The segments $\alpha_{1},\ldots,\alpha_{q}$  divide $X$ into subrectangles $X_{i}$ with two opposite edges $\alpha_{i},\alpha_{i+1}$, $0\le i\le q$.
 
The isotopy $\Phi_{n+1}$ is defined as a composition of the following very straightforward isotopies. One first  defines isotopies supported in  small neighborhoods of the $\alpha_{i}$, $0\le i\le q+1$, that slides the point of intersection of  $\psi_{n}(\ell_{n+1})$ with $\alpha_{i}$ along the arc $\alpha_{i}$ until it coincides with the point of intersection of $\ell^{\g}_{n+1}$ with $\alpha_{i}$. Then Theorem~\ref{3.1} is used to define isotopies supported on the $X_{i}$, $0\le i\le q$, to move the image of the arc $\psi_{n}(\ell_{n+1})$ under the first set of sliding isotopies to the arc $\ell^{\g}_{n+1}$.   
 
 Composing these finitely many isotopies  defines an isotopy $\Phi_{n+1}$, straightening $\psi_{n}(\ell_{n+1})$ to $\ell_{n+1}^{\g}$ while leaving the already straightened leaves $\ell_{k}^{\g}$ invariant, $1\le k\le n$.  Set $\psi_{n+1}=\Phi^{1}_{n+1}\o\psi_{n}$. 
\end{proof}

\begin{rem}
We emphasize that, if $\ell_{n}\in\GG_{\pm}$, then $\ell_{n}^{\g}\in\GG_{\pm}^{\g}$ is the corresponding  geodesic tightening of $\ell_{n}$.  Also, $\psi_{n}(\HH_{\pm})$ is transverse to  $\psi_{n}(\HH_{\mp})$. By abuse of notation, we denote this lamination again by $\HH_{\pm}$.  
\end{rem}

Let  $\WW_{n}=\{\ell_{1}^{\g},\ldots,\ell_{n}^{\g}\}$. Then  $|\RR|\sm|\WW_{n}|$ consists of finitely many rectangles. Let $\rho:L\x L\to[0,\infty)$ denote the (complete) hyperbolic metric.

\begin{defn}[mesh]\label{gridmesh}

The quantity $\text{mesh}(\WW_{n})$ is the diameter of the largest rectangular component of   $|\RR|\sm|\WW_{n}|$, measured in the hyperbolic metric $\rho$.

\end{defn}

Let $\delta_{n}$ denote  $\text{mesh}(\WW_{n})$. Then, by compactness and the fact that   $|\HH_{\pm}^{\g}|$  is dense in $\RR$ we have,

\begin{lemma}\label{goestozero}

$\delta_{1}\ge\delta_{2}\ge\cdots\ge\delta_{n}\ra 0$.

\end{lemma}

Let $\{\psi_{n}\}_{n=0}^{\infty}$ and $\{\Phi_{n}\}_{n=0}^{\infty}$ be the sequences of homeomorphisms and isotopies as in Proposition~\ref{key}.  On the space $\CC$ of continuous functions $s:\RR\to \RR$ we define a metric $d$ by setting 
$$d(s,r)=\max_{x\in \RR}\rho(s(x),r(x)).$$
As is well known, the topology associated to this metric is the compact-open topology.  The group $H\ss\CC$ of homeomorphisms $h:|\RR|\to |\RR|$ is a topological group under the metric $d$.

\begin{lemma}\label{cauchy}
Relative to the metric $d$, the sequence $\{\Phi^{1}_{n+k}\o\Phi^{1}_{n+k-1}\o\cdots\o\Phi^{1}_{n}\}_{n=0}^{\infty}$ converges to the identity uniformly for $k\ge0$ and the sequence $\{\psi_{n}\}_{n=0}^{\infty}$ is Cauchy.
\end{lemma}

\begin{proof}
By (3) of Proposition~\ref{key},  $\Phi^{1}_{n+k}\o\Phi^{1}_{n+k-1}\o\cdots\o\Phi^{1}_{n}$  sends the rectangles of $|\RR|\sm|\WW_{n-1}|$ into themselves.   Therefore, by Lemma~\ref{goestozero}, 
$$\rho(\Phi^{1}_{n+k}\o\Phi^{1}_{n+k-1}\o\cdots\o\Phi^{1}_{n}(x),x)\le\delta_{n-1}\ra 0,$$
 uniformly in $k\ge0$ and $x\in |\RR|$ as $n\to\infty$.      Since $$\psi_{n+k}=\Phi^{1}_{n+k}\o\Phi^{1}_{n+k-1}\o\cdots\o\Phi^{1}_{n}\o\psi_{n-1},$$ it follows immediately that the sequence $\{\psi_{n}\}_{n=0}^{\infty}$ is Cauchy.
\end{proof}

By a standard argument, we get the following.

\begin{cor}
The pointwise limit $\psi=\lim_{n\to\infty}\psi_{n}$ exists and equals the homeomorphism $|\mu|$.
\end{cor}

Recall that the group $H$ of homeomorphisms of $|\RR|$ is a topological group under the metric $d$. The homeomorphism  $\psi$ will be isotopic  to the identity if there  is a continuous path $s$ in $H$ starting at $\id$ and ending at $\psi$.  We  construct such a  path.  It will be convenient to parametrize it on the one point compactification $[1,\infty]$ of $[1,\infty)$. The isotopies $\Phi_{n}^{t}$, $0\le t\le1$, send the rectangles of $|\RR|\sm|\WW_{n-1}|$ into themselves and  $\Phi_{n}^{0}=\id$. 
Recall   that $\psi_{0}=\id$.
Define $s:[1,\infty]\to H$ by 
$$
s(t)=\begin{cases}
\Phi_{n}^{t-n}\o\psi_{n-1}, &n\le t\le n+1,\,\,1\le n<\infty\\
\psi, &t=\infty.
\end{cases}
$$

\begin{lemma}
The path $s:[1,\infty]\to H$ is continuous, hence $\psi$ is isotopic to the identity.
\end{lemma}

\begin{proof}
Continuity of $s(t)$ for $1\le t<\infty$ is clear. 
Since  $\Phi_{n+k}^{t}\o\Phi^{1}_{n-k-1}\o\cdots\o\Phi^{1}_{n}$ sends the rectangles of $|\RR|\sm|\WW_{n-1}|$ into themselves, $0\le t\le1$ and $k\ge0$, we see by Lemma~\ref{goestozero} that $d(s(t_{1}),s(t_{2}))$ is uniformly as small as desired, for all $t_{1},t_{2}\in[n,\infty)$ and $n$ sufficiently large.  Since $\lim_{n\to\infty}s(n)=\psi$ in the metric $d$, it follows that $\lim_{t\to\infty}s(t)=\psi$ in that metric, proving continuity at $t=\infty$.
\end{proof}

Since the homeomorphism $\psi$ isotopic to the identity is pointwise fixed on $\bd_{*}\RR$, it can be extended to a homeomorphism isotopic to the identity, again denoted by $\psi$, which is the identity on $K\sm\RR$.  We then have,

\begin{prop}\label{straightensall}

The homeomorphism $\psi$ of $K$,  isotopic to the identity, sends each leaf of $\ell\in\GG_{\pm}$ to its corrponding geodesic tightening in $\ell^{\g}\in\GG^{\g}_{\pm}$.

\end{prop}

\subsubsection{Extending the isotopy to $L$}\label{sect96}
Since $\psi(\fr_{\pm}K) = \fr_{\pm}K$ and  $\psi$ is isotopic to the identity there, it is easy to extend $\psi$ into a small collar of $\fr_{\pm}K$ on the side facing away from $K$, damping the extension off to the identity in this collar.  One then extends by the identity over the rest of $L$, obtaining a homeomorphism $\psi_{K}:L\to L$ isotopic to the identity by an isotopy compactly supported as near to $K$ as desired.  Since $\psi_{K}|K=\psi$, Proposition~\ref{straightensall} implies that $\psi_{K}(\gamma_{\alpha})$ and $\gamma^{\g}_{\alpha}$ cross $\fr_{\pm}K$ in exactly the same points.  

Enumerate the tiles of $\mathfrak T^{\g}$ in any convenient way as $P_{0},P_{1},P_{2},\dots,P_{n},\dots$.  We can take $P_{0}=K$ and assume that, for each $n\ge1$, $P_{n}$ has at least one edge in common with at least one tile $P_{i}$, $0\le i<n$.

Next proceed to $P_{1}$ and produce the homeomorphism $\psi_{P_{1}}$, isotopic to the identity, in the same way as above, noting that the isotopy will fix pointwise the interface of $P_{1}$ with $K$.  Remark that the procedure is easier here since,  if $P_{1}$ lies in a neighborhood of   a positive end, there will only be  rectangles $R\in\RR_{+}$ in $P_{1}$, and  if $P_{1}$ lies in a neighborhood of   a negative end, there will only be  rectangles $R\in\RR_{-}$ in $P_{1}$.  One obtains a   compactly supported  homeomorphism   $\psi_{P_{1}}:L\to L$ isotopic to the identity by an isotopy which is the identity on $K$ and outside a neighborhood of $P_{1}$ which is only slightly larger than $P_{1}$.  Proceed in this way with each $P_{i}$ is turn.The supports of the homeomorphisms $\psi_{K},\psi_{P_{1}},\dots,\psi_{P{n}},\dots$ and associated isotopies form a locally finite family of compact sets.
Thus, the infinite composition 
$$
\psi_{L}=\cdots\o\psi_{P_{n}}\o\psi_{P_{n-1}}\o\cdots\psi_{P_{1}}\o\psi_{K}
$$
is a well defined homeomorphism on $L$, isotopic to the identity, which carries $\Lambda_{\pm}$ to $\Lambda^{\g}_{\pm}$.  The proof of Theorem~\ref{isotlams} is complete.

\subsection{Good choice of junctures}

In Section~\ref{axsys} we axiomatize a system consisting of five elements $(f,\NN,\JJ,\Lambda_{+},\Lambda_{-})$ where $f:L\to L$ is an endperiodic automorphism, $\NN$ is a set of $f$-juncture components chosen as in Definition~\ref{famNNg}, $\JJ$ is a set of juncture components associated to $\NN$ (Definition~\ref{jnctfam}), and $\Lambda_{\pm}$ are pseudo-geodesic laminations.

The Isotopy Theorem (Theorem~\ref{isotlams}) provides a homeomorphism $\psi$ of $L$, isotopic to the identity, so that $\psi(\Lambda_{\pm}) = \Lambda^{\g}_{\pm}$.

\begin{theorem}\label{boldmove}
If the system $(f,\NN,\JJ,\Lambda_{+},\Lambda_{-})$ satisfies the axioms, the entire  theory for the Handel-Miller geodesic laminations developed in \emph{Sections~\ref{constr}~-~\ref{cordyn}}   carries over verbatim to the system 
$(f,\NN,\JJ' = \psi^{-1}(\JJ^{\g}),\Lambda_{+},\Lambda_{-})$.
\end{theorem}

Denote by $\iota^{\g}$ the geodesic tightening map of Definition~\ref{gtm}. 

\begin{defn}[$\iota'$]\label{gtm10}

Define the map  $\iota':\NN\to \JJ$   by $\iota' = \psi^{-1}\circ\iota^{\g}\circ\psi$.  

\end{defn}

The map $\iota':\NN\to \JJ$  serves the role  of the map $\iota:\NN\to \JJ$ of Definition~\ref{jnctfam}  for the sytem $(f,\NN,\JJ' = \psi^{-1}(\JJ^{\g}),\Lambda_{+},\Lambda_{-})$ and as such replaces the geodesic tightening map   of Definition~\ref{gtm}. It can be extended to a map $\iota':\NN^{\dagger}\to \JJ$ where $\NN^{\dagger}$ is the set of pseudo-geodesics $\gamma$ such that there exists an $f$-juncture component $\sigma\in\NN$ whose lifts have the same endpoints on $\Se$ as the lifts of $\gamma$. Then define $\iota'(\gamma) = \iota(\sigma)$. This definition then extends in a natural way to a function $\iota'$ with domain the finite unions of elements of $\NN^{\dagger}$.

\begin{proof}[Proof of  Theorem~\ref{boldmove}]
First, consider the endperiodic automorphism $f^{*} = \psi\circ f\circ\psi^{-1}:L\to L$ with set of $f^{*}$-juncture components $\NN^{*} = \{\psi(\gamma)\ |\ \gamma\in\NN\}$. Since $\gamma\in\NN$ and $\psi(\gamma)\in\NN^{*}$ have the same endpoints on $\Se$, they have the same geodesic tightenings. Thus, $\JJ^{\g}$ is both the set of geodesic tightenings of the set $\NN$ of $f$-juncture components and the set $\NN^{*}$ of $f^{*}$-juncture components.    Thus, the entire  theory for the Handel-Miller geodesic laminations developed in Sections~\ref{constr}~-~\ref{cordyn} is true for the system $(f^{*},\NN^{*},\JJ^{\g},\Lambda^{\g}_{+},\Lambda^{\g}_{-})$.  

Further,  $\psi^{-1}$ carries the  bilamination $(\Lambda_{+}^{\g},\Lambda_{-}^{\g}) = (\psi(\Lambda_{+}),\psi(\Lambda_{-}))$   to the bilamination $(\Lambda_{+},\Lambda_{-})$, the set of geodesic junctures $\JJ^{\g}$ to the set $\JJ' = \psi^{-1}(\JJ^{\g})$ of junctures, and the set $\NN^{*}$ of $f^{*}$-junctures  to the set $\NN$ of $f$-junctures  as well as all statements for the theory  developed in Sections~\ref{constr}~-~\ref{cordyn} for the system $(f^{*},\NN^{*},\JJ^{\g},\Lambda^{\g}_{+},\Lambda^{\g}_{-})$  to  the system $(f,\NN,\JJ',\Lambda_{+},\Lambda_{-})$. The truth of each statements for the system $(f,\NN,\JJ',\Lambda_{+},\Lambda_{-})$ then follows from the truth of the statement for the system $(f^{*},\NN^{*},\JJ^{\g},\Lambda^{\g}_{+},\Lambda^{\g}_{-})$ since $\psi^{-1}$ is a homeomorphism, isotopic to the identity (see examples below).
\end{proof}

\begin{example}

By Theorem~\ref{finmany}, $\Lambda^{\g}_{+}$ and $\Lambda^{\g}_{-}$ each have only finitely many semi-isolated leaves. Since $\psi^{-1}$ is a homeomorphism, $\Lambda_{+} = \psi^{-1}(\Lambda^{\g}_{+})$ and $\Lambda_{-} = \psi^{-1}(\Lambda^{\g}_{-})$ each have finitely many semi-isolated leaves. Thus, Theorem~\ref{finmany}  is valid in the system   $(f,\NN,\JJ' = \psi^{-1}(\JJ^{\g}),\Lambda_{+},\Lambda_{-})$.

\end{example}

\begin{example}

By Lemma~\ref{LGstrcl}, the laminations $\Gamma^{\g}_{\pm}$ and $\Lambda^{\g}_{\pm}$ are strongly closed. Since $\psi^{-1}$ is a homeomorphism, it follows immediately that the  laminations $\Gamma_{\pm} = \psi^{-1}(\Gamma^{\g}_{\pm})$ and $\Lambda_{\pm} = \psi^{-1}(\Lambda^{\g}_{\pm})$ are strongly closed. Thus, Lemma~\ref{LGstrcl}  is valid in the system   $(f,\NN,\JJ' = \psi^{-1}(\JJ^{\g}),\Lambda_{+},\Lambda_{-})$. It is much more difficult to prove Lemma~\ref{LGstrcl} for the system $(f,\NN,\JJ' = \psi^{-1}(\JJ^{\g}),\Lambda_{+},\Lambda_{-})$ directly from the axioms.

\end{example}

\begin{example}

By Theorem~\ref{geodext}, there exists an endperiodic automorphism $h$, isotopic to $f^{*}$ and permuting the elements  of each of the sets $\Lambda^{\g}_{+}$,  $\Lambda^{\g}_{-}$,  $\JJ^{\g}_{+}$, and  $\JJ^{\g}_{-}$. That is $h = \phi\circ f^{*}$ with $\phi$ a homeomorphism isotopic to the identity and $h(\Lambda^{\g}_{+}) = \Lambda^{\g}_{+}$,  $h(\Lambda^{\g}_{-}) = \Lambda^{\g}_{-}$,  $h(\JJ^{\g}_{+}) = \JJ^{\g}_{+}$, and  $h(\JJ^{\g}_{-}) = \JJ^{\g}_{-}$. Then,
\begin{eqnarray*}
h' &=& \psi^{-1}\circ h\circ\psi\\
&=& \psi^{-1}\circ\phi\circ f^{*}\circ\psi\\
&=& \psi^{-1}\circ\phi\circ\psi\circ f\circ\psi^{-1}\circ\psi\\
&=& \psi^{-1}\circ\phi\circ\psi\circ f
\end{eqnarray*}
is an endperiodic automporphism isotopic to $f$ and $h'(\Lambda_{+}) = \psi^{-1}\circ h\circ\psi(\psi^{-1}(\Lambda^{\g}_{+})) = \psi^{-1}(\Lambda^{\g}_{+}) = \Lambda_{+}$. Thus $h'$ permutes the elements of $\Lambda_{+}$. Similarly, $h'$ permutes the elements of $\Lambda_{-}$, $\JJ'_{+}$, and $\JJ'_{-}$. Thus, Theorem~\ref{geodext}  is valid in the system   $(f,\NN,\JJ' = \psi^{-1}(\JJ^{\g}),\Lambda_{+},\Lambda_{-})$.

\end{example}

\begin{rem}
We think of $\JJ'$ as a ``good'' choice of juncture components  in the sense that Theorem~\ref{boldmove} holds. The good choice is not always preferrable when one is trying to verify the axioms.  Thus, in the proof of the transfer theorem (Theorem~\ref{transfer}), the set $\JJ$ of junctures that we will construct for the transferred laminations is definitely not ``good'', but satisfies Axiom~\ref{trnsvrs}.
\end{rem}


\section{Smoothing $h$ and the laminations}\label{HMSM}

The goal of  this section is to prove the following.

\begin{theorem}\label{HMsmooth}
Given an endperiodic automorphism $f$, there exists  a smooth Handel-Miller pseudo-geodesic bilamination $(\Lambda_{+},\Lambda_{-})$ associated  to $f$   and a smooth endperiodic automorphism $h$ isotopic to $f$ and preserving $(\Lambda_{+},\Lambda_{-})$. 
\end{theorem} 

Recall that the  bilamination $(\Lambda_{+},\Lambda_{-})$ is called a   Handel-Miller pseudo-geodesic bilamination $(\Lambda_{+},\Lambda_{-})$ associated  to $f$  if it satisfies the four axioms (Definitions~\ref{HMbilam}).

\begin{rem}
Geodesic  laminations $\Lambda^{\g}_{\pm}$  associated to $f$ and satisfying the axioms are probably not generally smooth.  In fact,  geodesic laminations may not even be  $\CI$ or, if $\CI$ they may fail to be $\CII$ (Example~\ref{Denjoy}).  The problem is that if, in local laminated charts, there are infinitely many gaps clustering on plaques of the lamination, the mean value theorem generally obstructs attempts to extend the lamination to a $\CI$ foliation across these gaps and there are similar higher order obstructions to $C^{r}$ smoothness, $r\ge1$.  Furthermore, as in the Nielsen-Thurston theory, there seem to be   obstacles to choosing the endperiodic automorphism  $h^{\g}$, isotopic to $f$ and preserving  $(\Lambda^{\g}_{+},\Lambda^{\g}_{-})$,  to be a diffeomorphism. 
\end{rem}

We first prove Theorem~\ref{HMsmooth} for the case in which there are no principal regions. The changes that have to be made to handle the case in which there are principal regions  are not substantive and   are outlined in Section~\ref{SecPrinReg}. 

\begin{*hyp}

\textbf{Until Section~\ref{SecPrinReg}, we assume there are no principal regions.}

\end{*hyp}

\ni\textbf{Strategy.} We give a heuristic sketch of the proof. One begins with a Handel-Miller geodesic bilamination  $(\Lambda^{\g}_{+},\Lambda^{\g}_{-})$  associated to $f$  and  an endperiodic automorphism $h^{\g}$, isotopic  to $f$ and preserving  $(\Lambda^{\g}_{+},\Lambda^{\g}_{-})$. We produce a Markov partition for $h^{\g}$ by geodesic quadrilaterals, somewhat different from the ones described in Section~\ref{cordyn}, and define a diffeomorphism on the union of these rectangles.  Using standard techniques in differential topology, we extend this to an endperiodic diffeomorphism $h:L\to L$ which is isotopic to $h^{\g}$.  Positive iterates of $h$ ``stretch'' the Markov quadrilaterals deeper and deeper into the positive ends, limiting in a natural way on a smooth lamination $\Lambda_{+}$ in $L$.  Similarly, negative iteration produces a smooth bilamination $\Lambda_{-}$,  $(\Lambda_{+},\Lambda_{-})$ being an $h$-invariant bilamination.  We verify the axioms, thereby proving Theorem~\ref{HMsmooth} for the case that there are no principal regions.  We then give a detailed outline of how to modify this construction to accomodate the presence of principal regions.      \vs

Fix a set  $\NN$\label{fixNN}  of $f$-juncture components as in Definition~\ref{famNNg}.  This determines a set of geodesic juncture components $\JJ^{\g}$.

To begin with, since the intersection points of junctures with $\bd L$ do not cluster except at ends of $L$, the construction of $h^{\g}$ in Section~\ref{defineh} can be carried out so that the following holds.

\begin{lemma}\label{diffeonearbd}
The automorphism $h^{\g}:L\to L$ is a diffeomorphism in a neighborhood of $\bd L$.
\end{lemma}

Fix a core $K$.   Let $\JJ^{-}_{0}\ss\JJ^{\g}_{-}$ be the set of juncture components comprising $\fr_{-}K$ and  $\JJ^{+}_{0}\ss\JJ^{\g}_{+}$ be the set of juncture components comprising $\fr_{+}K$. Recall the decomposition $L=W^{-}\cup K\cup W^{+}$. Then $|\JJ^{-}_{0}|=W^{-}\cap K = \fr W^{-}$ and $|\JJ^{+}_{0}|=W^{+}\cap K = \fr W^{+}$. Let   $\JJ^{-}_{n}=(h^{\g})^{n}(\JJ^{-}_{0})$,   $\JJ^{+}_{n}=(h^{\g})^{n}(\JJ^{+}_{0})$,  and $W_{n}=(h^{\g})^{n}(W^{-})$, $n\in\Z$.  

The following is clear.

\begin{lemma}

\begin{enumerate}

\item $\bigcup_{n=-\infty}^{\infty} \JJ^{-}_{n} = \JJ^{\g}_{-}$\upn{;}

\item $\bigcup_{n=-\infty}^{\infty} \JJ^{+}_{n} = \JJ^{\g}_{+}$\upn{;}

\item $\fr W_{n} = |\JJ^{-}_{n}|$.

\end{enumerate}

\end{lemma}

\begin{defn}[quadrilateral and geodesic quadrilateral]\label{quadltrl}
We will call a geodesically convex figure with four edges that are geodesics a \emph{geodesic quadrilateral}.  When $h$ has been defined, the images of a geodesic quadrilateral under applications of $h$ and its powers will  be called quadrilaterals.
\end{defn}

\begin{rem}
We defined a rectangle (Definition~\ref{rctnglr}) to have a pair of opposite edges in $\Gamma_{-}$ and a pair of opposite edges in $\Gamma_{+}$.   In a geodesic quadrilateral, the edges    may be, but need not be,  in $\Gamma_{\pm}$.
\end{rem}

\begin{lemma}\label{maxrects}
For $n>0$ large enough, the set of components of $K\sm \intr W_{n}$ consists of a finite family $\{R_{1},\ldots,R_{k}\}$ of geodesic quadrilaterals with one pair of opposite edges in $\fr_{+}K$ and the other pair of opposite edges components of $|\JJ^{-}_{n}|\cap K$ in $\XX^{\g}_{-}|K$. 
\end{lemma}

\begin{proof}
By Lemma~\ref{LambdaK} the leaves of $\Lambda^{\g}_{+}\cap K$ are properly embedded, boundary incompressible  arcs with endpoints in $\fr_{+}K$ which fall into finitely many isotopy classes. By Lemma~\ref{extremals} each of these isotopy classes has two extreme arcs which together with two arcs on $\fr_{+}K$ bound rectangles $\{R'_{1},\ldots,R'_{\ell}\}$ such that $\bigcup_{i=1}^{\ell}R'_{i}\supset |\Lambda^{\g}_{+}|\cap K$. Since there are no principal regions, both edges of $R'_{i}$ in $|\Lambda^{\g}_{+}|\cap K$ are approached outside of $R'_{i}$ by arcs of $|\XX^{\g}_{-}|\cap K$. Since the negative junctures can only cluster on $|\Lambda^{\g}_{+}|$, the strongly closed property implies that there are only finitely many components of $\bigcup_{n=0}^{\infty}|\JJ^{-}_{n}|\cap K$ that do not lie in one of these isotopy classes and that $R'_{i}\ss R''_{i}$, a geodesic quadrilateral with a pair of opposite edges in $\fr_{+}K$ and a pair of opposite edges in $\bigcup_{n=0}^{\infty}|\JJ^{-}_{n}|\cap K$ such that $\bigcup_{i=1}^{\ell}R''_{i}$ contains every arc of $\bigcup_{n=0}^{\infty}|\JJ^{-}_{n}|\cap K$ that lies in the same isotopy class as an arc of $|\Lambda^{\g}_{+}|\cap K$.

If  $n$ is large enough, then  the components of $\fr W_{n}\cap K$ all lie in one of the isotopy classes and all are  components of $|\XX^{\g}_{-}|\cap K$. Thus, the components of $K\sm \intr W_{n}$ are  geodesic quadrilaterals $\{R_{1},\ldots,R_{k}\}$ that are subquadrilaterals of the geodesic quadrilaterals $\{R''_{1},\ldots,R''_{\ell}\}$.
\end{proof}

\begin{rem}
By choosing $n>0$  large enough, as in Lemma~\ref{maxrects}, we guarantee that the set  $\{R_{1},\dots,R_{k}\}$  of components of  $K\sm \intr W_{n}$ has the property  that the geodesic quadrilateral $h^{\g}(R_{j})$ completely crosses any $R_{i}$ that it meets, $1\le j\le k$.  The components of the $h^{\g}(R_{j})\cap R_{i}$ are the components of $K\sm \intr W_{n+1}$. Now take the $R_{i}$ to be the components of   $K\sm \intr W_{n+1}$. As in the proof of Proposition~\ref{just1}, any $h^{\g}(R_{j})$ meets any $R_{i}$ at most once.   Thus, the set of geodesic quadrilaterals $\{R_{1},\ldots,R_{k}\}$ making up $K\sm \intr W_{n+1}$ satisfy Properties~II, III, and~IV  of Definition~\ref{marfam} but is not a  Markov system as defined there because the edges of  $R_{i}$ are not subarcs of leaves of the laminations $\Lambda^{\g}_{\pm}$. 
\end{rem}

\begin{nota}
From now on in Section~\ref{HMSM}  we denote by $N>0$ this integer $n+1$.

\end{nota}

\begin{defn}[Markov chain]\label{markchn}
A sequence 
$$
\iota=(i_{0},i_{1},i_{2},\dots,i_{n},\dots)
$$ 
in which $i_{n}\in\{1,2,\dots,k\}$ and $R_{i_{n}}\cap h^{\g}(R_{i_{n+1}})\ne\0$, $n\ge0$, will be called a Markov chain.
\end{defn}

This language is borrowed from symbolic dynamics.  Our system of geodesic quadrilaterals is  a Markov partition as commonly defined in dynamics, but different from that defined in Section~\ref{cordyn}.  The following statement and proof are closely analogous to those  of Proposition~\ref{prop1013}.  Since our construction of the smooth bilamination is modelled on this, we give details.

\begin{prop}\label{infint}
If $\iota$  is a Markov chain, then there is a unique leaf $\ell^{\g}_{\iota}\in\Lambda^{\g}_{+}$ such that if
$$\ell^{\g}_{\iota_{n}}=(h^{\g})^{n}(R_{i_{n}})\cap(h^{\g})^{n+1}(R_{i_{n+1}})\cap\cdots,\quad n\ge 0$$
then
$$
\ell^{\g}_{\iota} =\bigcup_{n=0}^{\infty}\ell^{\g}_{\iota_{n}},
$$ 
an increasing union of compact arcs.
\end{prop}

\begin{proof}
Choose a lift $C^{\g}_{0} = \wt R_{i_{0}}$  of $R_{i_{0}}$  to  $\wt L\ss\Delta$. This lift determines lifts $C^{\g}_{n}$ of $(h^{\g})^{n}(R_{i_{n}})$ for all $n>0$. For $n\ge 0$, let $\wt\lambda^{\g}_{n},\wt\mu^{\g}_{n}\in\wt\XX^{\g}_{-}$ contain opposite sides of these lifts $C^{\g}_{n}$ of $(h^{\g})^{n}(R_{n})$. The sequences $\{\wt\lambda^{\g}_{n}\},\{\wt\mu^{\g}_{n}\}$ converge monotonically to $\wt\lambda^{\g},\wt\mu^{\g}\in\wt\Lambda^{\g}_{+}$.  Let $\delta^{\g}_{0}$ be an edge of $R_{i_{0}}$ contained in $\fr_{+}K$. Let $\delta^{\g}_{1}\ss h^{\g}(\fr_{+}K)$ be the edge of $h^{\g}(R_{i_{1}})$ bordering the piece of  $h^{\g}(R_{i_{1}})$ that exits  $R_{i_{0}}$ through $\delta^{\g}_{0}$. Inductively, let $\delta^{\g}_{n}\ss (h^{\g})^{n}(\fr_{+}K)$ be the edge of $(h^{\g})^{n}(R_{i_{n}})$ bordering the piece of  $(h^{\g})^{n}(R_{i_{1}})$ that exits  $R_{i_{n-1}}$ through $\delta^{\g}_{n-1}$. For $n\ge 0$, let $\wt\delta^{\g}_{n}$ be the edge of the lift  $C^{\g}_{n}$ of $(h^{\g})^{n}(R_{i_{n}})$ that covers $\delta^{\g}_{n}$ and let $\wt\sigma^{\g}_{n}$ be the geodesic in $\wt L$ that contains $\wt\delta^{\g}_{n}$. Note that the $\wt\sigma^{\g}_{n}$ are lifts of components of $(h^{\g})^{n}(\fr_{+}K)$. By Theorem~\ref{esctoe}, the $\wt\sigma^{\g}_{n}$ nest on a point $a\in\Si$ which is a common endpoint of $\wt\lambda^{\g}$ and $\wt\mu^{\g}$. Similarly, starting with the other edge  $\gamma^{\g}_{0}$ of $R_{i_{0}}$ which lies in $\fr_{+}K$, one gets a family of geodesics in $\wt L$ that nest on a point $b\in\Si$ which is the other common endpoint of $\wt\lambda^{\g}$ and $\wt\mu^{\g}$. It follows that $\wt\lambda^{\g}=\wt\mu^{\g}$ which projects to $\ell^{\g}_{\iota}$.
\end{proof}

\begin{rem}
For a similar characterization of the leaves $\ell^{\g}_{\kappa}\in\Lambda^{\g}_{-}$, first note that applying $(h^{\g})^{-N}$ to the geodesic quadrilaterals $R_{i}$, gives  similar geodesic quadrilaterals $R^{*}_{i}$ with a pair of opposite edges in $\fr_{-}K$ and the other pair of  edges  components of $|\JJ^{+}_{-N}|\cap K$.  We use  sequences 
$$
\kappa=(\dots,i_{-n},\dots,i_{-2},i_{-1},i_{0})
$$ 
in which $R^{*}_{i_{-n}}\cap (h^{\g})^{-1}(R^{*}_{i_{-n-1}})\ne\0$.  It is actually unnecessary to go to the geodesic quadrilaterals $R^{*}_{i}$ since one gets the same result by interchanging the roles of $R_{i}$ and $h^{\g}(R_{j})$ and using negative powers of $h^{\g}$.  Notice that this would not work if an edge of $R_{i}$ bordered a principal region.
\end{rem}

As remarked earlier, our approach to the proof of Theorem~\ref{HMsmooth} in the case that there are no principal regions will be to isotope $h^{\g}$ to an endperiodic \emph{diffeomorphism} $h$ such that $h(R_{i})=h^{\g}(R_{i})$, $1\le i\le k$, and define the new smooth bilamination $(\Lambda_{+},\Lambda_{-})$ by a construction analogous to the above,  defining $\XX_{\pm}$ in a fairly obvious way.  Verifying the axioms will complete the proof.  The proof in the presence of principal regions is similar, but notationally more complicated.

\subsection{Geodesic grids}\label{geogrs}

Given a 
geodesic quadrilateral $R$, let $\alpha^{\g}_{0}$ and $\alpha^{\g}_{1}$ be one pair of opposite sides, $\beta^{\g}_{0}$ and $\beta^{\g}_{1}$ the other. Orient the $\alpha^{\g}_{i}$ from $\beta^{\g}_{0}$ to $\beta^{\g}_{1}$ and the $\beta^{\g}_{i}$ from $\alpha^{\g}_{0}$ to $\alpha^{\g}_{1}$. Let $\phi:\alpha^{\g}_{0}\to\alpha^{\g}_{1}$ and $\psi:\beta^{\g}_{0}\to\beta^{\g}_{1}$ be orientation preserving diffeomorphisms. Thus $\phi$ pairs the endpoints of $\beta^{\g}_{0}$ and $\beta^{\g}_{1}$ and  $\psi$ likewise pairs the endpoints of $\alpha^{\g}_{0}$ and of $\alpha^{\g}_{1}$.  Since  $R$ is  geodesically convex, the geodesics joining $t\in\alpha^{\g}_{0}$ to $\phi(t)\in\alpha^{\g}_{1}$ lie in $R$, for all $t\in\alpha^{\g}_{0}$.  Since the geodesic quadrilateral is simply connected, we may as well be working in the hyperbolic plane where geodesic arcs depend smoothly on their endpoints.   Thus, our geodesics depend smoothly on  $t,\phi(t)$ and cannot intersect each other.  Indeed, since $\phi$ is orientation preserving, any proper intersections would produce geodesic digons.  Thus these geodesic arcs are the leaves of a smooth foliation $\FF_{\vv}$ of $R$ such that $\beta^{\g}_{i}\in\FF_{\vv}$, $i=0,1$.  Similarly, use $\psi$ to obtain a smooth, geodesic foliation $\FF_{\hh}$ of $R$, necessarily transverse to $\FF_{\vv}$ and incorporating the $\alpha^{\g}_{i}$ as leaves.  We call this a ``geodesic grid''.

Typically, we will have one geodesic quadrilateral completely crossing another once as in Figure~\ref{grid1}. We can put geodesic grids on each which agree on the intersection.  In Figure~\ref{grid1}, define the ``vertical'' geodesic foliation $\FF_{\vv}$ first on the tall geodesic quadrilateral $R$ and the ``horizontal'' geodesic foliation $\FF_{\hh}$ on the wide one $R'$.  In order to extend $\FF_{\hh}|R\cap R'$ across the unfoliated geodesic subquadrilaterals of $R$, one needs to choose diffeomorphisms from the left side to the right that match  smoothly at the common endpoints of domains with the ones already defined by $\FF_{\hh}|R\cap R'$.  This is a matter of making the two $\infty$-jets coincide and is standard. Similarly, $\FF_{\vv}$ is extended smoothly. 

\begin{figure}
\begin{center}
\begin{picture}(100,130)(75,100)
\includegraphics[width=300pt]{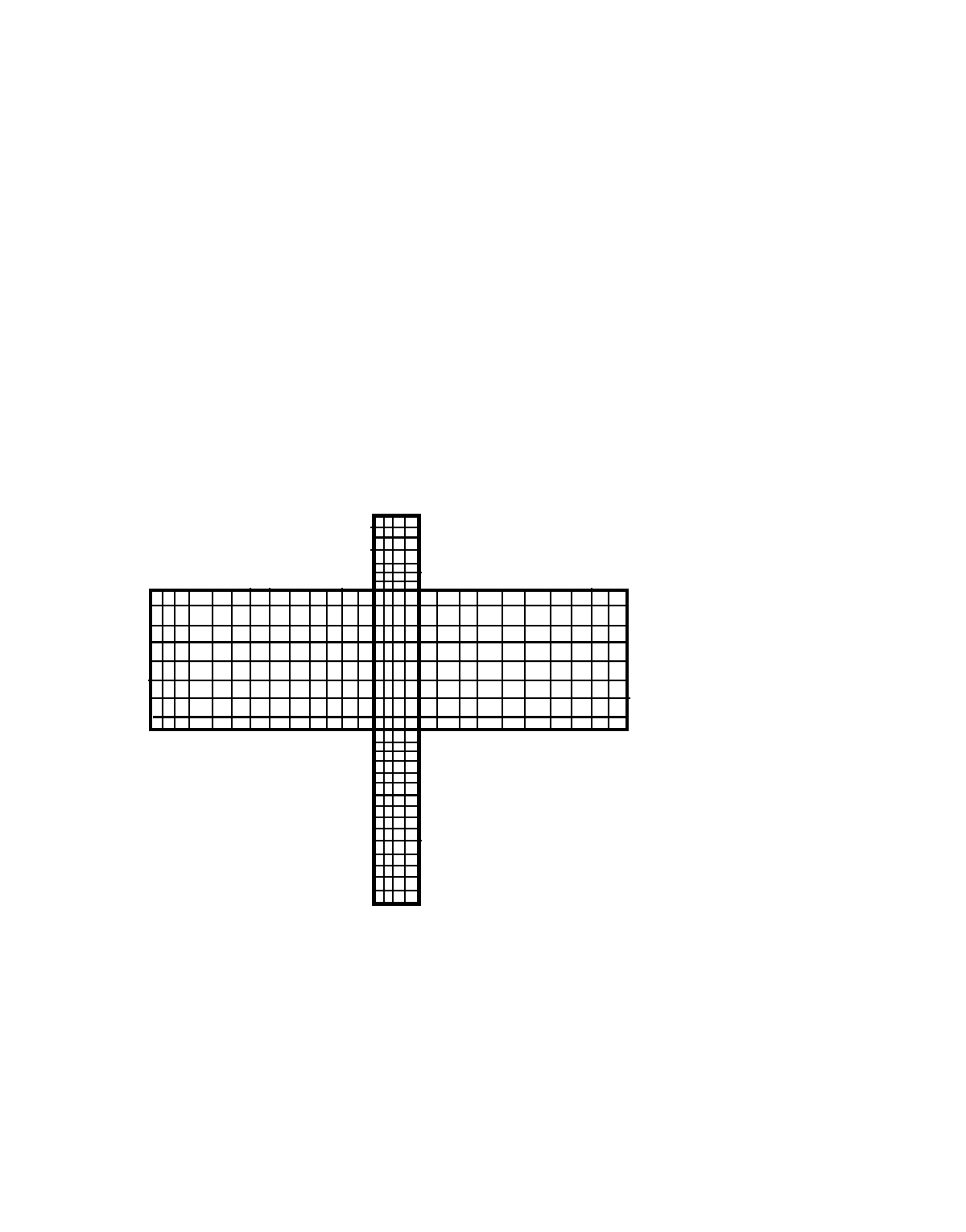}
\put(-100,175){\Small$R'$}
\put(-166,130){\Small$R$}

\end{picture}
\caption{Compatible grids}\label{grid1}
\end{center}
\end{figure}

Another situation is that the quadrilateral $R$ with the geodesic grid may be contained in a slightly larger quadrilateral $A$ as in Figure~\ref{larger}.  We will leave it to the reader to adapt the previous discussion to extend the geodesic grid on $R$ to a geodesic grid on $A$.

\begin{figure}
\begin{center}
\begin{picture}(100,100)(100,170)
\includegraphics[width=300pt]{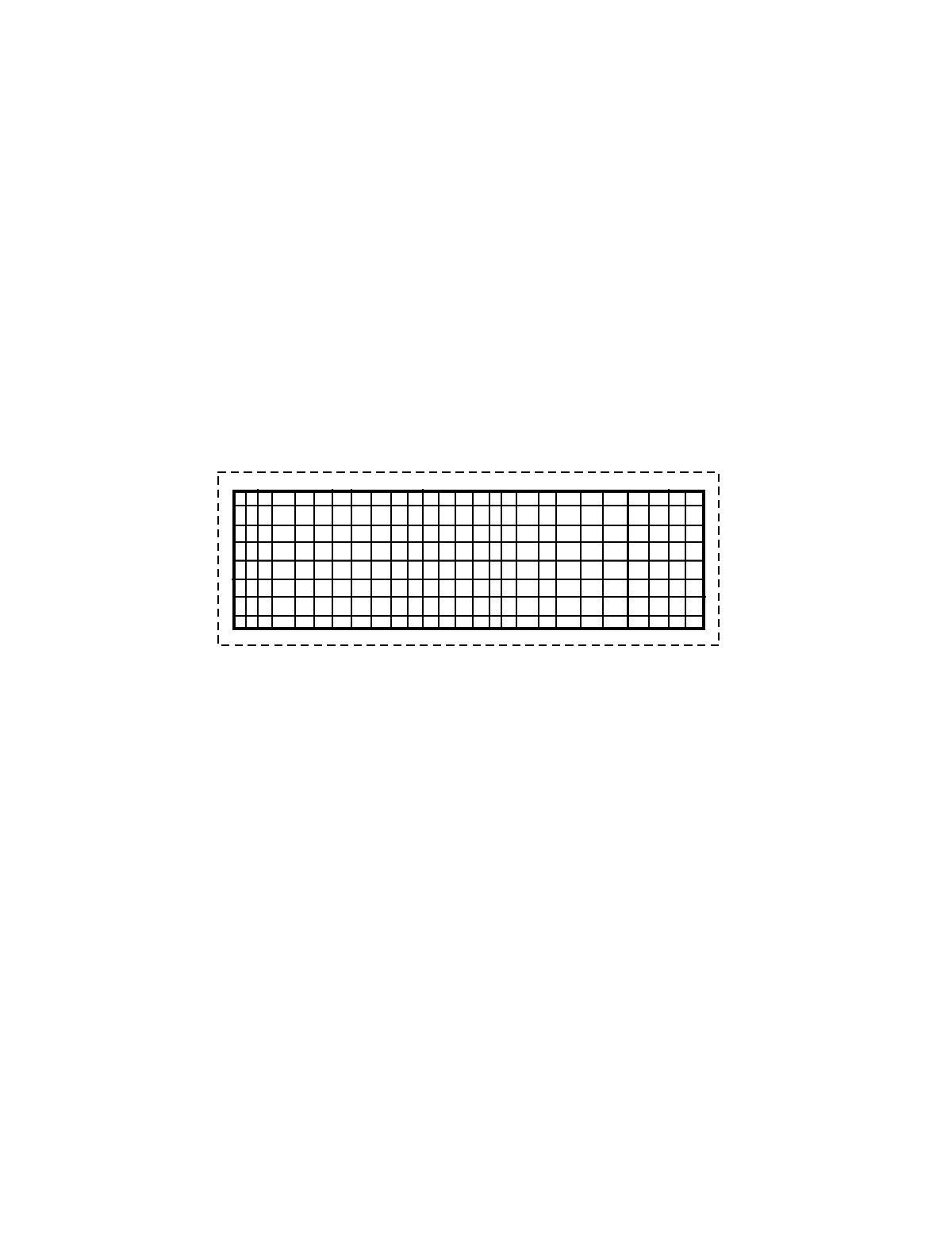}
\put(-100,180){\Small$A$}

\end{picture}
\caption{Expanded geodesic quadrilateral}\label{larger}
\end{center}
\end{figure}

\begin{figure}
\begin{center}
\begin{picture}(100,200)(100,70)
\includegraphics[width=300pt]{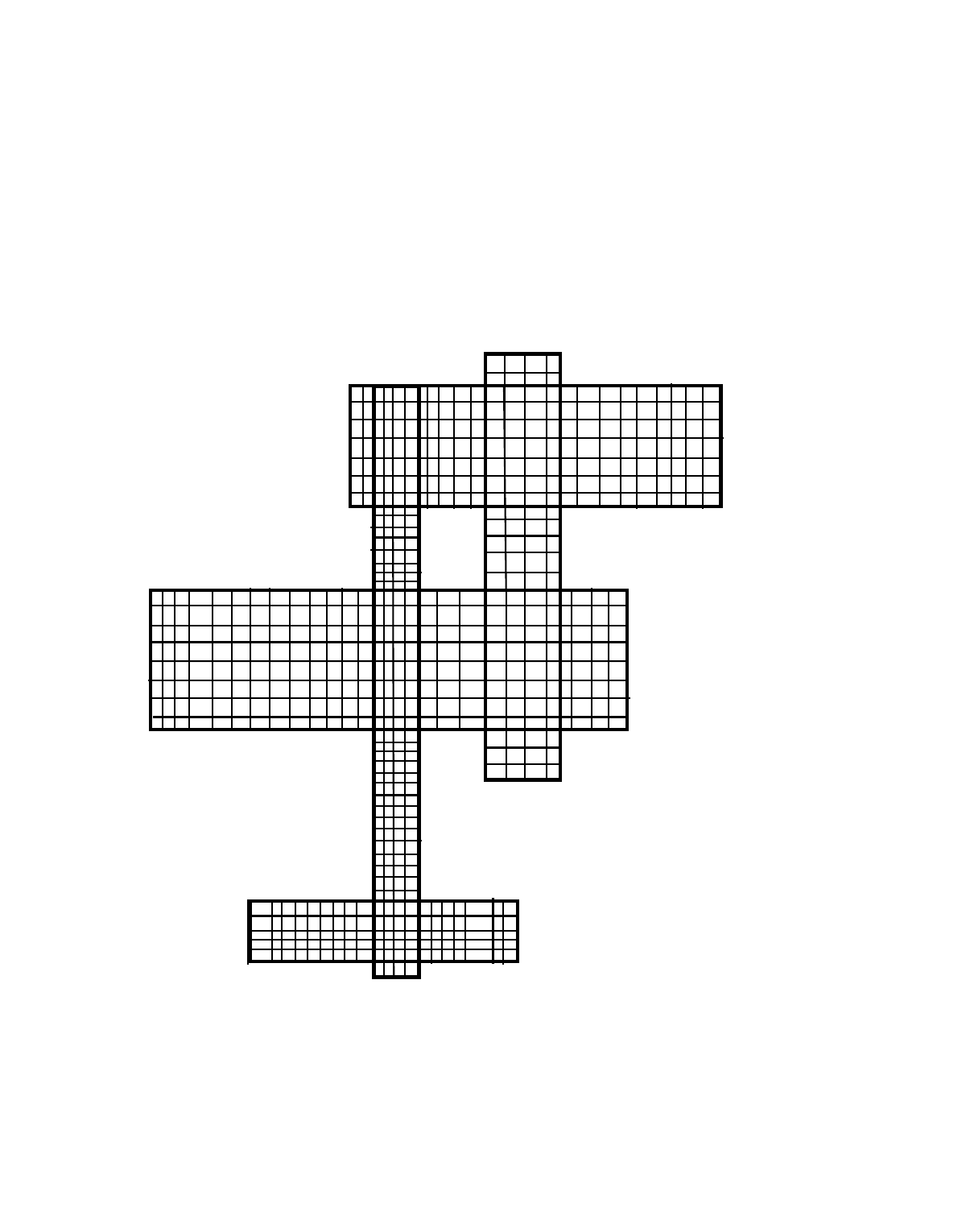}
\put(-100,175){\Small$R_{j}$}
\put(-166,130){\Small$h^{\g}(R_{i})$}

\end{picture}
\caption{Compatible grids}\label{grid}
\end{center}
\end{figure}

In the remark preceding Definition~\ref{markchn}, we  constructed  families $\{R_{1},\ldots,R_{k}\}$ and  $\{h^{\g}(R_{1}),\ldots,h^{\g}(R_{k})\}$of pairwise disjoint
 geodesic quadrilaterals. Furthermore, $h^{\g}(R_{i})$ will completely cross some of the $R_{j}$'s, but never more than once (cf.~Definition~\ref{completelycrosses}).  As above, we define smooth geodesic grids on $\RR=
\{R_{1},\ldots,R_{k}\}\cup\{h^{\g}(R_{1}),\ldots,h^{\g}(R_{k})\}$ which are compatible on overlaps.  In Figure~\ref{grid}, the vertical lines are to be leaves of $\FF_{\vv}$, the horizontal ones leaves of $\FF_{\hh}$. Note that the foliation $\FF_{\hh}$ incorporates as leaves the edges of $R_{i}$ which are subarcs of $\fr_{+}K$, while $\FF_{\vv}$ incorporates the other pair of opposite edges.  

\begin{rem}
Ultimately, the smooth bilamination $(\Lambda_{+},\Lambda_{-})$ that we are going to construct will be such that $\FF_{\vv}|R_{i}$ is an extension of $\Lambda_{+}|R_{i}$ and $\FF_{\hh}$ is an extension of $\Lambda_{-}|R_{i}$.  This will establish the smoothness of the laminations in each $R_{i}$ and will extend to global smoothness by iteration of the diffeomorphism $h$ that we are going to construct.
\end{rem}

\begin{rem}
If finitely many geodesic leaves are preassigned in any of the geodesic quadrilaterals, there  is no problem choosing the foliations to incorporate them.  If one wants to preassign infinitely many geodesic leaves, say arcs of $\Lambda^{\g}_{\pm}$, serious smoothness issues arise.  This is exactly why, in producing the  smooth   laminations $\Lambda_{\pm}$ and diffeomorphism $h$ out of the geodesic data we will not generally get back the geodesic laminations.
\end{rem}

\subsection{Smoothing in the geodesic quadrilaterals}  Let $\RR=\bigcup_{i=1}^{k}R_{i}$.  In order to smooth $h^{\g}$ on $\RR$, we first thicken each $R_{i}$ to a slightly larger quadrilateral $A_{i}$ (see Figure~\ref{bigrectangle}).  We can do this so that the $A_{i}$'s are pairwise disjoint.

We want to extend the geodesic grids of Section~\ref{geogrs} to a neighborhood   of $\RR\cup h^{\g}(\RR)$ of the form
$$
V=A_{1}\cup\cdots\cup A_{k}\cup h^{\g}(A_{1})\cup\cdots h^{\g}(A_{k}).
$$
The problem is that $h^{\g}(A_{i})$ will not generally be a geodesic quadrilateral, the edges failing to be geodesics. To correct this, we perform a small isotopy of $h^{\g}$.

\begin{lemma}
If $A_{i}$ approximates $R_{i}$ sufficiently well, there is a homeomorphism $\phi_{i}$, isotopic to the identity by an isotopy supported near $h^{\g}(A_{i})$ and away from $h^{\g}(A_{j})$, $j\neq i$, taking $h^{\g}(\fr_{+}K)$ to itself componentwise, fixing $h^{\g}(R_{i})$ pointwise, fixing the vertices of $h^{\g}(A_{i})$, and such that $\phi_{i}(h^{\g}(A_{i}))$ is a geodesic quadrilateral.
\end{lemma}

\begin{proof}
Let $D\ss L$ be a smoothly embedded disk   containing $h^{\g}(A_{i})$, having the four vertices of that image on its boundary, but such that the edges of $h^{\g}(A_{i})$ are  properly embedded in $D$. If the vertices of $A_{i}$ have been chosen sufficienly near the corresponding vertices of $R_{i}$, $D$ can be chosen so that the geodesic arcs joining the vertices of $h^{\g}(A_{i})$ to form a geodesic quadrilateral are also properly embedded in $D$.
Also, choose $D$ so that it meets no  $h^{\g}(A_{j})$, $j\neq i$.  
By our usual application of Theorem~\ref{3.1}, we can find an ambient isotopy $\Phi$ supported in $D$, such that $\Phi$  fixes $h^{\g}(R_{i})$ pointwise and takes the bottom two  edges of $h^{\g}(\bd A_{i})$ to the geodesics joining the corresponding  pairs of vertices. For the remaining two edges, simple sliding isotopies along the components of $\fr_{+}K$ crossed by these edges must first be performed and the Theorem~\ref{3.1} completes the isotopy. Take $\phi = \Phi^{1}$.
\end{proof}

We can assume that the homeomorphisms $\phi_{i}$ are isotopic to the identity by isotopies which have pairwise disjoint supports and set $\phi=\phi_{1}\o\cdots\o\phi_{k}$.  Thus, replacing $h^{\g}$ by $\phi\circ h^{\g}$ we can assume that $h^{\g}(A_{i})$ is a slightly larger geodesic quadrilateral than $h^{\g}(R_{i})$, $1\le i\le k$.

\begin{figure}[b]
\begin{center}
\begin{picture}(300,130)(-5,-130)

\rotatebox{270}{\includegraphics[width=140pt]{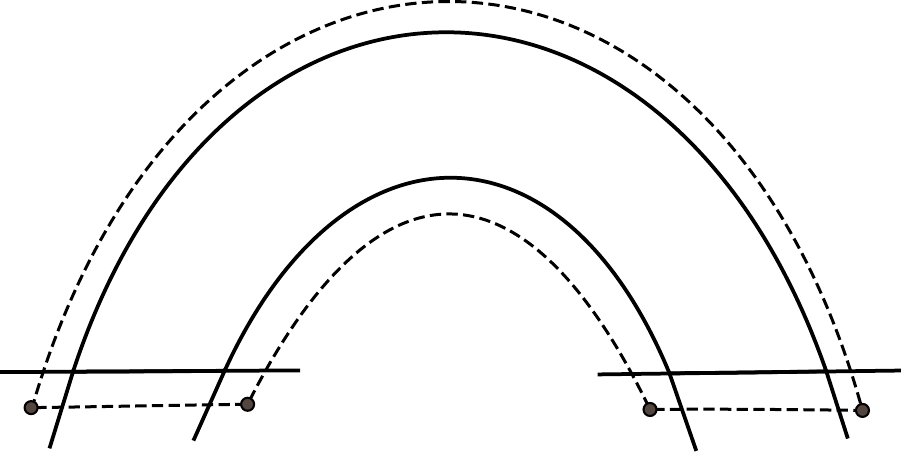}}

\put(-180,-118){$\fr_{+}K$}
\put(10,-118){$\fr_{+}K$}
\put(-65,-70){$R_{i}$}
\put(-40,-57){$A_{i}$}

\put(-165,-85){$\alpha'$}
\put(-237,-40){$\alpha$}

\end{picture}
\caption{Geodesic quadrilaterals $R_{i}$ and $A_{i}$ (dotted)}\label{bigrectangle}
\end{center}
\end{figure}

 We construct the geodesic grid on $V$ as above so that it restricts to the geodesic grid already constructed on $\RR$.  The extended foliations are again denoted by $\FF_{\hh}$ and $\FF_{\vv}$.
 
 These foliations have trivial holonomy, so it is easy to construct a smooth, transverse invariant measure $\mu_{\hh}$ for $\FF_{\hh}$ on $V$ and another $\mu_{\vv}$ for $\FF_{\vv}$ on $V$.  The measure $\mu_{\hh}$ can be viewed as a smooth measure on the leaves of $\FF_{\vv}$, invariant by holonomy translations along the leaves of $\FF_{\hh}$ and, similarly, $\mu_{\vv}$ is such a measure on the leaves of $\FF_{\hh}$.  
 
 \begin{rem}
 By a smooth measure $\mu$ on a smooth $1$-manifold $B$, we mean that if $\theta$ is Lebesgue measure on $B$, then the Radon-Nikodym derivative $d\mu/d\theta = f$ is a smooth function on $B$. Thus, for any Borel set $X\sseq B$, $\mu(X)=\int_{X}f\,d\theta$.
 \end{rem}
 
 On each $A_{i}$, using the above measures, we define smooth coordinates $(x,y)$ such that the leaves of $\FF_{\hh}$ are the level sets of $y$ and the leaves of $\FF_{\vv}$ are the level sets of $x$.  Simply designate the coordinates at one point $p_{0}$, say by $(a,b)$, designate an orientation of the leaves of $\FF_{\vv}|A_{i}$ and of the leaves of $\FF_{\hh}|A_{i}$ and use the holonomy invariant  measures to define the desired smooth coordinates on $A_{i}$.  More precisely, using the designated orientation of the leaf of $\FF_{\hh}$ through $p_{0}$, each point $p$ in that leaf to the right of $p_{0}$ defines a subarc $[p_{0},p]$ and we define $x(p)=a+\mu_{\vv}[p_{0},p]$.  If $p$ is to the left of $p_{0}$, set $x(p)=a-\mu_{\vv}[p,p_{0}]$.  By translation along the leaves of $\FF_{\vv}$ we extend the $x$ coordinate to all of $A_{i}$.  The $y$ coordinate is defined analogously using the measure $\mu_{\hh}$.  These do not give global coordinates on $V$, but the above construction makes the following clear.

 \begin{lemma}
 On overlaps $A_{i}\cap h^{\g}(A_{j})$,  the induced coordinates $(x,y) $ coming from $A_{i}$ and $(x',y')$ coming from $h^{\g}(A_{j})$ are related by 
 \begin{eqnarray*}
x'=\epsilon_{1}x + c\\
y'=\epsilon_{2}y+d,
\end{eqnarray*}
where $\epsilon_{i}=\pm1$, $i=1,2$, and $c,d$ are constants.
 \end{lemma}

The sign is due to the choices of orientation and the additive constants are due to the fact that the coordinates are well defined up to translation.

Thus the Jacobian matrix for this coordinate change is
\begin{equation}
d\Phi = \begin{bmatrix}
\epsilon_{1} & 0\\
0 & \epsilon_{2}
\end{bmatrix}.\tag{$*$}
\end{equation}

Normalize the measures $\mu_{\hh},\mu_{\vv}$ so that the ``width'' of each $A_{i}$, measured by $\mu_{\vv}$ is $1$ and similarly, the ``height'' of each $h^{\g}(A_{i})$ is $1$.

We now choose  diffeomorphisms $h_{i}:A_{i}\to h'(A_{i})$ which preserve the smooth geodesic grids, $1\le i\le k$. In the smooth local coordinates $(x,y)$ on  $V$ we have, $$h_{i}(x,y)=(\xi_{i}(x),\zeta_{i}(y)).$$ For technical reasons,   choose $h_{i}$ to satisfy, 
$$0< |d\xi_{i}/dx|<a<1,\qquad |d\zeta_{i}/dy| >b>1,$$
for suitable constants $a$  and $b$. This is possible since the width of each $A_{j}$ and the height of each $h^{\g}(A_{i})$ is $1$ and $h^{\g}(A_{i})$ intersects $A_{j}$, if at all, in a proper geodesic subquadrilateral of both $A_{i}$ and $h^{\g}(A_{j})$. Hence $h_{i}$ strictly diminishes width and increases height.
We summarize.

\begin{prop}\label{localprop}
For $1\le i\le k$ there are diffeomorphisms $$h_{i}=(\xi_{i},\zeta_{i}):A_{i}\ra h^{\g}(A_{i}),$$ preserving the geodesic grid and such that 
$$0< |d\xi_{i}/dx| <a<1,\qquad |d\zeta_{i}/dy| >b>1$$
in the coordinates on $A_{i}\cup h^{\g}(A_{i})$.
\end{prop}

\begin{rem}
Because of $(*)$, these inequalities hold for any of our choices of local coordinates.
\end{rem}

We begin the modification of $h^{\g}$ to  an isotopic diffeomorphism $h$ by replacing $h^{\g}|R_{i}$ by $h_{i}|R_{i}$, $1\le i\le k$.  The fact that these maps extend to diffeomorphisms on the slightly larger geodesic quadrilaterals $A_{i}$ will be useful in extending these local definitions to the global $h$.

\subsection{Smoothing outside $\RR$}\label{smoothoutside}

We first need to ``blend'' $h^{\g}$ with the diffeomorphisms $h_{i}$ so that they agree on a neighborhood of $R_{i}$. Fix a smooth, embedded disk $D_{i}\ss V$ with  $R_{i}\ss\intr D_{i}$ and set  $\bd D_{i}=S_{i}$.  Choose these so that $D_{i}\cap D_{j}=\0$, $i\neq j$. Note that $h^{\g}|S_{i}$ and $h_{i}|S_{i}$ are homotopic.  By Theorem~\ref{2.1}, there is a homeomorphism $\phi_{i}$, isotopic to the identity by an isotopy compactly supported outside of $R_{i}$ and inside of $V$, such that $\phi_{i}\o h^{\g}|S_{i}=h_{i}|S_{i}$.  (Note that we are working in the complement of $R_{i}$ where $S_{i}$ is essential.) Again, we can assume the $\phi_{i}$ and $\phi_{j}$, $i\neq j$, are isitopic to the identity by isotopies which have disjoint supports. In this way, we isotope $h^{\g}$  by a compactly supported isotopy to a homeomorphism $h''$ that agrees with $h_{i}$ on $S_{i}$, $1\le i\le k$.  Use Alexander's trick to isotope $h''\o h_{i}^{-1}|D_{i}$ to the identity by an isotopy that fixes   $S_{i}$ pointwise.  This shows that $h''|{D_{i}}$ is isotopic to $h_{i}|D_{i}$ by an isotopy throuhout which they continue to agree on $S_{i}$, $1\le i\le k$.  We summarize.

\begin{lemma}
After a compactly supported  isotopy, we can assume that $h^{\g}$  agrees with $h_{i}$ on a neighborhood of  $R_{i}$, $1\le i\le k$.  
\end{lemma}

Let $L'$ denote the complement in $L$ of the union of the open geodesic quadrilaterals $\intr R_{i}$.  Similarly, Let $L''=h^{\g}(L')$ be the complement in $L$ of the union of the open geodesic quadrilaterals $h^{\g}(\intr R_{i})$.  We work completely in $L'$. By the above lemma and Lemma~\ref{diffeonearbd},  $h^{\g}|L'$ is  a diffeomorphism in a neighborhood in $L'$ of $\bd L'$. It is well known   that $h^{\g}|L':L'\to L''$ is arbitrarily well approximated by a diffeomorphism $h^{\#}$, isotopic to $h^{\g}|L'$ by an isotopy that is fixed in a smaller neigborhood in $L'$ of $\bd L'$. For a particularly nice proof of this which does not use the Sch\"onflies theorem, see A.~Hatcher's unpublished note~\cite{ha:smooth} which is available on the author's website.  

Since $h^{\#}$ agrees with $h_{i}$ in a one-sided neighborhood of $\bd R_{i}$, on the side outside $R_{i}$, $1\le i\le k$, and the isotopy was constant in such a neighborhood, $h^{\#}$ and the $h_{i}$ combine to give  a diffeomorphism which we again denote by $h^{\#}
:L\to L$, isotopic to $h^{\g}$ and agreeing with $h_{i}$ in a neighborhood of each $R_{i}$. Since $h^{\#}$ agrees with $h^{\g}$ in  a neighborhood of $\bd L'$ and is isotopic to $h^{\g}$, we can apply the smooth versions of Theorems~\ref{2.1} and~\ref{3.1} (cf.~Theorem~\ref{epsteinsmooth}) to smoothly isotope $h^{\#}$ to a diffeomorphism $h$ that agrees with the permutation induced by $h^{\g}$ on the components of $|\JJ_{i}^{-}|$, $i\le N$, and of $|\JJ_{i}^{+}|$, $i\ge0$ thus preserving the structure we set up in Lemma~\ref{maxrects}. Again we work in $L'$ and the isotopies are supported away from $\bd L'$.   Evidently, $h$ is endperiodic and the  $h^{\g}$-junctures making up $\fr K$ propagate under applications of $h^{n}$, $n\in\Z$. In Section~\ref{smoothaxioms} we will take these to be the set of all $h$-junctures.  We summarize.

\begin{lemma}
The diffeomorphism $h:L\to L$ is endperiodic, isotopic to $h^{\g}$ and agrees with $h_{i}$ on a neighborhood of $R_{i}$, $1\le i\le k$.  Furthermore, if $\tau^{\g}$ is a component of $|\JJ_{i}^{-}|$, $i\le N$, or of $|\JJ_{i}^{+}|$, $i\ge0$, then $h(\tau^{\g})=h^{\g}(\tau^{\g})$.
\end{lemma}

\begin{rem}
In particular, since $h|R_{i}=h_{i}$, we have guaranteed that $h$ satisfies the inequalities of Proposition~\ref{localprop} on $R_{i}$, $1\le i\le k$.

\end{rem}

\subsection{Proof of Theorem~\ref{HMsmooth} in the case of no principal regions}\label{smoothaxioms}

We define the sets of positive/negative juncture components $\JJ_{\pm}$. Let $\JJ_{\pm}$ be the set of arcs and circles $h^{k}(\gamma)$, $k\in\Z$, $\gamma\in\JJ_{0}^{\pm}$. Let $\XX_{\pm}\ss\JJ_{\pm}$ be the set of nonescaping  juncture components.  Since the components of the junctures are either essential circles or properly embedded arcs, the following is evident.

\begin{lemma}\label{cor3ccepstein}

The components of the junctures are pseudo-geodesics.

\end{lemma}

Since the juncture components in $\JJ$ are isotopic to the corresponding geodesic juncture components in $\JJ^{\g}$ the following is evident.

\begin{lemma}\label{NNbiJJ}

There is a bijection $\iota:\NN\to\JJ$ where $\NN$ is the set of $f$-juncture components fixed on \emph{page~\pageref{fixNN}}.

\end{lemma}

It is now necessary to construct the smooth, $h$-invariant laminations and verify the axioms.  Remark that 
$$R_{i_{n}}\cap h^{\g}(R_{n_{i+1}})=R_{i_{n}}\cap h(R_{i_{n+1}}),\quad n\ge0.$$ 
Thus $h$ and $h^{\g}$ define the same Markov chains $\iota=(i_{0},i_{1},i_{2},\dots,i_{n},\dots)$.
In  analogy with the construction of $\ell^{\g}_{\iota}$ (Proposition~\ref{infint}), replacing $h^{\g}$ with $h$, we define $\ell_{\iota}$ by a Markov chain $\iota$. 
That this is a curve requires proof.

\begin{lemma}
The set $\ell_{\iota}$ is a one-one immersed smooth curve in $L$ for each Markov chain $\iota$.
\end{lemma}

\begin{proof}
Set 
$$R_{i_{0}i_{1}\cdots i_{n}}=R_{i_{0}}\cap h(R_{i_{1}})\cap h^{2}(R_{i_{2}})\cap\cdots\cap h^{n}(R_{i_{n}}).
$$
Iterated applications of the first inequality in  Proposition~\ref{localprop} to $h:R_{i_{n}}\to h(R_{i_{n}})$ show that the (geodesic) quadrilaterals in the nested sequence 
$$
R_{i_{0}}\supset R_{i_{0}i_{1}}\supset\cdots\supset R_{i_{0}i_{1}\cdots i_{n}}\supset\cdots
$$
 have widths decreasing monotonically to $0$.  Thus, the intersection of this nested sequence is a single leaf $\ell_{i_{0}}$ of $\FF_{\vv}|R_{i_{0}}$.  Likewise, the Markov chain $(i_{1},i_{2},\dots)$ defines a leaf $\ell_{i_{1}}$ of $\FF_{\vv}|R_{i_{1}}$ and the smooth, embedded arc $h(\ell_{i_{1}})\ss L$ contains $\ell_{i_{0}}$ as a proper subarc. Set $$\ell_{i_{n}}=R_{i_{n}}\cap h(R_{i_{n+1}})\cap\cdots$$ and obtain a strictly increasing nest
$$
\ell_{i_{0}}\ss h(\ell_{i_{1}})\ss h^{2}(\ell_{i_{2}})\ss\cdots\ss h^{n}(\ell_{i_{n}})\ss\dots
$$
of smoothly embedded curves.  Since $h$ is endperiodic, the two sequences of bottom and top edges of the quadrilaterals $\{h^{n}(R_{i_{n}})\}_{n\ge 0}$ both escape and the union $\ell_{\iota}$ of these arcs is a one-one, smoothly immersed copy of $\R$. 
\end{proof}

\begin{rem}
Since $h$ is only a diffeomorphism,  not a hyperbolic isometry, $\ell_{\iota}$ is not generally a geodesic, although it will repeatedly cross the quadrilaterals $R_{i}$ and $h(R_{j)})$ in geodesic arcs which are leaves of $\FF_{\vv}$.
\end{rem}

\begin{rem}
The curves $\ell_{\iota}$ are in one-one correspondence with the curves $\ell^{\g}_{\iota}\in\Lambda^{\g}_{\pm}$ of Proposition~\ref{infint}, each being uniquely determined by the Markov chain $\iota$.  This is crucial for proving that the bilamination $(\Lambda_{+},\Lambda_{-})$ that we construct satisfies the axioms.
Note that, despite the notation $\ell^{\g}_{\iota}$, we have not yet proven that this curve is the geodesic tightening of $\ell_{\iota}$. This will be proven in Lemma~\ref{basiclem}. Meanwhile, the correspondence is via the index $\iota$.
\end{rem}

Let   $\Gamma_{\pm}=\Lambda_{\pm}\cup\XX_{\mp}$ .

\begin{lemma}\label{bilams}

 $(\Lambda_{+},\Lambda_{-})$ and $(\Gamma_{+},\Gamma_{-})$ are smooth bilaminations. 

\end{lemma}

\begin{proof} 
Let $\iota=(i_{0},i_{1},\dots,i_{n},\dots)$ be a sequence with each $i_{n}\in\{1,2,\dots,k\}$ and with the property that $R_{i_{n}}\cap h(R_{i_{n+1}})\ne\0$, for all $n\ge0$.  Since $h:\RR\to h(\RR)$ is a diffeomorphism and preserves the smooth foliations $(\FF_{\hh},\FF_{\vv})$, every leaf $\ell_{\iota}$ of $\Lambda_{+}$ intersects a geodesic quadrilateral $R_{i}$, if at all, in plaques of $\FF_{\vv}$.  Via $h$, we define $(\FF^{i_{n}}_{\hh},\FF^{i_{n}}_{\vv})$ on each $h^{n}(R_{i_{n}})$, $n\ge0$, and remark that $\ell_{\iota}\cap h^{n}(R_{i_{n}})$ is a set of plaques of $\FF^{i_{n}}_{\vv}$ in $h^{n}(R_{i_{n}})$. Since $\FF^{i_{n}}_{\vv}$ and $\FF^{i_{n}}_{\hh}$ are transverse  smooth foliations of $h^{n}(R_{i_{n}})$, we see that the interiors of the sets $h^{n}(R_{i_{n}})$ are  laminated charts for $\Lambda_{+}$ belonging to the maximal smooth atlas of $L$.   That is, we have produced a smooth partial laminated atlas for $\Lambda_{+}$.  Similarly, we produce a smooth partial  laminated atlas for $\Lambda_{-}$.  Since the smooth foliations in each of our charts are transverse, one containing all plaques of $\Lambda_{+}$ meeting the chart and the other containing all plaques of $\Lambda_{-}$ meeting the chart, we have produced a smooth partial bilaminated atlas $\A$ for $(\Lambda_{+},\Lambda_{-})$.  

  By the definition of $\XX_{\pm}$ we see that the intersection of any of its leaves with a chart in $\A$ is also a plaque of the appropriate foliation.  That is, $\A$ consists of smooth bilaminated charts for $(\Gamma_{+},\Gamma_{-})$, but does not cover $|\Gamma_{+}|\cup|\Gamma_{-}|$.  Any point of $|\Gamma_{+}|\cup|\Gamma_{-}|$ outside of $\bigcup_{U\in\A}U$ either lies on a unique leaf of $\XX_{+}$ or on a unique leaf of $\XX_{-}$ or is a point of intersection of a leaf of $\XX_{+}$ and a leaf of $\XX_{-}$.   In all of these cases, constructing a smooth, bilaminated chart about $x$ for $(\Gamma_{+},\Gamma_{-})$ is easy since the leaves of $\XX_{\pm}$ are isolated from each other and from $|\Lambda_{+}|\cup|\Lambda_{-}|$.  This gives a smooth, partial bilaminated atlas for $(\Gamma_{+},\Gamma_{-})$.
\end{proof}

The pair $(\Lambda_{+}\cup\JJ_{-},\Lambda_{+}\cup\JJ_{+})$ may not be a bilamination. In fact, an escaping component of a positive juncture, if any, would  coincide with an  escaping component of a negative juncture. However, the following is an obvious corollary.

\begin{cor}\label{julams}

$\Lambda_{+}\cup\JJ_{-}$ and $\Lambda_{-}\cup\JJ_{+}$ are each smooth laminations.

\end{cor}

Let  $\UU_{+}$ (respectively  $\UU_{-}$) denote the positive (respectively  negative) escaping set for $h$. That is, $x\in\UU_{+}$ if and only if $\{h^{n}(x)\}_{n\ge 0}$ escapes and $x\in\UU_{-}$ if and only if $\{h^{n}(x)\}_{n\le 0}$ escapes.  These sets are open and 
  each component of $\XX_{\pm}$ lies in $\UU_{\pm}$.  In what follows, we use $\CC X$ to denote the complement in $L$ of a subset $X$.

\begin{lemma}\label{CCUU}
 $\CC\UU_{\pm} = |\Lambda_{\mp}|$.
\end{lemma}

\begin{proof}
We prove that $\CC\UU_{-} = |\Lambda_{+}|$. The proof that $\CC\UU_{+} = |\Lambda_{-}|$ is similar. Let $x\in\CC\UU_{-}$. Then no point in the orbit of $x$ is in $\UU_{-}$ and, by applying a large enough negative power of $h$, we can  assume that $x\in K$. For all $n\ge0$, $x\notin W_{N+n}\ss\UU_{-}$. (Recall the definition of the sets $W_{n}$ just before Definition~\ref{quadltrl}.) Thus, for each $n\ge0$,  there is a unique $1\le i_{n}\le k$ such that $x\in h^{n}(R_{i_{n}})\cap K$. Consequently, $x\in\ell_{\iota}\in\Lambda_{+}$ where $\iota = (i_{0},i_{1},\ldots,i_{n},\ldots)$.  Thus, $\CC\UU_{-}\sseq|\Lambda_{+}|$ and the reverse inclusion is obvious.
\end{proof}

\begin{lemma}\label{front}
The frontier of $\UU_{\mp}$ is exactly $|\Lambda_{\pm}|$ and $\ol{|\XX_{\mp}|}\sm |\XX_{\mp}| = |\Lambda_{\pm}|$.
\end{lemma}

\begin{proof}
  By Lemma~\ref{CCUU},  $\UU_{-}\cap|\Lambda_{+}| = \0$.
On the other hand, if $x\in|\Lambda_{+}|$, then it lies in a segment of a leaf $\ell_{\iota}$ in $h^{k}(R_{i_{k}})$ for arbitrarily large values of $k$.  Consequently, every neighborhood of $x$ meets a leaf of $\XX_{-}$, hence meets $\UU_{-}$. Thus $|\Lambda_{+}|\sseq  \fr\UU_{-}$. By Lemma~\ref{CCUU}, it follows that $|\Lambda_{+}| =  \fr\UU_{-}$. A similar argument proves that $|\Lambda_{-}|=\fr\UU_{+}$.  For the second assertion, the definition of $\Lambda_{+}$ makes it clear that $|\Lambda_{+}|\sseq \ol{|\XX_{-}|}\sm |\XX_{-}|$.  For the reverse   inclusion, let $\{x_{n}\}_{n=1}^{\infty}$ be a sequence of points in distinct leaves of $\XX_{-}$ converging to a point $x\in L$.  Since the negative $h$-junctures do not accumulate in $\UU_{-}$, $x\notin\UU_{-}$ and so $x\in\fr\UU_{-}=|\Lambda_{+}|$.  
\end{proof}

\begin{cor}\label{Lclsd}
The laminations $\Lambda_{\pm}$ and  $\Gamma_{\pm}$ are closed.
\end{cor}

\begin{proof}
Since the frontier of a set is a closed set it follows that the laminations $\Lambda_{\pm}$ are closed laminations. Further, the closure of $|\XX_{\mp}|$ is exactly $|\Lambda_{\pm}|\cup|\XX_{\mp}| = |\Gamma_{\pm}|$ so the lamination $\Gamma_{\pm}$ is a closed lamination.
\end{proof}

By construction, each leaf of $\XX_{\pm}$ is homotopic to a unique leaf of $\XX^{\g}_{\pm}$.  In fact, the components of  junctures in distinguished neighborhoods of ends  are the same.

\begin{lemma}\label{nodigs}

 The leaves of $\Lambda_{+}$ cannot intersect those of\, $\JJ_{+}$ so as to form digons.  The leaves of $\Lambda_{-}$ cannot intersect those of\, $\JJ_{-}$ so as to form digons.

\end{lemma}

\begin{proof}
First note that, if a leaf of $\Lambda_{+}$ intersects a leaf of $\JJ_{+}$ so as to form a digon, then a leaf of $\Lambda_{+}$ intersects $\fr_{+}K$ so as to form a digon.  Just apply a suitable power of $h$.  But this would imply that  an edge of $\bd R_{i}$ or of $\bd h(R_{i})$ would intersect $\fr_{+}K$ so as to form a digon, some $i\in\{1,2,\dots,k\}$.  Since these geodesic quadrilaterals are the same as the ones for the geodesic laminations, this would contradict Axiom~\ref{trnsvrs} for the geodesic laminations.  A similar proof works for the laminations $\Lambda_{-}$ and $\JJ_{-}$. 
\end{proof}

\begin{cor}\label{elnodigs}

 The leaves of $\Lambda_{+}$ cannot intersect those of $\Lambda_{-}$ so as to form digons.
 
 \end{cor}
 
 \begin{proof}
If a leaf $\ell_{+}$ of $\Lambda_{+}$ intersects a leaf $\ell_{-}$ of $\Lambda_{-}$ so as to form a digon, then a suitable locally uniform approximation of $\ell_{+}$ by leaves of $\XX_{-}$ will give a leaf of $\XX_{-}$ that intersects $\ell_{-}$ so as to form a digon.  
\end{proof}

The leaves of $\Lambda_{\pm}$ correspond one-one with the leaves of $\Lambda^{\g}_{\pm}$ by $\ell_{\iota}\leftrightarrow\ell^{\g}_{\iota}$ and the leaves of $\XX_{\pm}$ correspond one-one to those of $\XX^{\g}_{\pm}$ by homotopy.
This establishes a canonical one-one correspondence $\Gamma_{\pm}\leftrightarrow\Gamma^{\g}_{\pm}$.  Denote this correspondence by $\gamma\leftrightarrow\gamma^{\g}$.  The following shows that, in conformity with the notation of Section~\ref{uniq}, $\gamma^{\g}$ is the geodesic tightening of $\gamma$.

\begin{lemma}\label{basiclem}
The laminations $\Gamma_{\pm}$ are pseudo-geodesic and, 
under the correspondence $\gamma\leftrightarrow\gamma^{\g}$, the completed lifts $\wh\gamma$ and $\wh\gamma^{\g}$ have the same  endpoints on $S_{\infty}$.
\end{lemma}

\begin{proof}
By Lemma~\ref{cor3ccepstein}, the leaves of $\XX_{\pm}$ are pseudo-geodesics.  Let 
$$
\iota=(i_{0},i_{1},i_{2},\dots,i_{n},\dots)
$$ 
be a Markov chain. Let $\delta_{0} = \delta^{\g}_{0}$ and inductively define the edge $\delta_{n}\ss h^{n}(\fr_{+}K)$ of $h^{n}(R_{i_{n}})$ exactly as the edge $\delta^{\g}_{n}$ was defined in the proof of Proposition~\ref{infint}. The lift $C_{0} = C^{\g}_{0} = \wt R_{i_{0}}$ of $R_{i_{0}}$ to $\wt L$ determine the lifts $C_{n}$ of $h^{n}(R_{i_{n}})$ with edge $\wt\delta_{n}$ a lift of $\delta_{n}$. In fact $\delta_{n}$ and $\delta^{\g}_{n}$ are in the same lift $\wt\sigma^{\g}_{n}$ of a component of $h^{n}(\fr_{+}K)$  from which it follows that $\ell_{\iota}$ and $\ell^{\g}_{\iota}$ share the endpoint $a$.  Similarly, starting with the other edge $\gamma_{0} = \gamma^{\g}_{0}\ss\fr_{+}K$ of $R_{i_{0}}$ one shows that $\ell_{\iota}$ and $\ell^{\g}_{\iota}$ share the endpoint $b$. Thus $\ell_{\iota}$ is a pseudo-geodesic sharing endpoints with $\ell^{\g}_{\iota}$.
\end{proof}

At this point we have verified that  $(\Gamma_{+},\Gamma_{-})$ is a smooth bilamination preserved by the smooth endperiodic diffeomorphism  $h$. It remains to verify that $(\Lambda_{+},\Lambda_{-})$ is a Handel-Miller pseudo-geodesic bilamination associated to $h$ (and, thus, to $h^{\g}$), that is that $(\Lambda_{+},\Lambda_{-})$ satisfies the axioms:\vs

\ni\textbf{Axiom~\ref{muttran}}.  By   Lemma~\ref{bilams}, $(\Lambda_{+},\Lambda_{-})$ is a bilamination. By Corollary~\ref{Lclsd}, $(\Lambda_{+},\Lambda_{-})$ is closed.  By Lemma~\ref{basiclem}, the elements of $\Lambda_{\pm}$ are pseudo-geodesics.  The fact that the leaves of $\Lambda_{\pm}$ are disjoint from $\bd L$ is obvious from the construction.

\ni\textbf{Axiom~\ref{eachmeets}}.   By Corollary~\ref{elnodigs}, the  leaves of $\wt\Lambda_{+}$ and $\wt\Lambda_{-}$ can only intersect in a single point.

\ni\textbf{Axiom~\ref{ecorr}}.  By Lemma~\ref{basiclem}, the bilamination $(\Lambda_{+},\Lambda_{-})$ has the endpoint correspondence property with respect to $f$.

\ni\textbf{Axiom~\ref{trnsvrs}}.  By Lemma~\ref{NNbiJJ} there exists a set $\NN$ of $f$-juncture components and bijection $\iota:\NN_{\pm}\to\JJ_{\pm}$.  By Lemmas~\ref{cor3ccepstein} and~\ref{basiclem}, $\Lambda_{\pm}\cup\JJ_{\mp}$ are sets of pseudo-geodesics. By Corollary~\ref{julams},  the elements of each of the sets $\Lambda_{\pm}\cup\JJ_{\mp}$ are disjoint.    Item~(1) follows.    Item~(2) follows since, by construction, the elements of $\Lambda_{\pm}$ are transverse to the elements of $\JJ_{\mp}$.   By Lemma~\ref{nodigs}, no element of $\JJ_{\pm}$ meets an element of $\Lambda_{\pm}$ to form a digon.   Thus, item~(3) is true.

\subsection{Proving Theorem~\ref{HMsmooth} if there are principal regions}\label{SecPrinReg}

We will sketch the way the previous argument needs to be modified, leaving many details for the reader.

Let $\PP$ denote the union of the positive principal regions.  For $N\ge1$ large enough, the components of $K\sm (W_{N}\cup\PP)$ are   geodesic quadrilaterals, each having a pair of opposite edges that are subarcs of $\fr_{+}K$.  But the other pair of opposite edges are either both subarcs of $\JJ_{N}^{-}$ or one is a subarc of $\JJ_{N}^{-}$ and the other a subarc of a border leaf of a positive principal region.  For the structure of principal regions and their nuclei, the reader may wish to review Section~\ref{crown}.

Let $P$ be a positive principal region.  It has a dual negative principal region $P'$ which shares a compact nucleus $N_{P}=N_{P'}$ with $P$. See Figure~\ref{dual} where the dashed curves are border curves of $P'$, the solid ones border curves of $P$.

\begin{figure}[h]
\begin{center}
\begin{picture}(300,170)(-45,-170)

\rotatebox{270}{\includegraphics[width=170pt]{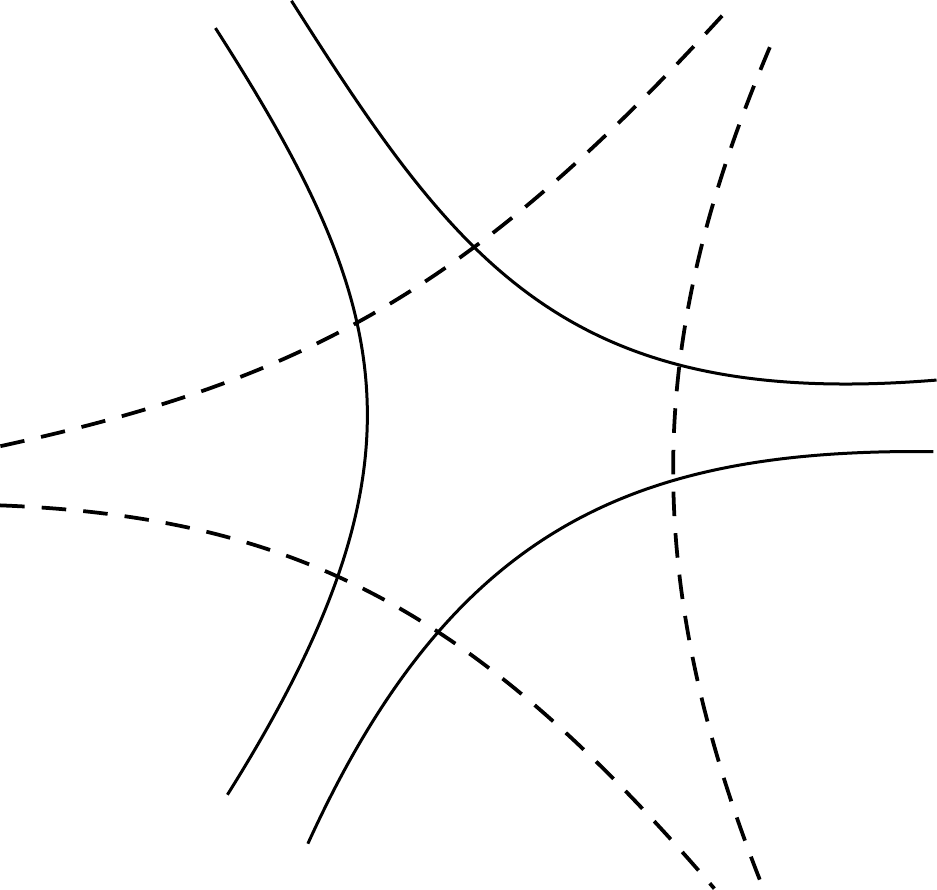}}
\put(-83,-95){\Small$N_{P}$}

\end{picture}
\caption{The dual principal regions $P$ and $P'$.}\label{dual}
\end{center}
\end{figure}

Actually, Figure~\ref{dual} depicts only the parts of $P$ and $P'$ including one boundary curve $\gamma$ of $N_{P}$.  There may be finitely many distinct components of $\bd N_{P}$.  
The curve $\gamma$ is made up alternately of $n_{\gamma}$ subarcs of $\delta P$ and $n_{\gamma}$ subarcs of $\delta P'$, $n_{\gamma}$ being a positive integer depending on $\gamma$, not on $P$ alone. In the figure we draw the case $n_{\gamma}=3$.  Out of  $\gamma$ there radiate $n_{\gamma}$ simply connected ``arms'' of $P$ and $n_{\gamma}$ such arms of $P'$.   Since we are in the case of geodesic laminations, the lifts of these arms to $\wt L$ are  cusps.    

The set $(P\sm\intr N_{P})\cap K$ may have infinitely many components, all but finitely many of which are geodesic quadrilateral components of the intersection of arms of $P$ with $K$.
The nonquadrilateral components of $(P\sm\intr N_{P})\cap K$ are each equal to the union of one of the components $\gamma$ of $\bd N_{P}$ and the ``stumps'' of the $n_{\gamma}$  arms (cf.~Definition~\ref{stump}).  In Figure~\ref{NucleusStumps}, we have drawn the nucleus $N_{P}$ (shaded) and stumps of  $P$ cut off by one of the curves $\gamma$ with $n_{\gamma}=3$.  The boundary curve $\gamma$ of $N_{P}$ is drawn in boldface.

\begin{figure}[h]
\begin{center}
\begin{picture}(300,170)(-45,-170)

\rotatebox{270}{\includegraphics[width=170pt]{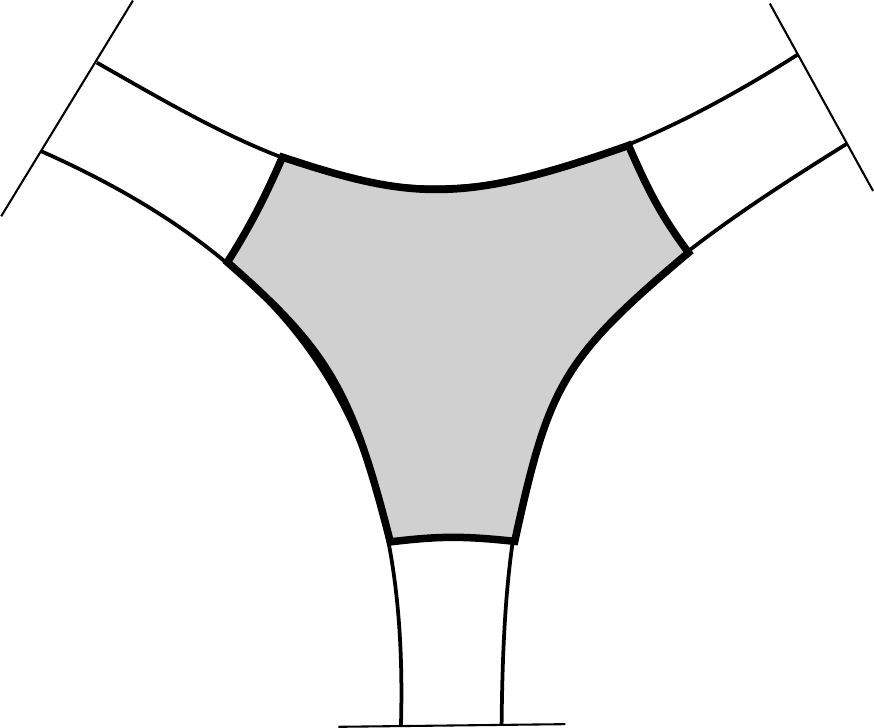}}
\put(-107,-80){\Small$N_{P}$}
\put(7,-20){\Small$\fr_{+}K$}
\put(-210,-20){\Small$\fr_{+}K$}
\put(-70,-174){\Small$\fr_{+}K$}
\put(-70,-95){\Small$\gamma$}

\end{picture}
\caption{The nucleus and stumps of a crown set.}\label{NucleusStumps}
\end{center}
\end{figure}

\begin{defn}[stump]\label{stump}
If $P$ is a positive principal region, we will call   the rectangular portion of an arm of $P$ that is  between the nucleus and $\fr_{+} K$ the  \emph{stump}  of the arm.  The stump  is bounded by the arc of $| \Lambda_{-}|$ separating the stump from the nucleus, two  arcs of $ |\Lambda_{+}|$, and an arc of $\fr_{+} K$.  It is simply connected.  The stumps of the arms of the dual negative principal region $P'$ are defined similarly.
\end{defn}

In Figure~\ref{RQ}, we have attached the $n_{\gamma}=3$ geodesic quadrilaterals $R_{i}^{\gamma}$, each with one edge in one of the three depicted border leaves of $P$, the opposite edges being arcs in $\JJ_{N}^{-}$. The other pair of opposite edges lie in $\fr_{+}K$.  These geodesic quadrilaterals and the stumps unite to produce an annulus $A_{\gamma}$ which we have shaded.  With dotted lines representing arcs in the border of $P' $ we have cut off three other geodesic subquadrilaterals of $A_{\gamma}$ which will be denoted as $Q_{j}^{\gamma}$. Note that the $R_{i}^{\gamma}$'s are pairwise disjoint, as are the $Q_{j}^{\gamma}$'s, but  each $R_{i}^{\gamma}$ overlaps two $Q_{j}^{\gamma}$'s.  

\begin{figure}[h]
\begin{center}
\begin{picture}(300,170)(-45,-175)

\rotatebox{270}{\includegraphics[width=170pt]{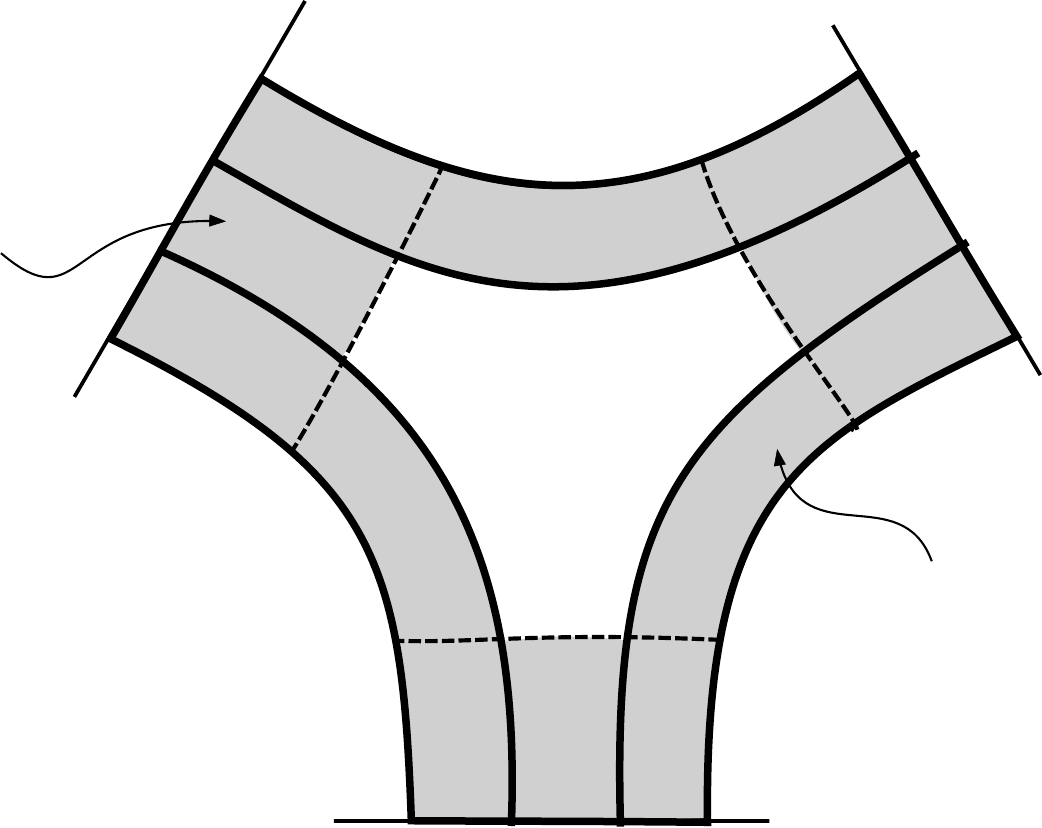}}
\put(-20,-30){\Small$\fr_{+}K$}
\put(-193,-30){\Small$\fr_{+}K$}
\put(-55,-174){\Small$\fr_{+}K$}
\put(-106,-28){\Small$\JJ_{N}^{-}$}
\put(-155,-122){\Small$\JJ_{N}^{-}$}
\put(-55,-122){\Small$\JJ_{N}^{-}$}
\put(-20,-120){\Small$R_{i}^{\gamma}$}
\put(-226,-53){\Small$Q_{j}^{\gamma}$}

\end{picture}
\caption{The annulus $A_{\gamma}$}\label{RQ}
\end{center}
\end{figure}

As $\gamma$ varies over boundary components of the nuclei of the finitely many principal regions in $\PP$, we get a family of geodesic quadrilaterals $R_{i}^{\gamma}$, $1\le i\le n_{\gamma}$.  There are also the ``normal'' geodesic quadrilateral components of $K\sm(W_{N}\cup\PP)$ with a pair of opposite edges in $\JJ_{N}^{-}$. These latter are simply indexed $R_{i}$, $1\le i\le k$  as before. The geodesic quadrilaterals of both types have a pair of opposite edges that are subarcs of $\fr_{+}K$.  As before, if $N$ has been chosen large enough, the $h^{\g}$-images of these geodesic quadrilaterals have connected intersection (possibly empty) with each  of  $R_{i}$ and $R_{i}^{\gamma}$, completely crossing any of these latter geodesic quadrilaterals that they meet.  

\begin{rem}
The reader should give some thought to the family of geodesic quadrilaterals consisting of $Q_{j}^{\gamma}$, $1\le j\le n_{\gamma}$ and $R_{j}$, $1\le j\le k$.  These are pairwise disjoint and each  has connected intersection with the geodesic quadrilaterals $h^{\g}(Q^{\gamma}_{i})$ and $h^{\g}(R_{i})$, completely crossing whichever of the latter that it meets.  This is entirely analogous to the situation without principal regions and will be used to generate the lamination $\Lambda_{-}$.
\end{rem}

While it is clear how the $h^{\g}$-images of  geodesic quadrilaterals intersect  geodesic quadrilaterals, one also should note that there is often an annular  component of $A_{\gamma}\cap h^{\g}(A_{\gamma'})$, for boundary circles $\gamma$ and $\gamma'$ of nuclei, as indicated in Figure~\ref{Annuli}, where the dotted lines together with $\gamma=h^{\g}(\gamma')$ bound $h^{\g}(A_{\gamma'})$.

\begin{figure}[h]
\begin{center}
\begin{picture}(300,170)(-45,-167)

\rotatebox{273}{\includegraphics[width=170pt]{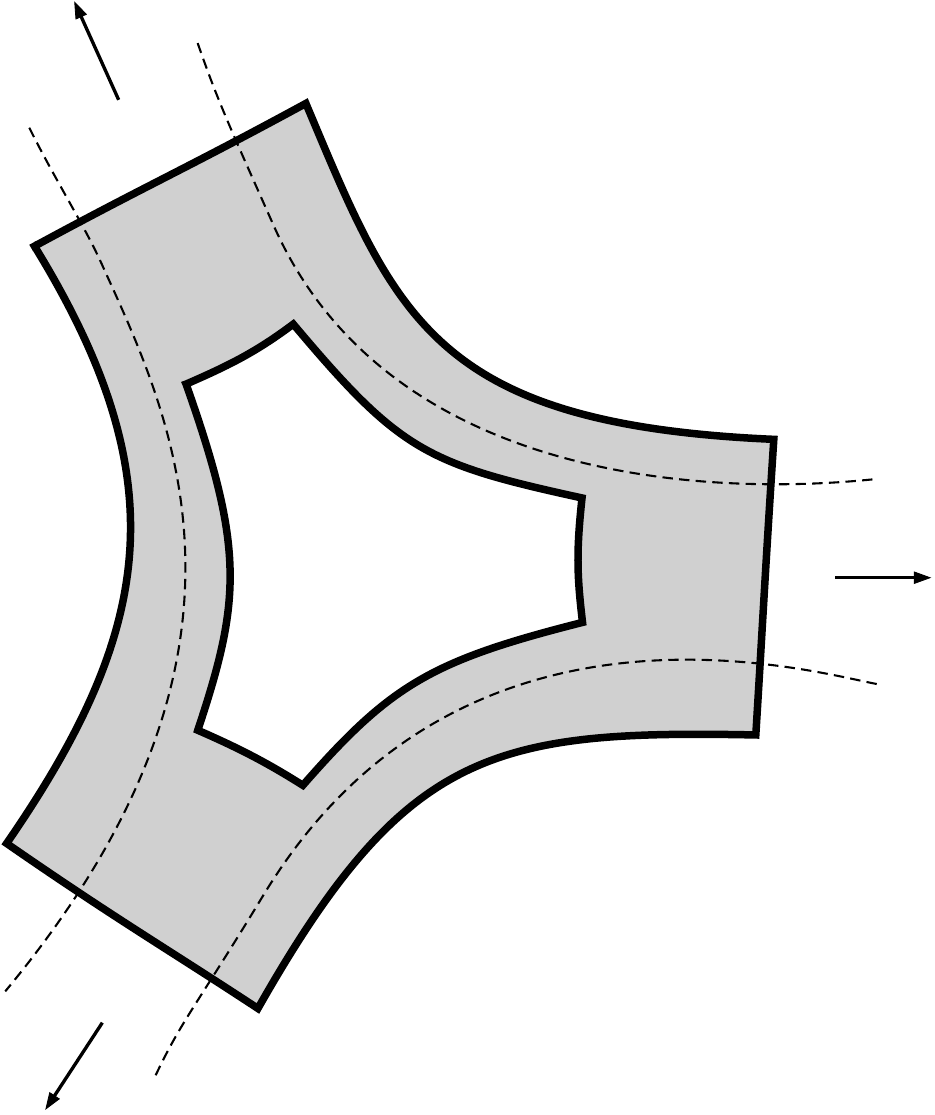}}

\end{picture}
\caption{An annular component of $A_{\gamma}\cap h^{\g}(A_{\gamma'})$.}\label{Annuli}
\end{center}
\end{figure}

Set $\RR$ equal to the union of all $R_{i}$ and all $A_{\gamma}$.
In analogy with the former construction, we define   a pair of transverse, smooth geodesic foliations $\FF_{\hh}$ and $\FF_{\vv}$ on $\RR\cup h^{\g}(\RR)$.  The foliation $\FF_{\hh}$  incorporates the edges in $\fr_{+}K$ and their $h^{\g}$-images as leaves. The foliation $\FF_{\vv}$ incorporates the edges in $\JJ_{N}^{-}$ and in $\delta\PP$ and their $h^{\g}$-images.  Furthermore, the geodesic quadrilaterals and annuli and their $h^{\g}$-images can be fattened to obtain a neighborhood $V$ of $\RR\cup h^{\g}(\RR)$ over which the foliations are extended.  Details, which are analogous to the case in which there are no principal regions, are left to the reader. Figure~\ref{Annulus} depicts the foliated neighborhood of one of the annuli. 

\begin{figure}[h]
\begin{center}
\begin{picture}(300,170)(-45,-170)

\rotatebox{270}{\includegraphics[width=170pt]{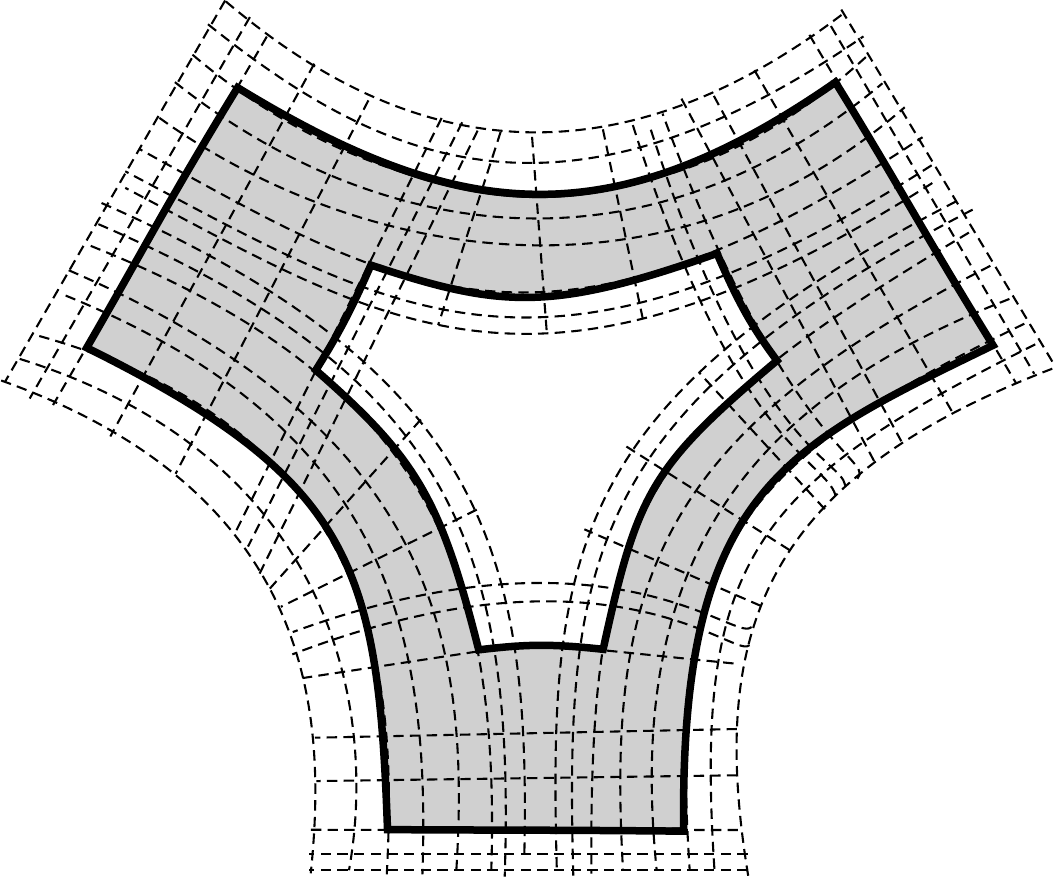}}

\end{picture}
\caption{Geodesic foliations of a neighborhood of $A_{\gamma}$.}\label{Annulus}
\end{center}
\end{figure}

 This smooth geodesic grid is again used to define a diffeomorphism which we will denote by  $h_{\RR}:\RR\to h^{\g}(\RR)$ and will extend over an open neighborhood of $\RR$. Notice that the geodesic quadrilaterals $Q_{j}^{\gamma}$ are also in the domain of $h_{\RR}$.  If $R$ is any of the geodesic quadrilaterals $R_{i}$, $R_{i}^{\gamma}$, or $Q_{i}^{\gamma}$, the grid defines smooth coordinates $(x,y)$ on $R$. Here, the leaves of $\FF_{\hh}$ are level sets of $x$ and the leaves of $\FF_{\vv}$ are level sets of $y$. On overlapping geodesic quadrilaterals of $\RR\cup h^{\g}(\RR)$ we can again choose the coordinates so that the first derivatives of the change of coordinate functions are $\pm1$.  The diffeomorphism $h_{\RR}|R$ has the form $h_{\RR}(x,y)=(\xi(x),\zeta(y))$.  As in Proposition~\ref{localprop}, $\xi$ and $\zeta$ can be chosen so that the following are true.
 
 \begin{prop}\label{localprop'}
 On $R=R_{i}$, $1\le i\le k$, 
 $$0< |d\xi/dx| <c<1,\qquad |d\zeta/dy| >b>1.$$
 On $R=R_{i}^{\gamma}$, the first of these inequalities holds and on $R=Q_{i}^{\gamma}$ the second holds.
 \end{prop}

 \begin{lemma}
 The homeomorphism $h^{\g}$ is isotopic to a  homeomorphism  which agrees with $h^{\g}$ outside a neighborhood of $\RR$ and with $h_{\RR}$ inside that neighborhood.
 \end{lemma}
 
 \begin{proof}
For each geodesic quadrilateral $R_{i}$, $1\le i\le k$, and $R_{i}^{\gamma}$ this is proven exactly as before.  On $A_{\gamma}$, there remain geodesic quadrilaterals  $D$ in each stump where the isotopy has not been defined.  Looking carefully at Figure~\ref{Annulus}, one sees that an application of Theorem~\ref{3.1} isotopes $h^{\g}$ to agree with $h_{\RR}$ on the boundary of a  geodesic quadrilateral in $V$ which covers $D$, this without affecting the isotopies already performed.  An application of Alexander's trick completes the isotopy. 
 \end{proof}
 
 We therefore assume that $h^{\g}$ agrees with the diffeomorphism $h_{\RR}$ in a neighborhood of $\RR$.
  One now proceeds exactly as before, using~\cite{ha:smooth} and the smooth versions of Theorem~\ref{2.1} and~\ref{3.1}, to prove the following.
 
 \begin{lemma}
 There is an isotopy of $h^{\g}$, which is constant in a neighborhood of $\RR$ and on  $\bd L$, to a diffeomorphism $h:L\to L$  such that $h|\tau^{\g}=h^{\g}|\tau^{\g}$, where $\tau^{\g}$ is a component of $\JJ_{i}^{-}$, $i\le N$, or of $\JJ_{i}^{+}$, $i\ge 0$.  In particular, $h$ is endperiodic.
 \end{lemma}
 
 At this point all of the machinery has been set up for defining $\Lambda_{+}$ using the geodesic quadrilaterals $R_{i}^{\gamma}$ and $R_{i}$ and positive powers of $h$.  One similarly produces $\Lambda_{-}$ using the geodesic quadrilaterals $R_{i}$ and $Q_{i}^{\gamma}$ and negative powers of $h$. Throwing in the nonescaping components of junctures gives the smooth laminations $\Gamma_{\pm}$ and verification of the axioms proceeds essentially as before.   Theorem~\ref{HMsmooth} is proven.  


\section{The transfer theorem}\label{AWH}

The principal applications of Handel-Miller theory are to the endperiodic automorphisms  that arise in foliations of $3$-manifolds by surfaces. See, for example,~\cite{fe:thesis,fe:endp}.  We introduce here the basics, setting the stage for such applications as in~\cite{cc:cone,cc:almostnohol}.    The main new result in this section will be Theorem~\ref{transfer}, the ``transfer theorem'',  which is absolutely fundamental to~\cite{cc:cone,cc:almostnohol}.

\subsection{Depth~one foliations}\label{dpthonefol}

Let $M$ be a compact, connected sutured $3$-manifold with boundary.  We do not require $M$ to be orientable, but we will only be considering smooth, codimension~$1$ foliations $\FF$ of $M$ which are transversely orientable.  The boundary of $M$ will decompose into 
$$
\bd M=\tb M\cup\trb M,
$$
where the components of $\tb M$ (the \emph{tangential} boundary) are compact leaves of $\FF$ and $\trb M$ (the \emph{transverse} boundary) is transverse to $\FF$.  The components of $\trb M$ must be annuli, tori and/or Klein bottles.  M\"obius strips are ruled out by the transverse orientability of $\FF$.  The annuli will meet $\tb M$ along corners that are convex with respect to the foliation.

\begin{defn}[taut foliation]\label{tautfoln}
The foliation $\FF$ is taut if each leaf meets either a closed transversal to $\FF$ or an arc transverse to $\FF$ with endpoints in $\tb M$.
\end{defn}

Taut foliations are very important in the theory of $3$-manifold topology.  For instance, the foliation can have no Reeb components, hence well known theorems of S.~P.~Novikov~\cite{nov} imply that  the leaves are $\pi_{1}$-injective and no closed transversal to $\FF$ is null homotopic.  It follows easily that the lift $\wt\FF$ of $\FF$ to the universal cover $\wt M$ has all leaves simply connected and, unless $M=S^{2}\x\SI$ with the product foliation,  that $\wt M$ itself is contractible.   

The foliation $\FF$ induces foliations of the components of $\trb M$.  Even when $\FF$ is taut, $\FF|\trb M$ might have 2-dimensional Reeb components.  We need to rule this out also.

Set $M^{\o}=M\sm\tb M$.

\begin{defn}[depth one foliation]\label{dpthonefoln}
The foliation $\FF$ is depth one if $\FF|M^{\o}$ is a fibration of $M^{\o}$ over $\SI$.
\end{defn}

Here it is possible that $M=F\x[-\infty,\infty]$ and the foliation is transverse to the compact interval fibers.  Such foliations, called foliated products of depth one, are well understood and easily classified.

\begin{rem}
Since $\FF $ is transversely oriented, some of the components of $\tb M$ are oriented transversely into $M$ and some out of $M$.  The annular components of $\trb M$ separate inwardly oriented components of $\tb M$ from outwardly oriented ones.  As $\FF$ varies, we keep the transverse orientations of the components of $\tb M$ fixed. This then is a constraint on the transverse orientations of these foliations.  In the case that $M$ is orientable, the structure $(M,\tb M,\trb M)$ is a \emph{sutured manifold} in the sense of D.~Gabai~\cite{ga1}.  In Gabai's notation,  the sutured manifold is denoted by $(M,\gamma)$, $\trb M=\gamma=A(\gamma)\cup T(\gamma)$ and $\tb M=R(\gamma)=R_{+}(\gamma)\cup R_{-}(\gamma)$. Here, $A(\gamma)$ is the union of the annular components of $\trb M$, $T(\gamma)$ the union of the toroidal components, $R_{+}(\gamma)$ is the union of the outwardly oriented components of $\tb M$ and $R_{-}(\gamma)$ the union of the inwardly oriented ones.
\end{rem}

\subsection{Monodromy}\label{defmonod}

It is possible to find a  $1$-dimensional foliation $\LL$ transverse to $\FF$.  For instance, construct a smooth, nowhere zero vector field $v$ on $M$, transverse to the leaves of $\FF$, tangent to $\trb M$  and oriented coherently with the transverse orientation of $\FF$.  This is easily done locally and the local fields are assembled into a global one by a suitable smooth partition of unity.  The integral curves to $v$ are the leaves of $\LL$.    This transverse foliation is smooth.  It is also possible to construct transverse foliations $\LL$ that are only $\CO$.  In any event, $\LL$ is oriented by the transverse orientation of $\FF$.  

If $\trb M$ has no $2$-dimensional Reeb components, the foliation $\LL|A$ induced on annular components $A$ of $\trb M$ can always be taken to be the product foliation by compact intervals.

Suppose $\FF$ is a depth one foliation of $M$ (Definition~\ref{dpthonefoln}) and let $L$ be a noncompact leaf of $\FF$.  That is, $L$ is a fiber of the fibration $M^{\o}\to\SI$ induced by $\FF$. For each point $x\in L$, let $\ell_{x}$ be the leaf of $\LL$ through $x$.  Follow $\ell_{x}$ starting at $x$ in the direction given by its orientation until the first time a point $y\in L$ is encountered. The fact that such a point $y\in L$ exists for each $x\in L$ follows from the compactness of $M$~\cite[Lemma~4.2]{fe:jdg}. The map $x\to y$ defines the first return map 
$$
f:L\to L.
$$
If $\LL$ is smooth,  $f$ is a diffeomorphism.  If $\LL$ is only $\CO$, $f$ is a homeomorphism.

\begin{defn}[monodromy]\label{monofoln}
The automorphism $f:L\to L$ is called the monodromy  induced by $\LL$.
\end{defn}

Sometimes we will denote $\LL$ by $\LL_{f}$.  One should note, however, that while $f$ is determined by $\LL_{f}$, distinct transverse foliations may determine the same monodromy.  

\subsubsection{Endperiodic monodromy}

\begin{lemma}\label{endpmono}
Let $M$ be a compact, connected sutured $3$-manifold and $\FF$  a taut, transversely oriented, depth one foliation of $M$ which induces no   $2$-dimensional Reeb components on $\trb M$. A leaf $L$ of $\FF|M^{\o}$ has only finitely many ends and the monodromy $f:L\to L$, defined by any transverse, $1$-dimensional foliation $\LL$, is an endperiodic automorphism.  
\end{lemma}

To prove Lemma~\ref{endpmono}, for any end $e\in\EE(L)$, one constructs the neighborhood $U_{e}$ of $e$ exactly as in the proof of Proposition~\ref{comcomp}  in Section~\ref{jnctrs}.  

\subsubsection{Realizing endperiodic monodromy}\label{realizing}

\begin{lemma}\label{realize}
Given an endperiodic automorphism $f:L\to L$ of a surface with finitely many ends, there exists a compact $3$-manifold $M$, a depth one foliation $\FF$ of $M$ and a transverse, $1$-dimensional foliation $\LL$ such that $L$ is homeomorphic to each noncompact leaf of $\FF$ and $f$ is the monodromy induced by $\LL$.  If $f$ is a diffeomorphism, $M$, $\FF$ and $\LL$ are smooth.  If $f$ is only a homeomorphism, $M$ and $\FF$ are smooth, but generally $\LL$ is not.
\end{lemma}

The proof of Lemma~\ref{realize} is standard.  One first constructs the suspension of $f$, 
$$M^{\o}=L\x[0,1]/\{(x,1)\equiv(f(x),0)\}.$$
The fact that $f$ is endperiodic, then enables one to glue on the compact leaves that form $\tb M$.

The annular components of $\trb M$ can arise in two ways. There might be an infinite sequence of compact components of $\bd L$ contained in $\UU$ permuted simply transitively by $f$ or a finite set of compact components of $\bd L$ contained in the set of principal regions permuted cyclically by $f$.  The $\LL$-saturation of any one of these circles contains the others and is an annular component of $\trb M$. There might be a necessarily finite set of noncompact components of $\bd L$ permuted cyclically by $f$.  The $\LL$-saturation of any of these contains the others and again, since there are no periodic points on such a component, we obtain an annular component of $\trb M$.  In these annuli, the leaves of $\LL$ are compact intervals, hence $\FF$ induces no $2$-dimensional Reeb components.  The tori and/or Klein bottles in $\trb M$ are induced similarly by finite $f$-cycles of compact components of $\bd L$.  

If  $\FF$ induces no 2-dimensional Reeb components on $\trb M$, each noncompact component of $\bd L$  joins a negative end to a positive end and  the monodromy of $\FF$ has no fixed points on $\bd L$. If  every component of $\tb M$ has negative Euler characteristic, $L$ is an admissible surface.  If $\FF$ is not a foliated product, the monodromy of $\FF$ is not isotopic to a translation.  Thus the theory of Sections~\ref{endp}~-~\ref{HMSM} applies to the monodromy of a depth one  leaf of a foliation satisfying Hypothesis~\ref{ongo}.

\begin{hyp}\label{ongo}
\textbf{Hereafter, unless we explicitly state otherwise, $M$ is a compact, connected sutured $3$-manifold with boundary, the foliation $\FF$ is smooth, transversely oriented, taut, depth~one, not a foliated product, induces no 2-dimensional Reeb components on $\trb M$, and every component of $\tb M$ has negative Euler characteristic}.

\end{hyp}

\subsection{The Transfer Theorem}\label{transfth}

We now turn to  the main goal of   Section~\ref{AWH}.  Suppose $M$ and $\FF$ satisfies Hypothesis~\ref{ongo}.

Let $f:L\to L$ be the monodromy associated to a transverse, $1$-dimensional foliation $\LL_{f}$.  Fix any standard hyperbolic metric $\mathfrak g$ on $L$ and recall that, if $(\Lambda_{+},\Lambda_{-})$ is a Handel-Miller pseudo-geodesic bilamination of $L$  associated to $f$ relative to this metric, it is Handel-Miller associated to $f$  relative to any other choice of standard hyperbolic metric (Corollary~\ref{metricindep}). Thus the following definition is independent of the choice of $\mathfrak g$.

\begin{defn}[Handel-Miller monodromy]\label{HMmonofoln}

Suppose that $\LL$ is a $1$-dimensional foliation transverse to $\FF$ which induces endperiodic monodromy $h:L\to L$,  isotopic to $f$ and preserving a Handel-Miller pseudo-geodesic bilamination $(\Lambda_{+},\Lambda_{-})$ associated to $f$. Then we say that $h$ is a Handel-Miller monodromy  for the depth one foliation $\FF$.
\end{defn}

Recall that the  bilamination $(\Lambda_{+},\Lambda_{-})$ is called a   Handel-Miller pseudo-geodesic bilamination $(\Lambda_{+},\Lambda_{-})$ associated  to $f$  if it satisfies the four axioms (Definitions~\ref{HMbilam}).

The following  is analogous to a result of Fried~\cite[p.~261]{FLP} for pseudo-Anosov monodromy in fibrations of hyperbolic $3$-manifolds and is key to proving the maximality of foliation cones defined by Handel-Miller monodromy (see Proposition~\ref{HMmax}).  

\begin{theorem}[Transfer Theorem]\label{transfer}
Let $\LL$ be a $1$-dimensional foliation transverse to $\FF$ and inducing Handel-Miller monodromy $h:L\to L$ on a depth one leaf $L$.  If $\FF'$ is another depth one foliation transverse to $\LL$ and satisfying \emph{Hypothesis~\ref{ongo}},  then $\LL$ induces Handel-Miller monodromy $h':L'\to L'$ on any depth one leaf $L'$ of $\FF'$. 
\end{theorem}

The proof of the transfer theorem is now our goal.  We first note that by Lemma~\ref{endpmono}, the monodromy $h':L'\to L'$ is an endperiodic automorphism. We  show how to produce an $h'$-invariant bilamination $(\Lambda'_{+},\Lambda'_{-})$ by transferring the bilamination $(\Lambda_{+},\Lambda_{-})$ from $L$ to $L'$ along $\LL$  and   then verify that this satisfies the axioms for a Handel-Miller pseudo-geodesic bilamination associated to $h'$.

\subsection{Transferring paths}

Given a parametrized path $s:[a,b]\to L$,  there are countably many ``transferred'' parametrized paths $s':[a,b]\to L'$ which are uniquely determined by $s'(a)$ and are obtained by projecting locally along the leaves of $\LL$.  This is a standard sort of local continuation process.  Likewise, a parametrized path $s:(-\infty,\infty)\to L$ transfers to parametrized paths $s':(-\infty,\infty)\to L'$, each uniquely determined by any one of its values. All of this works equally well for transfers of paths from $L'$ to $L$ and transfers of paths from $L$ to $L$ and from $L'$ to $L'$.  A choice of parametrization is convenient in defining the transferred paths, but the unparametrized paths underlying the parametrized transferred paths are independent of the choice. Thus, we will normally view the transfer operation as having to do with unparametrized paths.

\begin{rem}
If $s$ is a loop it is quite possible that some or all of its transferred paths open up. 
\end{rem}

\begin{defn}[transfer of a path]\label{transfpath}
The set of paths obtained by transferring  a path $s$ in $L$ to $L'$ will be called simply the transfer of $s$ to $L'$.  Similarly, one defines the transfer of a path $s$ in $L'$ to $L$, of a path in $L$ to $L$ and of a path  in $L'$ to $L'$.  We will distinguish \textbf{the} transfer of $s$, which is a set of paths, from  \textbf{a} transfer of $s$, which is a single path in the transfer of $s$.
\end{defn}

\begin{rem}
The union of the  paths  in the transfer of $s$ to $L'$  is exactly the set of points of intersections of $L'$ with leaves of $\LL$ which meet $s$. Similar remarks hold for the other three types of transfer.
\end{rem}

\begin{lemma}\label{transtrans}
The set of paths in the transfer of a  path $s$ on $L$ to $L'$ is permuted transitively by the map $h':L'\to L'$. Similar assertions hold for the other three types of transfer.
\end{lemma}

Indeed, $h'$ and $h$ are defined by flowing along $\LL$.

\begin{rem}
The transfer operation on a parametrized path is a continuation process using a parametrization of $s$ and gives a parametrization of $s'$.  Thus, a choice of orientation of $s$ induces canonically an orientation of any transfer $s'$.
\end{rem}

\begin{lemma}\label{nulltransf}
Let $s$ be a loop on $L$ and $s'$ a transferred path on $L'$. Then $s'$ is a  nullhomotopic loop on $L'$ if and only if $s$ is a nullhomotopic loop on $L$.  A similar assertion holds for the other three types of transfer.  
\end{lemma}

\begin{proof}
Suppose that $s$ is nullhomotopic on $L$. 
 If $s'$ opens, a classical construction (cf.~\cite[Lemma~3.3.7]{condel1}) gives a closed transversal to $\FF'$ which is homotopic in $M^{\o}$ to $s$.  By tautness, such a loop must be essential in $M$ and so $s$ is essential on $L$. This contradiction shows that $s'$ is a closed loop. Thus $s'$ is homotopic to $s$ in $M$ along the projecting arcs of $\LL$ and, again by tautness, it is nullhomotopic on $L'$.  The converse has the same proof.
\end{proof}

\begin{defn}[transfer of a set of curves]\label{trfamcur}
Given a set $\AAA$ of curves on $L$, the union of the transfers of the elements of $\AAA$ will be called the transfer $\AAA'$ of $\AAA$.
\end{defn}

\begin{rem}
Be careful of the logic here.  The transfer of a path is a \emph{set} of paths.  The union of the transfer of a set of paths is the union of sets of paths, hence is a set of paths. As usual, we can define the transfer of a set of curves on $L$ to $L$, etc.
\end{rem}

\subsubsection{The transferred laminations}
We  define laminations $\Lambda'_{\pm}$ on $L'$ as the transfers of the Handel-Miller pseudo-geodesic laminations $\Lambda_{\pm}$ on $L$  (forgetting parame\-trizations).  There are problems here as it is not generally true that the transfer of a lamination is a lamination. It turns out that the $h$-invariance of $\Lambda_{\pm}$ saves the day. 

\begin{lemma}\label{2disj}
Let $\AAA$ be an $h$-invariant family of pairwise disjoint, one-one immersed, nonparametrized   curves.  Then the transfer $\AAA'$ of $\AAA$ to $L'$ is an $h'$-invariant family of pairwise disjoint, one-one immersed, nonparametrized  curves.
\end{lemma}

\begin{proof}
Let $\lambda_{1}',\lambda'_{2}\in\AAA'$ and suppose there is a point $x\in\lambda_{1}'\cap\lambda'_{2}$ such that, for every simply connected neighborhood $V$ of $x$ in $L'$, the connected component of $x$ in $V\cap(\lambda'_{1}\cup\lambda'_{2})$ is not an embedded curve. (It is possible that $\lambda'_{1}=\lambda'_{2}$ and self intersects nontrivially.)  By the $h$-invariance of $\AAA$, every transfer of $\lambda'_{i}$ to $L$ is an element of $\AAA$, $i=1,2$.  Since a small enough choice of $V$ projects along $\LL$ homeomorphically   into $L$, we obtain $\lambda_{1},\lambda_{2}\in\AAA$ intersecting nontrivially.
\end{proof}

\begin{lemma}\label{symtransf}
For $\AAA'$ and $\AAA$ as in \emph{Lemma~\ref{2disj}}, $\AAA$ is the transfer of $\AAA'$ to $L$.
\end{lemma}

This is an easy exercise.

\begin{lemma}\label{hinvlam}
If $\Lambda$ is an $h$-invariant lamination of $L$ and $\Lambda'$ is its transfer to $L'$, then $\Lambda'$ is an $h'$-invariant lamination of $L'$.  
\end{lemma}

\begin{proof}
By Lemma~\ref{2disj}, we only need to show that there is a laminated partial atlas for $\Lambda'$. Let $\lambda'\in\Lambda'$ and let $x'\in\lambda'$.  There is $\lambda\in\Lambda$ such that $\lambda'$ is in the transfer of $\lambda$ and so there is $x\in\lambda$ that is carried by the transfer operation to $x'$.  Let $(U,Y,\phi)$ be a laminated chart associated to $\Lambda$, containing $x$ and small enough that $U$ projects along $\LL$ homeomorphically onto a neighborhood $U'$ of $x'$.  Let $\pi:U\to U'$ be this projection. Then $(U',Y,\phi\o\pi^{-1})$ is a laminated chart about $x'$.  It is clear  that the plaques in this chart are exactly the path components of the intersections of the curves in $\Lambda'$ with $U'$.  Since $x'\in|\Lambda'|$ is arbitrary, we have constructed a laminated partial atlas for $\Lambda'$, hence $\Lambda'$ is a lamination.
\end{proof}

\begin{lemma}
If $\Lambda$ is a closed, $h$-invariant lamination of $L$, then $\Lambda'$ is a closed, $h'$-invariant lamination of $L'$.
\end{lemma}

Indeed, a point  $x'\in L'\sm|\Lambda'|$ transfers to a point  $x\in L\sm|\Lambda|$.  But $x$ has an open neighborhood $U$ in $L\sm|\Lambda|$ which is small enough to project along $\LL$ to an open neighborhood $U'$ of $x'$.   Lemma~\ref{symtransf} implies that $U'$ misses $|\Lambda'|$.

\begin{cor}\label{(L,L')}
The transferred laminations $\Lambda'_{\pm}$ are closed,  disjoint from $\bd L'$, $h'$-invariant and $(\Lambda_{+}',\Lambda'_{-})$ is a bilamination.
\end{cor}

Indeed, the bilaminated charts are produced exactly as the laminated charts.

 \subsubsection{Behavior of the transferred laminations at the ideal boundary $E$}\label{sss1242}

Since there are only thirteen nonstandard surfaces~\cite[Theorem~8]{cc:epstein}, all of which are easily ruled out under our hypotheses as leaves of $\FF$ and $\FF'$,  $L$ and $L'$  can each be given a standard hyperbolic metric (cf.~Definition~\ref{standard} and the following remarks).  Fix a choice of such metrics.  By Corollary~\ref{metricindep}, the laminations satisfy the axioms relative to one choice of standard metric if and only if they do so relative to any choice of standard metric.

\begin{rem}
 It is not necessary to make all of the leaves of either foliation simultaneously hyperbolic.  By a well known theorem of A.~Candel~\cite{cand:hyp}, it is possible to find a metric on $M$ which makes all of the leaves of $\FF$ hyperbolic and another that makes the leaves of $\FF'$ hyperbolic, but nothing so sophisticated is needed here.
\end{rem}

Let $\wt M^{\o}$ be the universal cover of ${M}^{\o}$.  Let $\wt\LL$ be the lift of $\LL|M^{\o}$ to $\wt{M}^{\o}$. Fix a lift $\wt{L}\ss\wt M^{\o}$ of the   leaf $L\in\FF$.  Similarly, fix a lift $\wt{L}'\ss\wt M^{\o}$ of the   leaf $L'\in\FF'$. Because the foliations $\FF$ and $\FF'$ are taut, these lifts are the universal  covers of $L$ and $L'$, respectively, and we view $\wt{L}, \wt{L}'\ss\Delta$. 

\begin{lemma}\label{eachsingle}
Each of the surfaces $\wt{L}$ and $\wt{L}'$ meets each leaf of $\wt\LL$ in a single point.
\end{lemma}

\begin{proof}
The fibration $\pi:M^{\o}\to\SI$ defined by $\FF$ lifts to a (necessarily trivial) fibration $\wt\pi:\wt M^{\o}\to\R$ with fibers the leaves of $\wt\FF$.  The leaves of $\wt\LL$ are cross-sections of this bundle. Indeed, if $\wt\ell$ is a leaf of $\wt\LL$, $\wt\pi:\wt\ell\to\R$ is easily seen to be a surjective immersion,  hence a diffeomorphism (by the classificatiion of $1$-manifolds).  The assertion for the leaves $\wt L$ of $\wt\FF$ follow and the  same proof works for the leaves $\wt L'$ of $\wt\FF'$.
\end{proof}

\begin{cor}\label{LamLam'}

Projection along the leaves of $\wt\LL$ defines  $\nu:\wt{L}\to\wt{L}'$,  a homeomorphism which carries $\wt\Lambda_{\pm}$ onto $\wt\Lambda'_{\pm}$.

\end{cor}

\begin{proof}
Indeed, the fact that projection defines a homeomorphism $\nu$ is an obvious consequence of Lemma~\ref{eachsingle}.  Given a path $\sigma$ on $L$ and a lift $\wt\sigma$ on $\wt L$, it is evident that $\nu(\wt\sigma)$ is a lift of a transfer of $\sigma$ to $L'$. Thus, $\nu(\wt\Lambda_{\pm})\sseq\wt\Lambda'_{\pm}$.  For the reverse inclusion, note that the argument just given works equally well to prove that $\nu^{-1}(\wt\Lambda'_{\pm})\sseq\wt\Lambda_{\pm}$.  Thus,
$$
\nu(\wt\Lambda_{\pm})\supseteq\nu(\nu^{-1}(\wt\Lambda'_{\pm}))=\wt\Lambda'_{\pm}.
$$
\end{proof}

\begin{rem}
Note that $\nu$ generally is not the lift of a map from $L$ to $L'$.
\end{rem}

The main result of Section~\ref{sss1242} is,

\begin{prop}\label{main}
The map $\nu$ extends to a homeomorphism $\wh\nu:\wh{L}\to\wh{L}'$.
\end{prop}

 The proof of Proposition~\ref{main} will be modelled on that of~\cite[Theorem~2]{cc:epstein}, but there are important differences.  Note that doubling $M$ along $\trb M$ gives rise to foliations $2\FF$ and $2\FF'$ having  $2L$ and $2L'$, respectively, as leaves.  By this device, we can reduce the proof of Proposition~\ref{main} to the case that $L$ and $L'$ have empty boundary.  Thus $\wt L$ and $\wt L'$ can be identified with $\Delta$ and $\nu:\Delta\to\Delta$ is a homeomorphism. The real reason for doubling, however, is to assure that the components of $\tb M$ are closed surfaces, thereby eliminating finitely many counterexamples to the conclusion of Lemma~\ref{sepcurves}.
 
 \begin{rem}
 This doubling is only a temporary device. It will end when the proof of Proposition~\ref{main} is complete  and nowhere during this period will it be assumed that the doubled laminations satisfy the axioms.  Generally they do not, as Example~\ref{simpex2} shows.  We will explain shortly why this doubling is necessary.
 \end{rem}

\begin{defn}[essentially intersecting loops]\label{essintlps}
Two  (parametrized) loops $\sigma_{1},\sigma_{2}$ in a surface  \emph{essentially intersect} if,  whenever a loop $\tau_{1}$ is freely homotopic to $\sigma_{1}$ and a loop $\tau_{2}$ is freely homotopic to $\sigma_{2}$, then $\tau_{1}$ and $\tau_{2}$ have nonempty intersection.
\end{defn}

In particular, loops that essentially intersect must intersect.

\begin{lemma}\label{sepcurves}
If $F$ is a closed, connected surface other than the sphere, projective plane, or torus, it contains a pair of essentially intersecting loops each of which separates $F$.
\end{lemma}

For the case of the connected sum of $n\ge4$ tori or $n\ge4$ projective planes, this is easy.  The remaining cases are left as an exercise, but see Figure~\ref{SeparatingCurves} for the connected sum of two tori.

\begin{lemma}\label{rr'}
There is a pair of essentially intersecting loops $\rho$ and $\rho_{\pitchfork}$ in $L$ such that, 
\begin{enumerate}
\item Every transfer of $\rho$ to $L'$ is a  loop $\rho'$,
\item Given a transfer $\rho'$ as above, there exists a transfer $\rho'_{\pitchfork}$ of $\rho_{\pitchfork}$ to  $L'$ which  is a  loop essentially intersecting $\rho'$.
\end{enumerate}
 \end{lemma}
 
 \begin{proof}
Let $F$ be a component of $\tb M$
 and consider ends $e$ and $e'$ of $L$ and $L'$, respectively, which are asymptotic to $F$. 
 Fix  an $h$-neighborhood 
 $U_{e}\ss L$ of $e$ and an $h'$-neighborhood $U_{e'}\ss L'$ of $e'$ which spiral on $F$. As in Section~\ref{jnctrs}, let $N,N'$ be smoothly embedded, transversely oriented 1-manifolds in $F$ which define the respective cohomology classes corresponding to these semi-coverings. It is worth noting that, as in the cited section, these semicoverings are the restrictions of  honest coverings $p:\UU\to F$ and $p':\UU'\to F$, where $\UU\ss L$ is open and $h$-invariant and $\UU'\ss L'$ is open and $h'$-invariant. The  deck transformation groups are infinite cyclic, generated by $h$ and $h'$ respectively.

 By Lemma~\ref{sepcurves},  we can choose a  pair of essentially  intersecting   simple closed curves $\sigma$ and $\sigma_{\pitchfork}$ on $F$ which each separate $F$.  Let  $x\in F$ be a point of intersection. Since $\sigma$ and $\sigma_{\pitchfork}$ are separating, the homological intersection number of each of $N$ and $N'$ with each of $\sigma$ and $\sigma_{\pitchfork}$ is zero.   Thus, the curves $\sigma$ and $\sigma_{\pitchfork}$ lift along $\LL$ to essential loops  $\rho$ and $\rho_{\pitchfork}$ in $\UU$ (and therefore $L$) with $\rho\cap\rho_{\pitchfork}$ containing a point $y$ which is a lift of $x$.  Then $\rho$ and $\rho_{\pitchfork}$ intersect essentially. Similarly, the curves $\sigma$ and $\sigma_{\pitchfork}$ lift along $\LL$ to essentially intersecting loops  $\rho'$ and $\rho'_{\pitchfork}$ in  $L'$ with $\rho'\cap\rho'_{\pitchfork}$ containing a point $z$ which is a lift of $x$.  The curves $\rho$ and  $\rho'$ and the curves  $\rho_{\pitchfork}$  and $\rho'_{\pitchfork}$  are obviously transfers of each other.
 \end{proof}
 
 \begin{figure}
\begin{center}
\begin{picture}(210,150)(0,0)

\includegraphics[width=200pt]{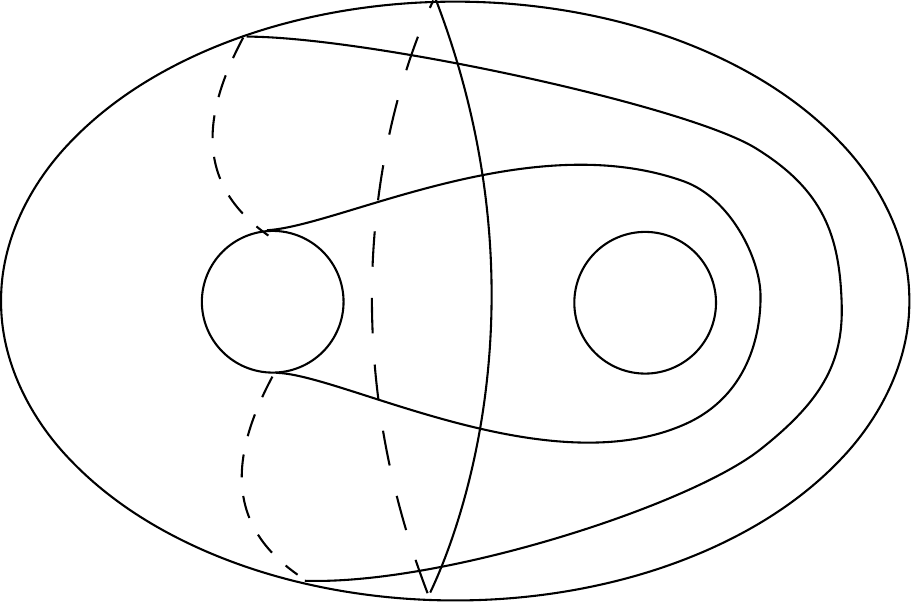}

\end{picture}
\caption{Separating loops which essentially intersect}\label{SeparatingCurves}
\end{center}
\end{figure}

\begin{nota}
Denote by $\RR$  the set of  loops $\rho$ on $L$ such that,
\begin{enumerate}
\item  there is a  loop $\rho_{\tr}$ essentially intersecting $\rho$;
\item the transfers  of $\rho$ and $\rho_{\tr}$ to $L'$ are sets of loops;
\item  for each  transfer $\rho'$ of $\rho$, there exists a transfer $\rho'_{\tr}$ of $\rho_{\pitchfork}$ which  is a  loop essentially intersecting $\rho'$.
\end{enumerate}
Let $\RR'$ denote the union of all transfers $\rho'$ of loops $\rho\in\RR$. 
\end{nota}

By Lemma~\ref{rr'}, both $\RR$ and $\RR'$ are nonempty. Further, $\RR'$ is the transfer of $\RR$ and vice versa.    Let $\wt\RR$  be the set of lifts to $\wt L$ of elements of $\RR$ and $\wt\RR'$ the set of lifts to $\wt L'$ of elements of $\RR'$.  

\begin{lemma}

\begin{enumerate}

\item $\wt\RR$ is invariant under the group of deck transformations of $\wt L$ and $\wt\RR'$ is invariant under the deck transformations of $\wt L'$.  

\item  $\nu$ induces a bijection $\wt\RR\to\wt\RR'$ which will again be denoted by $\nu$. 

\item Finally, every $\rho\in\RR$ admits an essentially transverse  closed  loop $\rho_{\pitchfork}$ such that some transfer of $\rho'_{\pitchfork}$ is a closed  loop essentially transverse to a given transfer $\rho'$ of $\rho$.

\end{enumerate}

\end{lemma}

Indeed, (1) and (3) are clear and (2) is proven exactly as Corollary~\ref{LamLam'}.

\begin{nota}
Let,
$$\ZZ = \{z\in S^{1}_{\infty}\mid   \text {$z$ is an ideal endpoint of $\sigma\in\wt\RR$}\},$$ 
$$\ZZ' = \{z\in S^{1}_{\infty}\mid   \text {$z$ is an ideal endpoint of $\tau\in\wt\RR'$}\}.$$
\end{nota}

\begin{rem}
The standard counterclockwise orientation of $\Si $ induces a cyclic order on $\ZZ$ and on $\ZZ'$.
\end{rem}

\begin{lemma}\label{olh}
The homeomorphism $\nu:\Delta\to\Delta$  induces a bijection $\ol\nu:\ZZ\to \ZZ'$ which either preserves or reverses the cyclic order.
\end{lemma}

\begin{proof}
The bijection $\nu:\wt\RR\to\wt\RR'$ induces a bijection $\ol\nu$ between their sets of endpoints.  If $\nu:\Delta\to\Delta$ is an orientation preserving (respectively reversing) homeomorphism, it is clear that $\ol\nu$ preserves (respectively reverses) cyclic order.  
\end{proof}

Evidently $\ZZ$ and $\ZZ'$ are invariant under (the extensions to $\Si$ of) deck transformations for the respective covering spaces.  Since $L$ and $L'$ have no simply connected ends, we can apply~\cite[Corollary~1]{cc:epstein} to obtain the following.

\begin{lemma}\label{zzdense}
The sets $\ZZ$ and $\ZZ'$ are dense in $ S^{1}_{\infty}$.
\end{lemma}

\begin{cor}
The map $\ol\nu:\ZZ\to \ZZ'$ extends to a homeomorphism $\ol\nu: S^{1}_{\infty}\to  S^{1}_{\infty}$. 
\end{cor}

We now define $\wh\nu:\D^{2}\to\D^{2}$ to be the bijection such that $\wh\nu|\Delta=\nu$ and $\wh\nu|\Si=\ol\nu$.  We need to prove it is a homeomorphism.  Since $\D^{2}$ is a compact Hausdorff space and $\wh\nu$ is bijective, it will be enough to prove that $\wh\nu$ is continuous.  It is only necessary to prove this  at points $z\in\Si$.

\begin{nota}
Denote by $\GG$ (respectively $\GG'$)  the set of lifts of pseudo-geodesics $\gamma$ in $L$ (respectively  $L'$) such that some, hence every, completed lift $\wh\gamma$ has both endpoints in $\ZZ$ (respectively  in $\ZZ'$).
\end{nota}

\begin{lemma}\label{inGG}
The homeomorphism $\nu$ defines a bijection $\nu:\GG\to\GG'$.
\end{lemma}

\begin{proof}
If $\wh\gamma$ has an endpoint $z\in \ZZ$, we will show that $\nu(\widetilde{\gamma})$ limits on the point $\ol\nu(z)\in \ZZ'$.  Applying this to both endpoints of $\wh\gamma$, $\wt\gamma\in\GG$, will prove that $\nu(\GG)\ss\GG'$.  The proof for $\nu^{-1}$ is completely similar, proving the lemma.

Fix a choice of $s\in\RR$ with a lift having $z$ as an ideal endpoint and let $\sigma$ be the closed geodesic homotopic to $s$.  Let $s_{\tr}\in\RR$ intersect $s$ essentially and $\tau$ be the closed geodesic homotopic to $\s_{\tr}$.  Then $\tau$ intersects $\sigma$ essentially and, since these loops are geodesics, the intersection is transverse. Let $a$ be one of these intersection points. 
Let  $\widetilde\sigma$ be the lift of $\sigma$ with $z$ as an ideal endpoint.      Let $\{a_{n}\}_{n\in\Z}$ be the lifts of $a$ in $\widetilde\sigma$, indexed so that $\lim_{n\to\infty}a_{n}=z$ monotonically along $\wt\sigma$. Let $\wt\tau_{n}$ be the lift of $\tau$ through $a_{n}$.  Each $\wh{\tau}_{n}$ has endpoints $u_{n},w_{n}$ bounding a subarc $\alpha_{n}\subset S^{1}_{\infty}$ containing the point $z$.  The sequences $u_{n}\to u$ monotonically and 
 and $w_{n}\to w$ monotonically  as $n\to\infty$.  We claim that $u=z=w$.  Otherwise,  the closed geodesic $\tau$ in $L$ accumulates locally uniformly on a geodesic which is impossible. The sets $V_{n}\subset\D^{2}$, $n\ge0$, bounded by $\wt\tau_{n}\cup\alpha_{n}$ form a fundamental system of closed neighborhoods of $z$. Since $z$ is an ideal endpoint of $\wt\gamma$, it follows that, for each $n\ge0$, there is a ray $\wt\gamma_{n}\ss\wt\gamma$ such that $\wt\gamma_{n}\ss V_{n}$.
 
 The  sets $ V_{n}\cap\Delta$ are nested with empty intersection. Thus, the homeomorphism $\nu$ carries them to a nested sequence of  sets with empty intersection.  Furthermore, the homeomorphism $\ol\nu$ carries the sequences $u_{n}$ and $w_{n}$ to sequences converging monotonically from opposite sides to $\ol\nu(z)$. By the definition of $\ol\nu$, it carries the ideal endpoints of $\wt\tau_{n}$ to the ideal endpoints of $\nu(\wt\tau_{n})$.  It clearly follows  that $\{\wh\nu(V_{n})\}_{n=1}^{\infty}$ is a fundamental system of neighborhoods of $\ol\nu(z)$. Furthermore, $\nu(\wt\gamma_{n})\ss\wh\nu(V_{n})$ for $n\ge0$, so $\nu(\wt\gamma)$ has $\ol\nu(z)$ as an ideal endpoint.
\end{proof}

\begin{rem}
The existence of $\tau$ in the above proof depends on $\rho$ being nonperipheral, that is, on $\sigma$ not being a component of $\bd L$.  By doubling we have assured that $\bd L=\0$.  Actually, we would be fine without doubling except for a few pesky cases.  For example, if each of the components $F$ of $\tb M$ is either a torus with an open disk deleted or a pair of pants, then, in the proof of Lemma~\ref{sepcurves}, it would not be possible to find a simple  separating loop $\sigma\ss F$ which is not peripheral and the proof of Lemma~\ref{rr'} would fail.  Hence, in this case, we cannot construct $\rho$ to be nonperipheral.  Doubling is not necessary for the application of~\cite[Corollary~1]{cc:epstein} in proving Lemma~\ref{zzdense} since the leaves $L$ and $L'$ are admissible, hence have no simply connected ends.
\end{rem}

\begin{proof}[Proof of \emph{Proposition~\ref{main}}]
We want to show that $\wh\nu:\D^{2}\to\D^{2}$ is continuous at $z\in\D^{2}$.  This is clear if $z\in \Delta$, so we assume $z\in S^{1}_{\infty}$. Let $U$ be an open neighborhood of $\wh\nu(z)$. Since $\ZZ'$ is dense in $S^{1}_{\infty}$, we can choose  an arc $\alpha\subset U\cap S^{1}_{\infty}$ with endpoints $a,b\in\ZZ'$ such that $\wh\nu(z)\in\alpha$.
Let $\wh{\gamma}$ be any curve in $U$ with endpoints $a$ and $b$ and $\gamma$ the projection of $\wh{\gamma}\cap\Delta$ to $L$. Then $\wt\gamma\in\GG'$ so by Lemma~\ref{inGG}, $\nu^{-1}(\wt\gamma)\in\GG$ and has  endpoints $\ol\nu^{-1}(a),\ol\nu^{-1}(b)\in\ZZ$.  Then the subset $V$ of $\D^{2}$ bounded by $\wh\nu^{-1}(\wh{\gamma})$ and $\wh\nu^{-1}(\alpha)$  is a closed neighborhood of $z$ in $\mathbb{D}^{2}$ and $\wh\nu(V)\subset U$.  This proves continuity at arbitrary $z\in S^{1}_{\infty}$.  Finally, in the general case before doubling, $\nu$ restricts to the desired homeomorphism of $\wh L$ onto $\wh L'$.
\end{proof}

 Because of the final statement in the above proof, we no longer work in the double and will not do so again. Here are two corollaries of Proposition~\ref{main}.

\begin{cor}\label{pseudotransfer}
Any transfer of a pseudo-geodesic on $L$ to a curve on $L'$ is a pseudo-geodesic \upn{(}and vice versa\upn{)}.
\end{cor}

\begin{proof}
Let $\lambda$ be a pseudo-geodesic on $L$ and let $\wt\lambda$ be a lift of $\lambda$ to $\wt L$. By the definition of $\nu$, it is clear that $\nu(\wt\lambda)=\wt\lambda'$ is a lift of some transfer $\lambda'$ of $\lambda$ to $L'$. Since $\wt\lambda$ has two well defined endpoints, $\wh\nu$ carries those endpoints to well defined endpoints of $\wt\lambda'$. The transfer  of $\lambda$ is the $h'$-orbit of $\lambda'$.  Since any lift $\wt h'$ extends canonically to $\bd\wh L'$, we see that the transfer of $\lambda$ consists entirely of  pseudo-geodesics.
\end{proof}

\begin{cor}\label{primestrcl}

The laminations $\Lambda'_{\pm}$ are strongly closed.

\end{cor}

\begin{proof}
The corollary follows since $\wh\nu$ is a homeomorphism and the laminations $\Lambda_{\pm}$ are strongly closed.
\end{proof}

\subsection{Proofs of Axioms~\ref{muttran} and~\ref{eachmeets}. The positive/negative escaping sets}

We fix  choices of $L\in\FF|M^{\o}$ and $L'\in\FF'|M^{\o}$ and transfer the laminations $\Lambda_{\pm}$ on $L$ to laminations $\Lambda'_{\pm}$on $L'$. 

 By Corollary~\ref{pseudotransfer}, the leaves of $\Lambda'_{\pm}$ are pseudo-geodesics. By Corollary~\ref{(L,L')}, we obtain the following.
 
 \begin{lemma}\label{ax1}
 \emph{Axiom~\ref{muttran}} holds  for $\Lambda'_{\pm}$.
 \end{lemma}

Since we have  identified the lifts of $\Lambda_{\pm}$ with those of $\Lambda'_{\pm}$ via the homeomorphism $\nu$ (Corollary~\ref{LamLam'}), it follows that,

\begin{lemma}\label{ax3}
\emph{Axiom~\ref{eachmeets}} holds for $\Lambda'_{\pm}$.
\end{lemma}

\subsubsection{The positive/negative escaping sets $\UU_{\pm}'$}

By Lemma~\ref{endpmono}, $h':L'\to L'$ is an endperiodic automorphism. Even though we have not verified all the axioms for $\Lambda_{\pm}'$, we can still define the positive and negative escaping sets $\UU'_{\pm}$ as in Definition~\ref{pmesc}. That is for $e'$ an end of $L'$, choose an $h'$-neighborhood $U'_{e'}$ of $e'$ (Definition~\ref{fnbhd}) and define  $\UU'_{e'}=\bigcup_{n=-\infty}^{\infty}(h')^{np_{e'}}(U'_{e'})$. The union of the sets $\UU'_{e'}$ as $e'$ ranges over the positive (respectively  negative) ends will be denoted by $\UU'_{+}$ (respectively  $\UU'_{-}$).  We will call $\UU'_{+}$ the \emph{positive escaping set} and $\UU'_{-}$ the \emph{negative escaping set}.

Clearly,

\begin{lemma}

The set $\UU'_{+}$ \upn{(}respectively $\UU'_{-}$\upn{)} consists of the set of points $x\in L'$ such that the sequence $\{(h')^{n}(x)\}_{n\ge 0}$  \upn{(}respectively $\{(h')^{n}(x)\}_{n\le 0}$\upn{)} escapes \upn{(}{\rm Definition~\ref{seqesc}}\upn{)}.

\end{lemma}

\begin{lemma}\label{samesc'}
 The point $x\in L'$ belongs to $\UU'_{\pm}$ if and only if the leaf $\ell_{x}$ of $\LL$ through $x$ satisfies $\ell_{x}\cap L\ss\UU_{\pm}$.
\end{lemma}

\begin{proof}
The negative ends of $L$ and $L'$ are exactly the ones that wind in on inwardly oriented components of $\bd M$ as semi-coverings, the positive ends likewise winding in on the outwardly oriented components of $\bd M$.  Thus, the points $x\in\UU_{+}$ are characterized by the fact that flowing them forward along $\LL$ causes them to approach the outwardly oriented components of $\bd M$, the points of $\UU_{-}$ being characterized analogously.  Since the same characterizations hold for $\UU'_{\pm}$, the claim follows.
\end{proof}

Since all axioms hold in $L$ and since projections along $\LL$ define local homeomorphisms of $L$ to $L'$, Lemma~\ref{front'} for $\Lambda_{\pm}\ss L$ now implies:

\begin{cor}\label{U'fr}
The frontier of $\UU'_{\mp}$ is $|\Lambda'_{\pm}|$ and, consequently, the border leaves of $\UU'_{\mp}$ accumulate locally uniformly exactly on the leaves of $\Lambda'_{\pm}$.
\end{cor}

Likewise, Corollary~\ref{LambdadoesnotmeeUU'}  for $\Lambda_{\pm}\ss L$ implies:

\begin{cor}\label{staysout}
 No leaf of $\Lambda'_{+}$ \upn{(}respectively, $\Lambda'_{-}$\upn{)} meets $\UU'_{-}$ \upn{(}respectively, $\UU'_{+}$\upn{)}.
\end{cor}

\subsection{Semi-isolated leaves and escaping ends in $\Lambda'_{\pm}$} 

Let $e'$ be an end of $L'$ and denote its full $h'$-cycle by $c'=\{e'=e'_{1},e'_{2},\dots,e'_{p_{e'}}\}$.  Set 
$$\UU'_{c'}=\bigcup_{n=-\infty}^{\infty}(h')^{n}(U'_{e'})=\UU'_{e'_{1}}\cup\UU'_{e'_{2}}\cup\cdots\cup\UU'_{e'_{p_{e'}}}.$$  
Remark that $\UU'_{c'}$ is an open, $h'$-invariant set with no periodic points.  The connected components $\UU'_{e'_{i}}$ of $\UU'_{c'}$ are permuted cyclically by $h'$.    In $M$, the ends $e'=e'_{1},e'_{2},\dots,e'_{p_{e'}}$ are exactly the ends of $L'$ asymptotic to a compact leaf $F=F_{c'}\ss\tb M$.  There is an $h$-cycle $c$ of ends of $L$ asymptotic to this same compact leaf $F=F_{c}$ and one similarly defines an open $h$-invariant subset $\UU_{c}\ss L$.  As in Section~\ref{jnctrs}, one has infinite cyclic coverings
\begin{eqnarray*}
q':\UU'_{c'}\to F\\
q:\UU_{c}\to F.
\end{eqnarray*}
Here, these coverings are defined by projection along the compact subarcs of leaves of $\LL$ issuing from points of $\UU'_{c'}$ and of $\UU_{c}$, respectively,  and terminating on $F$. 

The group of deck transformations for $q'$ is generated by $h'|\UU_{c'}$ and for $q$ by $h|\UU_{c}$. Thus,

\begin{lemma}
We can identify the surface $F$  \upn{(}which is a leaf of the foliation of $M$\upn{)} with the surface $F$ of \emph{Section~\ref{jnctrs}}.
\end{lemma}

For definiteness, we consider the case that $c'$ is an $h'$-cycle of negative ends of $L'$, $c$ the corresponding $h$-cycle of negative ends of $L$, with $F=F_{c'}=F_{c}$.
Since $\UU'_{c'}$ and $\Lambda'_{-}$ (respectively  $\UU_{c}$ and $\Lambda_{-}$) are $h'$-invariant (respectively  $h$-invariant), it follows that the induced lamination $\Lambda'_{-}|\UU'_{c'}$ (respectively  $\Lambda_{-}|\UU_{c}$) is invariant under the group of deck transformations of $q':\UU'_{c'}\to F$ (respectively  $q:\UU_{c}\to F$).  As in Section~\ref{inducedlam}, the laminations $\Lambda'_{-}|\UU'_{c'}$ and  $\Lambda_{-}|\UU_{c}$ descend to  closed laminations  of  $F$. Since $\Lambda'_{-}|\UU'_{c'}$ and  $\Lambda_{-}|\UU_{c}$ are transfers of each other, they descend to the same lamination of $F$. We have,

\begin{lemma}\label{samesame}
The laminations $\Lambda'_{-}|\UU'_{c'}$ and  $\Lambda_{-}|\UU_{c}$ descend to the  same well defined closed lamination $\Lambda_{F}$ of  $F$. 
\end{lemma}

In Section~\ref{inducedlam}, we analyzed the structure of $\Lambda_{F}$ and the traintrack $T$ that carries it using well understood properties of semi-isolated leaves in $\Lambda_{+}$ and escaping ends in $\Lambda_{-}$.  By Lemma~\ref{samesame}, the same lamination $\Lambda_{F}$ and traintrack are induced by the laminations $\Lambda'_{-}$ and the covering map $q':\UU'_{c'}\to F$. Hence it is reasonable to expect that these properties of semi-isolated leaves also hold for  $\Lambda_{+}$ and that these properties of escaping ends also hold for $\Lambda'_{-}$.  Since the axioms have not all been verified for $\Lambda'_{\pm}$, this requires proof which we now provide.

Remark that the negative escaping set $\UU'_{-}\ss L'$ is the disjoint union of the $\UU'_{c'}$'s corresponding to $h'$-cycles $c'$ of negative ends, with the parallel assertion about $\UU'_{+}\ss L'$. 
The following discussion will be carried out for the negative escaping sets and the cycles of negative ends, the positive case  being entirely parallel.

 \begin{prop}\label{finmanyper}
The laminations $\Lambda'_{+}$     of $L'$ have only finitely many semi-isolated leaves.  They are periodic under $h'$ and each contains either one $h'$-periodic point or one nondegenerate, compact $h'$ periodic subarc. In the case of one $h'$-periodic point, it is  repelling    on both sides.  In the case of a compact, nondegenerate $h'$-periodic subarc $A'$, the endpoints of $A'$ are repelling     on the sides not containing $A'$.  Furthermore, $\intr A'$ meets no  leaves of $\Lambda'_{-}$. 
\end{prop}

\begin{proof}
 Because the transfer operation is locally a homeomorphism, it is clear that the semi-isolated leaves of $\Lambda'_{+}$ are exactly the transfers of the semi-isolated leaves of $\Lambda_{+}$.  Since any of these latter contains at least one $h$-periodic point $y$, the leaf of $\LL$ through $y$ is closed and can only intersect $L'$ in finitely many points.  Thus the laminations $\Lambda'_{+}$ have only finitely many semi-isolated leaves and they each contain $h'$-periodic points.  

If $\lambda\in\Lambda_{+}$ is semi-isolated and contains exactly one $h$-periodic point $y$,  then the leaves of $\LL$ passing through $\lambda$ are asymptotic (in backward time) in the $\LL$-saturation of $\lambda$ exactly to the closed leaf $\ell$ through $y$. Remark that $\ell$ meets $\lambda$ only in the point $y$. If $\lambda'$ is a transfer of $\lambda$, we claim that $\ell$ can only meet $\lambda'$ in a single point $y'$.  Otherwise, a transfer back to $\lambda$ would send a subarc of $\lambda'$ with distinct endpoints onto a nontrivial, compact curve lying in $\lambda$  and having the same endpoints.  This would imply that $\lambda$ is compact, contrary to Axiom~\ref{muttran}.   Thus, the leaves of $\LL$ passing through $\lambda'$ are asymptotic to $\ell$ in the $\LL$-saturation of $\lambda'$. 
Since $\lambda$ borders $\UU_{-}$, $y$ is a repelling periodic point of $h$ on $\lambda$ and is the only periodic point (cf.~Section~\ref{periodicleaves}).  Consequently $y'$ is a repelling periodic point of $h'$ on $\lambda'$. It is the only one as a second would would generate another fixed point in $\lambda$.

  In a similar way, the compact, nondegenerate, $h$-periodic subarc $A$ of a semi-isolated leaf $\lambda\in\Lambda_{+}$ determines a corresponding $h'$-periodic subarc $A'$ on any  transfer $\lambda'$ of $\lambda$, the endpoints of which are both  $h'$-repelling  on the sides not containing $A'$,   and $A'$ is the only $h'$-periodic subarc of $\lambda'$.

The final assertion transfers from the corresponding fact for $\Lambda_{+}$, where the leaves of $\Lambda_{-}$ passing through the endpoints of $A$ lie in the arms of  principal regions.
\end{proof}

The following is an easy consequence of the above and Lemma~\ref{samesc'}.

\begin{cor}
If $c'$ is a cycle of negative  ends  of $L'$, then the border $\delta\UU'_{c'}$ has finitely many components consisting   of $h'$-cycles of semi-isolated leaves of $\Lambda'_{+}$.
\end{cor}

\begin{prop}\label{escray}

The ends represented by rays of leaves of $\Lambda'_{-}$   issuing from $\delta\UU'_{-}$  into $\UU'_{-}$,  either from a unique $h'$-periodic point on a semi-isolated leaf $\lambda$ of $\Lambda'_{+}$   or from the endpoints of a unique nondegenerate, compact, $h'$-periodic subarc of  such a leaf are exactly the escaping ends \emph{(Definition~\ref{defnescpend})}  of leaves of $\Lambda'_{-}$. Only finitely many leaves of $\Lambda'_{-}$  have an escaping end.
\end{prop}

\begin{proof}
Such a ray $r'=[a',\infty)$ is clearly the transfer of such a ray $r$ in $L$.  In particular, $\intr r'=(a',\infty)$ lies in the negative $h'$-escaping set.  We will suppose it lies in  $\UU'_{e'}$ for a negative end $e'\in\EE(L')$.  As the transfer of an $h$-periodic ray, $r'$ is $h'$-periodic.  Let $p$ be an $h'$-period of $r'$.   (If $(h')^p$ is orientation reversing on $\lambda$, replace $p$ by $2p$.)

Let $x\in(a',\infty)$ and consider the compact arc  $\alpha\ss(a',\infty)$ bounded by $x$ and $(h')^{-p}(x)$.  Let $U_{e'}$ be an $h'$-neighborhood of $e'$ for the  automorphism $h'$ which by Lemma~\ref{endpmono} is endperiodic.   Since each point $y\in\alpha$ converges to $e'$ under negative iterations of $(h')^{p}$, the compactness of $\alpha$ implies that there is a sufficiently large integer $n\ge1$ such that $(h')^{-kp}(\alpha)\ss U_{e'}$, all $k\ge n$. Thus, $\bigcup_{k=n}^{\infty}(h')^{-kp}(\alpha)\ss U_{e'}$.  Therefore,  the ray $r'$ represents an escaping end as asserted.

Conversely, if the ray $r' =  [a',\infty)$ represents an escaping end, then $r'$ is the transfer of a ray $r = [a,\infty)$ with $(a,\infty)\ss\UU_{-}$ and $a\in|\delta\UU_{-}|$. By Lemma~\ref{onlynoncpt}, the ray $r$ represents an escaping end so by Lemma~\ref{doesnotreturn} $a$ is $h$-periodic. Thus, $a'$ is $h'$-periodic and the converse is proven.

Proposition~\ref{finmanyper} now implies that only finitely many leaves of $\Lambda'_{-}$  have an escaping end.
\end{proof}

By transferring from $L'$ to $L$,  we obtain the following.

\begin{prop}
The completions of the leaves of $\Lambda'_{-}|\UU'_{c'}$ are all compact arcs with endpoints in $\delta\UU'_{c'}$, except for the rays $[a,\infty)$ corresponding to escaping ends.  The compact completions fall into parallel packets permuted by $h'$ and there are only finitely many $h'$-orbits of parallel packets.
\end{prop}

We repeat that the analogous statements for $\Lambda'_{+}|\UU'_{c'}$, where $c'$ is an $h'$-cycle of positive ends of $L$, are true and have completely parallel proofs.

Thus, we see  exactly how $\Lambda'_{\pm}|\UU'_{c'}$ descends via the covering map $q'$ to $\Lambda_{F}$. 

For future reference, we also note the following.

\begin{prop}\label{borderdense'}
The union of the border components of $\UU'_{\mp}$ is dense in $|\Lambda'_{\pm}|$.
\end{prop}

\begin{proof}
Equivalently, the border components of $\UU'_{\mp}$ accumulate locally uniformly on every leaf of $\Lambda_{\pm}$.  Since the border leaves are the transfers of the border leaves of $\UU_{\mp}$, this will be true if the corresponding assertion is true for $\UU_{\mp}$ and $\Lambda_{\pm}$ in $L$.  Since the axioms are satisfied for $\Lambda_{\pm}$, this is given by Proposition~\ref{borderdense}.
\end{proof}

\subsection{The $h'$-junctures}\label{modjnct}

In this subsection we choose and fix a countable set of $h'$-junctures.  

Let $c'$ denote an $h'$-cycle of  positive or negative ends of $L'$.   The covering projection $q':\UU'_{c'}\to F$ defines a cohomology class $\kappa$ on $F$ exactly as in Section~\ref{jnctrs} which we represent by a compact, properly embedded, transversely oriented $1$-manifold $J_{\kappa}$ which is ``weakly groomed'' (Definition~\ref{groomed}).   Thus, \emph{as a cocycle}, we can replace  $J_{\kappa}$ with a transversely oriented $1$-manifold $|J_{\kappa}|$ to which each component is assigned a positive weight.  Each component $\sigma$ of $|J_{\kappa}|$ is a properly embedded, transversely oriented arc or circle, with a positive integral weight $w_{\sigma}$, representing a ``packet'' of $w_{\sigma}$ parallel arcs or parallel circles, coherently  transversely oriented. By an abuse of notation, we generally use $J_{\kappa}$ to signify this weighted $1$-manifold also.  For certain purposes, especially lifting $J_{\kappa}$ to define a system of $h'$-juncture components in $\UU'_{c'}$, it is best to view $J_{\kappa}$ in this way, but for other purposes, such as ``tunneling'' (Section~\ref{Tunnel}), it is better to view $J_{\kappa}$ as the honest $1$-manifold produced as in Lemma~\ref{Poincdu}.  For details, consult Sections~\ref{jnctrs} and~\ref{Tunnel}.

Regard $\T$ as a \textit{finite}, nonorientable 1-complex having as vertices the switches. Its structure is identical with that of the traintrack for $\Lambda_{F}$.  In particular, it has no ``Reeb annuli'' (see the remark on page~\pageref{noreebs}), although this is not essential for the following application of 
standard general position arguments. Clearly,

\begin{lemma}
An isotopy of $|J_{\kappa}|$ guarantees that $|J_{\kappa}|$ is transverse to the $1$-complex $\T$,  meets no vertex of $\T$ and intersects $\T$ so as to form no digons. 
\end{lemma}

\begin{cor}
The isotopy of $|J_{\kappa}|$ can be chosen so that this properly embedded $1$-manifold meets the leaves of $\Lambda_{F}$ transversely and so as to form no digons.
\end{cor}

\begin{proof}
Wherever $|J_{\kappa}|$ crosses $\T$, put a laminated chart for $\Lambda_{F}$.  An isotopy makes $|J_{\kappa}|$ cross the chart in a transversal without introducing digons.  
\end{proof}

 Since every escaping ray must cross junctures, we have the following.

\begin{lemma}
Each circle in $\T$ crosses at least one component of $J_{\kappa}$.
\end{lemma}

\begin{lemma}\label{poscross}
The juncture $J_{\kappa}$  can be chosen so that each oriented circle leaf $C$ of $\Lambda_{F}$ crosses  $J_{\kappa}$ only in the positive sense. 
\end{lemma}

\begin{proof}
By Proposition~\ref{tunnel}, one rechooses $J_{\kappa}$, if necessary, so that each $C$ either crosses it everywhere in the positive sense, or everywhere in the negative sense. This introduces no new intersections, hence does not undo any features we have established. But the transverse orientation of the negative  $h'$-junctures in $\UU'_{c'}$ are into the neighborhood of the end of $L'$ or $L$ that they cut off.  The orientation of an escaping ray is toward that end.  Since juncture and ray will have all intersections positive or all negative, they must have all intersections positive.  Down in $F$, this means that all circle leaves of $\Lambda_{F}$ have positive intersections with $J_{\kappa}$.  
 \end{proof}

 We fix  these properties of $J_{\kappa}$.

 As in Section~\ref{jnctrs}, we now lift $J_{\kappa}$ to a set $\{(h')^{n}(N'_{c'})\ |\ n\in\Z\}$ of $h'$-juncture components in $\UU'_{c'}$.  Letting $c'$ vary over all $h'$-cycles  of negative (respectively positive) ends of $L'$, we obtain a countable  set of  $h'$-junctures in $\UU'_{-}$ (respectively in $\UU'_{+}$) which we fix.  As in Definition~\ref{famNNg}, we define $\NN'_{-}$ to be the set of components of these negative $h'$-junctures, $\NN'_{+}$ to be the set of components of these positive $h'$-junctures, and $\NN' = \NN'_{-}\cup\NN'_{+}$.
 
 We define the set of juncture components $\JJ'$ accociated to $\NN'$ by $\JJ' = \NN'$ and the map $\iota:\NN'\to \JJ'$ to be the identity (see Definition~\ref{jnctfam}). Thus the set of junctures is identical to the set of $h'$-junctures.
 
 By Corollary~\ref{U'fr}, an element of $\JJ'_{-}$ $h'$-escapes if and only if it is disjoint from $|\Lambda'_{-}|$, with a similar statement for the elements of $\JJ'_{+}$. 
 
  \begin{lemma}\label{singleton}
 A component of a negative juncture intersects a negative escaping end $[a,\infty)$ in at most one point.
 \end{lemma}
 
 \begin{proof}
 Indeed, the intersection must be a singleton because, by Lemma~\ref{poscross}, an escaping ray can only meet a given juncture in positive intersections, hence only once.
 \end{proof}

\begin{defn}[$\YY'_{\pm}$]\label{yyprime}
The set $\YY'_{\pm}$ consists of the nonescaping  components of the elements of $\NN'_{\pm}$.
\end{defn}

Note that $\YY'_{\pm}$ is a discrete, nonclosed lamination.

\begin{lemma}\label{noaccpt}
If a sequence $\{y_{n}\}_{n=1}^{\infty}$ of points in distinct leaves of $\YY'_{\mp}$  converges  to a point  $y\in L'$, then $y\in\Lambda'_{\pm}$.
\end{lemma}

\begin{proof}
Evidently, $y\in\ol\UU'_{\mp}$ and, by Corollary~\ref{U'fr}, $\ol\UU'_{\mp}=\UU'_{\mp}\cup\Lambda'_{\pm}$.  
If  $y\in\UU'_{\pm}$, all but finitely many terms of the sequence lie in $\UU'_{c}$ for an $h'$-cycle of ends.  The covering projection  $q':\UU'_{c}\to F_{c}$ would then carry $y$ to a point $q'(y)=y'\in F_{c}$ such that infinitely many strands of $J_{\kappa}$ pass through a neigborhood of $y'$.  Since $J_{\kappa}$ is a compact $1$-manifold, this is impossible.
\end{proof}

\subsection{The last two axioms}

It remains that we verify   Axiom~\ref{ecorr} and Axiom~\ref{trnsvrs}.  

 Let $\NN'$ denote the set of $h'$-juncture components just constructed and let the set $\JJ'$ of juncture components be the same set, taking for $\iota:\NN'\to\JJ'$ the identity map.  Then, the condition that $\iota(N)$ be homotopic to $N$ for each $N\in\NN$ is  tautological.  By our construction, the following has already been proven.
 
 \begin{lemma}
 \emph{Axiom~\ref{trnsvrs}} holds for $\Lambda'_{\pm}$, with $\JJ'_{\pm}=\NN'_{\pm}$.
 \end{lemma}
 
  It remains to check Axiom~\ref{ecorr}, the endpoint correspondence property.

\subsubsection{Convergence properties of $\YY'_{\pm}$}

We work explicitly with the set $\YY'_{-}$ in $\UU'_{-}$, but everything carries through for the positive case in a completely parallel way.  
Let $\lambda$ be a border leaf of $\UU'_{e}$, where $e$ is a negative end of $L'$, and let $[a,\infty)$ be an escaping end issuing from $a\in\lambda$.  There may or may not be a second escaping end $[a',\infty)$ issuing from $a'\in\lambda$.  Let $\{\sigma_{n}\}_{n=-\infty}^{\infty}$ be an enumeration of the elements of $\YY'_{-}$ intersecting $[a,\infty)$.  By Lemma~\ref{singleton}, $\sigma_{n}\cap[a,\infty)$ is a singleton, $x_{n}$, $-\infty<n<\infty$.  The indexing is chosen so  that $\{x_{n}\}_{-\infty<n<\infty}$ is monotonically increasing in $[a,\infty)$.

\begin{defn}[assemblage]\label{assemblg}
The connected set
$$
\lambda\cup[a,\infty)\cup\bigcup_{n=-\infty}^{\infty}\sigma_{n}\ss L'
$$
will be called an \emph{assemblage}. 
\end{defn}

\begin{rem}
We can also view an assemblage as contained in the internal completion  $\ddot\UU'_{-}$ and define similar assemblages in $\ddot\UU'_{+}$.
 For an end $e$ of $L$, there are finitely many  assemblages in $\ddot \UU'_{e}$, hence in $\ddot\UU'_{\pm}$.   An assemblage in $\ddot\UU'_{-}$ may be called a negative assemblage and one in $\ddot\UU'_{+}$ a positive assemblage.  As mentioned above, we are explicitly considering the negative ones, but will mention the minor adjustment necessary in the positive case.
 \end{rem}

Let $p$ be the least positive integer such that $(h')^{p}$ carries $\lambda$ to itself, preserving orientation.  Thus $(h')^{p}$ carries $[a,\infty)$ to itself and permutes the set $\{x_{n}\}$ so as to preserve the order. The sequence $\{x_{n}\}_{n=-\infty}^{\infty}$ falls into finitely many $(h')^{p}$-orbits. The point $a$ is $(h')^{p}$-contracting in $[a,\infty)$.  If it is an isolated fixed point in $\lambda$, then $a$ is $(h')^{p}$-expanding in $\lambda$ on both sides.  Otherwise, there is an invariant subinterval $[a',a]\ss\lambda$ and $a$ and $a'$ are both $(h')^{p}$-expanding in $\lambda$  on the sides not containing $[a',a]$.  All of this is easily deduced from the corresponding properties in $L$, being known there since the axioms are satisfied in $L$.

\begin{lemma}\label{convergestobothends}
If $e$ is the end of $L'$ to which $[a,\infty)$ escapes, then $\lim_{n\to\infty}x_{n}=e$ and $\lim_{n\to\infty}x_{-n}=a$.
\end{lemma}

\begin{proof}
The first equality is due to the fact that the $x_{n}$'s lie in the component $\UU'_{e}$ of the negative $h'$-escaping set.  The second sequence is monotonically decreasing in $[a,\infty)$, hence converges to a point $b$ in that ray. Clearly $b$ is $h'$-periodic, hence cannot lie in $\UU'_{-}$.  Therefore, $b=a$.  Alternatively, $b=a$ is a consequence of  Lemma~\ref{noaccpt}.
\end{proof}

Let
$$
\lambda\cup[a,\infty)\cup\bigcup_{n=-\infty}^{\infty}\sigma_{n}\ss L'
$$
be an assemblage. A choice of lift $\wt\lambda$ gives a lift $\wt a$ of $a$, hence a lift $[\wt a,\infty)$ of $[a,\infty)$, hence lifts $\wt x_{n}\in[\wt a,\infty)$ and lifts $\wt\sigma_{n}$ through $\wt x_{n}$, $-\infty<n<\infty$. This defines a lifted assemblage  
$$
\wt \lambda\cup[\wt a,\infty)\cup\bigcup_{n=-\infty}^{\infty}\wt\sigma_{n}\ss\wt L'
$$
of which there are a countable infinity.
An application of $\nu^{-1}$ (Corollary~\ref{LamLam'}) transfers this assemblage to 
$$
\wt\gamma\cup[\wt b,\infty)\cup\bigcup_{n=-\infty}^{\infty}\wt\tau_{n}\ss\wt L,
$$
a lift of a \emph{particular} transfer of the  assemblage in $L'$ to $L$.  In this transfer,
$$
\gamma\cup[ b,\infty)\cup\bigcup_{n=-\infty}^{\infty}\tau_{n}\ss L,
$$
the respective transfers of the points $x_{n}\in[a,\infty)$ will be denoted by $y_{n}\in[b,\infty)$.

\begin{rem}
There is a possible surprise here.  Examples show that the curves $\sigma_{n}$ may be loops that open up under transfer.  As nonparametrized curves, there may be only finitely many distinct transferred curves $\tau_{n}$, each meeting $[b,\infty)$  in a bi-infinite sequence of points  each.  By parametrizing these transfers so that  $\tau_{n}(0)=y_{n}$, we view $\{\tau_{n}\}_{n\in\Z}$ as an infinite family of distinct parametrized curves. 
\end{rem}

There is a smallest positive integer $q$ such that $h^{q}$ carries $\gamma\cup[b,\infty)$ to itself, preserving orientation of $\gamma$. It is not necessarily true that $q=p$.  

\begin{lemma}
the map $h^{q}$ permutes the set $\{y_{n}\}$  and also the set $\{\tau_{n}\}$ of \emph{paramet-rized} curves, having finitely many distinct orbits in each.  
\end{lemma}

\begin{proof}
Denote by $\mu:[a,\infty)\to[b,\infty)$ the transfer such that $\mu(x_{n})=y_{n}$, $-\infty<n<\infty$.  Then $\mu^{-1}\o h^{q}\o\mu:[a,\infty)\to[a,\infty)$ is a transfer operation, hence is of the form $(h')^{kp}$, for some integer $k$.  That is, $h^{q}\o\mu=\mu\o(h')^{kp}$.  Thus, for each $n\in\Z$, there is $m\in\Z$ such that 
$$
h^{q}(y_{n})=h^{q}(\mu(x_{n})) = \mu ((h')^{kp}(x_{n})) = \mu(x_{m}) = y_{m}.
$$
That is, $h^{q}$ maps the set $\{y_{n}\}$ into itself.  The same argument applies to $h^{-q}$, hence $h^{q}$ maps $\{y_{n}\}$ bijectively to itself.  Evidently, the same argument applies to $\{\tau_{n}\}$.  That there are only finitely many $h^{q}$-orbits is evident.
\end{proof}

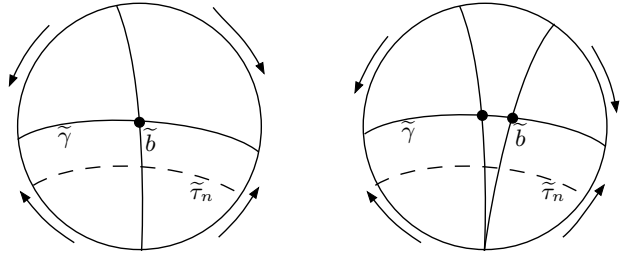
\begin{figure}
\begin{center}
\begin{picture}(300,120)(55,-205)
\rotatebox{270}{\includegraphics[width=300pt]{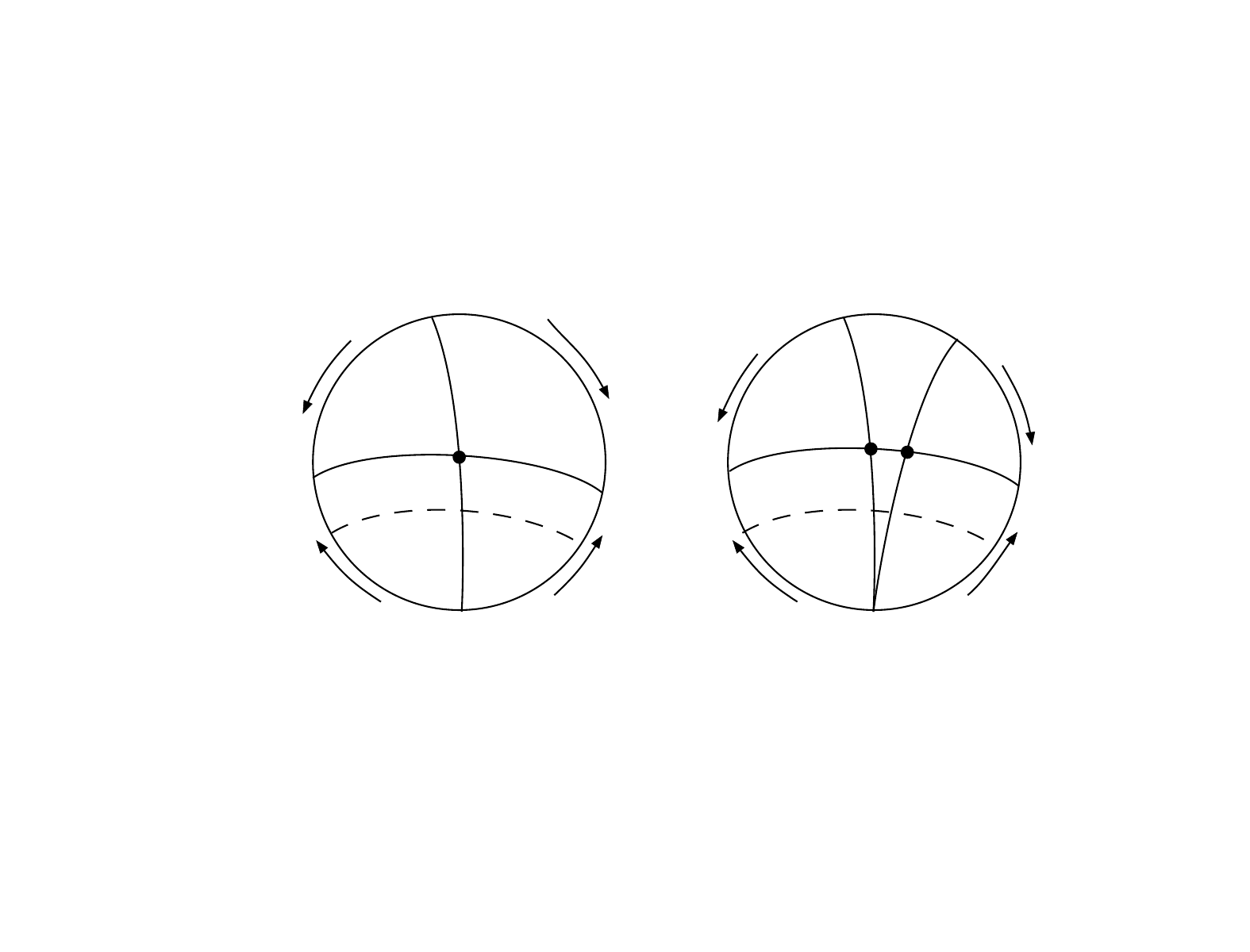}}

\put(-150,-150){\small$\wt\gamma$}
\put(-280,-152){\small$\wt\gamma$}
\put(-247,-155){\small$\wt b$}
\put(-107,-153){\small$\wt b$}
\put(-97,-173){\small$\wt\tau_{n}$}
\put(-230,-173){\small$\wt\tau_{n}$}

\end{picture}
\caption{Case of juncture components meeting one or two escaping ends of a leaf of $\Lambda_{-}$.}\label{2cases}
\end{center}
\end{figure}

Since the laminations on $L$ satisfy the axioms, Proposition~\ref{3cases} applies.   In Figure~\ref{2cases} we reproduce the pertinent two cases of Figure~\ref{PR}, relabeling with the current notation.  The arrows indicate the action  on $\Se$ of the completion $\wh g$ of a suitable lift $g$ of $h^{q}$ (one which fixes the ideal endpoints of $\wt\gamma$).  Since the  curves $\wt\tau_{n}$ are permuted by $g$, they have endpoints (finite or ideal) converging to the ideal endpoint of $[\wt b,\infty)$ as $n\to\infty$ and to the ideal endpoints of $\wt\gamma$ as $n\to-\infty$.  Under application of $\wh\nu$ (Proposition~\ref{main}), the same convergence property holds for the lifted set in $\wt L'$.  This establishes the following.

\begin{prop}\label{endpoints1}
Given an assemblage 
$$
\lambda\cup[a,\infty)\cup\bigcup_{n=-\infty}^{\infty}\sigma_{n}\ss L',
$$
and a lift of this assemblage to $\wt L$,
the endpoints of $\wt\sigma_{n}$ in $\Se$ converge to the ideal endpoints of $\wt\lambda$ as $n\to-\infty$ and to the ideal endpoint of $[\wt a,\infty)$ as $n\to\infty$.  
\end{prop}

\begin{rem}
The above discussion was for a negative assemblage.  In the positive case, $g$ should be a suitable lift of $h^{-q}$.  This gives Proposition~\ref{endpoints1} in both the positive and the negative case.
\end{rem}

\subsubsection{Verifying the endpoint correspondence property}

Let $\JJ^{\g}$ be the set of geodesic juncture componentss obtained as the geodesic tightening of the elements of $\JJ'$.  Let $\UU^{\g}_{-}$ denote the negative escaping set for $(\Lambda^{\g}_{+},\Lambda^{\g}_{-})$, the geodesic Handel-Miller bi\-lamination associated to $h'$.  The set $\YY^{\g}_{\mp}$ consists of the geodesic tightening of the elements of $\YY'_{\mp}$.  

In order to verify Axiom~\ref{ecorr} for $(\Lambda'_{+},\Lambda'_{-})$, we must show that $\Lambda^{\g}_{\pm}$ is exactly the set of geodesic tightenings of the elements of $\Lambda'_{\pm}$.

\begin{lemma} 
A component of $\JJ'$ escapes if and only if its geodesic tightening escapes.
Equivalently, the set $\YY^{\g}_{\mp}$ is exactly the set of nonescaping components of the geodesic junctures.
\end{lemma}

\begin{proof}
``Only if'' is given by Theorem~\ref{essc}.  Since there are no digon intersections of elements of $\JJ'_{\pm}$ with leaves of $\Lambda'_{\pm}$, any such intersection is essential, hence the same is true of the geodesic tightenings of such juncture components.  That is, nonescaping elements of $\JJ'$ tighten to nonescaping elements of $\JJ^{\g}$, proving the ``if'' assertion.
\end{proof}

By  the construction of the Handel-Miller geodesic bilamination in Section~\ref{constr}, we have the following.

\begin{cor}
The geodesic laminations and nonescaping juncture components satisfy the following:
\begin{enumerate}

\item $\Gamma^{\g}_{\pm}=\Lambda^{\g}_{\pm}\cup\YY^{\g}_{\mp}$\upn{;}

\item $\Gamma^{\g}_{\pm}$ is a closed geodesic lamination\upn{;}

\item $|\YY^{\g}_{\mp}|$ is dense in $|\Gamma^{\g}_{\pm}|$.

\end{enumerate}

\end{cor}

\begin{nota}
Let $X^{\g}_{\pm}$ denote the set of pairs $\{x,y\}\ss\Si$ such that there exists $\lambda\in\Lambda^{\g}_{\pm}$ having a lift $\wt\lambda$ with endpoints $x,y$.

\end{nota}

\begin{prop}\label{thesetXgeo}

The following two statements are equivalent,

\begin{enumerate}

\item There exists a sequence $\{\sigma_{n}\}_{n\ge 1}\ss\YY'_{\mp}$ and lifts $\wt\sigma_{n}$ having endpoints $x_{n},y_{n}\in\Se$ such that $x_{n}\to x$  and $y_{n}\to y$ as $n\to\infty$\upn{;}

\item $\{x,y\}\in X^{\g}_{\pm}$.

\end{enumerate}

\end{prop}

\begin{proof}
This follows immediately from the facts that lifts of $\sigma_{n}\in\YY'_{\mp}$ and $\sigma^{\g}_{n}\in\YY^{\g}_{\mp}$ have the same endpoints on $\Se$ and that the geodesic laminations $\Gamma^{\g}_{\pm}$ have the strongly closed property (Definition~\ref{strnglyclsed}).
\end{proof}

\begin{nota}
Let $X'_{\pm}\ss $ denote the set of pairs $\{x,y\}\ss\Si$ such that there exists $\lambda\in\Lambda'_{\pm}$ having a lift $\wt\lambda$ with endpoints $x,y$. We have,

\end{nota}

\begin{prop}\label{thesetXpr1}

\begin{enumerate}

\item For every pair $\{x,y\}\in X'_{\pm}$,  there exists a sequence $\{\sigma_{n}\}_{n\ge1}\ss\YY'_{\mp}$ and lifts $\wt\sigma_{n}$ having endpoints $x_{n},y_{n}\in\Se$ such that $x_{n}\to x$  and $y_{n}\to y$ as $n\to\infty$\upn{;}\label{itemoneprime}

\item $X'_{\pm}\ss X^{\g}_{\pm}$.\label{itemtwoprime} 

\end{enumerate}

\end{prop}

\begin{proof}
By Proposition~\ref{thesetXgeo}, $(\ref{itemoneprime})\Ra(\ref{itemtwoprime})$.  To prove (\ref{itemoneprime}), we prove  that for every $\{x,y\}\in X'_{+}$ there exists  a sequence $\{\sigma_{n}\}\ss\YY'_{-}$ and lifts $\wt\sigma_{n}$ having endpoints $x_{n},y_{n}\in\Se$ such that $x_{n}\to x$  and $y_{n}\to y$ as $n\to\infty$. The other case is analogous.  

If $x,y$ are the endpoints of the lift of a border leaf of $\UU'_{-}$ in $\Lambda'_{+}$, then such a sequence exists by Proposition~\ref{endpoints1}.   

 If $x,y$ are the endpoints of $\wt\lambda$, where $\lambda\in\Lambda'_{+}$ is not a border leaf of $\UU'_{-}$, then Proposition~\ref{borderdense'} implies that there is a sequence $\{\wt\lambda_{n}\}$ of lifts of border leaves that strongly converges to $\wt\lambda$.  Since $\Lambda'_{+}$ is strongly closed (Corollary~\ref{primestrcl}), the lifts $\wt\lambda_{n}$ have endpoints $a_{n},b_{n}\in\Si$ such that $a_{n}\to x$ and $b_{n}\to y$.  By Proposition~\ref{endpoints1},  there exists a sequence $\{\sigma_{n}\}\ss\YY'_{-}$ and lifts $\wt\sigma_{n}$ having endpoints $x_{n},y_{n}\in\Se$ such that $d(x_{n},a_{n})<1/n$ and $d(y_{n},b_{n})<1/n$. Here the distance is Euclidean distance on the closed unit disk.  Thus, $x_{n}\to x$ and $y_{n}\to y$ as required.
\end{proof}

\begin{cor}\label{correctside}

If $\ell\in\wt\Lambda'_{\pm}$ is the leaf with endpoints $\{x,y\}$ and a side of $\ell$ is specified, the lifts $\wt\sigma_{n}$ in part \emph{(\ref{itemoneprime})} of \emph{Proposition~\ref{thesetXpr1}} can be chosen to lie on that side unless $\ell$ is a lift of a border leaf of a positive principal region $P$ and the side in question borders a lift $\wt P$.

\end{cor}

\begin{proof}
If the specified side of the leaf $\ell$ is not isolated, the argument in the proof of (\ref{itemoneprime}) of  Proposition~\ref{thesetXpr1}  goes through on that side of $\ell$.  If the specified side of $\ell$ is isolated and borders $\UU'_{-}$ on that side, we   explicitly handled that case  in the proof of (\ref{itemoneprime}) of  Proposition~\ref{thesetXpr1}.  If the side of $\ell$ is isolated and the leaf $\ell$ borders a positive principal region $P$ on that side, there are no negative junctures meeting $P$, by definition.
\end{proof}

\begin{prop}\label{thesetXpr2}
$X^{\g}_{\pm}\ss X'_{\pm}$.
\end{prop}

\begin{proof}
We prove that $X^{\g}_{+}\ss X'_{+}$. The proof that $X^{\g}_{-}\ss X'_{-}$ is analogous. Suppose that $\{x,y\}\in X^{\g}_{+}$. By Proposition~\ref{thesetXgeo},  there exists  a sequence $\{\sigma_{n}\}\ss\YY'_{-}$ and lifts $\wt\sigma_{n}$ having endpoints $x_{n},y_{n}\in\Se$ such that $x_{n}\to x$  and $y_{n}\to y$ as $n\to\infty$. By definition, $\{x,y\}\in X^{\g}_{+}$ implies that there exists $\lambda^{\g}\in\Lambda^{\g}_{+}$ having a lift $\wt\lambda^{\g}$ with endpoints $x,y$. Since $x_{n}\to x$  and $y_{n}\to y$ as $n\to\infty$, it follows that the sequence of hyperbolic geodesics $\wt\sigma^{g}_{n}$ converges strongly to the hyperbolic geodesic $\wt\lambda^{\g}$. In particular, the sequence $\{\sigma^{g}_{n}\}$ does not escape so, by Theorem~\ref{essc}, the sequence $\{\sigma_{n}\}$ does not escape.

Remark that the points $x,y\not\in\bd\wt L'$, hence lie in $E$.  Let $B_{1}\ss B_{2}\ss\cdots\ss B_{j}\ss\cdots$ be an exhaustion of $\wt L'$ by compact sets homeomorphic to disks.  Since the sequence $\{\sigma_{n}\}_{n\ge 1}$ does not escape, for some $j_{0}\ge1$ and all large enough values of $n$, $\wt\sigma_{n}\cap B_{j_{0}}\ne\0$.  Thus, there exists points $z_{n}\in\wt\sigma_{n}\cap B_{j_{0}}$.  These points cluster on at least one point $z\in B_{j_{0}}$. By passing to a subsequence, we can assume that $z_{n}\to z$. The point $z$ projects to a point $w\in L'$ on which the sequence of negative juncture components $\{\sigma_{n}\}$ accumulates.  By Lemma~\ref{noaccpt}, there is a leaf $\ell\in\Lambda'_{+}$ containing $w$.  Thus, there is a lift $\wt\ell\in\wt\Lambda'_{+}$ containing $z$.  Let $a,b\in\Si$ be the endpoints of $\wt\ell$.

Without loss we can assume that the $\wt\sigma_{n}$ all lie on the same side of $\wt\ell$. By Corollary~\ref{correctside},  there exists a sequence $\{\tau_{k}\}_{k\ge1}\ss\YY'_{-}$ and lifts $\wt\tau_{k}$ having endpoints $a_{k},b_{k}\in\Se$ such that $a_{k}\to a$  and $b_{k}\to b$ as $k\to\infty$ with the $\wt\tau_{k}$ all lying on the same side  of $\wt\ell$ as the $\wt\sigma_{n}$'s. For each $k\ge 1$, the lifts $\wt\sigma_{n}$ must be caught between $\wt\tau_{k}$ and $\wt\ell$, for all large enough values of $n$.  This implies that the endpoints $x_{n}.y_{n}$ converge to the endpoints $a,b$ of $\wt\ell$. Choose notation so that $x_{n}\to a$ and $y_{n}\to b$ as $k\to\infty$. Therefore $\{x,y\} = \{a,b\}\in X'_{+}$. 
\end{proof}

By Propositions~\ref{thesetXpr1}  and~\ref{thesetXpr2},  $X^{\g}_{\pm} = X'_{\pm}$. Thus, the laminations $\Lambda'_{\pm}$ satisfy the endpoint correspondence  property and we have,

\begin{cor}\label{endpoints2222}
 \emph{Axiom~\ref{ecorr}} holds for  $\Lambda'_{\pm}$.
 \end{cor}
 
 All four axioms have been verified for $(\Lambda'_{+},\Lambda'_{-})$, hence the proof of Theorem~\ref{transfer} is complete.

\section{Foliations}\label{appfolns}

In Section~\ref{FoCones}, we give  a heuristic,  detailed outline of the first two authors' theory of ``foliation cones''~\cite{cc:cone,cc:almostnohol} for taut, depth one foliations of compact $3$-manifolds.   More generally, the monodromy of a proper leaf in an open saturated foliated subset of a $C^{2}$ foliation of a compact $3$-manifold $M$ is an endperiodic automorphism. If such a leaf is not at depth one it will have infinite endset. In Section~\ref{OpFolSet}, we give a sketch of the structure of open saturated foliated sets and in Section~\ref{endpinf} we show how the theory of endperiodic automorphisms can be extended to surfaces with infinite endset.

\subsection{A technical result about isotopic monodromies.}

Let $(M,\FF)=(M_{0},\FF_{0})$ be a depth one  (Definition~\ref{dpthonefoln}), compact, foliated $3$-manifold with $\tb M$ as sole compact leaves. As usual we set $M^{\o}=M\sm\tb M$. Let $f=f_{0}:L\to L$ be a monodromy for this foliation and let $f_{1}:L\to L$ be an endperiodic automorphism isotopic to $f_{0}$ by an isotopy $f_{t}$, $0\le t\le1$. If $f_{0}$ and $f_{1}$ are diffeomorphisms, the isotopy can be taken to be smooth. It will not be necessary to assume that $f_{t}$ is endperiodic, $0<t<1$. Let  $(M_{1},\FF_{1})$ be the depth one foliated manifold with monodromy $f_{1}$ given by Lemma~\ref{realize}.  

We are going to realize $M^{\o}=M_{0}^{\o}$ as $L\x[-1,1]/\{(x,1)\equiv(f_{0}(x),-1)\mid x\in L\}$.  The factors $\{x\}\x[-1,1]$ hook together to form a one dimensional foliation $\LL_{0}$ transverse to $\FF|M^{\o}$ inducing the monodromy $f=f_{0}$ on $L\x\{0\}$.  We modify $\LL_{0}$ on $L\x[-1,0]$   by replacing the segment $(f_{0}(x),t)$, $-1\le t\le0$, with the segment $(f_{t+1}(x),t)$, for all $x\in L$. Denote the new transverse foliation by $\LL_{1}$ and note that it induces monodromy $f_{1}$ on $L\x\{0\}$. In the smooth case, $\LL_{1}$ may have corners at $L\x\{-1\}$ and $L\x\{0\}$, but these can be smoothed by standard techniques. Thus, in constructing $(M_{1},\FF_{1})$ by Lemma~\ref{realize}, we can assume that $M^{\o}=M_{0}^{\o}=M_{1}^{\o}$ canonically.  Since the inclusions $M^{\o}\hra M_{0}$ and $M^{\o}\hra M_{1}$ are homotopy equivalences, we obtain the following result.

\begin{lemma}\label{homotequiv}
If $f_{0},f_{1}:L\to L$ are isotopic endperiodic automorphisms and $(M_{0},\FF_{0})$, $(M_{1},\FF_{1})$ are the corresponding depth one foliated $3$-manifolds given by \emph{Lemma~\ref{realize}}, Then there is a canonical identification $M_{0}^{\o}=M_{1}^{\o}$, $\FF_{0}|M_{0}^{\o}=\FF_{1}|M_{1}^{\o}$ and a canonical homotopy equivalence of $M_{0}$ to $M_{1}$.
\end{lemma}

In particular, $M_{0}$ and $M_{1}$ have the same homology and cohomology.  This will  be critical in what follows.

\begin{rem}
It seems certain that $M_{0}=M_{1}$ by a canonical homeomorphism (diffeomorphism if $f_{0}$ and $f_{1}$ are diffeomorphisms), but a proof seems surprisingly delicate.    Lemma~\ref{homotequiv} is adequate for our current purposes.
\end{rem}

\subsection{Foliation cones}\label{FoCones}

We give a brief synopsis of the results of~\cite{cc:cone,cc:almostnohol}, indicating how  Handel-Miller theory plays a key role in the proofs.

\begin{rem}
The theory of endperiodic automorphisms is basic to our theory of foliation cones. The development in~\cite[Section~5]{cc:cone}, particularly the proof of the transfer theorem~\cite[Theorem~5.8]{cc:cone}, is not adequate. We have rectified these problems in this paper and~\cite{cc:almostnohol}. In addition we have replaced the proof of~\cite[Lemma~4.10]{cc:cone} by an elementary proof of~\cite[Theorem~4.1]{cc:almostnohol}.

\end{rem}

Let $(M,\FF)$ be a depth one foliated $3$-manifold satisfying Hypothesis~\ref{ongo}.  That is, $M$ is  a compact connected $3$-manifold that is not a product such that  every component of $\tb M$ has negative Euler characteristic and the foliation $\FF$ is smooth, transversely oriented, taut, depth~one, and induces no 2-dimensional Reeb components on $\trb M$.

\begin{rem}
If $M$ has a tangential boundary component that is a toral leaf $T$, then each end of a leaf of $\FF$ that approaches $T$ has a neighborhood of the form $S^{1}\times [0,\infty)$ which spirals in on $T$. The torus $T\ss M$ has a neighborhood of the form $T\times [0,1]$ with $T\times\{0\} = T$ and $T\times\{1\} = T'$, a torus tansverse to $\FF$. Then $M' = M\sm T\times[0,1)$ has $T'$ as transverse boundary and $\FF|T'$ is a foliation of $T'$ by circles. The $3$-manifold $M$ has at most finitely many tangential toral boundary components all of which can be converted to transverse toral boundary components in this way to yield a $3$-manifold $M'$ that is either fibered by $\FF|M'$ or is such that $\FF|M'$ is a depth one foliation satisfying Hypothesis~\ref{ongo}.

\end{rem}

Since $(M,\FF)$ can be assumed to be of class $\Ci$,  we will carry out  all of our constructions in the smooth category. Let $\LL_{f}$ be a smooth, transverse, $1$-dimensional foliation inducing an endperiodic diffeomorphism $f:L\to L$ on a depth one leaf.  Let $\X_{f}$ denote the compact sublamination of $\LL_{f}$ consisting of the leaves that do not meet $\tb M$.  This is the $\LL_{f}$-saturation of the maximal compact $f$-invariant subset $X_{f}\ss L$. That is, $\X_{f}$ is the union of the leaves of $\LL_{f}$ which intersect $L$ in points of $X_{f}$. Since we assume that $f$ is not a translation, $X_{f}\ne\0$.  

\begin{defn}
The lamination $\X_{f}$ will be called the core (sub)lamination of $\LL_{f}$.
\end{defn}

The Schwartzmann-Sullivan theory of asymptotic cycles~\cite{sull:cycles,sch_cycles} associates to the core lamination $\X_{f} $ a closed convex cone with compact base in the infinite dimensional space of closed de~Rham $1$-currents.  For example, every closed leaf of $\X_{f}$ is a homology cycle asymptotic to $\X_{f}$.  More generally, any leaf $\ell\in\X_{f}$ defines ``long almost closed orbits'' which converge to asymptotic cycles.  One takes a sequence of longer and longer subarcs $P_{k}$ of $\ell$, divides by the length $\|P_{k}\|$ of the subarc, obtaining a sequence of  de~Rham 1-currents containing subsequences that converge to asymptotic cycles called ``homology directions''. All asymptotic cycles are limits of sequences of linear combinations of homology directions with positive coefficients.    One also allows $0$ to be considered an asymptotic cycle so that the cone has a vertex.  The cone of asymptotic cycles passes to a closed convex cone $\C'_{f}\ss H_{1}(M;\R)$ with compact base.  The dual cone $\C_{f}\ss H^{1}(M;\R)$ consists of exactly those classes which take all values $\ge0$ on $\C'_{f}$.  This cone is closed and convex, has nonempty interior, but does not necessarily have a compact base.

\begin{rem}
There are extreme cases in which $\C'_{f}$ reduces to a single ray issuing from the origin, in which case $\C_{f}$ is an entire half space.
\end{rem}

\begin{defn}[foliated form]\label{ff}
A 1-form $\eta\in A^{1}(M^{\o})$ is a foliated form if it is closed, nowhere vanishing and becomes unbounded at $\tb M$ in such a way that the corresponding foliation $\FF^{\o}_{\eta}$ that it defines on $M^{\o}$ extends  by adjunction of $\tb M$ to a transversely oriented $\Ci$ foliation ${\FF}_{\eta}$ of $M$, $\Ci$-flat  at $\tb M$.
\end{defn} 

The following is a compilation of~\cite[Theorem~4.9 and Theorem~7.1]{cc:almostnohol}, the last assertion about the structure of the foliations being rather elementary.

\begin{prop}\label{folforms}
 The  open cone $\intr\C_{f}$ consists of  classes in $H^{1}(M)$ that can be represented by foliated forms transverse to $\LL_{f}|M^{\o}$.  The ray $\left<\eta\right>$ issuing from the origin through any of these classes $[\eta]$ \upn{(}called a foliated ray\upn{)} determines the foliation $\FF_{\eta}$ uniquely up to a $\CO$ ambient isotopy.  Those foliated rays passing through nontrivial elements of the integer lattice $H^{1}(M;\Z)$ \upn{(}called rational rays\upn{)} determine foliations of depth one, while the irrational foliated rays define foliations which, in $M^{\o}$, are dense leaved without holonomy.
 \end{prop}

 \begin{defn}
 If $\GG=\FF_{\eta}$ for a foliated form $\eta$, we denote the foliated ray $\left<\eta\right>$ by $\left<\GG\right>$.
 \end{defn}
 
  \begin{defn}
  The cones $\C_{f}$, obtained from smooth, one dimensional foliations $\LL_{f}$ transverse to $\FF$, are called the \emph{foliation cones} associated to $\FF$. 
  \end{defn}

 One wants to find a monodromy $h$ for $\FF$ with  ``tightest'' dynamics in the sense that the   cone $\C'_{h}$ is contained in every    cone $\C'_{g}$ as $g$ varies over the smooth monodromies for $\FF$. Equivalently, $\C_{h}$ will contain every $\C_{g}$.   As the reader will have guessed, $h$ will be a Handel-Miller monodromy (Definition~\ref{HMmonofoln}). We are going to  sketch the proof after some preliminary considerations.
 
 To begin with, by Theorem~\ref{HMsmooth}, there is a smooth Handel-Miller  automorphism $h$, that is  a smooth endperiodic automorphism  $h:L\to L$  isotopic to a monodromy automorphism $f:L\to L$ for $\FF$ and preserving a Handel-Miller pseudo-geodesic bilamination $(\Lambda_{+},\Lambda_{-})$  associated to $f$.  Using Lemma~\ref{realize}, construct a depth one foliated manifold $(M_{h},\FF_{h})$ with monodromy $h$.  By Lemma~\ref{homotequiv}, we have $M^{\o}=M_{h}^{\o}$.  Given any smooth one dimensional foliation $\LL$ transverse to $\FF$, its core sublamination $\XX$ lives in $M^{\o}=M_{h}^{\o}$  The restriction of $\LL$ to $M_{h}^{\o}$ is integral to a smooth vector field which, near $\tb M_{h}$ and outside a neighborhood of $\X$ can be smoothly modified so as to be defined and nonsingular on $M_{h}$ and transverse to $\FF_{h}$ there.  This provides a one dimensional foliation $\LL'$ of $M_{h}$, transverse to $\FF_{h}$ and having $\X$ as its core sublamination.  By the equalities $H_{1}(M)=H_{1}(M_{h})$ and $H^{1}(M)=H^{1}(M_{h})$, we see that every foliation cone associated to $\FF$ is identical with a foliation cone associated to $\FF_{h}$.  Reasoning similarly, every foliation cone associated to $\FF_{h}$ is also a foliation cone associated to $\FF$.  We summarize.
 
 \begin{lemma}
 Let $h$ be a smooth Handel-Miller automorphism isotopic to a monodromy of $\FF$.  Then, under the canonical identification $H^{1}(M;\R)=H^{1}(M_{h};\R)$,  the foliation cones associated to $\FF$ are identical with those associated with $\FF_{h}$.
 \end{lemma}

\begin{*hyp}
 \textbf{From now on in Section~\ref{FoCones}, we assume that $\FF$ itself has a smooth Handel-Miller monodromy $h$.}
\end{*hyp}

   Let $\LL_{h}$ be a smooth, transverse, $1$-dimensional foliation defining $h$, $\X_{h}$ the compact sublamination of $\LL_{h}$ which is the $\LL_{h}$-saturation of the maximal, compact, $h$-invariant subset $X_{h}\ss L$.
  In general,  $X_{h}$ consists of $|\Lambda_{+}|\cap|\Lambda_{-}|\cup N$, where $N$ is a compact surface, generally not connected, which is the union of the nuclei of principal regions.  The components of $N$ are permuted by $h$.  An isotopy of $h$  supported in $N$ puts $h|N$ in Nielsen-Thurston canonical form~\cite{FLP} without destroying the fact that $h$ leaves the laminations invariant.  This means that $N$ is partitioned into connected subsurfaces $N_{i}$, each of which is invariant under $h^{p}$, some $p\ge1$, such that $h^{p}|N_{i}$ is either pseudo-Anosov or periodic.  These subsurfaces are bordered by annuli in $N$ and $h^{p}|N$ is smooth except at finitely many multi-pronged singularities.
  
  \begin{defn}
  If  $h$ is Handel-Miller and $h|N$  is in Nielsen-Thurston canonical form, $h$ is said to be a tight Handel-Miller monodromy automorphism for $\FF$.
  \end{defn}
 
 \begin{rem}
 The transfer theorem (Theorem~\ref{transfer}) easily extends to show that tight Handel-Miller monodromy $h$ on $L$ transfers to tight Handel-Miller monodromy $h'$ on $L'$~\cite[Theorem~6.20]{cc:almostnohol}.   Since $\X_{h}=\X_{h'}$, we see that $\C_{h}=\C_{h'}$.
 \end{rem}
 
  Remark that $h$ is not uniquely determined by its isotopy class.  We state without proof Theorem~6.18 of~\cite{cc:almostnohol}.
 
 \begin{prop}\label{HMuniq}
 The cone $\C_{h}\ss H^{1}(M)$ is independent of the choice of the representative $h$ of tight Handel-Miller monodromy automorphism for $\FF$.
 \end{prop}
 
 Because of this proposition, we will often denote $\C_{h}$ by $\C_{\FF}$.
 
 \begin{defn}[Handel-Miller foliation cone]\label{HMfolncone}
 The cone  $\C_{h}=\C_{\FF}$ is called the \emph{Handel-Miller foliation cone} associated to $\FF$.
 \end{defn}
 
 The Handel-Miller foliation cones have a nice geometric structure which often makes explicit computations possible.
 
\begin{prop}
Each Handel-Miller foliation cone $\C_{\FF}\ss H^{1}(M;\R)$ is polyhedral.
\end{prop}

\begin{proof}[Outline of the proof]
Let $X^{*}_{h}=|\Lambda_{+}|\cap|\Lambda_{-}|$, called the \emph{meager} invariant set,  and let $\X^{*}_{h}$ denote the $\LL_{h}$-saturation of the meager invariant set.  For each Nielsen-Thurston component $N_{i}$, let $\X_{N_{i}}$ denote the $\LL_{h}$-saturation.  If $A$ is the union of annuli bordering these components, as well as any annular or M\"obius nuclei of principal regions, let $\X_{A}$ denote the $\LL_{h}$-saturation of $A$.  We can ignore nuclei that are disks since their $\LL_{h}$-saturations contribute no asymptotic cycles not already contributed by $\X^{*}_{h}$.  The Markov partition of Section~\ref{cordyn} determines finitely many minimal period $h$-orbits in $X^{*}_{h}$ which correspond to minimal loops in $\X^{*}_{h}$.  Homologically, all loops in $X^{*}_{h}$ are linear combinations of minimal loops with coefficients nonnegative integers.  In turn, homology directions corresponding to leaves of $\X^{*}_{h}$ are limits, homologically, of sequences of nonnegative linear combinations of loops, while every asymptotic cycle for $\X^{*}_{h}$ is in the closure of the linear span with nonnegative coefficients  of the homology directions.  Bottom line, the finitely many minimal loops span the homology cone represented by asymptotic cycles  of $\X^{*}_{h}$.  Similarly, if $N_{i}$ is a pseudo-Anosov component of the Nielsen-Thurston decomposition of $N$, invariant under $h^{p}$, there is a Markov partition for $h^{p}|N_{i}$ and the same reasoning shows that the homology cone represented by  asymptotic cycles of $\X_{N_{i}}$ is spanned by finitely many minimal loops.  If $N_{i}$ is a periodic component, $\X_{N_{i}}$ is a compact $3$-manifold which is Seifert fibered by $\LL_{h}$.  The cone has a single generator.  It is elementary that $\X_{A}$ has finitely many generating cycles, at most $2$ for each annulus or M\"obius strip.  All of this gives a finite spanning set, with possible redundancies, for $\C'_{h}$.  It follows easily that $\C'_{h}$ and $\C_{h}=\C_{\FF}$ are polyhedral.
\end{proof}

 \begin{prop}\label{HMmax}
 Let $h$ be a tight Handel-Miller monodromy for $\FF$. Then $\C_{h}$ is the maximal foliation cone $\C_{u}$ as $u$ ranges over the \upn{(}smooth\upn{)} monodromies for $\FF$.
 \end{prop}
 
 \begin{proof}[Sketch of the proof]
 If not, one easily finds a rational foliated ray $\left<\GG\right>$ lying in the boundary of $\C_{h}$.  Let $g$  be a tight Handel-Miller monodromy for $\GG$.  Then  $\left<\GG\right>\ss\intr\C_{g}$ and so $\intr\C_{h}$ and $\intr\C_{g}$ intersect.  Let $\left<\HH\right>$ be a rational foliated ray in the interior of both cones.  By Proposition~\ref{folforms}, this ray is represented by a foliated form $\eta$ transverse to  $\LL_{h}|M^{\o}$. By~\cite{cc:isotopy}, it follows that $\HH$ is isotopic to a foliation $\HH'$ transverse to $\LL_{h}$ by an ambient isotopy that is smooth in $M^{\o}$.  By Theorem~\ref{transfer}, $\LL_{h}$ induces tight Handel-Miller monodromy $h'$ on the depth one leaves of $\HH'$.  Similarly, $\HH$ is isotopic to a foliation $\HH''$ transverse to $\LL_{g}$ which induces tight Handel-Miller monodromy on the depth one leaves of $\HH''$. An ambient isotopy of $\HH''$ to $\HH'$, applied to $\GG$ and $\LL_{g}$ allows us to assume that $\HH'$ is transverse to both $\LL_{h}$ and $\LL_{g}$, each inducing tight Handel-Miller monodromy $h'$ and $g'$, respectively, on the depth one leaves of $\HH'$.  By Proposition~\ref{HMuniq}, $\C_{g}=\C_{g'}=\C_{h'}=\C_{h}$, contrary to the assumption that $\left<\GG\right>$ lies on the boundary of $\C_{h}$.    
 \end{proof}
 
\begin{rem}
The use of the transfer theorem in the above argument is critical. We know of no other way to prove the maximality of the Handel-Miller foliation cones. 

\end{rem}
 
Finally, there are only finitely many Handel-Miller foliation cones by~\cite[Theorem~6.4]{cc:cone} or~\cite[Theorem~6.25]{cc:almostnohol},
hence we have classified all depth one foliations of $M$ up to isotopy by a finite set of combinatorial data.  Explicit computations of the cones for some knot complements have been made. For instance, cf.~\cite[Section~7]{cc:cone}.

\begin{rem}
This theory is  closely analogous to the classification of smooth foliations without holonomy transverse to $\bd M$.  These foliations are either fibrations of $M$ over $\SI$ or they are dense leaved.  A  well known theorem of W.~Thurston~\cite{th:norm} shows that, if $M$ has any such foliations, certain top dimensional faces of the ``Thurston ball'' (a convex polyhedron which  is the unit ball of the Thurston norm) subtend polyhedral cones $\C_{1},\C_{2},\dots,\C_{r}\ss H^{1}(M;\R)$ such that the rational rays in the interior of these cones correspond one-to-one to the smooth isotopy classes of fibrations.  Furthermore, combining the Laudenbach-Blank theorem~\cite{LB} with a theorem of the first two authors~\cite{cc:LB}, the irrational rays in the interiors of the Thurston cones correspond one-to-one to the $\CO$ isotopy classes of dense leaved foliations without holonomy.
\end{rem}

\begin{rem}
Recently,  I.~Altman~\cite{alt:depth1} has shown that, with some important restrictions on the sutured manifold $(M,\gamma)$, our maximal foliation  cones are subtended by certain top dimensional faces of the dual Juh\'asz polytope~\cite{juh:poly}, the unit ball for a \emph{nonsymmetric} norm defined via sutured Floer homology.  For these sutured manifolds,  the nonsymmetry of the norm implies that our cones are not permuted under multiplication by $-1$.  In the general case, this is a consequence of the fact that, as part of the definition of a sutured manifold, there is a given transverse orientation on the boundary leaves.  This poses a restraint on which transversely oriented foliations  are allowed.  Reversing the boundary orientations changes the set of maximal foliation cones by multiplication by $-1$.
\end{rem}

\subsection{Open foliated sets of relative depth one}\label{OpFolSet}

In this subsection, we shetch how the  relation between depth one foliations and endperiodic automorphisms  can be extended to open foliated sets without holonomy in compact $3$-manifolds. See~\cite[Section~5.2]{condel1} for a treatment of open foliated sets without holonomy. In~\cite{cc:withouthol} we begin a theory of foliation cones for open foliated sets without holonomy.

\subsubsection{The $\CII$ case}\label{C2case}

Let $\FF$ be a  $\CII$ foliation  of a compact $3$-manifold $M$.  We continue to require $\FF$ to be transversely oriented,  taut and without $2$-dimensional Reeb components in $\trb M$.  We also require that every compact leaf of $\FF$ has strictly negative Euler characteristic and that no noncompact leaf is a plane, open annulus or open M\"obius strip. We fix a transverse, $1$-dimensional foliation $\LL$.

Let $W\sseq M$ be an open, connected $\FF$-saturated set (i.e., an open, connected union of leaves). There is a   notion of ``transverse completion'' $\wh W$~\cite[Page~130]{condel1} of such a set, analogous to the notion ``internal completion''  (Definition~\ref{intcompl}).  Essentially, it is a generally noncompact, foliated $3$-manifold with finitely many boundary leaves and perhaps infinitely many transverse boundary components. We continue to write $\bd\wh W=\tb\wh W\cup\trb\wh W$.  Some or all of the finitely many boundary leaves may be noncompact and some or all of the  components of $\trb\wh W$ may be noncompact. Again, we require that the induced foliation of $\trb \wh W$ contain no Reeb components. Also,  $\tb\wh W$ and $\trb\wh W$ are separated by convex corners wherever they meet. The natural inclusion $\iota:W\hra M$ extends to a natural immersion $\wh \iota:\wh W\looparrowright M$ which may identify some boundary leaves pairwise.  We set $\wh\FF=\wh \iota^{-1}(\FF)$, and $\wh\LL=\wh \iota^{-1}(\LL)$, the foliations induced on $\wh W$ by $\FF$ and $\LL$, respectively. These foliations are $C^{2}$. 

\begin{defn}[relative depth one]\label{reldpthone}
The foliation $\wh \FF$ has relative depth one if $\wh\FF|W$ ($=\FF|W$) fibers $W$ over $\SI$.
\end{defn}

\begin{rem}
If $W$ is an open saturated set without holonomy  in a foliation $\FF$ with $\wh W$ not foliated as a product,  then $\wh\FF$ has relative depth one.

\end{rem}

In order to properly visualize these foliations, we need the notion of an ``octopus decomposition''~\cite[Definition~5.2.13]{condel1} of $\wh W=K\cup A_{1}\cup\cdots\cup A_{r}$.  Here, K is a compact, connected $3$-manifold with boundary and corners, foliated by $\wh\FF|K$, the corners dividing $\tb K=K\cap\tb\wh W$ from $\trb K$. We call $K$ the nucleus of the octopus decomposition and $A_{i}$ the arms.   The arms are of the form $A_{i}= L_{i}\x I$, the $I$-fibers being leaves of $\wh\LL|A_{i}$ and $L_{i}$ a noncompact, connected surface with boundary. The arms attach to $K$ along annular components of $\trb K$ and/or rectangular subsets of $\trb K$. 

As a consequence of the Generalized Kopell Lemma~\cite[Lemma~8.1.24]{condel1},  one proves that, when $\wh\FF$ has relative depth one, the junctures in $\tb\wh W$ are compactly supported cohomology classes $\kappa$. This requires differentiability of class at least $\CII$. This is a special case of~\cite[Theorem~8.1.26]{condel1}.  The following is an easy consequence.

\begin{prop}
If $\wh\FF$ has relative depth one, the nucleus of the octopus decomposition can be chosen large enough that, in the arms $A_{i}$, $\wh\FF$ is the product foliation with leaves $L_{i}\x\{t\}$.
\end{prop}

Thus, everything interesting happens in $\wh\FF|K$, which is an honest depth one foliation.  Since $\tb K\sseq\tb\wh W$, we can assume that every component $F$ of $\tb K$ has negative Euler characteristic.  Indeed, either $F$ is a full component of $\tb \wh W$, in which case it is a compact leaf of $\FF$ and has negative Euler characteristic by hypothesis, or $F$ has nonempty boundary and  is contained in a noncompact component $N$ of $\tb\wh W$.   By hypothesis, $N$ is not a plane, open annulus or open M\"obius strip.  Thus, choosing $K$ large enough guarantees that $F$ is not a disk, closed annulus or closed M\"obius strip.  This shows that the depth one foliation $\wh\FF|K$ of the compact $3$-manilold $K$ satisfies all the requirements of our previous discussion.

 Let $L$ be a leaf of $\FF|W$ and $L'=L\cap K$.  Since the leaves of $\wh \FF|A_{i}$ are of the form $L_{i}\x\{t\}\ss L_{i}\x I$, $1\le i\le r$, the fact that $L$ is connected implies that $L'$ is connected.  By Lemma~\ref{endpmono}, the monodromy of $L'$ is endperiodic and so the monodromy of $L$ is endperiodic (in the sense of Definition~\ref{epdefn}) if infinite endsets are allowed.  
 
 \begin{rem}
 In Section~\ref{endpinf} we show how the theory of endperiodic automorphisms can be extended to surfaces with infinite endset.
  
  \end{rem}
 
 \begin{prop}
If $f:L\to L$ is the monodromy of a leaf of $\FF|W$, where $\FF$ is $C^{2}$ and $\wh\FF$ has relative depth one, then $f$ is endperiodic.
\end{prop}

\begin{rem}
We refer to $L'$ as the \emph{soul} of $L$ (Definition~\ref{soul}). The above construction provides motivation for the construction of the soul  (Proposition~\ref{soulexists}) in the general situation of an endperiodic automorphism of a surface with infinite endset.   In the  construction of Proposition~\ref{soulexists}, the ends that are not periodic are pared off to obtain the soul. In the above construction  these  ends that are pared off of $L$ to obtain the soul $L'$ are the surfaces $L_{i}\x\{t\}$, $1\le i\le r$,  which attach to $L'$. In the above construction the choice of $L'$ is not unique but  clearly depends on how large we choose $K$. In the   construction of Proposition~\ref{soulexists} the soul also depends on choices.
\end{rem}

\subsubsection{Finite depth foliations}\label{fdfol}  

One says that a leaf of $\FF$ has depth~$0$ if it is compact.  Inductively, a leaf $L$ has depth~$r\ge1$ if $\ol L\sm L$ is a  union of leaves of depths~$\le r-1$, at least one of which has depth~$r-1$.  It is possible that a leaf has no well defined depth (for example, a leaf that is dense in $M$).  

\begin{defn}[depth $r$]\label{dpthr}
The foliation $\FF$ has depth $r$ if all leaves have finite depth and $r$ is the least upper bound of the depths.
\end{defn}

The leaves at depth~$r$ of a depth~$r$ foliation unite to form an open saturated set and one can establish the following.

\begin{prop}
If $\FF$ is a  depth~$r$ foliation and $W$ is a component   of the union of depth~$r$ leaves, then the foliation $\wh\FF$ of $\wh W$ has relative depth one.
\end{prop}

It is easy to construct $\Ci$ depth~$r$ foliations with all leaves proper for arbitrary $r>1$. For examples, see~\cite[Theorems~2,~3]{cc:lp}.  These examples have leaves of type $r-1$ (Definition~\ref{typek}) and thus monodromy an endperiodic automorphism of a surface with infinite endset.

In  many of the examples of finite depth foliations constructed by Gabai'~\cite{ga1}, it is possible to choose all the junctures to be compact. If all junctures are compact, ~\cite[Main Theorem]{cc:smth2} implies that the finite depth foliation can be $\Ci$ smoothed giving many examples of $\Ci$ finite depth foliations.

\subsubsection{Leaves at relative depth one with a Cantor set of ends}\label{duminy}

Let $X\ss M$ be an exceptional minimal set of a $C^{2}$ foliation $\FF$ of $M$.  This is a compact, nonempty, saturated set which does not properly contain another such and does not reduce to a single compact leaf nor to $M$ itself.  Such a set is transversely Cantor.  A theorem of G.~Duminy (unpublished, but cf.~\cite{cc:dum}) asserts that the semi-proper leaves in $X$ (i.e., those which border a gap in the Cantor set) have a Cantor set of ends.  If $W$ is a connected component of $M\sm X$ and $\wh\FF$ is of relative depth one in $\wh W$, then the leaves of $\FF|W$ will have a Cantor set of ends.

There are many examples of exceptional minimal sets. See, for example, Raymond's example~\cite[Corollary~8.4.2]{condel2}. In these examples the  leaves will have monodromy which is a normal endperiodic automorphism (Definition~\ref{normal}) and thus, by Propositions~\ref{finm} and~\ref{att-repinf},  finitely many of the ends will be attracting and finitely many will be repelling.

\subsubsection{Hyperbolic knots and the $\COO$ case}

A $\CO$ foliation is $\COO$ it its leaves are integrable to a $\CO$ $2$-plane field. If $\gamma\ss S^{3}$ is a hyperbolic knot, then $M=S^{3}\sm\gamma$ is a cusped hyperbolic $3$-manifold.  In general, a cusped hyperbolic $3$-manifold has finitely many cusps which are topologically of the form $T^{2}\x[0,\infty)$.  Amputating these cusps along $T^{2}\x\{0\}$ leaves behind a compact $3$-manifold $M'$ with finitely many new boundary components, all tori,  A finite depth foliation $\FF$ of a cusped $3$-manifold is to be a product foliation in each cusp $T^{2}\x[0,\infty)$, the leaves there being of the form $C\x[0,\infty)$ where $C$ ranges over a family of circles in $T^{2}$ which fiber that torus.  The foliation $\FF|M'$ will be finite depth in the usual sense, meeting the new tori in $\bd M'$ in circles that fiber those tori  and in general are only known to be  $\COO$.  If $\FF|M'$ is $\CII$, then the monodromy of the leaves of an open saturated set without holonomy will be endperiodic and the methods of Section~\ref{fdfol} apply. If $\FF|M'$ can not be $\CII$ smoothed,  then some junctures will not be compact (\cite[Main Theorem]{cc:smth2}) and the  mondromy of the leaves of an open saturated set without holonomy will not be endperiodic. An approach to the type of monodromy that occurs in these $\COO$ foliations that can not be $\CII$ smoothed is given in~\cite{cc:withouthol}. 

D.~Gabai~\cite[Theorem~3.1]{ga3} shows how to construct taut, finite depth foliations of a knot complements in $S^{3}$ which meet the boundary of the knot complement in circles. The construction of these foliations is intrinsically $\COO$ and not all of these foliations  can be smoothed. In fact, In~\cite[Theorem~III]{cc:depth} the first two authors give an example of a knot in $S^{3}$ that has no taut, finite depth, $C^{2}$ foliation which meets the boundary of the knot complement in circles.

\subsection{Endperiodic automorphisms of surfaces with infinite endsets}\label{endpinf}

In this subsection we treat endperiodic automorphisms of surfaces with infinite endset.  The definitions and results here are all exemplified by the foliations $\wh\FF$ of $\wh W$ of relative depth one of Section~\ref{OpFolSet}.
\medskip

Recall that $\EE(L)$   is a totally disconnected separable metric space. Define the $0^{\thh}$ derived endset to be $\EE^{(0)}(L)=\EE(L)$.  If, for an ordinal $\alpha\ge0$, the $\alpha^{\thh}$ derived endset $\EE^{(\alpha)}(L)$ has been defined and is a compact subset of $\EE(L)$, then the $(\alpha+1)^{\thh}$ derived set  $\EE^{(\alpha+1)}(L)$ is the set of cluster points in $\EE^{(\alpha)}(L)$.  If $\beta$ is a limit ordinal and $\EE^{(\alpha)}(L)$ has been defined for all $\alpha<\beta$, we define the $\beta^{\thh}$ derived set to be $\EE^{(\beta)}(L)=\bigcap_{\alpha<\beta}\EE^{(\alpha)}(L)$.  It can be shown that, for a first \emph{countable} ordinal $\gamma$, either $\EE^{(\gamma)}(L)$ is finite and nonempty, or $\EE^{(\gamma)}(L)=\EE^{(\Omega)}(L)$ is a Cantor set, where $\Omega$ is the first uncountable ordinal~\cite{pierce}.\label{endnote}

The following terminology has been used by the first two authors elsewhere  and will also be found  in~\cite[Section~4.1]{condel1}, but may not be considered standard.

\begin{defn}[type of surface]\label{typek}
Let $\gamma\ge0$ be the first countable ordinal such that $\EE^{(\gamma)}(L)$ is either finite or a Cantor set. If $\EE^{(\gamma)}(L)$ is finite and nonempty, the surface $L$ has \emph{topological type} $\gamma$.   Otherwise, $L$  has topological type $\Omega$.  If $L$ is  compact it is said to have  topological type $-1$.
\end{defn}

\begin{defn}[type of end]\label{typeend}
An end of $L$ is of type $\alpha$ if it is isolated in the $\alpha^{\thh}$ derived endset.  If the end lies in $\EE^{(\Omega)}(L)$, it has type $\Omega$.
\end{defn}

\begin{rem}
In  $C^{2}$ foliations, leaves at finite depth~$k$ have topological type~$k-1$ and growth type exactly polynomial of degree~$k$ \cite[Theorem 6.0]{cc:pb}.  This was one of the first theorems relating the topology of a leaf and its volume growth function.  It is valid for codimension~$1$ foliations of compact $n$-manifolds, $n\ge3$.

\end{rem}

In what follows $L$ is allowed to have any topological type $\le\Omega$.

\begin{lemma}\label{nin}
If $f:L\to L$ is endperiodic, then a neighborhood of a periodic end $e'$ cannot lie in an $f$-neighborhood $U_{e}$ for another periodic end $e$.
\end{lemma}

\begin{proof}
For then the iterates  $f^{k}(e')$ would all have to be distinct, contradicting the fact that $e'$ is periodic. 
\end{proof}

In our   applications to foliations in Section~\ref{OpFolSet}, it was necessary to consider endperiodic automorphisms $f:L\to L$ where $L$
has topological type $\Omega$ (cf.~Section~\ref{duminy}).  Without some added hypothesis, examples show that some very undesirable bizarre behavior can arise.   The following condition will always be verified in the foliation setting.

\begin{defn}[normal endperiodic automorphism]\label{normal}
An endperiodic automorphism $f:L\to L$ is normal if  the  $f$-minimal sets in $\EE(L)$ are the finite periodic orbits of prime period.
\end{defn}

Of course, the finite periodic orbits of prime period are $f$-minimal.  By the standard Zorn's Lemma argument, there are $f$-minimal sets.  By the following, normality only becomes a restriction if $L$ has type $\Omega$.

\begin{lemma}\label{alphanormal}
If $f:L\to L$ is endperiodic and $L$ has topological type $\gamma<\Omega$, then $f$ is normal.
\end{lemma}

\begin{proof}
Let $X\sseq\EE(L)$ be an $f$-minimal set.  Let $e\in X$.  If this is not a periodic end, then $\{f^{k}(e)\}_{k\in\Z}$ is infinite and clusters at every point of $X$, including $e$ itself.  If $e$ has type $\alpha$, then $\alpha\le\gamma<\Omega$ and every element of $X$ must have type $>\alpha$, contradicting the fact that $e$ has type $\alpha$. This contradiction means that the end $e$ is periodic, so $f$ is normal.
\end{proof}

\begin{rem}
Further, if $L$ with topological type $\Omega$ is a leaf of a $C^{2}$ foliation of an open saturated set $W$ of relative depth one as in Section~\ref{C2case}, then the monodromy map $f:L\to L$ is a normal endperiodic automorphism.

\end{rem} 

\begin{prop}\label{finm}
If $f:L\to L$ is endperiodic and normal, there are only finitely many periodic ends.
\end{prop}

\begin{proof}
Let $\PP(L)\sseq\EE(L)$ be the set of periodic ends.  If this set is infinite, the set $\QQ(L)\sseq\EE(L)$ of cluster points of $\PP(L)$ must be compact, $f$-invariant and nonempty. Thus, $\QQ(L)$ contains an $f$-minimal set.  Since $f$ is normal, this is a finite periodic orbit on which $\PP(L)$ accumulates, contradicting Lemma~\ref{nin}.
\end{proof}

\begin{example}\label{normal1}
The methods of~\cite[Section~5]{cc:examples} can be used to give an example of an endperiodic automorphism $f:L\to L$ in which the endset of $L$ is the union of one negative, isolated, nonplanar  end and a Cantor set of nonplanar ends containing a countable infinite set of positive ends and  no other periodic ends. Thus, Proposition~\ref{finm} is not true without the assumption that the endperiodic automorphism is normal.
\end{example}

\begin{prop}\label{att-repinf}
If $L$ has topological type $\alpha\ge1$, a normal endperiodic automorphism $f:L\to L$ has both positive  and negative ends.  
\end{prop}

\begin{proof}
 There will be finitely many periodic ends of type $\alpha$, each necessarily negative or positive. If there are no positive ends of type $\alpha$ replace $f$ by $f^{-1}$ in the following proof. Therefore we can assume there is a positive end $e$ of type $\alpha$.

The fact  that $\alpha\ge1$  implies that $\EE(L)$ is infinite. Hence Proposition~\ref{finm} implies that there are nonperiodic ends. Let $U_{e}$ be an $f$-neighborhood of $e$ and let $e_{0}\in U_{e}$ be an end with $e_0\ne e$. Then by Definition~\ref{perends}, $e_0$ is  a nonperiodic end. If $e_{i} = f^{ip_{e}}(e_{0})$, the set of accumulation points of $\{e_{i}\}_{i=0}^{-\infty}$ contains a minimal set which is a finite periodic  orbit since $f$ is normal. Clearly this orbit cannot contain $e$.  By Lemma~\ref{emptint}, the ends in this orbit cannot  be positive ends since each end in the orbit has an $f$-neighborhood meeting an $f$-neighborhood of the positive end $e$. Thus, the set of negative ends is nonempty.
\end{proof}

\begin{example}\label{normal2}
In~\cite[Section~5]{cc:examples}, we give an example of an endperiodic automorphism $f:L\to L$ in which the endset of $L$ is the union of one negative, isolated, nonplanar  end and a Cantor set of nonplanar ends, none of which are periodic. Thus, Proposition~\ref{att-repinf} is not true without the assumption that the endperiodic automorphism is normal.
\end{example}

 \begin{example}\label{normal3}
 
Let  $C\ss S^{2}$ be a Cantor set, $L = S^{2}\sm C$ be an the open, planar surface with $\EE(L) = C$, and $\nu:C\to C$ be the homeomorphism without periodic points  given in~\cite[Proposition~2]{cc:examples}. Then it is easy to define a homeomorphism of $S^2$ which induces the homeomorphism $\nu$ on $C$ and thus restricts to a homeomorphism  $f:L\to L$  which has no periodic ends. Thus $f$ is an endperiodic automorphism by default and has no positive nor negative ends. The endperiodic automorphism $f$ obviously cannot be normal.

\end{example}

\begin{rem}
Examples~\ref{normal1},~\ref{normal2},~\ref{normal3},  Lemma~\ref{alphanormal}, and the remark following it, indicate that the concept ``normal endperiodic automorphism'' is a natural concept.

\end{rem}

 \begin{defn}[escaping end of surface]\label{escpend}
 An end $e$ of $L$ is escaping if $\{f^{k}(e)\}_{k\in\Z}$ accumulates exactly on an $f$-cycle of positive periodic ends as $k\to\infty$ and on an $f$-cycle of negative ones as $k\to-\infty$.
 \end{defn}
 
 \begin{lemma}\label{clusterper}
 If $f:L\to L $ is a normal endperiodic automorphism, every nonperiodic end of $L$ is escaping.
  \end{lemma}
 
 \begin{proof}
  If $e\in\EE(L)$ is nonperiodic, the set $\QQ_{+}(e)$ of accumulation points of $\{f^{k}(e)\}_{k\ge0}$ is compact, nonempty and $f$-invariant.  It contains a periodic point $\epsilon$ such that every    $f$-neighborhood $U_{\epsilon}^{i}$ of $\epsilon$ in $L$ is a neighborhood of some $f^{k}(e)$, $k\ge0$.  Since this is true as $k\to\infty$, $\epsilon$ must be a positive end.  But once $f^{k}(e)\in U_{\epsilon}^{i}$, then $f^{k+pn_{\epsilon}}(e)\in U_{\epsilon}^{i+p}$, for all $p\ge0$.  It follows that $\QQ_{+}(e)$ is exactly the finite orbit of $\epsilon$.  In a similar way, define $\QQ_{-}(e)$ and prove that it is exactly the $f$-cycle of a negative end.
 \end{proof}

Finally, we turn to the process of ``paring off'' the nonperiodic ends of $L$ to reduce the ``interesting'' dynamics of $f$ to its action on an $f$-invariant subsurface $L'\ss L$ with finite endset, the Handel-Miller situation. 

\begin{prop}\label{soulexists}

If $f:L\to L$ is a normal endperiodic automorphism, then there exists a subsurface $L'\ss L$ with finite endset such that,

\begin{enumerate}

\item  $f|L':L'\to L'$ is endperiodic\upn{;}\label{item1}

\item $\fr L'$ in $L$ has only compact components\upn{;}\label{item2}

\item The components of $L\sm L'$ are neighborhoods of all the escaping ends of $L$.\label{item3}

\end{enumerate}

\end{prop}

\begin{proof}
Let $e_{0}$ be a negative end and consider the cycle 
$$c=\{e_{0},f(e_{0}),f^{2}(e_{0}),\dots,f^{p_{e_{0}}-1}(e_{0})\}.$$
of negative ends. Let $U_{c}$ be an $f$-neighborhood   of $e_{0}$ and let $X_{c} = U_{c}\sm\intr f^{p_{e_{0}}}(U_{c})$. Let $\EE_{c} \ss\EE(L)$ consists of the ends of $L$ contained in $X_{c}$.

For each $e\in\EE_{c}$ choose an open neighborhood $V_{e}$ of $e$ such that $\{f^{n}(\ol V_{e})\}_{n\in\Z}$ escapes, $\bd V_{c}$ consists of one simple closed geodesic, and $\ol V_{c}\ss\intr X_{c}$. The set $\{V_{e}\ |\ e\in\EE_{c}\}$ is an open cover of the compact set $\EE_{c}$ so there exists a finite subcover $\{V_{e_{1}},\ldots,V_{e_{k}}\}$. Then $O_{c} = \bigcup_{i=1}^{k}V_{e_{i}}$ has finitely many piecewise geodesic boundary components $\gamma_{1},\ldots,\gamma_{\ell}$, the set $\{f^{n}(\ol O_{c})\}_{n\in\Z}$ escapes, and $\ol O_{c}\ss\intr X_{c}$. 

Construct such an $O_{c}$ for every cycle $c$ of negative ends. Then $L' = L\sm \bigcup_{c}O_{c}$, where the union is over all cycles of negative ends, is a standard surface, satisfying (\ref{item1}),  (\ref{item2}), and  (\ref{item3}) of the proposition and thus having no escaping ends. By Propsition~\ref{finm}, $L$ has finitely many periodic ends. Thus, $L'$ has finite endset and the proposition is proven.
\end{proof}

\begin{defn}[soul]\label{soul}
 The $f$-invariant, type~$0$ subsurface $L'\ss L$ will be called the \emph{soul} of $L$.
\end{defn}

 Of  course, the soul is not unique.

\vs

\centering{\textbf{Glossary of Symbols}}

\begin{description}

\item[$\bd$, $\fr$] Page~\pageref{bdfr}
\item[$\delta$] Definition~\ref{bordcomp}
\item[$E$] Definition~\ref{idbd}
\item[$\EE(L)$] Page~\pageref{endsetL}
\item[$\EE_{\pm}(L)$] Definition~\ref{perends}
\item[$\EE^{\alpha}(L)$, $\EE^{\Omega}(L)$] Page~\pageref{endnote}
\item[$\fr$] Page~\pageref{bdfr}
\item [$\gamma^{\g}$] Definition~\ref{geotightpg}
\item[$\iota$] Definitions~\ref{gtm},~\ref{jnctfam}
\item[$\iota'$] Definition~\ref{gtm10}
\item[$\ddot\iota$] Definition~\ref{intcompl}
\item[$J_{\kappa}$] Definition~\ref{kapJunct}, Lemma~\ref{descends}
\item[$J_{n}$] Definition~\ref{Jndef},~\ref{Jndef10}
\item[$\JJ$, $\JJ_{+}$, $\JJ_{-}$] Definitions~\ref{junctdefn},~\ref{jnctfam}
\item[$\JJ_{W}$] Definition~\ref{jw}
\item[$K$, $K_{i}$] Definition~\ref{coredef}
\item[$\Lambda_{F}$] Lemma~\ref{descends}
\item[$\NN$, $\NN_{+}$, $\NN_{-}$]\ Definition~\ref{famNNg}
\item[$p_{e}$] Definition~\ref{perend}
\item[$\QQ^{+}$] Definition~\ref{QQ+}
\item[$\S$] Sections~\ref{redcurv},~\ref{redcircs},~\ref{redcrv2}, and Definition~\ref{famredc}
\item[$\Se$] Definition~\ref{se}
\item[$\Si$, $\Delta$, $\D^{2}$] Page~\pageref{s1inf}
\item[$\T$] Proposition~\ref{T'}
\item[$\mathfrak T^{\g}$, $\mathfrak T^{\g}_{*}$] Definition~\ref{defntile}
\item[$U_{e}$] Definition~\ref{hnbhd}
\item[$U_{e}^{n}$] Definition~\ref{disnbhdsys}
\item[$\UU_{e}$, $\UU_{\pm}$] Definitions~\ref{pmesc}
\item[$\UU$] Definition~\ref{espset}
\item[$\ddot U$] Definition~\ref{intcompl}
\item[$W^{\pm}$, $W_{i}^{\pm}$] Definition~\ref{coredef}
\item[$W$] Definition~\ref{jw}
\item[$\XX_{\pm}$] Definition~\ref{nonesc}
\item[$\YY'_{\pm}$] Definition~\ref{yyprime}

\end{description}

\centering{\textbf{Glossary of Terms}}

\begin{description}

\item[admissible]  Definition~\ref{admi}
\item[arm] Definition~\ref{defnarm}
\item[assemblage] Definition~\ref{assemblg}
\item[attracting/repelling ends]  Page~\ref{attrrep}
\item[bilamination] Definition~\ref{bilam}
\item[bilaminated chart] Definition~\ref{bichart}
\item[border component] Definition~\ref{bordcomp}
\item[border component of the first/second kind] Section~\ref{fskind}
\item[completely crosses] Definition~\ref{completelycrosses}
\item[converges strongly] Definition~\ref{convstrly}
\item[core and $i^{\thh}$ core] Definition~\ref{coredef}
\item[core dynamical system] Definition~\ref{CoreDy}
\item[depth one foliation] Definition~\ref{dpthonefoln}
\item[depth $r$] Definition~\ref{dpthr}
\item[distinguished neighborhood] Definition~\ref{hnbhd}
\item[dual crown sets] Definition~\ref{defndualSC}
\item[dual principal regions] Definition~\ref{defndualPR}
\item[end, endset] Page~\pageref{endendset}
\item[endperiodic automorphism] Definition~\ref{epdefn}
\item[endpoint correspondence property] Definition~\ref{epcorrp}
\item[escapes (juncture)] Definition~\ref{escapes}
\item[escapes (sequence of sets)] Definition~\ref{seqesc}
\item[escaping end of leaf] Definition~\ref{defnescpend}
\item[escaping end of surface] Definition~\ref{escpend}
\item[escaping set] Definition~\ref{espset}
\item[escaping set (positive/negative)] Definitions~\ref{pmesc}
\item[essentially intersecting loops] Definition~\ref{essintlps}
\item[extreme rectangle] Definition~\ref{extquad}
\item[$f$-domain] Definition~\ref{fdom}
\item[$f$-juncture] Definition~\ref{fjunct}, Hypothesis~\ref{hypjunct}
\item[$f$-junctures (fixed set of)] Definition~\ref{famNNg}
\item[$f$-neighborhood of an end] Definition~\ref{fnbhd}
\item[first/second kind] Definition~\ref{fskind}
\item[foliated form] Definition~\ref{ff}
\item[frontier] Page~\pageref{bdfr}
\item[geodesic lamination] Definition~\ref{geolam}
\item[geodesic tightening of a curve] Definition~\ref{geodtite},~\ref{geotightpg}
\item[geodesic tightening map] Definitions~\ref{gtm},~\ref{jnctfam},~\ref{gtm10}
\item[Handel-Miller (geodesic) bilaminations] Definitions~\ref{HMgeobilamin}
\item[Handel-Miller pseudo-geodesic bilaminations] Definitions~\ref{HMbilam}
\item[Handel-Miller foliation cone] Definition~\ref{HMfolncone}
\item[Handel-Miller monodromy] Definition~\ref{HMmonofoln}
\item[Hypothesis~\ref{hyp1}] page~\pageref{hyp1}
\item[Hypothesis~\ref{hypjunct}] page~\pageref{hypjunct}
\item[Hypothesis~\ref{noperpt}] page~\pageref{noperpt}
\item[Hypothesis~\ref{hypstan}] page~\pageref{hypstan}
\item[Hypothesis~\ref{hypadmiss}] page~\pageref{hypadmiss}
\item[Hypothesis~\ref{hyp5}] page~\pageref{hyp5}
\item[Hypothesis~\ref{ongo}] page~\pageref{ongo}
\item[ideal boundary] Definition~\ref{idbd}
\item[internal completion] Definition~\ref{intcompl}
\item[invariant set] Definition~\ref{invset}
\item[juncture] Definitions~\ref{junctdefn},~\ref{jnctfam}
\item[juncture intersection property] Definition~\ref{intprop}
\item[$\kappa$-juncture] Definition~\ref{kapJunct}
\item[lamination] Definition~\ref{lamination}
\item[laminated atlas, laminated partial atlas] Definition~\ref{lamatlas}
\item[laminated chart] Definition~\ref{lamchart}
\item[locally uniform accumulation] Definition~\ref{locunif}
\item[Markov chain] Definition~\ref{markchn}
\item[Markov family, pre-Markov family] Definition~\ref{marfam}
\item[mesh] Definition~\ref{gridmesh}
\item[monodromy] Definition~\ref{monofoln}
\item[neighborhood of an end] Page~\pageref{nhbend}
\item[normal endperiodic automorphism] Definition~\ref{normal}
\item[nucleus] Definition~\ref{defnnucleus}
\item[parallel packets] Definition~\ref{parpackets}
\item[passes arbitrarily near] Definition~\ref{arbitrarilynear}
\item[peripheral border components] Definition~\ref{perif}
\item[positive/negative ends] Definition~\ref{perends}
\item[principal regions (positive/negative)] Definition~\ref{pnprinreg}
\item[pseudo-anosov] Definition~\ref{psdansv}
\item[pseudo-geodesic] Definition~\ref{pseudogeodesic}
\item[pseudo-geodesic lamination] Definition~\ref{psgeolam}
\item[quadrilateral and geodesic quadrilateral] Definition~\ref{quadltrl}
\item[rectangle, rectangular] Definition~\ref{rctnglr}
\item[reducing curves] Section~\ref{redcurv}
\item[relative depth one] Definition~\ref{reldpthone}
\item[rim]\ Definition~\ref{prreduce}
\item[semi-isolated]\ Definition~\ref{defnsemiis}
\item[simple end] Definition~\ref{simpend}
\item[sliding isotopy] Definition~\ref{slidingisotopy}
\item[smooth lamination, bilamination] Definition~\ref{Cilam}
\item[soul] Definition~\ref{soul}
\item[strongly closed property] Definition~\ref{strnglyclsed}
\item[standard metric, standard surface] Definition~\ref{standard}
\item[stump] Definition~\ref{stump}
\item[taut foliation] Definition~\ref{tautfoln}
\item[tile, tiling] Definition~\ref{defntile}
\item[transfer of a path] Definition~\ref{transfpath}
\item[transfer of parametrized paths]  page~\pageref{transfpath}
\item[transfer of a set of curves] Definition~\ref{trfamcur}
\item[translation] Definition~\ref{U'=L}
\item[transversely totally disconnected] Definition~\ref{trtotdis} 
\item[type of end] Definition~\ref{typeend}
\item[type of surface] Definition~\ref{typek}
\item[vertex/edge] Lemma~\ref{veredg}
\item[virtually escapes] Definition~\ref{vescapes}
\item[weakly groomed] Definition~\ref{groomed}
\end{description}

\centering{\textbf{Glossary of Axioms, Main Results, and Main Theorems}}
\begin{description}

\item[Axioms] Pages~\pageref{muttran} - \pageref{trnsvrs}
\item[constructing the geodesic bilaminations $(\Lambda_{+},\Lambda_{-})$] Section~\ref{HMconstruct}
\item[core dynamics conjugate Markov shift] Theorem~\ref{basicmar}
\item[defining $h$] Theorem~\ref{geodext}
\item[defining the reducing curves] Sections~\ref{redcurv},~\ref{redcircs}, and~\ref{redcrv2}
\item[Epstein-Baer Theorems] Theorem~\ref{2.1} and~~\ref{3.1}
\item[escaping ends of leaves] Section~\ref{ee}
\item[escapes implies virtually escapes] Theorem~\ref{essc} (see also Theorem~\ref{esctoe})
\item[finitely many semi-isolated leaves] Theorem~\ref{finmany}
\item[Isotopy Theorem] Theorem~\ref{isotlams}
\item[periodic points and leaves] Section~\ref{periodicleaves}
\item[Smoothing Theorem] Theorem~\ref{HMsmooth}
\item[Transfer Theorem] Theorem~\ref{transfer}
\item[Uniqueness Theorem] Theorem~\ref{uniqlams}

\end{description}

\providecommand{\bysame}{\leavevmode\hbox to3em{\hrulefill}\thinspace}
\providecommand{\MR}{\relax\ifhmode\unskip\space\fi MR }
\providecommand{\MRhref}[2]{%
  \href{http://www.ams.org/mathscinet-getitem?mr=#1}{#2}
}
\providecommand{\href}[2]{#2}


\begin{thebibliography}{10}

\bibitem{alt:depth1}
I.~Altman, \emph{The sutured floer polytope and taut depth-one foliations},
  Algebraic and Geometric Topology \textbf{{\bf 14}} (2014), 1881--1923.

\bibitem{calegari}
D.~Calegari, \emph{Foliations and the {G}eometry of 4-{M}anifolds}, Oxford
  Univ. Press, Oxford, 2007.

\bibitem{cand:hyp}
A.~Candel, \emph{Uniformization of surface laminations}, Ann. Sci. \'Ecole
  Norm. Sup. \textbf{{\bf 26}} (1993), 489--516.

\bibitem{condel1}
A.~Candel and L.~Conlon, \emph{Foliations, {I}}, American Mathematical Society,
  Providence, Rhode Island, 1999.

\bibitem{condel2}
\bysame, \emph{Foliations, {II}}, American Mathematical Society, Providence,
  Rhode Island, 2003.

\bibitem{cc:lp}
J.~Cantwell and L.~Conlon,  \emph{{L}eaf prescriptions for closed $3$-manifolds}, Trans. Amer. Math. Soc. \textbf{{\bf 236}} (1978), 239--261.

\bibitem{cc:pb}
\bysame,  \emph{{P}oincar\'{e}--{B}endixson theory for leaves of codimension
  one}, Trans. Amer. Math. Soc. \textbf{{\bf 265}} (1981), 181--209.

\bibitem{cc:tisch}
\bysame, \emph{{T}ischler fibrations of open foliated sets}, Ann. Inst.
  Fourier, Grenoble \textbf{{\bf 31}} (1981), 113--135.

\bibitem{cc:depth} 
\bysame, \emph{Depth of knots}, Topology and its Applications 
  {\bf 42} (1991), 277--289.

\bibitem{cc:isotopy}
\bysame, \emph{Isotopy of depth one foliations}, Proceedings of the
  {I}nternational {S}ymposium and {W}orkshop on Geometric Study of Foliations,
  Tokyo, World Scientific, November 1993, pp.~153--173.

\bibitem{cc:smth2}
\bysame, \emph{Topological obstructions to smoothing proper foliations},
  Contemporary Mathematics \textbf{{\bf 161}} (1994), 1--20.

\bibitem{cc:cone}
\bysame, \emph{Foliation cones}, Geometry and Topology Monographs, Proceedings
  of the Kirbyfest, vol.~{\bf 2}, 1999, pp.~35--86.

\bibitem{cc:LB}
\bysame, \emph{Isotopies of foliated $3$--manifolds without holonomy}, Adv. in
  Math. \textbf{{\bf 144}} (1999), 13--49.

\bibitem{cc:dum}
\bysame, \emph{Endsets of exceptional leaves; a theorem of {G}. {D}uminy},
  Proceedings of the {E}uroworkshop on Foliations, Geometry and Dynamics, World
  Scientific, 2002, pp.~225--261.

\bibitem{cc:epstein}
\bysame, \emph{Hyperbolic geometry and homotopic homeomorphisms of surfaces},
  Geom. Dedicata \textbf{177} (2015), 27--42.
  
\bibitem{cc:almostnohol}
\bysame, \emph{Cones of foliations almost without holonomy},
 Geometry, Dynamics,  and Foliations 2013, Advanced Studies in Pure
  Mathematics   \textbf{{\bf 72}} (2017), 301--348.

\bibitem{cc:withouthol}
\bysame, \emph{Open saturated sets without holonomy}, 
 arXiv:1108.0714v3.

 \bibitem{cc:examples}
 \bysame,  {Examples of endperiodic automorphisms},
 arXiv:1008.2549v2.

\bibitem{bca}
A.~J. Casson and S.~A. Bleiler, \emph{Automorphisms of surfaces after {N}ielsen
  and {T}hurston}, Cambridge Univ. Press, Cambridge, 1988.

\bibitem{Epstein:isotopy}
D.~B.~A. Epstein, \emph{Curves on $2$-manifolds and isotopies}, Acta Math.
  \textbf{{\bf 115}} (1966), 83--107.

\bibitem{FLP}
A.~Fathi, F.~Laudenbach, and V.~Po\'enaru, \emph{Travaux de {T}hurston sur les
  {S}urfaces {(}2de \'ed.{)}}, Ast\'erisque \textbf{{\bf 66--67}} (1991).

\bibitem{fe:thesis}
S.~Fenley, \emph{Depth one foliations in hyperbolic $3$-manifolds}, {\it
  Thesis, Princeton University, 1990}.

\bibitem{fe:jdg}
\bysame, \emph{Asymptotic properties of depth one foliations in hyperbolic $3$-manifolds}, J. Diff. Geom.
  \textbf{{\bf 36}} (1992), 269–313.

\bibitem{fe:endp}
\bysame, \emph{Endperiodic surface homeomorphisms and $3$-manifolds}, Math. Z.
  \textbf{{\bf 224}} (1997), 1--24.

\bibitem{fried}
D.~Fried, \emph{Fibrations over ${S}^{1}$ with pseudo--{A}nosov monodromy},
  Ast\'erisque \textbf{{\bf 66--67}} (1991), 251--266.

\bibitem{ga1}
D.~Gabai, \emph{Foliations and the topology of $3$--manifolds}, J. Diff. Geo.
  \textbf{{\bf 18}} (1983), 445--503.

\bibitem{ga0}
\bysame, \emph{Foliations and genera of links}, Topology \textbf{{\bf 23}}
  (1984), 381--394.

\bibitem{ga3}
\bysame, \emph{Foliations and the topology of $3$--manifolds {III}}, J. Diff.
  Geo. \textbf{{\bf 26}} (1987), 479--536.

\bibitem{ha:th}
M.~Handel and W.~Thurston, \emph{New proofs of some results of {N}ielson}, Adv.
  in Math. \textbf{{\bf 56}} (1985), 173--191.

\bibitem{ha:smooth}
A.~Hatcher, \emph{The {K}irby torus trick for surfaces},\hfill\break
  http://www.math.cornell.edu/$\sim$hatcher/Papers/TorusTrick.pdf.

\bibitem{juh:poly}
A.~Juh\'asz, \emph{The sutured {F}loer homology polytope}, Geom. Topol.
  \textbf{{\bf 14}} (2010), 1303--1354.

\bibitem{LB}
F.~Laudenbach and S.~Blank, \emph{Isotopie de formes ferm\'ees en dimension
  trois}, Inv. Math. \textbf{{\bf 54}} (1979), 103--177.

\bibitem{mill}
R.~T. Miller, \emph{Geodesic laminations from {N}ielsen's viewpoint}, Adv. in
  Math. \textbf{{\bf 45}} (1982), 189--212.

\bibitem{ms}
J.~W. Morgan and P.~B. Shalen, \emph{Degenerations of hyperbolic structures,
  ii: Measured laminations in $3$-manifolds}, Ann. Math. \textbf{127} (1988),
  403--456.

\bibitem{nov}
S.~P. Novikov, \emph{Topology of foliations}, Trans. Moscow Math. Soc.
  \textbf{{\bf 14}} (1965), 268--305.

\bibitem{pierce}
R.~S. Pierce, \emph{Existence and uniqueness theorems for extensions of
  zero--dimensional compact metric spaces}, Trans. Amer. Math. Soc.
  \textbf{{\bf 148}} (1970), 1--21.

\bibitem{rolf}
D.~Rolfsen, \emph{Knots and {L}inks}, Publish or Perish, Inc., Berkeley, CA,
  1976.

\bibitem{sa:pseudo}
R.~Sacksteder, \emph{Foliations and pseudogroups}, Amer. J. Math. \textbf{{\bf
  87}} (1965), 79--102.

\bibitem{sch_cycles}
S.~Schwartzmann, \emph{Asymptotic cycles}, Ann. of Math. \textbf{{\bf 66}}
  (1957), 270--284.

\bibitem{sch:seifert}
P.~Schweitzer, \emph{Counterexamples to the {S}eifert conjecture and opening
  closed leaves of foliations}, Ann. of Math. \textbf{{\bf 100}} (1974),
  386--400.

\bibitem{sull:cycles}
D.~Sullivan, \emph{Cycles for the dynamical study of foliated manifolds and
  complex manifolds}, Inv. Math. \textbf{{\bf 36}} (1976), 225--255.

\bibitem{th:norm}
W.~Thurston, \emph{A norm on the homology of three--manifolds}, Mem. Amer.
  Math. Soc. \textbf{{\bf 59}} (1986), 99--130.

\bibitem{th:surfaces}
\bysame, \emph{On the geometry and dynamics of diffeomorphisms of surfaces},
  Bull. Amer. Math. Soc. \textbf{\bf 19} (1988), 417--431.

\bibitem{whjhc}
J.~H.~C. Whitehead, \emph{A certain open manifold whose group is unity}, Quart.
  J. Math. \textbf{\bf 6} (1936), 268--279.

\end{thebibliography}
\end{document}